\documentclass{compositio}

\renewcommand{\contentsname}{Contents\\{\footnotesize\normalfont(A table
of contents should normally not be included)}}
\usepackage[english]{babel}
\usepackage{newlfont}
\usepackage{amssymb,amsfonts}
\usepackage{mathtools,braket,mathrsfs}
\usepackage[all]{xy}
\usepackage{booktabs}
\usepackage{stmaryrd}
\usepackage{bbm}
\usepackage{hyperref}

\usepackage[colorinlistoftodos,bordercolor=black,
backgroundcolor=orange!70!white,linecolor=black,]{todonotes}

\newcommand{\bb}{\mathbb}
\newcommand{\bs}{\boldsymbol}
\newcommand{\mbf}{\mathbf}
\newcommand{\msf}{\mathsf}
\newcommand{\scr}{\mathscr}
\newcommand{\mrm}{\mathrm}
\newcommand*{\bfcdot}{\scalebox{0.6}{$\bullet$}}

\newcommand{\et}{{\acute{\mathrm{e}}\mathrm{t}}} 

\let\det\relax
\DeclareMathOperator{\det}{det} 

\DeclareFontFamily{U}{wncy}{}
\DeclareFontShape{U}{wncy}{m}{n}{<->wncyr10}{}
\DeclareSymbolFont{mcy}{U}{wncy}{m}{n}
\DeclareMathSymbol{\Sh}{\mathord}{mcy}{"58} 

\DeclareSymbolFont{extraup}{U}{zavm}{m}{n}
\DeclareMathSymbol{\varheart}{\mathalpha}{extraup}{86}

\theoremstyle{definition}
\newtheorem{theorem}{Theorem}[section]
\newtheorem{lemma}[theorem]{Lemma}
\newtheorem{proposition}[theorem]{Proposition}
\newtheorem{corollary}[theorem]{Corollary}
\newtheorem{definition}[theorem]{Definition}
\newtheorem{conjecture}[theorem]{Conjecture}

\newtheorem{thmx}{Theorem}

\theoremstyle{remark}
\newtheorem{remark}[theorem]{Remark}

\begin{document}

\author{Michele Fornea}
\email{mfornea@math.columbia.edu}
\address{Columbia University, New York, USA.}
\author{Zhaorong Jin}
\email{zhaorong@math.princeton.edu}
\address{Princeton University, Princeton, USA.}
\classification{11F33, 11F41, 11F67, 11F80, 11G35, 11G40.} 

\title{Hirzebruch--Zagier classes and rational elliptic curves over quintic fields}

\begin{abstract}

Conditionally on a conjecture on the \'etale cohomology of Hilbert modular surfaces and some minor technical assumptions, we establish new instances of the equivariant BSD-conjecture in rank $0$ with applications to the arithmetic of rational elliptic curves over quintic fields. The key ingredients are a refinement of twisted triple product \emph{$p$-adic $L$-functions}, the construction of a compatible collection of \emph{Hirzebruch--Zagier cycles} and an \emph{explicit reciprocity law} relating the two. 
\end{abstract}
\maketitle

\vspace*{6pt}\tableofcontents  

\section{Introduction}
The most general result towards the BSD-conjecture was established by Shouwu Zhang and his school \cite{Heights} as a major generalization of the methods of Gross-Zagier \cite{GZformula} and Kolyvagin \cite{Koly}. The result states that if $E_{/F}$ is a modular elliptic curve over a totally real field $F$ such that either $E_{/F}$ has at least one prime of multiplicative reduction or $[F:\bb{Q}]$ is odd, then 
\[
 r_\mrm{an}(E/F)\in\{0,1\}\implies r_\mrm{an}(E/F)=r_\mrm{alg}(E/F).
\]
It is important not to forget that the modularity of $E_{/F}$ is currently the only known way to access the analytic properties of the $L$-function $L(E/F,s)$. Because of this, it becomes natural to expect that cycles on Shimura varieties will play a role in any strategy to establish the BSD-conjecture. 

The three pillars of Gross-Zagier and Kolyvagin's approach are: $(i)$ the existence of a non-constant map $X_{/F}\to E_{/F}$ from a Shimura curve to the elliptic curve, $(ii)$ the existence of CM points on $X_{/F}$ with their significance for Selmer groups, and $(iii)$ formulas for the derivative of certain base-change $L$-functions of $E_{/F}$ in terms of the height of images of CM points, called Heegner points.  These three items are at the same time the strengths and the limitations of the most effective strategy developed so far to prove instances of the BSD-conjecture. Firstly, the strong form of geometric modularity in $(i)$ can only be realized for certain elliptic curves over totally real fields, hence the first pillar topples down right away when considering elliptic curves defined over general number fields. However, a lot can be proved using congruences for general rank zero elliptic curves over totally real fields, as it was shown by Longo \cite{Longo} and Nekovar  \cite{LevelRaisingNekovar}. Secondly, suppose one fixed an elliptic curve over a totally real field $F$ and took a finite extension $K/F$; what could be said about the BSD-conjecture for $E_{/K}$? In this case, even though a modular parametrization could still be available, one would lack a systematic way to produce points over extensions of a general $K$. Indeed, Heegner points are defined over certain dihedral extensions of $F$, and therefore miss all non-solvable extensions. Finally, what if one contented themselves with tackling the BSD-conjecture over totally real fields? In this case all the pillars could still be standing, but $(ii)$ and $(iii)$  would have nothing to say about higher rank situations. The striking feature of CM points is their explicit relation to \emph{first} derivative of $L$-functions; thus, as soon as the rank is greater than or equal to two, they become torsion.

In recent years the $p$-adic approach to the BSD-conjecture of Coates and Wiles \cite{CoatesWiles} has been revitalized by Bertolini, Darmon and Rotger (\cite{BDR}, \cite{DR2}). In their works the focus is to explore the arithmetic of rational elliptic curves over field extensions that are not contained in ring class fields of quadratic imaginary fields.
The present paper is part of this new line of inquiry and studies the arithmetic of rational elliptic curves over non-solvable quintic fields. 

\subsubsection{The equivariant BSD-conjecture.}
 Let $F$ be a totally real number field and $K/F$ a finite Galois extension. For any elliptic curve $E_{/F}$, the Galois group $G(K/F)$ naturally acts on the $\bb{C}$-vector space $E(K)\otimes\bb{C}$ generated by the group of $K$-rational points. Since complex representations of finite groups are semisimple, the representation $E(K)\otimes\bb{C}$ decomposes into a direct sum of $\varrho$-isotypic components $E(K)^\varrho=\mrm{Hom}_{G(K/F)}(\varrho,E(K)\otimes\bb{C})$, indexed by irreducible representations $\varrho\in\mrm{Irr}\big( G(K/F)\big)$, each with its multiplicity. It is natural to define the algebraic rank of $E$ with respect to some $\varrho$ as
\[
r_\mrm{alg}(E,\varrho):=\dim_\bb{C}E(K)^\varrho.
\]  
On the analytic side, for any $\varrho\in\mrm{Irr}\big( G(K/F)\big)$ one can define a twisted $L$-function $L(E,\varrho,s)$ as the $L$-function associated to the Galois representation $\varrho\otimes\mrm{V}_p(E)$ of the absolute Galois group of $F$. If $L(E,\varrho,s)$ admitted meromorphic continuation to the whole complex plane, then the analytic rank of $E$ with respect to some $\varrho$ could be defined as  
\[
r_\mrm{an}(E,\varrho):=\mrm{ord}_{s=1}L(E,\varrho,s).
\] 
The Artin formalism of $L$-functions can be used to show that the BSD-conjecture for an elliptic curve $E_{/F}$ base-changed to $K$ should be equivalent to the equality of ranks
\[
r_\mrm{an}(E,\varrho)\overset{?}{=}r_\mrm{alg}(E,\varrho)\quad \text{for all}\quad \varrho\in\mrm{Irr}\big( G(K/F)\big).
\]
 The advantage of this point of view resides in the fact that it splits the BSD-conjecture into more manageable pieces. Specifically, when the considered representation $\varrho$ arises from an automorphic form, the right framework to be explored becomes apparent. For example, Bertolini, Darmon and Rotger \cite{BDR} proved new instances of the conjecture in rank $0$ for rational elliptic curves in the case of $\varrho$ an odd, irreducible two dimensional Artin representation. By modularity, both $\varrho$ and $E_{/\bb{Q}}$ correspond to some automorphic representation of $\mrm{GL}_{2,\bb{Q}}$, thus it should not come as a total surprise that the main theorem of \cite{BDR} is obtained by a careful analysis of elements of higher Chow groups of a product of modular curves. Another noteworthy result was obtained by Darmon and Rotger \cite{DR2} when $\varrho$ is the tensor product of two odd, irreducible two dimensional Artin representations. In this case, generalized Kato classes -- constructed from diagonal cycles on triple products of modular curves --  are used to establish the first cases of the BSD-conjecture in rank $0$ for rational elliptic curves over $A_5$-quintic extensions of $\bb{Q}$. 
 
 Bhargava \cite{Bar5} showed that $100\%$ of quintic fields have Galois group isomorphic to $S_5$. To access some of these quintic fields, in this paper we are interested in the case of $\varrho$  the tensor induction of a totally odd, irreducible two dimensional Artin representation of the absolute Galois group of a real quadratic field. Conditionally on a conjecture on the \'etale cohomology of Hilbert modular surfaces (\ref{wishingOhta1}) and some technical assumptions, we prove new instances of the equivariant BSD-conjecture in rank $0$ by analyzing Hirzebruch--Zagier (HZ) cycles on a product of a Hilbert modular surface and a modular curve.

\subsection{Main results}
Let $L$ be a real quadratic field and  $\varrho:\Gamma_L\to\mrm{GL}_2(\bb{C})$ a totally odd, irreducible two-dimensional Artin representation  of the absolute Galois group of $L$. The Asai representation
\[
\mrm{As}(\varrho)=\otimes\mbox{-}\mrm{Ind}_L^\bb{Q}(\varrho)
\] 
is a $4$-dimensional complex representation obtained as the tensor induction of $\varrho$ from $\Gamma_L$ to $\Gamma_\bb{Q}$.
 We suppose that $\varrho$ has conductor $\mathfrak{Q}$ and that the tensor induction of the determinant $\det(\varrho)$ is the trivial character so that $\mrm{As}(\varrho)$ is self-dual. 
 For any rational elliptic curve $E_{/\bb{Q}}$ of conductor $N$ prime to $\frak{Q}$, we are interested in understanding when 
\[
r_\mrm{an}(E,\mrm{As}(\varrho))\overset{?}{=}r_\mrm{alg}(E,\mrm{As}(\varrho)).
\]
We rely on the modularity of totally odd Artin representations and rational elliptic curves (\cite{W}, \cite{TW}, \cite{PS}) to establish the meromorphic continuation, functional equation and analyticity at the center of the twisted triple product $L$-function $L\big(E,\mrm{As}(\varrho),s\big)$.

\begin{thmx}
Suppose that $N$ is coprime to $\mathfrak{Q}$, split in $L$, and there exists an ordinary prime $p\nmid 2N\cdot\frak{Q}$ for $E_{/\bb{Q}}$ such that  
	\begin{itemize}
		\item[($1$)] $p$ splits in $L$ with narrowly principal factors;
		\item[($2$)] there is no totally positive unit in $L$ congruent to $-1$ modulo $p$;
		\item[($3$)] the eigenvalues of $\mrm{Fr}_p$ on $\mrm{As}(\varrho)$ are all distinct modulo $p$.
	\end{itemize}
If, additionally, $\varrho$ is residually not solvable and Conjecture \ref{wishingOhta1} holds, then 
\[
r_\mrm{an}\big(E,\mrm{As}(\varrho)\big)=0\quad\implies\quad r_\mrm{alg}\big(E,\mrm{As}(\varrho)\big)=0.
\]
\end{thmx}

\begin{remark}
Conditions ($1$),($2$),($3$) on the auxiliary ordinary prime $p$ are minor technical assumptions and they can always be satisfied in our applications (Proposition \ref{choiceofp}).

\noindent We introduce ($1$) in Section \ref{section AJ p-adic}  to relate the action of certain Hecke correspondences on different Shimura varieties, while ($2$) appears in Proposition \ref{prop comparison different models} to compare Hilbert modular surfaces with different level structures. Furthermore, we use ($3$) in Proposition \ref{somekindoffil} to obtain a Galois stable filtration in a projective limit of \'etale cohomology groups, and in Theorem \ref{Main Theorem} to ensure that the four global cohomology classes we constructed are linearly independent.
	 
\noindent Regarding the remaining assumptions, we require $\varrho$ to be residually not solvable in Section \ref{geomrealiz} to apply recent results of Caraiani--Tamiozzo \cite{Caraiani-Tamiozzo}. This assumption is satisfied  in our applications because we consider Artin representations with projective image isomorphic to $A_5$. 
The final hypothesis, Conjecture \ref{wishingOhta1} which appears in Section \ref{motivic p-adic L-function}, is used to deduce an Eichler--Shimura morphism for Hida families of Hilbert modular forms.  We refer to Section \ref{ontheconjectures} for a discussion of Conjecture \ref{wishingOhta1}, but we note here that Sangiovanni--Skinner announced a proof of our conjecture  as part of their new construction of Euler Systems (\cite{SangiovanniSkinner}).
\end{remark}

\noindent We apply Theorem A to Artin representations constructed in \cite{MicAnalytic}. The outcome is a result concerning the arithmetic of rational elliptic curves over $S_5$-quintic fields.

\begin{corollary}\label{CorolQuintic}
	Let $K/\bb{Q}$ be a non-totally real $S_5$-quintic extension whose Galois closure contains a real quadratic field $L$. Suppose that $N$ is odd, unramified in $K/\bb{Q}$ and split in $L$, and that Conjecture \ref{wishingOhta1} holds, then 
		\[
		r_\mrm{an}(E/K)=r_\mrm{an}(E/\bb{Q})\quad\implies\quad r_\mrm{alg}(E/K)=r_\mrm{alg}(E/\bb{Q}).
		\]
\end{corollary}
\begin{remark}
 The collection of $S_5$-quintic number fields satisfying the hypothesis of Corollary \ref{CorolQuintic} for a fixed elliptic curve of odd conductor $N$ has positive proportion in all quintic fields by (\cite{BSW}, Theorem 2).
\end{remark}
\noindent The strategy of proof of these results consists in producing enough global cohomology classes functioning as annihilators, and whose non-triviality is controlled by an automorphic $L$-value. As one expects cycles on Shimura varieties to play a prominent role in any plan to establish cases of the BSD-conjecture, the \'etale Abel--Jacobi map becomes a pivotal tool to convert null-homologous cycle classes into Selmer classes. However, the representation 
\[
\mrm{V}_{\varrho,E}:=\mrm{As}(\varrho)\otimes\mrm{V}_p(E)
\] 
is not known to appear in the \'etale cohomology of a Shimura variety and even if it was, we are looking for annihilators of Mordell-Weil groups, not Selmer classes. To our rescue comes the idea of \emph{$p$-adic deformation}: Corollary \ref{correctspec} allows us to realize $\mrm{V}_{\varrho,E}$ as the $p$-adic limit of Galois representations appearing in the \'etale cohomology of Shimura threefolds. Thus, we can obtain the sought-after annihilators as limits of Abel--Jacobi images of HZ-cycles, that need not remain Selmer at $p$.
The compatible collection of HZ-cycles also gives rise to a \emph{motivic} $p$-adic $L$-function via Perrin-Riou's machinery. Conjecture \ref{wishingOhta1} ensures that the non-triviality of a value of that function implies the non-triviality of the annihilators. Finally, the last step of the proof's strategy entails an explicit reciprocity law (Theorem \ref{comparison aut-mot}) comparing the motivic with the \emph{automorphic} $p$-adic $L$-function, the latter retaining information about the automorphic $L$-value.   

In the remaining of the introduction we present in more detail the main steps of the proof.

\subsection{Overview of the proof}
Let $\msf{g}_\varrho$ be the Hilbert cuspform of parallel weight one associated to the Artin representation $\varrho:\Gamma_L\to\mrm{GL}_2(\bb{C})$, and $\msf{f}_E$ the elliptic cuspform associated to $E_{/\bb{Q}}$. Choose a rational prime $p$ ordinary for $\msf{g}_\varrho$ and $\msf{f}_E$, and a valuation ring $O\subseteq \overline{\bb{Q}}_p$ finite flat over $\bb{Z}_p$ containing all the Hecke eigenvalues of $\msf{g}_\varrho$. Then, there exist Hida families $\scr{G}, \scr{F}$ -- over $\bs{\cal{W}}_{\scr{G}}=\mrm{Spf}(\mbf{I}_\scr{G})^\mrm{rig}$ and $\bs{\cal{W}}_{\scr{F}}=\mrm{Spf}(\mbf{I}_\scr{F})^\mrm{rig}$ respectively -- passing through a choice of ordinary $p$-stabilizations $\msf{g}_\varrho^{\mbox{\tiny $(p)$}}$, $\msf{f}_E^{\mbox{\tiny $(p)$}}$. These families are equipped with big Galois representations of the absolute Galois group of $\bb{Q}$,
$\mrm{As}(\mbf{V}_\scr{G})$ of rank $4$ and $\mbf{V}_\scr{F}$ of rank $2$, each interpolating the representations of the eigenforms in the families. Specifically, there are arithmetic points $\mrm{P}_\circ\in\bs{\cal{W}}_\scr{G}$, $\mrm{Q}_\circ\in\bs{\cal{W}}_\scr{F}$ such that
\[
\mrm{As}(\mbf{V}_{\scr{G}_{\mrm{P}_\circ}})\cong\mrm{As}(\varrho)\qquad \text{and}\qquad\mbf{V}_{\scr{F}_{\mrm{Q}_\circ}}\cong\mrm{V}_p(E).
\]
One can show that there exists a twist $\mbf{V}_{\scr{G},\scr{F}}^\dagger$ of
\[
\mbf{V}_{\scr{G},\scr{F}}:=\mrm{As}(\mbf{V}_{\scr{G}})(-1)\otimes \mbf{V}_{\scr{F}}
\]
interpolating Kummer self-dual representations: for any pair of arithmetic points $\mrm{P}\in \bs{\cal{W}}_\scr{G}$ and $\mrm{Q}\in \bs{\cal{W}}_\scr{F}$ the specialization
\[
\mbf{V}_{\scr{G}_{\mrm{P}},\scr{F}_{\mrm{Q}}}^\dagger=\Big( \mrm{As}(\mbf{V}_{\scr{G}_{\mrm{P}}})(-1)\otimes \mbf{V}_{\scr{F}_{\mrm{Q}}}\Big)^\dagger
\] 
is the Kummer self-dual twist of the $8$-dimensional Galois representation attached to $\scr{G}_\mrm{P}$ and $\scr{F}_\mrm{Q}$. 
In particular, the specialization at the arithmetic points $\mrm{P}_\circ\in\bs{\cal{W}}_\scr{G}$, $\mrm{Q}_\circ\in\bs{\cal{W}}_\scr{F}$ is
\[
\mbf{V}_{\scr{G}_{\mrm{P_\circ}},\scr{F}_{\mrm{Q}_\circ}}^\dagger= \mrm{V}_{\varrho,E}.
\] 
 When the pair $(\mrm{P},\mrm{Q})$ is $\bb{Q}$-dominated  (\cite{BlancoFornea}, Definition 1.3), Ichino's formula (\cite{I}) gives an expression for the central $L$-value
\[
L\Big(\mbf{V}_{\scr{G}_\mrm{P},\scr{F}_\mrm{Q}}^\dagger, c\Big)
\]
which is suitable for $p$-adic interpolation. It is then possible to construct a rigid meromorphic function $\scr{L}_p(\breve{\scr{G}},\scr{F}):\bs{\cal{W}}_{\scr{G},\scr{F}}\to\bb{C}_p$ whose values at crystalline $\bb{Q}$-dominated pairs satisfy
\[
\scr{L}_p(\breve{\scr{G}},\scr{F})(\mrm{P},\mrm{Q})\overset{\cdot}{\sim} L^\mrm{alg}\Big(\mbf{V}_{\scr{G}_\mrm{P},\scr{F}_\mrm{Q}}^\dagger, c\Big),
\]
and whose values at certain crystalline points outside the range of interpolation are related to the syntomic Abel--Jacobi image of \emph{generalized Hirzebruch--Zagier cycles} (\cite{BlancoFornea}, Theorem 1.7). 

\noindent In the present work we focus our attention on the one-variable Hida family $\scr{G}$ interpolating parallel weight Hilbert cuspforms. By restricting the number of variables we are able to refine the construction of $\scr{L}_p(\breve{\scr{G}},\scr{F})$ and obtain a rigid analytic function on the disk $\bs{\cal{W}}_{\scr{G}}$, the \emph{automorphic $p$-adic $L$-function} (Definition \ref{autpadicLfun})
\begin{equation}\label{apLf}
\scr{L}^\mrm{aut}_p(\breve{\scr{G}},\msf{f}_E):\bs{\cal{W}}_{\scr{G}}\longrightarrow\bb{C}_p,
\end{equation}
whose value at the arithmetic point  $\mrm{P}_\circ$ of weight one is equal, up to a non-zero constant, to 
\[
\scr{L}^\mrm{aut}_p(\breve{\scr{G}},\msf{f}_E)(\mrm{P}_\circ)\overset{\cdot}{\sim} L^\mrm{alg}(E,\mrm{As}(\varrho), 1).
\]
Furthermore, its value at \emph{any} arithmetic point $\mrm{P}\in\bs{\cal{W}}_\scr{G}$ of weight $2$ is explicitly given in terms of $p$-adic cuspforms. In other words, we construct a rigid analytic function containing the information about the vanishing or non-vanishing of an automorphic $L$-value at $\mrm{P}_\circ\in\bs{\cal{W}}_\scr{G}$, and whose values at \emph{every} arithmetic points of weight $2$ have the potential of being related to syntomic Abel--Jacobi images of algebraic cycles.

\subsubsection{Hirzebruch--Zagier classes.}
When working with the one-variable ordinary Hida family $\scr{G}$ and the single form $\msf{f}_E$, the big Galois representation we want to consider has a simple form: there is a twist $\mrm{As}(\mbf{V}_{\scr{G}})^\dagger$ of $\mrm{As}(\mbf{V}_{\scr{G}})$ by a $\mbf{I}_\scr{G}^\times$-valued character (see Definition \ref{selfdual remark}) such that
\[
\mbf{V}_{\scr{G},E}^\dagger:= \mrm{As}(\mbf{V}_{\scr{G}})^\dagger(-1)\otimes \mrm{V}_p(E)
\] 
interpolates Kummer self-dual Galois representations.
The realization of this Galois representation in the cohomology of a tower of threefolds with increasing level at $p$ plays a crucial role in the construction of the sought-after annihilators. 
Suppose $p$ is a rational prime splitting in the real quadratic field $L$, and write $p\cal{O}_L=\frak{p}_1\frak{p}_2$.

\begin{definition}
 For any $\alpha\ge1$ and any compact open $K \le \mrm{GL}_2(\bb{A}_{L,f})$ hyperspecial at $p$, we set 
\[
	K_{\diamond,t}(p^\alpha):=\left\{\begin{pmatrix}a&b\\c&d \end{pmatrix}\in K_0(p^\alpha)\Big\lvert\  a_{\mathfrak{p}_1}d_{\mathfrak{p}_1}\equiv  a_{\mathfrak{p}_2}d_{\mathfrak{p}_2},\ d_{\mathfrak{p}_1}d_{\mathfrak{p}_2}\equiv 1 \pmod{p^{\alpha}}\right\}
\]
and denote by $S(K_{\diamond,t}(p^\alpha))$ the corresponding Hilbert modular surface.
\end{definition}
\noindent The reason for considering these unusual level structures is that for any arithmetic point $\mrm{P}\in\bs{\cal{W}}_\scr{G}$ of weight $2$ and level $p^\alpha$ the cohomology of $S(K_{\diamond,t}(p^\alpha))$ admits a Galois equivariant surjection
\[\xymatrix{
\mrm{H}^2_\et\big(S(K_{\diamond,t}(p^\alpha))_{\bar{\bb{Q}}},E_\wp(2)\big)\ar@{->>}[r]& \mrm{As}(\mbf{V}_{\scr{G}_{\mrm{P}}}).
}\]
Inspired by \cite{DR2}, for every $\alpha\ge1$ we produce a null-homologous codimension $2$ cycle, called \emph{Hirzebruch--Zagier cycle},
\[
\Delta_\alpha^\circ\in\mrm{CH}^2\big(Z_\alpha(K)\big)\big(\bb{Q}(\zeta_{p^\alpha})\big)\otimes\bb{Z}_p
\]
on the Shimura threefold $Z_\alpha(K)=S(K_{\diamond,t}(p^\alpha))\times X_0(Np)$. Moreover, the action of $\mrm{Gal}(\bb{Q}(\zeta_{p^\alpha})/\bb{Q})$ is such that  $\Delta^\circ_\alpha$ corresponds to a null-homologous rational cycle class 
\[
\Delta_\alpha^\circ\in\mrm{CH}^2\big(Z^\dagger_\alpha(K)\big)(\bb{Q})\otimes\bb{Z}_p
\]
  on a twisted threefold $Z^\dagger_\alpha(K)$ with the following appealing property: for every arithmetic point $\mrm{P}\in\bs{\cal{W}}_\scr{G}$ of weight $2$ and level $p^\alpha$ there are Galois equivariant surjections
\[\xymatrix{
\mrm{H}^3_\et\big(Z^\dagger_\alpha(K)_{\bar{\bb{Q}}},E_\wp(2)\big)\ar@{->>}[r]& \mrm{As}(\mbf{V}_{\scr{G}_{\mrm{P}}})^\dagger(-1)\otimes \mrm{V}_p(E).
}\]
The ordinary parts of the Abel--Jacobi images  
\[
\mrm{AJ}^\et_p(\Delta_\alpha^\circ)\in\mrm{H}^1\big(\bb{Q},\mrm{H}^3_\et\big(Z^\dagger_\alpha(K)_{\bar{\bb{Q}}},O(2)\big)\big)
\] can be made compatible under the degeneracy maps $\varpi_2:Z^\dagger_{\alpha+1}(K)\to Z^\dagger_\alpha(K)$ and packaged together to form a global big cohomology class (Definition \ref{BigCohomology})
\[
\bs{\kappa}_{\scr{G},E}\in\mrm{H}^1\big(\bb{Q},\bs{\cal{V}}_{\scr{G},E}\big).
\]
This class retains information about the Abel--Jacobi image of algebraic cycles at arithmetic points of weight two and it can be specialized at the arithmetic point  $\mrm{P}_\circ\in\bs{\cal{W}}_\scr{G}$ of weight one. Then, Corollary \ref{correctspec} implies that we obtain a $\mrm{V}_{\varrho,E}$-valued global class by specialization
\[
\kappa_E(\msf{g}_\circ^{\mbox{\tiny $(p)$}})\in\mrm{H}^1\big(\bb{Q},\mrm{V}_{\varrho,E}\big).
\]
In order to make apparent the relationship between $\bs{\kappa}_{\scr{G},E}$ and the automorphic $p$-adic $L$-function, we use Perrin-Riou's machinery to fabricate the \emph{motivic $p$-adic $L$-function}.

\subsubsection{The motivic $p$-adic $L$-function.}
This construction is naturally divided in two steps.
 First, the localization at $p$ of the big cohomology class can be projected to a Galois cohomology group
\[
\boldsymbol{\kappa}_p:=\mrm{Im}(\bs{\kappa}_{\scr{G},E}) \in \mrm{H}^1\big(\mathbb{Q}_p,\bs{\cal{U}}^E_\scr{G}(\bs{\Theta})\big)
\]
valued in a subquotient $\bs{\cal{U}}^E_\scr{G}(\bs{\Theta})$ of the Galois module $\bs{\cal{V}}_{\scr{G},E}$ on which $\Gamma_{\bb{Q}_p}$ acts through characters. Perrin-Riou's big logarithm (Proposition \ref{prop: big log}) valued in the big Dieudonn\'e module $\bb{D}(\bs{\cal{U}}^E_\scr{G})$ gives an element
\[
\bs{\cal{L}}(\boldsymbol{\kappa}_p)\in\bb{D}\big(\bs{\cal{U}}^E_\scr{G}\big)
\]
 interpolating the Bloch--Kato logarithm of the specialization of the class at arithmetic points of weight $\ge2$, and the Bloch--Kato dual exponential at the arithmetic point $\mrm{P}_\circ\in\bs{\cal{W}}_\scr{G}$ of weight one.
 The second step entails the definition of a linear map ($\mbf{I}_\scr{G}$-valued assuming Conjecture \ref{wishingOhta1})
	\begin{equation}\label{linmap}
	\big\langle\ ,\omega_{\breve{\scr{G}}}\otimes\eta_E'\big\rangle: \bb{D}\big(\bs{\cal{U}}^E_\scr{G}\big)\longrightarrow \mbf{I}_\scr{G}
	\end{equation}
producing rigid-analytic functions out of elements of the Dieudonn\'e module, and depending on the same input $(\breve{\scr{G}}, \msf{f}_E)$ used in constructing the automorphic $p$-adic $L$-function \eqref{apLf}. 
 Then, the motivic $p$-adic $L$-function is defined by setting (Definition \ref{motpadicLfun})
\[
\scr{L}_p^\mrm{mot}(\breve{\scr{G}},\msf{f}_E):=\big\langle\bs{\cal{L}}(\boldsymbol{\kappa}_p) ,\ \omega_{\breve{\scr{G}}}\otimes\eta_E'\big\rangle.
\]
As the notation suggests, the value at \emph{every} arithmetic point $\mrm{P}\in\bs{\cal{W}}_\scr{G}$ of weight two
\[
\scr{L}_p^\mrm{mot}(\breve{\scr{G}},\msf{f}_E)(\mrm{P})\overset{\cdot}{\sim}\big\langle\log_\mrm{BK}(\boldsymbol{\kappa}_p(P)) ,\ \omega_{\breve{\scr{G}_\mrm{P}}}\otimes\eta_E'\big\rangle_\mrm{dR}
\]
is computed by the de Rham pairing between the Bloch--Kato logarithm of the specialization of the big class and a de Rham class associated to the cuspforms $\breve{\scr{G}}_\mrm{P}$ and $\msf{f}_E$. Crucially, these quantities are values of \emph{syntomic Abel--Jacobi images} of HZ-cycles.

\subsubsection{Explicit reciprocity law.}
The comparison of the two $p$-adic $L$-functions is the bridge between the automorphic and the algebro-geometric worlds. It transfers information about the non-vanishing of an automorphic $L$-value into information on the non-triviality of annihilators of Mordell-Weil groups. It is achieved by an explicit reciprocity law (Theorem \ref{comparison aut-mot}) 
\[
\bs{\zeta}_{\scr{G},\msf{f}_E}(\mrm{P})\cdot \scr{L}^\mrm{mot}_p(\breve{\scr{G}},\msf{f}_E)(\mrm{P})
=
\scr{L}^\mrm{an}_p(\breve{\scr{G}},\msf{f}_E)(\mrm{P})
\]
for arithmetic points $\mrm{P}\in\bs{\cal{W}}_\scr{G}$ of weight two, where $\bs{\zeta}_{\scr{G},\msf{f}_E}\in\mbf{I}_\scr{G}$ is an explicit fudge factor non-vanishing at arithmetic points of $\mbf{I}_\scr{G}$. The proof relies on an explicit expression of syntomic Abel--Jacobi images of HZ-cycles in terms of $p$-adic modular forms (Theorem \ref{AJ formula}). Then, the key implication is given by the following theorem.

\begin{thmx}
	Suppose that $N$ is coprime to $\mathfrak{Q}$, split in $L$, and there exists an ordinary prime $p\nmid 2N\cdot\frak{Q}$ for $E_{/\bb{Q}}$ such that  
		\begin{itemize}
			\item[($1$)] $p$ splits in $L$ with narrowly principal factors;
			\item[($2$)] there is no totally positive unit in $L$ congruent to $-1$ modulo $p$;
			\item[($3$)] the eigenvalues of $\mrm{Fr}_p$ on $\mrm{As}(\varrho)$ are all distinct modulo $p$.
		\end{itemize}
	If, additionally, $\varrho$ is residually not solvable and Conjecture \ref{wishingOhta1} holds, then
for any choice of an ordinary $p$-stabilization $\msf{g}_\varrho^{\mbox{\tiny $(p)$}}$  of $\msf{g}_\varrho$, 
\[
L(E,\mrm{As}(\varrho), 1)\not=0 \qquad\implies\qquad\kappa_E(\msf{g}_\circ^{\mbox{\tiny $(p)$}})\in\mrm{H}^1\big(\bb{Q},\mrm{V}_{\varrho,E}\big)\quad \text{not Selmer at $p$}.
\]
\end{thmx}
\noindent  The result is used as follows: by assuming that $p$ splits in $L$ and the eigenvalues of $\mrm{Fr}_p$ on $\mrm{As}(\varrho)$ are all distinct, the eigenform $\msf{g}_\varrho$ has four distinct ordinary $p$-stabilizations.  Hence, we obtain four global cohomology classes by repeatedly applying Theorem B, and their images are linearly independent in the singular quotient at $p$ (Theorem \ref{criterion crystalline}). As the self-dual representation $\mrm{As}(\varrho)$ is four dimensional, these annihilators suffice to prove that the relevant part of the Mordell-Weil group is trivial (Lemma \ref{zerolocalization}).

\subsection{On the conjecture}\label{ontheconjectures}
We conclude the introduction by discussing the conjecture on the cohomology of Hilbert modular surfaces that  we assume in our work. 
Consider the nearly ordinary part of the \'etale cohomology
\[
\cal{V}_\alpha:=e_\mrm{n.o.}\mrm{H}^2_{\et,c}\big(S(K_{\diamond,t}(p^\alpha))_{\bar{\bb{Q}}},O(2)\big)
\]
where $e_\mrm{n.o.}$ is the nearly ordinary projector defined as in \eqref{eno} using Section \ref{def Up}.
The $0$-th graded piece $\mrm{Gr}^0\cal{V}_\alpha$ of the \'etale cohomology, with respect to the ordinary filtration, 
is an unramified $\Gamma_{\bb{Q}_p}$-representation (see Proposition \ref{somekindoffil}), therefore 
\[
\bb{D}\big(\mrm{Gr}^0\cal{V}_\alpha\big):=\big(\mrm{Gr}^0\cal{V}_\alpha\otimes\widehat{\bb{Z}}_p^\mrm{ur}\big)^{\Gamma_{\bb{Q}_p}}
\] is a lattice in the de-Rham cohomology group $\mrm{D}_\mrm{dR}\big(\mrm{Gr}^0\cal{V}_\alpha\otimes_OE_\wp\big)$. We are interested in comparing two integral structure for the de Rham cohomology of Hilbert modular surfaces: one coming from integral \'etale cohomology, the other arising from ordinary Hilbert modular forms of parallel weight two.

\begin{conjecture}\label{wishingOhta1}
	For every large enough prime $p$ and every $\alpha\ge1$ the image of the natural map
	\[
	S^\mrm{ord}_{2t_L,t_L}\big(K_{\diamond,t}(p^\alpha);O\big)
	\longrightarrow
	\mrm{D}_\mrm{dR}\big(\mrm{Gr}^0\cal{V}_\alpha\otimes_OE_\wp\big)
	\]
	is contained in the lattice $\bb{D}\big(\mrm{Gr}^0\cal{V}_\alpha\big)$.
\end{conjecture}

 This conjecture (Conjecture \ref{wishingOhta} in the body of the article) is a generalization of (\cite{Ohta95}, Proposition 3.3.6). Ohta's proof relies on Jacobians of modular curves, and therefore it cannot be directly generalized to higher dimensional Shimura varieties.
Recently, Sangiovanni--Skinner announced a proof of our conjecture  for general Hilbert and Siegel modular varieties (\cite{SangiovanniSkinner}). Their work relies on the new $p$-adic Hodge theoretic methods of Bhatt--Morrow--Scholze \cite{BMS}, Bhatt--Scholze \cite{Prismatic}, Bhatt--Lurie \cite{BhattLurie},
and was inspired by an argument of Faltings \cite{FaltingsSiegel}.

The crucial input of Conjecture \ref{wishingOhta1} in our work is in establishing the integrality of the linear map \eqref{linmap}. In particular, for $\mrm{P}_\circ\in\cal{A}_{\bs{\chi}}(\mbf{I}_\scr{G})$ the arithmetic point of weight one, it provides the implication
\[
\scr{L}_p^\mrm{mot}(\breve{\scr{G}},\msf{f}_E)(\mrm{P}_\circ)\not=0\qquad\implies\qquad \bs{\kappa}_p(\mrm{P}_\circ)\quad \text{non-trivial}.
\]

\begin{remark}
	The existence of the automorphic $p$-adic $L$-function can be used to meromorphically continue $\scr{L}_p^\mrm{mot}(\breve{\scr{G}},\msf{f}_E)$ to the whole disk $\bs{\cal{W}}_\scr{G}$ without assuming Conjecture \ref{wishingOhta1} -- even though that is insufficient for our arithmetic applications.
Indeed, the motivic $p$-adic $L$-function is a priori only an element of the huge ring $\bs{\Pi} \otimes_{\bs{\Lambda}} \mbf{I}_\scr{G}$ of functions defined only at arithmetic points of weight $2$ of $\bs{\cal{W}}_\scr{G}$ (for the definition of $\bs{\Pi}$ see \ref{bigpi}).
 However, there is a natural inclusion $\mbf{I}_\scr{G}\hookrightarrow\bs{\Pi} \otimes_{\bs{\Lambda}} \mbf{I}_\scr{G}$ which in terms of functions corresponds to the restriction of the domain from $\bs{\cal{W}}_\scr{G}$ to the subset of arithmetic points of weight $2$. As Hida families are \'etale at arithmetic points of weight $2$, an element of $\bs{\Pi} \otimes_{\bs{\Lambda}} \mbf{I}_\scr{G}$ is zero if and only if \emph{all} its specializations are zero. Therefore, Theorem \ref{comparison aut-mot} shows that $\scr{L}_p^\mrm{mot}(\breve{\scr{G}},\msf{f}_E)$ extends to a rigid meromorphic function on $\bs{\cal{W}}_\scr{G}$.	
\end{remark}


\begin{acknowledgements}
We would like to express deep gratitude to our Ph.D. advisors, Henri Darmon, Adrian Iovita and Christopher Skinner, who introduced us to this subject and provided valuable guidance throughout our Ph.D. careers. We also  thank Mladen Dimitrov, Dimitar Jetchev, David Lilienfeldt, David Loeffler, Alice Pozzi, Giovanni Rosso, Jan Vonk and Sarah Zerbes for many fruitful discussions and enriching conversations. A special thank goes to Daniel Disegni for pointing out a gap in a previous version of the paper, and to the anonymous referees whose comments helped to correct minor mistakes and to significantly improve the exposition.

This collaboration originated during a conference hosted by the Bernoulli Center at \'Ecole Polytechnique F\'ed\'erale de Lausanne, and continued at McGill University and Princeton University. We are grateful to these institutions for their hospitality and for providing ideal working conditions.
\end{acknowledgements}




\section{Review of Hilbert cuspforms}
In this section $L$ denotes a totally real number field with ring of integers $\cal{O}_L$ and different $\mathfrak{d}_L$. The following algebraic groups play a prominent role in the article
\begin{equation}
D = \mrm{Res}_{L/\bb{Q}}\big(\bb{G}_{m,L}\big)
,\qquad
G = \mrm{Res}_{L/\bb{Q}}\big(\mrm{GL}_{2,L}\big)
,\qquad
G^* = G \times_D \bb{G}_m.
\end{equation}
We denote by $\mrm{I}_L$ the set of field embeddings of $L$ into $\overline{\bb{Q}}$, then there is an identification of $L_\infty := L\otimes_\bb{Q}\bb{R}$ with $\bb{R}^{\mrm{I}_L}$. If $\frak{H}$ denotes the Poincar\'e upper half plane, the identity component 
$G(\bb{R})_+ $ of $G(\bb{R})= \mrm{GL}_2(L_\infty)$ naturally acts on $\frak{H}^{\mrm{I}_L}$. We denote by  $i = \sqrt{-1} \in \frak{H}$ the  square root of $-1$ belonging to $\frak{H}$ and $\mbf{i} = (i,...,i) \in \frak{H}^{\mrm{I}_L}$.

\begin{definition}
Every element $s = \sum_{\tau}s_\tau\cdot[\tau]$ of the free group $\bb{Z}[\mrm{I}_L]$ gives a power map 
\[ (-)^s:L\otimes_\bb{Q}\overline{\bb{Q}}_v\longrightarrow\overline{\bb{Q}}_v,\qquad \ell\otimes c\mapsto  \prod_{\tau} \big(c\cdot\tau(\ell)\big)^{s_\tau}
\]
for any place $v$ of $\bb{Q}$. 
\end{definition} 

\noindent The group $Z_L(1)$  is defined by the following short exact sequence 
\[\xymatrix{
 1\ar[r]&\overline{L^\times \widehat{\cal{O}}_L^{p,\times} L_{\infty,+}^\times}\ar[r]& \bb{A}_L^\times \ar[r]&Z_L(1)\ar[r]&1,
}\]  
and the $p$-adic cyclotomic character can be expressed as
\begin{equation}
\varepsilon_L: Z_L(1) \longrightarrow \bb{Z}^\times_p,\qquad y\mapsto y_p^{-t_L}\lvert y^\infty\rvert_{\bb{A}_L}^{-1}
\end{equation}
where $t_L := \sum_{\tau \in \mrm{I}_L} [\tau]\ \in\ \bb{Z}[\mrm{I}_L]$. Moreover, the canonical isomorphism $\bb{Z}_p^\times\cong(1+p\bb{Z}_p)\times\mu_{p-1}$ induces the factorization 
\begin{equation}\label{splitcyc}
\varepsilon_L=\eta_L\cdot\theta_L.
\end{equation}

\subsection{Adelic Hilbert cuspforms}
Let $K\le G(\bb{A}_f)$ be a compact open subgroup and $(k,w)\in\bb{Z}[\mrm{I}_L]^2$ an element satisfying $k-2w=m\cdot t_L$, then a holomorphic Hilbert cuspform of weight $(k,w)$ and level $K$ is a function $\msf{f}:G(\bb{A})\to\bb{C}$ that satisfies the following properties: 
\begin{itemize}
\item[$\bullet$] $\msf{f}(\alpha x u)=\msf{f}(x)j_{k,w}(u_\infty,\mathbf{i})^{-1}$ where $\alpha\in G(\bb{Q})$, $u\in K\cdot C_{\infty}^+$ for $C_{\infty}^+$ the stabilizer of $\mathbf{i}$ in $G(\bb{R})^+$, and where the automorphy factor is given by 
 $j_{k,w}\big(\gamma,z\big)=(ad-bc)^{-w}(cz+d)^k$ for 
$\gamma=\begin{pmatrix}a& b\\
c&d\end{pmatrix}
\in G(\bb{R})$, $z\in\mathfrak{H}^{\mrm{I}_L}$;
\item[$\bullet$] for every finite adelic point $ x\in G(\bb{A}_f)$ the well-defined function $\msf{f}_x:\frak{H}^{\mrm{I}_L}\to\bb{C}$ given by $\msf{f}_x(z)=\msf{f}(xu_\infty)j_{k,w}(u_\infty,\mathbf{i})$ is holomorphic, where for each $z\in\mathfrak{H}^{\mrm{I}_L}$ one chooses $u_\infty\in G(\bb{R})_+$ such that $u_\infty\mathbf{i}=z$.
\item[$\bullet$] for all adelic points $x\in G(\bb{A})$ and for all additive measures on 
$L\backslash\bb{A}_L$ we have
 \[\int_{L\backslash\bb{A}_L}\msf{f}\bigg(\begin{pmatrix}1&a\\0&1\end{pmatrix}x\bigg)da=0.\]
\item[$\bullet$] If the totally real field is the field of rational numbers, $L=\bb{Q}$, we need to impose the extra condition that for all finite adelic point $x\in G(\bb{A}^\infty)$ the function $\lvert \text{Im}(z)^\frac{k}{2}\msf{f}_x(z)\rvert$ is uniformly bounded on $\frak{H}$.
\end{itemize}
The $\bb{C}$-vector space of Hilbert cuspforms of weight $(k,w)$ and level $K$ is denoted by $S_{k,w}(K;\bb{C})$.
Let $dx$ be the Tamagawa measure on the quotient 
$[G(\bb{A})] := \bb{A}_L^\times G(\bb{Q}) \backslash G(\bb{A})$. For any pair $\msf{f}_1, \msf{f}_2 \in S_{k,w}(K;\bb{C})$ of cuspforms whose weight satisfies $k-2w = m\cdot t_L$, their Petersson inner product is given by
\begin{equation}
\langle \msf{f}_1, \msf{f}_2 \rangle :=
\int_{[G(\bb{A})]} \msf{f}_1(x) \overline{\msf{f}_2(x)} \cdot\lvert\det(x)\rvert^m_{\bb{A}_L} dx.
\end{equation}

\begin{definition}
For any $\cal{O}_L$-ideal $\mathfrak{N}$ we consider the following compact open subgroups of $G(\bb{A}_f)$
\begin{itemize}
\item[\bfcdot] $V_0(\mathfrak{N})=\bigg\{\begin{pmatrix}a&b\\c&d\end{pmatrix}\in G(\widehat{\bb{Z}})\bigg\lvert\ c\in\mathfrak{N}\widehat{\cal{O}}_L\bigg\}$,
\item[\bfcdot] $V_1(\mathfrak{N})=\bigg\{\begin{pmatrix}a&b\\c&d\end{pmatrix}\in V_0(\mathfrak{N})\bigg\lvert\ d\equiv1 \pmod{\mathfrak{N}\widehat{\cal{O}}_L}\bigg\}$,
\item[\bfcdot] $V^1(\mathfrak{N})=\bigg\{\begin{pmatrix}a&b\\c&d\end{pmatrix}\in V_0(\mathfrak{N})\bigg\lvert\ a\equiv1 \pmod{\mathfrak{N}\widehat{\cal{O}}_L}\bigg\}$,
\item[\bfcdot] $V(\mathfrak{N})=V_1(\mathfrak{N})\cap V^1(\mathfrak{N})$.
\end{itemize}
\end{definition}
When $K=V_1(\frak{N})$ for some $\cal{O}_L$-ideal $\frak{N}$ we will write
$S_{k,w}(\frak{N};\bb{C})$ instead of $S_{k,w}(K;\bb{C})$.

\subsubsection{Adelic q-expansion.} 
 Let 
$\mrm{cl}_L^+(\mathfrak{N}) := L^\times_+\backslash \bb{A}_{L,f}^\times /\det V(\mathfrak{N})$ be a narrow class group of $L$ of cardinality $h_L^+(\mathfrak{N})$, and fix a set of representatives $\{ a_i\}_i \subset \bb{A}_{L,f}^\times$ for $\mrm{cl}_L^+(\mathfrak{N})$. Then the adelic points of $G$ can be written as a disjoint union 
\[
G(\bb{A}) =
\coprod_{i=1}^{h_L^+(\mathfrak{N})} G(\bb{Q}) t_i V(\mathfrak{N})G(\bb{R})_+,\qquad \text{for}\quad t_i=\begin{pmatrix}a_i^{-1}&0\\ 0&1\end{pmatrix},
\]
using strong approximation.
Given a Hilbert cuspform $\msf{f}\in S_{k,w}(V(\frak{N});\bb{C})$ one can consider the holomorphic function $\msf{f}_i:\frak{H}^{\mrm{I}_L}\to\bb{C}$
\[
\msf{f}_i(z)=y_\infty^{-w}\msf{f}\left(t_i\begin{pmatrix}
y_\infty& x_\infty\\
0&1
\end{pmatrix}\right)=\underset{\xi\in(\frak{a}_i\frak{d}_L^{-1})_+}{\sum}a(\xi,\msf{f}_i)e_L(\xi z)
\]
where $z=x_\infty+\mathbf{i}y_\infty$, $\frak{a}_i=a_i\cal{O}_L$ and $e_L(\xi z)=\text{exp}\big(2\pi i\sum_{\tau\in \mrm{I}_L}\tau(\xi)z_\tau\big)$ for every index $i$. The Fourier expansions of these functions can be packaged together into a single adelic $q$-expansion: 

\noindent fix  a finite idele $\mathsf{d}_L\in\bb{A}_{L,f}^{\times}$ such that $\mathsf{d}_L\cal{O}_L=\frak{d}_L$. Let $L^\text{Gal}$ be the Galois closure of $L$ in $\overline{\bb{Q}}$ and write $\mathcal{V}$ for the ring of integers or a valuation ring of a finite extension $L_0$ of $L^\text{Gal}$ such that for every ideal $\frak{a}$ of $\cal{O}_L$, for all $ \tau\in \mrm{I}_L$, the ideal $\frak{a}^\tau\mathcal{V}$ is principal.
Choose a generator $\{\frak{q}^\tau\}\in\mathcal{V}$ of $\frak{q}^\tau\mathcal{V}$ for each prime ideal $\frak{q}$ of $\cal{O}_L$ and by multiplicativity define $\{\frak{a}^v\}\in\cal{V}$ for each fractional ideal $\frak{a}$ of $L$ and each $v\in\bb{Z}[\mrm{I}_L]$. Then, we set $\{y^v\}:=\{y^v\cal{O}_L\}\in\cal{V}$ for each idele of $L$.
Every idele $y$ in $\bb{A}^\times_{L,+}:=\bb{A}_{L,f}^{\times} L^\times_{\infty,+}$ can be written as $y=\xi a_i^{-1}\mathsf{d}_Lu$ for $\xi\in L_+^\times$ and $u\in\det V(\frak{N})L^\times_{\infty,+}$, then the following functions 
\[
\msf{a}(-,\msf{f}):\bb{A}^\times_{L,+}\longrightarrow\bb{C},\qquad \msf{a}_p(-,\msf{f}):\bb{A}^\times_{L,+}\longrightarrow\overline{\bb{Q}}_p
\]
are defined by 
\[
\msf{a}(y,\msf{f}):=a(\xi,\msf{f}_i)\{y^{w-t_L}\}\xi^{t_L-w}\lvert a_i\rvert_{\bb{A}_L}\qquad\text{and}\qquad \msf{a}_p(y,\msf{f}):=a(\xi,\msf{f}_i)y_p^{w-t_L}\xi^{t_L-w}\varepsilon_L(a_i)^{-1}
\] 
if $y\in\widehat{\cal{O}_L}L^\times_{\infty,+}$ and zero otherwise.  
 
 \begin{theorem}{(\cite{pHida}, Theorem 1.1)}\label{thm: adelic q-exp}
	Consider the additive character of the ideles $\chi_L:\bb{A}_L/L\to\bb{C}^\times$ which satisfies $\chi_L(x_\infty)=e_L(x_\infty)$. Each cuspform $\msf{f}\in S_{k,w}(V(\frak{N});\bb{C})$ has an adelic $q$-expansion of the form
	 \[
	 \msf{f}\left(\begin{pmatrix}
	 y & x \\ 
	 0 & 1
	 \end{pmatrix}\right)=\lvert y\rvert_{\bb{A}_L}\underset{\xi\in L_+}{\sum}\mathsf{a}(\xi y\mathsf{d}_L,\msf{f})\{(\xi y\mathsf{d}_L)^{t_L-w}\}(\xi y_\infty)^{w-t_L}e_L(\mathbf{i}\xi y_\infty)\chi_L(\xi x) 
	 \]
	 for $y\in\bb{A}^\times_{L,+}$, $x\in\bb{A}^\times_L$, and the function $\mathsf{a}(-,\msf{f}):\bb{A}^\times_{L,+}\to \bb{C}$ vanishes outside $\widehat{\cal{O}}_LL^\times_{\infty,+}$. 
 \end{theorem}

\subsubsection{Diagonal restriction.}
The degree map $\bb{Z}[\mrm{I}_L]\to\bb{Z}$ denoted by $\ell\mapsto \lvert\ell\rvert$ satisfies $\lvert t_L\rvert=[L:\bb{Q}]$. For any positive integer $N$ the natural inclusion $\zeta:\text{GL}_2(\bb{A})\hookrightarrow\text{GL}_2(\bb{A}_L)$ defines by composition a \emph{diagonal restriction} map 
\[
\zeta^*:S_{k,w}(V(N\cal{O}_L);\bb{C})\to S_{\lvert k\rvert, \lvert w\rvert}(V(N);\bb{C})
\]
from Hilbert cuspforms over $L$ to elliptic cuspforms.
\begin{lemma}\label{q-exp diagonal}
	Let $\msf{g}\in S_{\ell,x}(V(N\cal{O}_L);\overline{\bb{Q}})$ be a Hilbert cuspform over $L$, then for $y \in \widehat{\bb{Z}}\cdot \bb{R}^\times_{+}$ written as $y=\xi a_i^{-1}u$ for $\xi\in \bb{Q}^\times_+$ and $u\in\det(V(N))\bb{R}^\times_{+}$, we have
	\[
	\msf{a}_p(y,\zeta^*\msf{g})=y_p^{\lvert x\rvert-1}\xi^{1-\lvert x\rvert}\varepsilon_\bb{Q}(a_i)^{-1}\sum_{\mrm{Tr}_{L/\bb{Q}}(\eta)=\xi}\msf{a}_p(y_\eta,\msf{g})(y_\eta)_p^{t_L-x}\eta^{x-t_L}
	\]
	where $\eta\in L^\times_+$ and $y_\eta=\eta a_i^{-1}\msf{d}_Lu$.
\end{lemma}
\begin{proof}
	A direct computation.
\end{proof}
\begin{remark}
	When $p$ is unramified in $L/\bb{Q}$ one sees that $y_p=(\xi u)_p$, $(y_\eta)_p=(\eta u)_p$ and that the formula becomes
	\begin{equation}
	\msf{a}_p(y,\zeta^*\msf{g})=u_p\varepsilon_\bb{Q}(a_i)^{-1}\sum_{\mrm{Tr}_{L/\bb{Q}}(\eta)=\xi}\msf{a}_p(y_\eta,\msf{g}).
	\end{equation}
\end{remark}

\subsection{Hecke Theory}
Let $K \le G(\bb{A}_f)$ be an open compact subgroup satisfying $V(\mathfrak{N})\le K\le V_0(\mathfrak{N}) $. Suppose $\cal{V}$ is the valuation ring corresponding to the fixed embedding $\iota_p: L^\mrm{Gal} \hookrightarrow \overline{\bb{Q}}_p$, then we assume $\{\frak{q}\} = 1$ whenever the ideal $\frak{q}$ is prime to $p\cal{O}_L$.
For every $g \in G(\bb{A})$, one can consider the following double coset operator $[KgK]$. By decomposing the double coset into a disjoint union
\[
K g K
=
\coprod_i \gamma_i K,
\]
its action on Hilbert cuspforms of level $K$ is given by 
\begin{equation}
\big([KgK]\msf{f}\big)   (x)
=
\sum_i \msf{f}(x\gamma_i).
\end{equation}

\begin{definition}
For every prime ideal $\mathfrak{q}\le\cal{O}_L$ and a choice of uniformizer $\varpi_\mathfrak{q}$ of $\cal{O}_{L,\mathfrak{q}}$, the Hecke operators at $\frak{q}$ acting on $S_{k,w}(K;\bb{C})$ are defined as
\[
T_0(\varpi_\mathfrak{q}) = \{\varpi_\mathfrak{q}^{w-t_L}\}
\Big[ K\begin{pmatrix} \varpi_\mathfrak{q} & 0 \\ 0 & 1 \end{pmatrix} K \Big].
\]
For every invertible element
$a \in \cal{O}_{L,\mathfrak{N}}^\times = \Pi_{\mathfrak{q} | \mathfrak{N}} \cal{O}_{L,\mathfrak{q}}^\times$ there a Hecke operator
\[
T(a,1) =  
\Big[ K \begin{pmatrix} a & 0 \\ 0 & 1 \end{pmatrix} K \Big].
\]
For any element $z \in Z_G(\bb{A}_f)$ in the center of $G(\bb{A}_f)$, the associated diamond operator $\langle z \rangle$ acts through the rule $(\langle z \rangle\msf{f})(x) = \msf{f}(xz)$.
\end{definition}

\noindent It turns out that if the ideal $\mathfrak{q}$ is coprime to the level $\mathfrak{N}$, then $T_0(\varpi_\mathfrak{q})$  and $\langle \varpi_\mathfrak{q} \rangle$  are independent of the particular choice of  uniformizer $\varpi_\mathfrak{q}$, thus we simply denote them by $T(\mathfrak{q})$ and $\langle \mathfrak{q} \rangle$. However, if $\mathfrak{q}\mid \mathfrak{N}$, then  $T_0(\varpi_\mathfrak{q})$ does depend on $\varpi_\mathfrak{q}$, and we denote it $U_0(\varpi_\mathfrak{q})$. Any $y \in \widehat{\cal{O}_L}\cap\bb{A}_L^\times$ can be written as 
\[
y = au \prod_\mathfrak{q} \varpi_\mathfrak{q}^{e(\mathfrak{q})}\qquad\text{for}\qquad a \in \cal{O}_{L,\mathfrak{N}}^\times,\ u \in \det V(\mathfrak{N}).
\]
 Write $\mathfrak{n}$ for the ideal
$\big(\prod_{\mathfrak{q} \nmid \mathfrak{N}} \varpi_\mathfrak{q}^{e(\mathfrak{q})}\big) \cal{O}_L$, then we define the Hecke operators associated to the adele $y$ by
\begin{equation}
T_0(y) = T(a,1) T_0(\mathfrak{n}) \prod_{\mathfrak{q}\mid \mathfrak{N}} U_0(\varpi_\mathfrak{q}^{e(\mathfrak{q})}).
\end{equation}

\begin{definition}
A cuspform $\msf{f}\in S_{k,w}(K;\bb{C})$ is said to be an eigenform if it is an eigenvector for all the Hecke operators $T_0(y)$, and it is normalized if $\msf{a}(1,\msf{f}) = 1$. 
\end{definition}

\noindent     For finite ideles $b \in \bb{A}_{L,f}^\times$ there are other operators $V(b)$ on cuspforms defined by
    \begin{equation}
    (V(b)\msf{f})(x) = \mrm{N}_{L/\bb{Q}}(b\cal{O}_L)^{-1}\msf{f}\left(x\begin{pmatrix} b^{-1} & 0 \\ 0 & 1 \end{pmatrix}\right).
    \end{equation}
These operators are right inverses of the operators $U(\varpi_\mathfrak{q}):=\{\varpi_\frak{q}^{t_L-w}\}U_0(\varpi_\mathfrak{q})$, i.e.,
 \begin{equation}
 U(\varpi_\mathfrak{q})\circ V(\varpi_\mathfrak{q}) = 1.
 \end{equation}

\subsection{Hida families}\label{sect Hida families}
Let $O$ be a valuation ring  in $\overline{\bb{Q}}_p$ finite flat over $\bb{Z}_p$ and containing $\iota_p(\cal{V})$. For $\mathfrak{N}$ an $\cal{O}_L$-ideal prime to $p$ and compact open subgroups satisfying $V_1(\frak{N})\le K\le V_0(\mathfrak{N})$ we set
$K(p^\alpha) = K \cap V(p^\alpha)$ and $K(p^\infty) = \cap_{\alpha\ge1} K(p^\alpha)$. The projective limit of $p$-adic Hecke algebras
\[
\mbf{h}_L(K;O):=\varprojlim_\alpha\ \msf{h}_{k,w}(K(p^\alpha);O)
\qquad
\text{acts on} 
\qquad
\varinjlim_\alpha\ S_{k,w}(K(p^\alpha);O)
\]
 through the Hecke operators $\mbf{T}(y) = \varprojlim_\alpha T_0(y) $ and it is independent of the weight $(k,w)$. Since $\mathbf{h}_L(K;O)$ is a compact ring, it can be written as a direct sum of algebras 
\[
\mathbf{h}_L(K;O) = \mathbf{h}^{\text{n.o.}}_L(K;O) \oplus \mathbf{h}_L^{\text{ss}}(K;O)
\] 
such that $\mathbf{T}(\varpi_p)$ is a unit in $\mathbf{h}_L^{\text{n.o.}}(K;O)$ and topologically nilpotent in $\mathbf{h}_L^{\text{ss}}(K;O)$. We denote by 
\begin{equation}\label{eno}
e_{\text{n.o.}} = \underset{n \to \infty}{\lim} \mathbf{T}(\varpi_p)^{n!}
\end{equation}
the idempotent corresponding to the nearly ordinary part $\mathbf{h}^{\text{n.o.}}_L(K;O)$. 
For any $\alpha\ge1$ we set
\begin{equation}
Z^\alpha_L(K) := \bb{A}_L^\times/L^\times (\bb{A}_{L,f} \cap  K(p^\alpha)) L_{\infty,+}^\times\quad\text{and}\quad \bb{G}_L^\alpha(K):= Z^\alpha_L(K)\times(\cal{O}_L/p^\alpha\cal{O}_L)^\times.
\end{equation}
If we denote the projective limits by
\begin{equation}
Z_L(K):= \varprojlim_\alpha\ Z^\alpha_L(K),\qquad \bb{G}_L(K) :=\varprojlim_\alpha\ \bb{G}^\alpha_L(K),
\end{equation}
then $\bb{G}_L(K) = Z_L(K)\times\cal{O}_{L,p}^\times$ and there is a group homomorphism 
\begin{equation}
\bb{G}_L(K)\to \mathbf{h}_L(K;O)^\times,\qquad (z,a)\mapsto \langle z,a\rangle:=\langle z\rangle T(a^{-1},1)
\end{equation}
that endows  $\mbf{h}_L(K;O)$ with a structure of $O\llbracket\bb{G}_L(K)\rrbracket$-algebra.

\noindent Let $\mrm{cl}_L^+(\mathfrak{N}p)$ be the strict ray class group of modulus $\mathfrak{N}p$ and $\overline{\cal{E}}^+_{\mathfrak{N}p}$ the closure in $\cal{O}_{L,p}^\times$ of the totally positive units of $\cal{O}_L$ congruent to $1$ $\pmod{\mathfrak{N}p}$. There is a short exact sequence
\[
\xymatrix{
1\ar[r]
&
\overline{\cal{E}}^+_{\mathfrak{N}p}\backslash\big(1+p\cal{O}_{L,p}\big) \ar[r]& Z_L(K) \ar[r]
&
\mrm{cl}_L^+(\mathfrak{N}p) \ar[r]
& 
1,
}
\] 
that splits when $p$ is large enough because then the group $\mrm{cl}_L^+(\mathfrak{N}p)$ has order prime to $p$. We denote the canonical decomposition by
\begin{equation}\label{GaloisDecomposition}
	Z_L(K) \overset{\sim}{\to} \overline{\cal{E}}^+_{\mathfrak{N}p}\backslash\big(1+p\cal{O}_{L,p}\big) \times \mrm{cl}_L^+(\mathfrak{N}p),\qquad z\mapsto \big(\xi_z,\bar{z}\big).	
\end{equation} 
If we set
$\boldsymbol{\mathfrak{I}}_L = \overline{\cal{E}}^+_{\mathfrak{N}p}\backslash\big(1+p\cal{O}_{L,p}\big) \times (1+p \cal{O}_{L,p})$, then the short exact sequence
\[\xymatrix{
1\ar[r]& \boldsymbol{\mathfrak{I}}_L \ar[r]& \bb{G}_L(K)\ar[r]& \mrm{cl}_L^+(\mathfrak{N}p)\times (\cal{O}_L/p)^\times\ar[r]& 1
}\]
 splits canonically when $p$ is large enough. The group $\boldsymbol{\mathfrak{I}}_L$ is a finitely generated $\bb{Z}_p$-module of $\bb{Z}_p$-rank $[L:\bb{Q}]+1+\delta$, where $\delta$ is Leopoldt's defect for $L$. Let $\mbf{W}$ be the torsion-free part of $\boldsymbol{\mathfrak{I}}_L$ and denote by
$\bs{\Lambda}_L = O\llbracket \mbf{W} \rrbracket$ the associated completed group ring.

\begin{theorem}{(\cite{nearlyHida}, Theorem 2.4)}
The nearly ordinary Hecke algebra $\mathbf{h}_L^{\text{n.o.}}(K;O)$ is finite and torsion-free over $\boldsymbol{\Lambda}_L$.
\end{theorem}

\noindent The completed group ring $O\llbracket\bb{G}_L(\mathfrak{N})\rrbracket$ naturally decomposes as the direct sum $\bigoplus_\chi \boldsymbol{\Lambda}_{L,\chi}$ ranging over all the characters of the torsion subgroup
$\bb{G}_L(\mathfrak{N})_\mrm{tor} = \mathfrak{I}_{L,\mrm{tor}}\times \mrm{cl}_L^+(\mathfrak{N}p)\times (\cal{O}_L/p)^\times$.
It induces a decomposition of the nearly ordinary Hecke algebra $
\mbf{h}_L^{\text{n.o.}}(K;O) = \bigoplus_{\chi}\mathbf{h}_L^{\text{n.o.}}(K;O)_\chi$.

\begin{definition}\label{def I-adic cuspforms}
    Let $\chi: \bb{G}_L(K)_\mrm{tor} \rightarrow O^\times$ be a character. For any $\bs{\Lambda}_{L,\chi}$-algebra  $\mbf{I}$, the space of nearly ordinary $\mbf{I}$-adic cuspforms of tame level $K$ and character $\chi$ is 
    \[
    \bar{\mbf{S}}_L^\mrm{n.o.}(K,\chi;\mbf{I}) :=
    \mrm{Hom}_{\bs{\Lambda}_{L,\chi}\mbox{-}\mrm{mod}} \big( \mbf{h}_L^\mrm{n.o.}(K(p^\infty);O)_\chi, \mbf{I}\big).
    \]
	 When an $\mbf{I}$-adic cuspform is also a $\bs{\Lambda}_{L,\chi}$-algebra homomorphism, we call it a Hida family.
\end{definition}

Let
$\psi: \mrm{cl}_L^+(\mathfrak{N}p^\alpha) \rightarrow O^\times$,
$\psi': (\cal{O}_L/p^\alpha)^\times \rightarrow O^\times$  be a pair of characters 
and  $(k,w)$ a weight satisfying $k-2w = mt_L$. The group homomorphism
\[
\bb{G}_L(K)\to O^\times,\qquad (z,a) \mapsto \psi(z)\psi'(a)\varepsilon_L(z)^ma^{t_L-w}
\]
determines a $O$-algebra homomorphism $\mrm{P}_{k,w,\psi,\psi'}:O\llbracket\bb{G}_L(K) \rrbracket\rightarrow O$.

\begin{definition}
For  a $\boldsymbol{\Lambda}_{L,\chi}$-algebra $\mbf{I}$ the set of \emph{arithmetic points}, denoted by $\cal{A}_\chi(\mathbf{I})$, is the subset of $\mrm{Hom}_{O\mbox{-}\mrm{alg}}(\mathbf{I},\overline{\bb{Q}}_p)$ consisting of homomorphisms that coincide with some $\mrm{P}_{k,w,\psi,\psi'}$ when restricted to $\boldsymbol{\Lambda}_{L,\chi}$. 
\end{definition}

\begin{definition}
    Let $(k,w)$ be a weight such that $k-2w=mt_L$. For any pair of characters 
    $\psi: \mrm{cl}_L^+(\mathfrak{N}p^\alpha) \rightarrow O^\times$
    and
    $\psi': (\cal{O}_L/p^\alpha)^\times \rightarrow O^\times$,
   one defines
    \[
    S_{k,w}(K(p^\alpha);\psi,\psi';O)\subseteq S_{k,w}(K(p^\alpha);O)
	\]
  to be the submodule of cuspforms satisfying
   \[
    \langle z,a \rangle \msf{f} = \varepsilon_L(z)^m\psi(z)\psi'(a)\cdot\msf{f} \qquad \forall (z,a) \in \bb{G}_L(K).
   \]
    
\end{definition}

\subsection{Twists of cuspforms}
First we recall Hida's twists of Hilbert cuspforms by Hecke characters (\cite{pHida}, Section 7F), then we define a twist of cuspforms by local characters and relate it to the Atkin-Lehner involution.

\noindent Let $\Psi:\bb{A}_L^\times/L^\times\to\bb{C}^\times$ be a Hecke character of conductor $C(\Psi)$ and infinity type $m\cdot t_L$, $m\in\bb{Z}$. Since $\Psi$ has algebraic values on finite ideles, Hida defined the map
\[\begin{split}
-\otimes\Psi: S_{k,w}\big(\mathfrak{N}p^\alpha,\psi,\psi';&O\big)\to S_{k,w+m\cdot t_L}\big(C(\Psi)\mathfrak{N}p^\alpha,\psi\Psi^2,\psi'\Psi_p^{-1};O\big)\\
&\msf{f}\mapsto\msf{f}\otimes\Psi
\end{split}\]
where the cuspform $\msf{f}\otimes\Psi$ has adelic Fourier coefficients given by 
\begin{equation}
\msf{a}_p(y,\msf{f}\otimes\Psi)=\Psi(y^\infty)\msf{a}_p(y,\msf{f})y_p^{m\cdot t_L}.
\end{equation}
When $\Psi=\lvert-\rvert_{\bb{A}_L}^m$ is a integral power of the adelic norm character, one finds that
 \[
	\msf{f}\otimes\lvert-\rvert_{\bb{A}_L}^m\in S_{k,w+m\cdot t_L}\big(\mathfrak{N}p^\alpha,\psi,\psi';O\big)
	\]
	and 
	\[
	\msf{a}_p(y,\msf{f}\otimes\lvert-\rvert_{\bb{A}_L}^m)=\varepsilon_L(y)^{-m}\msf{a}_p(y,\msf{f}).
	\]
Note that twisting does not affect the classical Fourier expansion on the identity component of the Hilbert modular surface. 

\begin{lemma}\label{twist classical expansion}
    Let $\msf{g}\in S_{k,w}\big(\mathfrak{N}p^\alpha,\psi,\psi';O\big)$ and $\msf{g}_1$ be the first component of the corresponding tuple of classical Hilbert modular forms. Then for any Hecke character $\Psi:\bb{A}_L^\times/L^\times\rightarrow \bb{C}^\times$, we have
    \[
    (\msf{g}\otimes \Psi)_1 = \Psi(\msf{d}_L) \msf{g}_1.
    \]
\end{lemma}
\begin{proof}
   If $\xi\in (\frak{d}_L)_+$, it follows directly from the definitions that
    \[
    \msf{a}_p(\xi \msf{d}_L,\msf{g}) = a(\xi,\msf{g}_1).
    \]
    Suppose $\Psi$ has infinity type $m\cdot t_L$, then one can compute that
    \[\begin{split}
   a(\xi,(\msf{g}\otimes\Psi)_1)= \msf{a}_p(\xi \msf{d}_L,\msf{g}\otimes \Psi)
    &=
    \msf{a}_p(\xi \msf{d}_L,\msf{g})\Psi((\xi \msf{d}_L)^\infty) (\xi\msf{d}_L)_p^{m\cdot t_L}\\
    &=
    \Psi(\msf{d}_L)\cdot\msf{a}_p(\xi \msf{d}_L,\msf{g})\\
    &=
    \Psi(\msf{d}_L)\cdot a(\xi ,\msf{g}_1).
    \end{split}\]
\end{proof}

\noindent Now let $\mathfrak{N}, \mathfrak{A}$ be integral $\cal{O}_L$-ideals prime to $p$ and let $\chi:\mrm{Cl}_L^+(\mathfrak{A}p^\alpha)\to O^\times$ be a finite order Hecke character. Hida defined a function
\[\begin{split}
-_{\lvert\chi}: S_{k,w}\big(\mathfrak{N}p^\alpha,\psi,\psi';&O\big)\to S_{k,w}\big(\mathfrak{N}p^\alpha\mathfrak{A}^2,\psi\chi^2,\psi';O\big)\\
 &\msf{f}\mapsto\msf{f}_{\lvert \chi}
\end{split}\]
where $\msf{f}_{\lvert \chi}$ is a cuspform whose adelic Fourier coefficients are given by 
\[
\msf{a}_p(y,\msf{f}_{\lvert \chi})
=
\begin{cases}\chi(y)\msf{a}_p(y,\msf{f})
& \text{if}\ 
y_{\mathfrak{A}p}\in\cal{O}_{L,\mathfrak{A}p}^\times\\
0&\text{otherwise}.
\end{cases}
\]
We define a new kind of twist by slightly modifying Hida's work. Let $p>[L:\bb{Q}]$ be a rational prime, $\mathfrak{p}\mid p$  an $\cal{O}_L$-prime ideal and $\chi:\cal{O}_{L,\mathfrak{p}}^\times\to O^\times$ a finite order character of conductor $c(\chi)$.
\begin{proposition}\label{TwistHMF}
	There is a function 
	\[\begin{split}
	-\star\chi: S_{k,w}\big(\mathfrak{N}p^\alpha,\psi,\psi';&O\big)\to S_{k,w}\big(\mathfrak{N}p^{\mrm{max}\{\alpha,c(\chi)\}},\psi,\psi'\chi^{-1};O\big)\\
	 &\msf{f}\mapsto\msf{f}\star\chi
	\end{split}\]
	where $\msf{f}\star\chi$ has adelic Fourier coefficients given by 
	\[
	\msf{a}_p(y,\msf{f}\star\chi)=\begin{cases}\chi(y_\mathfrak{p})\msf{a}_p(y,\msf{f})& \text{if}\ y_{\mathfrak{p}}\in\cal{O}_{F,\mathfrak{p}}^\times\\
0&\text{otherwise}.
\end{cases}
	\]
\end{proposition}
\begin{proof}
	Define
	\[
	\msf{h}(x):=\sum_{u\in\left(\cal{O}_{L}/\mathfrak{p}^{c(\chi)}\right)^\times}\chi^{-1}(u)\msf{f}\left(x\begin{pmatrix}
	1& u\varpi_\mathfrak{p}^{-c(\chi)}\msf{d}_L\\
	0&1
	\end{pmatrix}\right).
	\]
	Looking at the adelic $q$-expansion we see that  for all $y\in\widehat{\cal{O}_L}$
	\[
	\msf{a}_p(y,\msf{h})=\left(\sum_{u\in\left(\cal{O}_{L}/\mathfrak{p}^{c(\chi)}\right)^\times}\chi^{-1}(u)\chi_L\big(yu\varpi_\mathfrak{p}^{-c(\chi)}\big)\right)\msf{a}_p(y,\msf{f}).
	\]
	Then one notices that $\chi_L\big(yu\varpi_\mathfrak{p}^{-c(\chi)}\big)$ is a $p$-power root of unity of order $c(\chi)-\mrm{val}_\frak{p}(y_\frak{p})$ so that
	\[
	\sum_{u\in\left(\cal{O}_{L}/\mathfrak{p}^{c(\chi)}\right)^\times}\chi^{-1}(u)\chi_L\big(yu\varpi_\mathfrak{p}^{-c(\chi)}\big)=\begin{cases}
	\chi(y_\mathfrak{p})G(\chi)&\text{if}\ y_{\mathfrak{p}}\in\cal{O}_{L,\mathfrak{p}}^\times\\
		0&\text{otherwise} 
	\end{cases}
	\]
	where $G(\chi)=\sum_{u}\chi^{-1}(u)\chi_L\big(u\varpi_\mathfrak{p}^{-c(\chi)}\big)$ is a Gauss sum. The cuspform
	\[
	\msf{f}\star\chi:=G(\chi)^{-1}\msf{h}
	\]
	has the claimed adelic $q$-expansion and another direct calculation shows that the operators $T(a^{-1},1)$'s act on it by
	\[
	\msf{a}_p\big(y,(\msf{f}\star\chi)_{\lvert T(a^{-1},1)}\big)=\chi^{-1}(a_\mathfrak{p})\psi'(a)\msf{a}_p(y,\msf{f}\star\chi).
	\]
\end{proof}

\noindent Let $\msf{f}\in S_{k,w}(\mathfrak{N}p^\alpha,\psi,\psi';O)$ be an eigencuspform. For $\mathfrak{p}$ a prime $\cal{O}_L$-ideal, let $\tau_{\mathfrak{p}^\alpha}\in\mrm{GL}_2(\bb{A}_L)$ be defined by
\[
(\tau_{\mathfrak{p}^\alpha})_\mathfrak{p}=\begin{pmatrix}
	0&-1\\
	\varpi_\mathfrak{p}^\alpha&0
\end{pmatrix},\qquad
(\tau_{\mathfrak{p}^\alpha})_v=\mathbbm{1}_2\qquad \text{for}\qquad v\not=\mathfrak{p}.
\] 
We have the following operator
\[
\msf{f}\lvert\tau_{\mathfrak{p}^\alpha}^{-1}(x):=\msf{f}(x\tau_{\mathfrak{p}^\alpha}^{-1}).
\]
Moreover, for any prime ideal $\frak{p}\mid p$, the $\frak{p}$-depletion of a cuspform $\msf{f}$ is the cuspform 
\[
\msf{f}^{[\frak{p}]}=\left(1-V(\varpi_\frak{p})\circ U(\varpi_\frak{p})\right)\msf{f}
\]
 whose Fourier coefficient $\mathsf{a}_p(y,\msf{f}^{[\frak{p}]})$ equals $\mathsf{a}_p(y,\msf{f})$ if $y_\frak{p}\in \cal{O}_{F,\frak{p}}^\times$ and $0$ otherwise.
\begin{lemma}\label{Atkin-Lehner}
	Let $\msf{f}\in S_{2t_L,t_L}(\mathfrak{N}p^\alpha;\psi,\psi';O)$ be an eigenform, then 
	\[
	\msf{a}_p(\varpi_\mathfrak{p},\msf{f})^\alpha \cdot
	\Big(\msf{f}\lvert\tau^{-1}_{\mathfrak{p}^\alpha}\Big)^{\mbox{\tiny $[\mathfrak{p}]$}}
	=
	G(\psi_\mathfrak{p}(\psi'_\mathfrak{p})^{2})\cdot\Big(\msf{f}\star\psi_\mathfrak{p}(\psi'_\mathfrak{p})^{2}\Big)
	\]
	the equality taking place in $ S_{2t_L,t_L}\big(\mathfrak{N}p^\alpha;\psi,\psi'\cdot\psi_\mathfrak{p}^{-1}(\psi_\mathfrak{p}')^{-2};O\big)$.
\end{lemma}
\begin{proof}
	For $u,v\in\cal{O}^\times_{L,\mathfrak{p}}$ such that $uv\equiv-1  \pmod{\mathfrak{p}^\alpha}$ we have the following identity in $\mrm{GL}_2(L_\frak{p})$
	\begin{equation}\label{matrix identity}
	\begin{pmatrix}
		0&-1\\
		\varpi_\mathfrak{p}^\alpha&0
	\end{pmatrix}^{-1}
	\begin{pmatrix}
		\varpi_\mathfrak{p}^\alpha&u\\
		0&1
	\end{pmatrix}
	=
	\begin{pmatrix}
		1&v\varpi_\mathfrak{p}^{-\alpha}\\
		0&1
	\end{pmatrix}
	\begin{pmatrix}
		v & (1+uv)\varpi_\mathfrak{p}^{-\alpha}\\
		-\varpi_\mathfrak{p}^\alpha & -u
	\end{pmatrix}.
	\end{equation}
	Let $\delta_{v,\alpha}=\begin{pmatrix}
		1&v\varpi_\mathfrak{p}^{-\alpha}\msf{d}_L\\
		0&1
	\end{pmatrix}\in G(\bb{A})$, denote by
	$\gamma_{u,\alpha}\in G(\bb{A})$ the matrix satisfying
	\[
	(\gamma_{u,\alpha})_\frak{p}=\begin{pmatrix}
		\varpi_\mathfrak{p}^\alpha&u\\
		0&1
	\end{pmatrix},
	\qquad (\gamma_{u,\alpha})_v=\mathbbm{1}_2\qquad \text{for}\qquad v\not=\frak{p}
	\]
	and by
		$\beta_{u,v,\alpha}\in G(\bb{A})$ the matrix satisfying
		\[
		(\beta_{u,v,\alpha})_\frak{p}=\begin{pmatrix}
		v & (1+uv)\varpi_\mathfrak{p}^{-\alpha}\\
		-\varpi_\mathfrak{p}^\alpha & -u
	\end{pmatrix},
		\qquad (\beta_{u,v,\alpha})_v=\begin{pmatrix}
		1 & -(\msf{d}_L)_v\\
		0 & 1
	\end{pmatrix}\qquad \text{for}\qquad v\not=\frak{p}.
		\]
		Then equation (\ref{matrix identity}) implies that
		\[
		(\tau_{\frak{p}^\alpha})^{-1}\cdot\gamma_{u,\alpha}=\delta_{v,\alpha}\cdot\beta_{u,v,\alpha}
		\]
	Right translation by $\beta_{u,v,\alpha}$ on a Hilbert cuspform
	corresponds to the action of the element $\langle v^{-1}, v^{-2}\rangle$ in $\bb{G}_L(K)$ and we can write
	\begin{equation}\label{AL equation}
	\msf{f}{\lvert\gamma_{u,\alpha}\lvert \tau_{\mathfrak{p}^\alpha}^{-1}}
	=
	(\langle v^{-1}, v^{-2}\rangle\msf{f}){\lvert\delta_{v,\alpha}}=
	\psi(v^{-1})\psi'(v^{-2})\cdot\msf{f}{\lvert\delta_{v,\alpha}}.
	\end{equation}
	On one hand, summing the right hand side of (\ref{AL equation}) over $v\in(\cal{O}_{L}/\mathfrak{p}^\alpha)^\times$ gives
	\[
	\sum_{v\in(\cal{O}_{L}/\mathfrak{p}^\alpha)^\times}
		(\psi(\psi')^{2})^{-1}(v) \msf{f}{\lvert\delta_{v,\alpha}}
	=
	G(\psi_\mathfrak{p}(\psi'_\mathfrak{p})^{2})\cdot\big(\msf{f}\star\psi_\mathfrak{p}(\psi'_\mathfrak{p})^{2}\big).
	\]
	On the other hand, summing the left hand side of (\ref{AL equation}) over $u\in(\cal{O}_{L}/\mathfrak{p}^\alpha)^\times$ gives
	\[\begin{split}
	\sum_{u\in(\cal{O}_L/\mathfrak{p}^\alpha)^\times}
	\msf{f}{\lvert\gamma_{u,\alpha}}{\lvert \tau_{\mathfrak{p}^\alpha}^{-1}}
	&=
	\Big(\sum_{u\in(\cal{O}_{L}/\mathfrak{p}^\alpha)}
	\msf{f}{\lvert\gamma_{u,\alpha}}\Big){\lvert \tau_{\mathfrak{p}^\alpha}^{-1}}
	-
	\Big(\sum_{u'\in(\mathfrak{p}\cal{O}_{L}/\mathfrak{p}^{\alpha})}
	\msf{f}{\lvert\gamma_{u',\alpha-1}}
	\lvert \tau_{\mathfrak{p}^\alpha}^{-1}
	\Big)\lvert\mbox{\tiny $\begin{pmatrix}
		1&\\
		&\varpi_\mathfrak{p}
	\end{pmatrix}$} \\
	&=
	\Big(U(\varpi_{\mathfrak{p}})^\alpha\msf{f} \Big){\lvert\tau_{\mathfrak{p}^\alpha}^{-1}}
	- 
	\Big(U(\varpi_{\mathfrak{p}})^{\alpha-1}\msf{f}\Big)\lvert \tau_{\mathfrak{p}^\alpha}^{-1}\lvert \mbox{\tiny $\begin{pmatrix}
		1&\\
		&\varpi_\mathfrak{p}
	\end{pmatrix}$},
		\end{split}\]
	where in the first equality we used the following identity for
	$u\in\mathfrak{p}(\cal{O}_L/\mathfrak{p}^\alpha)$:
	\[\begin{pmatrix}
				0&-1\\
				\varpi_\mathfrak{p}^\alpha&0
			\end{pmatrix}^{-1}
			\begin{pmatrix}
				\varpi_\mathfrak{p}^\alpha&u\\
				0&1
			\end{pmatrix}=\begin{pmatrix}
				1&0\\
				0&\varpi_{\mathfrak{p}}
			\end{pmatrix}\begin{pmatrix}
				0&-1\\
				\varpi_\mathfrak{p}^\alpha&0
			\end{pmatrix}^{-1}
			\begin{pmatrix}
				\varpi_\mathfrak{p}^{\alpha-1}&u/\varpi_\mathfrak{p}\\
				0&1
			\end{pmatrix}.\]
	Finally, taking the $\mathfrak{p}$-depletion of both expressions we obtain
	\[
	\msf{a}_p(\varpi_\mathfrak{p},\msf{f})^\alpha \cdot
	\Big(\msf{f}\lvert\tau^{-1}_{\mathfrak{p}^\alpha}\Big)^{\mbox{\tiny $[\mathfrak{p}]$}}
	=
	G(\psi_\mathfrak{p}(\psi'_\mathfrak{p})^{2})\cdot\Big(\msf{f}\star\psi_\mathfrak{p}(\psi'_\mathfrak{p})^{2}\Big)
	\]
	because
	\[
	\left[\big(U_{\mathfrak{p}}^{\alpha-1}\msf{f}\big)\lvert \tau_{\mathfrak{p}^\alpha}^{-1}\lvert \mbox{\tiny $\begin{pmatrix}
		1&\\
		&\varpi_\mathfrak{p}
	\end{pmatrix}$}\right]^{\mbox{\tiny $[\mathfrak{p}]$}}
	= 0
	\qquad  \text{and}\qquad \Big(\msf{f}\star\psi_\mathfrak{p}(\psi'_\mathfrak{p})^{2}\Big)^{\mbox{\tiny $[\mathfrak{p}]$}}=\msf{f}\star\psi_\mathfrak{p}(\psi'_\mathfrak{p})^{2}.
		\]
\end{proof}

\section{Automorphic $p$-adic $L$-functions}
Let $L/\bb{Q}$ be a real quadratic extension, $N$ a positive integer and $\frak{Q}$ an $\cal{O}_L$-ideal. 
We consider a primitive Hilbert cuspform $\msf{g}_\circ\in S_{t_L,t_L}(\mathfrak{Q};\chi_\circ; \overline{\bb{Q}})$ over $L$ with associated Artin representation $\varrho:\Gamma_L\to\mrm{GL}_2(\bb{C})$, and a primitive elliptic cuspform $\msf{f}_\circ\in S_{2,1}(N; \psi_\circ; \overline{\bb{Q}})$  such that the characters $\chi_\circ:\mrm{Cl}^+_L(\mathfrak{Q})\to O^\times$, $\psi_\circ:\mrm{Cl}^+_\bb{Q}(N)\to O^\times$  satisfy 
\begin{equation}\label{assumption characters}
\chi_{\circ\lvert\bb{Q}}\cdot\psi_\circ\equiv 1.
\end{equation} 
Denote by $\Pi=\pi_{\msf{g}_\circ}^u\otimes\sigma_{\msf{f}_\circ}^u$ the unitary cuspidal automorphic representation of $\mrm{GL}_2(\bb{A}_{L\times\bb{Q}})$ associated to the two cuspforms. The  restriction to the ideles $\bb{A}_\bb{Q}^\times$ of its central character $\omega_\Pi$ is trivial under hypothesis ($\ref{assumption characters}$). Therefore the twisted triple product $L$-function $L(s,\Pi,\mrm{r})$ admits meromorphic continuation to $\bb{C}$, functional equation $L(s,\Pi,\mrm{r})=\epsilon(s,\Pi,\mrm{r})L(1-s,\Pi,\mrm{r})$ and it is holomorphic at the center $s=1/2$ (\cite{BlancoFornea}, Section 3.1). By a direct inspection of Euler products (\cite{MicAnalytic}, Section 5) one deduces that the twisted triple product $L$-function does not vanish at its center if and only if the $L$-function $L\big(\msf{f}_\circ,\mrm{As}(\varrho),s\big)$ associated to the Galois representation $\mrm{As}(\varrho)\otimes\mrm{V}_{\msf{f}_\circ}$ -- where $\mrm{V}_{\msf{f}_\circ}$ is a $\msf{f}_\circ$-isotypic quotient of the \'etale cohomology of a modular curve -- does not vanish at its center:
\[
L\Big(\frac{1}{2},\Pi,\mrm{r}\Big)\not=0\qquad \iff\qquad L\Big(\msf{f}_\circ,\mrm{As}(\varrho),1\Big)\not=0.
\]
We let $\mu:\bb{A}_\bb{Q}^\times\to\bb{C}^\times$ denote the quadratic character attached to $L/\bb{Q}$ by class field theory and assume from now on that
\begin{equation}\label{epsilonfactors}
\epsilon_\ell\left(\frac{1}{2},\Pi_\ell,\mrm{r}\right)\cdot\mu_\ell(-1)=+1\qquad\qquad\forall\ \text{prime}\ \ell.
\end{equation}
\begin{remark}
The assumption on local $\epsilon$-factors can be satisfied by requiring that 
\begin{equation}\label{humhum}
   \text{$N$ splits in $L$ and is coprime to  $\frak{Q}$.}
\end{equation} 
Indeed, at a prime $\ell$ split in $L$, Prasad's calculations for triple product $L$-functions (\cite{prasadtrilinear}, Theorems 1.2, 1.4 and their proofs) ensure that as long as one on the local automorphic representations is an unramified principal series, condition \eqref{epsilonfactors} is satisfied (see also \cite{DR}, Introduction). At a non-split prime $\ell$, \eqref{humhum} implies that the local automorphic representation attached to $\msf{f}_\circ$ is an unramified principal series. Therefore, Prasad's calculations for twisted triple product $L$-functions (\cite{epsilonprasad}, Theorems $\text{B},\text{D}$ and their proofs) ensure that condition \eqref{epsilonfactors} is satisfied (see also \cite{YLiu}, Assumption 1.1, E2-E3).
\end{remark}

\begin{theorem}\label{Ichino}
	 The central $L$-value $L\big(\msf{f}_\circ,\mrm{As}(\varrho),1\big)$ does not vanish if and only if there exists an integer $M$ supported on the prime divisors of $N\cdot\mrm{N}_{L/\bb{Q}}(\frak{Q})\cdot d_{L/\bb{Q}}$, and an eigenform for the good Hecke operators $\breve{\msf{g}}_\circ\in S_{t_L,t_L}(M\cal{O}_L;O)[\msf{g}_\circ]$ such that the Petersson inner product
	\[
	\Big\langle\zeta^*(\breve{\msf{g}}_\circ)\otimes\lvert-\rvert_{\bb{A}_\bb{Q}}^{-1},\ \msf{f}_\circ^*\Big\rangle\not=0
	\]
	does not vanish. Here $\msf{f}_\circ^*\in S_{2,1}(N;\psi_\circ^{-1};\overline{\bb{Q}})$ denotes the elliptic eigenform whose Hecke eigenvalues are the complex conjugates of those of $\msf{f}_\circ$.
\end{theorem}
\begin{proof}
	This is a refinement of a special case of (\cite{BlancoFornea}, Theorem 3.2 $\&$ Lemma 3.4). In \emph{loc. cit.} the theorem is proved for some $\breve{\msf{g}}_\circ\in S_{t_L,t_L}(M_0\cal{O}_L;O)[\msf{g}_\circ]$ and some $\breve{\msf{f}}_\circ^*\in S_{2,1}(M_0;O)[\msf{f}_\circ^*]$ where $M_0$ is an integer supported on the prime divisors of $N\cdot\mrm{N}_{L/\bb{Q}}(\frak{Q})\cdot d_{L/\bb{Q}}$. We claim that, up to modifying  $\breve{\msf{g}}_\circ$, the statement holds for $\breve{\msf{f}}_\circ^*=\msf{f}_\circ^*$. Indeed, let $\phi=(\breve{\msf{g}}_\circ)^u\otimes(\breve{\msf{f}}^{\frak{J}}_\circ)^u\in \Pi$ be the test vector provided by (\cite{BlancoFornea}, Lemma 3.4). Now, for any place $v$, Ichino's local functional
 \[
 \mrm{I}_v\in \mrm{Hom}_{\mrm{GL}_2(\bb{Q}_v)\times \mrm{GL}_2(\bb{Q}_v)}\big(\Pi_v\otimes\overline{{\Pi}_v},\bb{C}\big)
 \]
 is (an explicit multiple of) an integral of matrix coefficients
 \[
 \mrm{I}_v(\psi_v\otimes\psi_v')\overset{\cdot}{=}\int_{\bb{Q}_v^\times\backslash\mrm{GL}_2(\bb{Q}_v)}\langle\Pi_v(x_v)\psi_v,\psi_v'\rangle_v\hspace{0.5mm}\mrm{d}^\times x_v,
 \]
 where $\langle\hspace{2mm},\hspace{1mm}\rangle_v$ is a $\big(\mrm{GL}_2(L_v)\times\mrm{GL}_2(\bb{Q}_v)\big)$-invariant pairing between $\Pi_v$ and its complex conjugate $\overline{{\Pi}_v}$ (which we have identified with the contragradient $\widetilde{\Pi}_v)$. As the local representation $(\sigma_{\msf{f}_\circ}^u)_v$ is irreducible, it is spanned by $\mrm{GL}_2(\bb{Q}_v)$-translates of the local component at $v$ of $(\msf{f}^{\frak{J}}_\circ)^u$. Therefore, using the invariance of the pairings  $\langle\hspace{2mm},\hspace{1mm}\rangle_v$ and by changing coordinates in the integrals of matrix coefficients, the non-vanishing of the $\mrm{I}_v(\phi_v\otimes\overline{\phi}_v)$'s is equivalent to the non-vanishing of Ichino's local functionals on the local components of a test vector $\psi$ of the form $(\widetilde{\msf{g}}_\circ)^u\otimes(\msf{f}^{\frak{J}}_\circ)^u$ for some $\widetilde{\msf{g}}_\circ\in S_{t_L,t_L}(M\cal{O}_L;O)[\msf{g}_\circ]$ where $M$ is an integer (possibly different from $M_0$) supported on the prime divisors of $N\cdot\mrm{N}_{L/\bb{Q}}(\frak{Q})\cdot d_{L/\bb{Q}}$.
\end{proof}

\subsection{Construction}
Let $p$ be a rational prime split in $L$, coprime to the levels $\mathfrak{Q},N$ and such that $\msf{g}_\circ$, $\msf{f}_\circ$ are $p$-ordinary.

\begin{definition}
    Let $\Gamma$ denote the $p$-adic group $1+p\bb{Z}_p$ and
    $\bs{\Lambda} = O\llbracket \Gamma \rrbracket$ its completed group ring. We consider the weight space 
    $\bs{\cal{W}}=\mrm{Spf}(\bs{\Lambda})^\mrm{rig}$ whose arithmetic points correspond to continuous homomorphisms of the form
	\[
	\msf{w}_{\ell,\psi}:\bs{\Lambda} \longrightarrow \overline{\bb{Q}}_p ,\qquad [u]\mapsto\psi(u)u^{\ell-2}
	\]
	for $\ell\in\bb{Z}_{\ge1}$ and $\psi:\Gamma\to\bar{\bb{Q}}_p^\times$ a finite order character.
\end{definition}
\noindent Given the character
\begin{equation}
\boldsymbol{\chi}:\mrm{Cl}_L^+(\mathfrak{Q}p)\longrightarrow O^\times,\qquad z\mapsto \chi_\circ(z)\theta_L^{-1}(\bar{z})
\end{equation}
one can define the surjection
\begin{equation}
\phi_{\boldsymbol{\chi}}:O\llbracket\bb{G}_L(\mathfrak{Q})\rrbracket\twoheadrightarrow \bs{\Lambda},\qquad [(z,a)]\mapsto\bs{\chi}(z) [\xi_z^{-t_L}]. 
\end{equation}

\begin{remark}
Any arithmetic point $\mrm{P}:O\llbracket\bb{G}_L(\mathfrak{Q})\rrbracket\to\overline{\bb{Q}}_p$ of weight $(\ell t_L,t_L)$ and character $(\chi_\circ\theta_L^{1-\ell}\chi^{-1},\mathbbm{1})$,
for $\chi:Z_L(\mathfrak{Q})\to O^\times$ a $p$-power order character factoring through the norm, $\chi=\chi_{\mbox{\tiny $\spadesuit$}}\circ\mrm{N}_{L/\bb{Q}}$, factors through $\phi_{\boldsymbol{\chi}}$:
\begin{equation} \mrm{P}_{\ell t_L,t_L,\chi_\circ\theta_L^{1-\ell}\chi^{-1},\mathbbm{1}}=\msf{w}_{\ell,\chi_{\mbox{\tiny $\spadesuit$}}}\circ\phi_{\bs{\chi}}
\end{equation}
Furthermore, any arithmetic point $\mrm{P}:O\llbracket\bb{G}_L(\mathfrak{Q})\rrbracket\to\overline{\bb{Q}}_p$ factoring through $\phi_{\bs{\chi}}$  has that form.
\end{remark}

\begin{definition}
Given $\chi:Z_L(\mathfrak{Q})\to O^\times$ a $p$-power order character factoring through the norm, we denote by $\chi_{\mbox{\tiny $\spadesuit$}}: \Gamma\to O^\times$ the character satisfying $\chi=\chi_{\mbox{\tiny $\spadesuit$}}\circ\mrm{N}_{L/\bb{Q}}$.
\end{definition}

\noindent Fix  $\msf{g}_\circ^{\mbox{\tiny $(p)$}}$ an ordinary $p$-stabilization of $\msf{g}_\circ$ and let $\scr{G}_\mrm{n.o.}\in \overline{\mbf{S}}_L^\mrm{n.o.}(\mathfrak{Q};\boldsymbol{\chi};\mbf{I}_{\scr{G}_\mrm{n.o.}})$ be the nearly ordinary Hida family passing through it, that is, there exists an arithmetic point $\mrm{P}_\circ\in\cal{A}(\mbf{I}_{\scr{G}_{\mrm{n.o.}}})$ of weight $(t_L,t_L)$ and character $(\chi_\circ,\mathbbm{1})$ such that  $\scr{G}_\mrm{n.o.}(\mrm{P}_\circ)=\msf{g}_\circ^{\mbox{\tiny $(p)$}}$. If we set
	\begin{equation}\label{OrdinaryFamily}
	\mbf{I}_{\scr{G}}:=\mbf{I}_{\scr{G}_\mrm{n.o.}}\otimes_{\phi_{\bs{\chi}}} \bs{\Lambda}
	\qquad\text{and}\qquad \scr{G}:=(1\otimes\phi_{\bs{\chi}})\circ\scr{G}_\mrm{n.o.},
	\end{equation}
then $\scr{G}\in \overline{\mbf{S}}_L^\mrm{ord}(\mathfrak{Q};\boldsymbol{\chi};\mbf{I}_{\scr{G}})$ is the ordinary Hida family passing through $\msf{g}_\circ^{\mbox{\tiny $(p)$}}$ (\cite{Wiles-rep}, Theorem 3).

\subsubsection{The $\Lambda$-adic cuspform.}
Write $\cal{P}$ for the set of prime $\cal{O}_L$-ideals dividing $p$. The choice of ordinary $p$-stabilization of $\msf{g}_\circ$ determines an ordinary $p$-stabilization $\breve{\msf{g}}_\circ^{\mbox{\tiny $(p)$}}$ of any cuspform $\breve{\msf{g}}_\circ$ arising from Theorem $\ref{Ichino}$.  Let $\breve{\scr{G}}$ be the $\mathbf{I}_{\scr{G}}$-adic cuspform passing through $\breve{\msf{g}}_\circ^{\mbox{\tiny $(p)$}}$ defined as in  (\cite{DR} Section 2.6). 	
For a choice a prime $\frak{p}\in\cal{P}$ we define a homomorphism of $\bs{\Lambda}_{L,\bs{\chi}}$-modules $d_\frak{p}^{\bfcdot}\breve{\scr{G}}^{\mbox{\tiny $[\cal{P}]$}}:\mathbf{h}_L(M\cal{O}_L;O)\longrightarrow\mathbf{I}_{\scr{G}}$  by
\[
d_\frak{p}^{\bfcdot}\breve{\scr{G}}^{\mbox{\tiny $[\cal{P}]$}}\left(\langle z\rangle\mathbf{T}(y)\right)
=
\begin{cases} 
\breve{\scr{G}}\left(\langle z\rangle\mathbf{T}(y)\right)
\phi_{\boldsymbol{\chi}}\big([ y_\mathfrak{p}, 1]\big)
y_\mathfrak{p}^{-1} \qquad &\text{if}\ y_p\in\cal{O}_{L,p}^\times
\\
0\qquad&\text{otherwise}.
\end{cases}
\]
 Using diagonal restriction (\cite{BlancoFornea}, Section 2.3)
\[
\zeta: \mbf{h}_\bb{Q}(M;O)\to\mathbf{h}_L(M\cal{O}_L;O),\qquad \zeta([z,a])=[\Delta(z),\Delta(a)]a^{-1},
\] 
we define the  ordinary
$\mathbf{I}_{\scr{G}}$-adic cuspform  
\begin{equation}
e_\mrm{ord}\zeta^*\big(d_\frak{p}^{\bfcdot}\breve{\scr{G}}^{\mbox{\tiny $[\cal{P}]$}}\big)^\dagger\in \overline{\mbf{S}}_\bb{Q}^\mrm{ord}\big(M;\psi_\circ^{-1};\mathbf{I}_{\scr{G}}\big)
\end{equation}
by setting 
\begin{equation}\label{our rule}
\zeta^*\big(d_\frak{p}^{\bfcdot}\breve{\scr{G}}^{\mbox{\tiny $[\cal{P}]$}}\big)^{\dagger}\big(\langle z\rangle\mbf{T}(y)\big)=\theta_\bb{Q}(\bar{y})\cdot\phi_{\bs{\chi}}\big([\Delta(\xi_z)^{-1}\Delta(\xi_y)^{-1/2},1]\big)\cdot d_\frak{p}^{\bfcdot}\breve{\scr{G}}^{\mbox{\tiny $[\cal{P}]$}}\Big(\zeta\left[\langle z\rangle\mbf{T}(y)\right]\Big),
\end{equation}
where $\Delta:1+p\bb{Z}_p\hookrightarrow1+p\cal{O}_{L,p}$ is the diagonal embedding.
\begin{remark}
By Definition \ref{def I-adic cuspforms}, proving that $e_\mrm{ord}\zeta^*\big(d_\frak{p}^{\bfcdot}\breve{\scr{G}}^{\mbox{\tiny $[\cal{P}]$}}\big)^\dagger$ is an ordinary $\mbf{I}_\scr{G}$-adic cuspform amounts to showing that \eqref{our rule} determines a $\bs{\Lambda}_{\psi_\circ^{-1}}$-linear homomorphism. Then, Proposition \ref{specializations} justifies our claim.
\end{remark}
\noindent Recall that for each $\mu\in \mrm{I}_L$ there is a differential operator on $p$-adic cuspforms  given on adelic $q$-expansions by $\msf{a}_p(y,d_\mu\msf{g})=y_p^\mu\msf{a}_p(y,\msf{g})$.  (\cite{pHida}, Section $6\text{G}$). 
\begin{proposition}\label{specializations}
	Let $\mrm{P}\in \cal{A}_{\bs{\chi}}(\mbf{I}_\scr{G})$ be an arithmetic point of weight $(\ell t_L,t_L)$ and character $(\chi_\circ\theta_L^{1-\ell}\chi^{-1},\mathbbm{1})$. If we define the local character $\chi_\mathfrak{p}\theta^{\ell-1}_{L,\mathfrak{p}}:\ \cal{O}_{L,\mathfrak{p}}^\times\longrightarrow O^\times$ by 
	$x\mapsto \chi\theta^{\ell-1}_{L}(x)$ and $\mu\in\mrm{I}_L$ is the embedding inducing $\frak{p}$ then 
	\[
	e_\mrm{ord}\zeta^*\big(d_\frak{p}^{\bfcdot}\breve{\scr{G}}^{\mbox{\tiny $[\cal{P}]$}}\big)^\dagger(\mrm{P})=e_\mrm{ord}\zeta^*\Big[d_\mu^{1-\ell}\big(\breve{\msf{g}}^{\mbox{\tiny $[\cal{P}]$}}_\mrm{P}\star\chi^{-1}_\mathfrak{p}\theta_{L,\mathfrak{p}}^{1-\ell}\big)\Big]\otimes\theta_\bb{Q}^{\ell-1}\chi_{\mbox{\tiny $\spadesuit$}}\lvert-\rvert_{\bb{A}_\bb{Q}}^{\ell-2}
	\]
	is a classical cuspform of weight $(2,1)$ and character $(\psi_\circ^{-1},\mathbbm{1})$.
\end{proposition}
\begin{proof}
	The direct computation 
	\[\begin{split}
	\mrm{P}\circ e_\mrm{ord}\zeta^*&\big(d_\frak{p}^{\bfcdot}\breve{\scr{G}}^{\mbox{\tiny $[\cal{P}]$}}\big)^{\dagger}([z,a])=\theta_\bb{Q}(a^{-1})\cdot\mrm{P}\circ\phi_{\bs{\chi}}([\Delta(\xi_z)^{-1}\Delta(\xi_a)^{1/2},1])\cdot\mrm{P}\circ d_\frak{p}^{\bfcdot}\breve{\scr{G}}^{\mbox{\tiny $[\cal{P}]$}}\Big([\Delta(z),\Delta(a)]a^{-1}\Big)\\
	&=\theta_{\bb{Q}}^{-1}(a)\cdot \chi^{2}_{\mbox{\tiny $\spadesuit$}}(z)\eta_\bb{Q}^{2(2-\ell)}(z)\chi^{-1}_{\mbox{\tiny $\spadesuit$}}(a)\eta_\bb{Q}^{\ell-2}(a)\cdot \mrm{P}\circ\phi_{\bs{\chi}}([\Delta(z)\Delta(a)_\mathfrak{p}^{-1},\Delta(a)]) \\
&=\theta_{\bb{Q}}^{-1}(a)\cdot \chi^{2}_{\mbox{\tiny $\spadesuit$}}(z)\eta_\bb{Q}^{2(2-\ell)}(z)\chi^{-1}_{\mbox{\tiny $\spadesuit$}}(a)\eta_\bb{Q}^{\ell-2}(a)\cdot (\chi_\circ\theta_L^{1-\ell}\chi^{-1})(\Delta(z)\Delta(a)_\mathfrak{p}^{-1})\varepsilon_L^{\ell-2}(\Delta(z)\Delta(a)_\mathfrak{p}^{-1})\\
	&=\psi^{-1}_\circ(z)
	\end{split}\] 
	shows that $e_\mrm{ord}\zeta^*\big(d_\frak{p}^{\bfcdot}\breve{\scr{G}}^{\mbox{\tiny $[\cal{P}]$}}\big)^\dagger(\mrm{P})$ is a cuspform of weight $(2,1)$ and character $(\psi_\circ^{-1},\mathbbm{1})$. Then we compute 
	\[
	\msf{a}_p(y,d_\frak{p}^{\bfcdot}\breve{\scr{G}}^{\mbox{\tiny $[\cal{P}]$}}(\mrm{P}))= y_\mathfrak{p}^{1-\ell}\msf{a}_p(y,\breve{\msf{g}}^{\mbox{\tiny $[\cal{P}]$}}_\mrm{P})\cdot\chi^{-1}(y_\mathfrak{p})\theta_L^{1-\ell}(y_\mathfrak{p})
	\] which implies 
	\[
	\zeta^*d_\frak{p}^{\bfcdot}\breve{\scr{G}}^{\mbox{\tiny $[\cal{P}]$}}(\mrm{P})=\zeta^*\Big[d_\mu^{1-\ell}\big(\breve{\msf{g}}^{\mbox{\tiny $[\cal{P}]$}}_\mrm{P}\star\chi^{-1}_\mathfrak{p}\theta_{L,\mathfrak{p}}^{1-\ell}\big)\Big].
	\]
	Finally,
	\[\begin{split}
	\msf{a}_p\Big(y,\zeta^*\big(d_\frak{p}^{\bfcdot}\breve{\scr{G}}^{\mbox{\tiny $[\cal{P}]$}}\big)^\dagger(\mrm{P})\Big)&=\theta_{\bb{Q}}(y)\cdot \chi_{\mbox{\tiny $\spadesuit$}}(y)\eta_\bb{Q}^{2-\ell}(y)\cdot \msf{a}_p\Big(y, \zeta^*d_\frak{p}^{\bfcdot}\breve{\scr{G}}^{\mbox{\tiny $[\cal{P}]$}}(\mrm{P})\Big)\\
	&= \theta^{\ell-1}_{\bb{Q}}(y) \chi_{\mbox{\tiny $\spadesuit$}}(y)\varepsilon_\bb{Q}^{2-\ell}(y)\cdot \msf{a}_p\Big(y, \zeta^*d_\frak{p}^{\bfcdot}\breve{\scr{G}}^{\mbox{\tiny $[\cal{P}]$}}(\mrm{P})\Big)
	\end{split}\]
	proves the last claim
	\[
	e_\mrm{ord}\zeta^*\big(d_\frak{p}^{\bfcdot}\breve{\scr{G}}^{\mbox{\tiny $[\cal{P}]$}}\big)^\dagger(\mrm{P})=e_\mrm{ord}\zeta^*\Big[d_\mu^{1-\ell}\big(\breve{\msf{g}}^{\mbox{\tiny $[\cal{P}]$}}_\mrm{P}\star\chi^{-1}_\mathfrak{p}\theta_{L,\mathfrak{p}}^{1-\ell}\big)\Big]\otimes\theta_\bb{Q}^{\ell-1}\chi_{\mbox{\tiny $\spadesuit$}}\lvert-\rvert_{\bb{A}_\bb{Q}}^{\ell-2}.
	\]
\end{proof}

\begin{corollary}\label{diagonal restriction family}
	We have
	\[
	e_{\mrm{ord}}\zeta^*\big(d_\frak{p}^{\bfcdot}\breve{\scr{G}}^{\mbox{\tiny $[\cal{P}]$}}\big)^\dagger\in S^\mrm{ord}_{2,1}\big(Mp;\psi^{-1}_\circ;O\big)\otimes_O\mbf{I}_{\scr{G}}.
	\]
\end{corollary}
\begin{proof}
By Proposition $\ref{specializations}$
\[
e_{\mrm{ord}}\zeta^*\big(d_\frak{p}^{\bfcdot}\breve{\scr{G}}^{\mbox{\tiny $[\cal{P}]$}}\big)^\dagger(\mrm{P})\in S^\mrm{ord}_{2,1}\big(Mp;\psi^{-1}_\circ;O\big)
\] for any arithmetic crystalline point $\mrm{P}\in\cal{A}_{\bs{\chi}}(\mbf{I}_\scr{G})$ of weight $(\ell t_L,t_L)$ and character $(\chi_\circ\theta_L^{1-\ell},\mathbbm{1})$.
By the density of such crystalline points, the homomorphism $e_{\mrm{ord}}\zeta^*\big(d_\frak{p}^{\bfcdot}\breve{\scr{G}}^{\mbox{\tiny $[\cal{P}]$}}\big)^\dagger$ factors through the reduction to weight $(2,1)$, level $Mp$ and character $\psi^{-1}_\circ$ of $\mbf{h}_\bb{Q}^\mrm{ord}(M;O)$.	
\end{proof}

\subsubsection{The automorphic $p$-adic $L$-function.}
Let  $\msf{f}_\circ^*\in S_{2,1}(N,\psi_\circ^{-1};\overline{\bb{Q}})$ be the elliptic eigenform whose Hecke eigenvalues are the complex conjugates of those of $\msf{f}_\circ$, then we can write 
\[
e_{\msf{f}_\circ^{*\mbox{\tiny $(p)$}}}e_{\mrm{ord}}\zeta^*\big(d_\frak{p}^{\bfcdot}\breve{\scr{G}}^{\mbox{\tiny $[\cal{P}]$}}\big)^\dagger=\sum_{d\mid{(M/N)}}\boldsymbol{\lambda}_d\cdot \msf{f}_\circ^{*\mbox{\tiny $(p)$}}(q^d)\qquad\text{for}\quad\{\bs{\lambda}_d\}_d\subset\mbf{I}_\scr{G},
\] and define 
\[
\mbf{L}_p(\breve{\scr{G}},\msf{f}_\circ)=\sum_{d\mid{(M/N)}}\boldsymbol{\lambda}_d\cdot \frac{\big\langle\msf{f}_\circ^{*\mbox{\tiny $(p)$}}(q^d),\msf{f}_\circ^{*\mbox{\tiny $(p)$}}(q)\big\rangle}{\big\langle\msf{f}_\circ^{*\mbox{\tiny $(p)$}},\msf{f}_\circ^{*\mbox{\tiny $(p)$}}\big\rangle}\in\mbf{I}_\scr{G}.
\]

\begin{definition}\label{autpadicLfun}
Set $\bs{\cal{W}}_\scr{G}=\mrm{Spf}(\mbf{I}_\scr{G})^\mrm{rig}$, then the automorphic $p$-adic $L$-function attached to $\big(\breve{\scr{G}}, \msf{f}_\circ\big)$ is the rigid-analytic function 
\[
\mathscr{L}^\mrm{aut}_p(\breve{\scr{G}},\msf{f}_\circ):\boldsymbol{\cal{W}}_{\scr{G}}\longrightarrow\bb{C}_p
\]
determined by
$\mbf{L}_p(\breve{\scr{G}},\msf{f}_\circ)\in\mbf{I}_\scr{G}$. 
\end{definition}
\noindent For any arithmetic point $\mrm{P}\in\boldsymbol{\cal{W}}_\scr{G}$ of weight $(\ell t_L,t_L)$ and character $(\chi_\circ\theta_L^{1-\ell}\chi^{-1},\mathbbm{1})$ we have
\[
\mathscr{L}^\mrm{aut}_p(\breve{\scr{G}},\msf{f}_\circ)(\mrm{P}) =\frac{\Big\langle e_\mrm{ord}\zeta^*\Big[d_\mu^{1-\ell}\big(\breve{\msf{g}}^{\mbox{\tiny $[\cal{P}]$}}_\mrm{P}\star\chi^{-1}_\mathfrak{p}\theta_{L,\mathfrak{p}}^{1-\ell}\big)\Big]\otimes\theta_\bb{Q}^{\ell-1}\chi_{\mbox{\tiny $\spadesuit$}}\lvert-\rvert_{\bb{A}_\bb{Q}}^{\ell-2},\ \msf{f}_\circ^{*\mbox{\tiny $(p)$}}\Big\rangle}{\Big\langle \msf{f}_\circ^{*\mbox{\tiny $(p)$}}, \msf{f}_\circ^{*\mbox{\tiny $(p)$}}\Big\rangle}.
\]

\begin{remark}
A different choice of  prime $\frak{p}$ above $p$ changes the automorphic $p$-adic $L$-function $\mathscr{L}^\mrm{aut}_p(\breve{\scr{G}},\msf{f}_\circ)$ by a sign. 
\end{remark}

\subsection{Weight one specialization}
 Consider the $T_0(\varpi_p)$-Hecke polynomial of $\msf{f}_\circ^*\in S_{2,1}(N;\psi_\circ^{-1};\overline{\bb{Q}})$
 \[ 1-\msf{a}_p(\varpi_p,\msf{f}_\circ^*)X+\psi^{-1}_\circ(p)pX^2=(1-\alpha_{\msf{f}_\circ^*}X)(1-\beta_{\msf{f}_\circ^*}X)
 \] 
  and suppose $\alpha_{\msf{f}_\circ^*}$  denotes the inverse of the root which is a $p$-adic unit. By defining
\[
\msf{E}(\msf{f}_\circ^*):=(1-\beta_{\msf{f}_\circ^*}\alpha_{\msf{f}_\circ^*}^{-1})
\] 
we can rewrite the values of the $p$-adic $L$-function at every arithmetic point $\mrm{P}\in\boldsymbol{\cal{W}}_\scr{G}$ as
	\begin{equation}\label{eq: second expression p-adic L-function}
	\mathscr{L}^\mrm{aut}_p(\breve{\scr{G}},\msf{f}_\circ)(\mrm{P}) =\frac{1}{\msf{E}(\msf{f}_\circ^*)}\frac{\left\langle e_\mrm{ord}\zeta^*\Big[d_\mu^{1-\ell}\big(\breve{\msf{g}}^{\mbox{\tiny $[\cal{P}]$}}_\mrm{P}\star\chi^{-1}_\mathfrak{p}\theta_{L,\mathfrak{p}}^{1-\ell}\big)\Big]\otimes\theta_\bb{Q}^{\ell-1}\chi_{\mbox{\tiny $\spadesuit$}}\lvert-\rvert_{\bb{A}_\bb{Q}}^{\ell-2},\ \msf{f}_\circ^{*}\right\rangle}{\big\langle \msf{f}_\circ^{*}, \msf{f}_\circ^{*}\big\rangle}.
\end{equation}
\begin{lemma}\label{specialvalue}
	We have
	\[
	\mathscr{L}^\mrm{aut}_p(\breve{\scr{G}},\msf{f}_\circ)(\text{P}_\circ)
	=
	\frac{\cal{E}_p^\mrm{sp}(\msf{g}_\circ,\msf{f}_\circ^*)}
	{\msf{E}(\msf{f}_\circ^*)\cal{E}_{1,p}(\msf{g}_\circ,\msf{f}_\circ^*)}
	\frac{\left\langle \zeta^*(\breve{\msf{g}}_\circ)\otimes\lvert-\rvert^{-1}_{\bb{A}_\bb{Q}}, \msf{f}_\circ^{*}\right\rangle}
	{\big\langle \msf{f}_\circ^{*}, \msf{f}_\circ^{*}\big\rangle}
	\]
	for
	\[
	\cal{E}^\mrm{sp}_{p}(\msf{g}_\circ,\msf{f}^*_\circ)=\underset{\bfcdot,\star\in\{\alpha,\beta\}}{\prod}\left(1-\bfcdot_1\star_2 \beta_{\msf{f}_\circ^*}^{-1}\right),\qquad\cal{E}_{1,p}(\msf{g}_\circ,\msf{f}^*_\circ)=1-\alpha_1\beta_1\alpha_2\beta_2(\beta_{\msf{f}^*_\circ})^{-2}.
	\]
	where $\alpha_i,\beta_i$ are the inverses of the roots of the $T(\mathfrak{p}_i)$-Hecke polynomial for $\msf{g}_\circ$, $i=1,2$.
\end{lemma}
\begin{proof}
	The value of the $p$-adic $L$-function at $\mrm{P}_\circ\in\bs{\cal{W}}_\scr{G}$ can be expressed as
	\[
	\mathscr{L}^\mrm{aut}_p(\breve{\scr{G}},\msf{f}_\circ)(\mrm{P}_\circ) =
	\frac{1}{\msf{E}(\msf{f}_\circ^*)}\frac{\left\langle e_\mrm{ord}\zeta^*\big(\breve{\msf{g}}^{\mbox{\tiny $[\cal{P}]$}}_\mrm{P}\big),\ \msf{f}_\circ^{*}\otimes\lvert-\rvert_{\bb{A}_\bb{Q}}\right\rangle}{\big\langle \msf{f}_\circ^{*}, \msf{f}_\circ^{*}\big\rangle}.
	\]
	The $T_0(\varpi_p)$-Hecke polynomial for the eigenform  $\msf{f}_\circ^{*}\otimes\lvert-\rvert_{\bb{A}_\bb{Q}}\in S_{2,2}(N;\psi_\circ^{-1};\overline{\bb{Q}})$ is 
	\[
	1-\msf{a}_p(\varpi_p,\msf{f}_\circ^*)p^{-1}X+\psi_\circ^{-1}(p)p^{-1}X^2=(1-\alpha_{\msf{f}^*_\circ\otimes\lvert-\rvert}X)(1-\beta_{\msf{f}^*_\circ\otimes\lvert-\rvert}X).
	\] 
	Therefore the inverse of the root which is a $p$-adic unit is 
	\[
	\alpha_{\msf{f}^*_\circ\otimes\lvert-\rvert}=\beta_{\msf{f}^*_\circ}\cdot p^{-1}.
	\]  
	The result follows applying (\cite{BlancoFornea}, Lemma 3.11):
	\[\begin{split}
	\mathscr{L}^\mrm{aut}_p(\breve{\scr{G}},\msf{f}_\circ)(\mrm{P}_\circ) &=
	\frac{1}{\msf{E}(\msf{f}_\circ^*)}\frac{\underset{\bfcdot,\star\in\{\alpha,\beta\}}{\prod}\left(1-\bfcdot_1\star_2 \big(\alpha_{\msf{f}_\circ^*\otimes\lvert-\rvert}\cdot p\big)^{-1}\right)}{1-\alpha_1\beta_1\alpha_2\beta_2\big(\alpha_{\msf{f}^*_\circ\otimes\lvert-\rvert}\cdot p\big)^{-2}}
	\frac{\Big\langle e_\mrm{ord}\zeta^*\big(\breve{\msf{g}}_\mrm{P}\big),\ \msf{f}_\circ^{*}\otimes\lvert-\rvert_{\bb{A}_\bb{Q}}\Big\rangle}{\big\langle \msf{f}_\circ^{*}, \msf{f}_\circ^{*}\big\rangle}\\
	&=
	\frac{\cal{E}_p^\mrm{sp}(\msf{g}_\circ,\msf{f}_\circ^*)}
	{\msf{E}(\msf{f}_\circ^*)\cal{E}_{1,p}(\msf{g}_\circ,\msf{f}_\circ^*)} \frac{\Big\langle e_\mrm{ord}\zeta^*\big(\breve{\msf{g}}_\mrm{P}\big),\ \msf{f}_\circ^{*}\otimes\lvert-\rvert_{\bb{A}_\bb{Q}}\Big\rangle}{\big\langle \msf{f}_\circ^{*}, \msf{f}_\circ^{*}\big\rangle}\\
	&=
	\frac{\cal{E}_p^\mrm{sp}(\msf{g}_\circ,\msf{f}_\circ^*)}
	{\msf{E}(\msf{f}_\circ^*)\cal{E}_{1,p}(\msf{g}_\circ,\msf{f}_\circ^*)}
	\frac{\left\langle \zeta^*(\breve{\msf{g}}_\circ)\otimes\lvert-\rvert^{-1}_{\bb{A}_\bb{Q}}, \msf{f}_\circ^{*}\right\rangle}
	{\big\langle \msf{f}_\circ^{*}, \msf{f}_\circ^{*}\big\rangle}.
	\end{split}\]
\end{proof}

\begin{corollary}\label{firststep}
	\[
	L\Big(\msf{f}_\circ,\mrm{As}(\varrho),1\Big)\not=0\qquad\iff\qquad\mathscr{L}^\mrm{aut}_p(\breve{\scr{G}},\msf{f}_\circ)(\text{P}_\circ)\not=0.
	\]
\end{corollary}
\begin{proof}
	It follows from Theorem \ref{Ichino} and Lemma $\ref{specialvalue}$. The $p$-adic valuations of  $\msf{E}(\msf{f}_\circ^*)$, $\cal{E}_{1,p}(\msf{g}_\circ,\msf{f}_\circ^*)$, and $\cal{E}_p^\mrm{sp}(\msf{g}_\circ,\msf{f}_\circ^*)$ show they are not zero.
\end{proof}


\section{Review of Hilbert modular varieties}\label{review HMV}

Recall the algebraic groups
\[
D = \mrm{Res}_{L/\bb{Q}}\big(\bb{G}_{m,L}\big)
,\qquad
G = \mrm{Res}_{L/\bb{Q}}\big(\mrm{GL}_{2,L}\big)
,\qquad
G^* = G \times_D \bb{G}_{m,\bb{Q}}.
\]
Given $K \le G(\widehat{\bb{Z}})$ a compact open subgroup we define
\begin{equation}
K^*= K\cap G^*(\bb{A}_f),\qquad K'=K \cap \mrm{GL}_2(\bb{A}_f)
\end{equation}
 with associated Shimura varieties
\begin{equation}
	\begin{split}
&S(K)(\bb{C}) := G(\bb{Q})_+ \backslash  \mathfrak{H}^2 \times G(\bb{A}_f)/K,\\
&
S^*(K^*)(\bb{C}) := G^*(\bb{Q})_+ \backslash  \mathfrak{H}^2 \times G^*(\bb{A}_f)/K^*,\\
&
Y(K')(\bb{C}) := \mrm{GL}_2(\bb{Q})_+ \backslash  \mathfrak{H} \times \mrm{GL}_2(\bb{A}_f)/K'.
\end{split}
\end{equation}
If $K$ is sufficiently small, then $S(K)(\bb{C})$, $S^*(K^*)(\bb{C})$ and $Y(K')(\bb{C})$ have smooth canonical models $S(K)$, $S^*(K^*)$ and $Y(K')$  defined over $\bb{Q}$. 

\subsubsection{Special level subgroups.}
Let $p$ be a rational prime split in $L$,
$p\cal{O}_L = \mathfrak{p}_1\mathfrak{p}_2$ and fix isomorphisms
$\cal{O}_{L,\mathfrak{p}_1} \simeq \bb{Z}_p$,
$\cal{O}_{L,\mathfrak{p}_2} \simeq \bb{Z}_p$
to identify elements of
$\cal{O}_{L,p} $
with pairs $(a_{\frak{p}_1},a_{\frak{p}_2})\in\bb{Z}_p\times\bb{Z}_p$.

\begin{definition}\label{LevelSubgroups1}
For any $\alpha\ge1$ and any compact open $K \subseteq G(\bb{A}_f)$ hyperspecial at $p$ we set 
\[
K_\diamond(p^\alpha):=\left\{\begin{pmatrix}a&b\\c&d \end{pmatrix}\in K_0(p^\alpha)\Big\lvert\  a_{\mathfrak{p}_1}\equiv a_{\mathfrak{p}_2},\ d_{\mathfrak{p}_1}\equiv d_{\mathfrak{p}_2}\pmod{p^{\alpha}}\right\},\quad
K_{\diamond,1}(p^\alpha):=K_\diamond(p^\alpha)\cap V_1(p^\alpha)
\]
\[
	K_{\mbox{\tiny $\mrm{X}$},1}(p^\alpha):=\left\{\begin{pmatrix}a&b\\c&d \end{pmatrix}\in K_0(p^\alpha)\Big\lvert\ a_{\mathfrak{p}_1}\equiv d_{\mathfrak{p}_2},\ d_{\mathfrak{p}_1}\equiv a_{\mathfrak{p}_2}\equiv1\pmod{p^{\alpha}}\right\},
\]
and
\[
	K_{\diamond,t}(p^\alpha):=\left\{\begin{pmatrix}a&b\\c&d \end{pmatrix}\in K_0(p^\alpha)\Big\lvert\  a_{\mathfrak{p}_1}d_{\mathfrak{p}_1}\equiv  a_{\mathfrak{p}_2}d_{\mathfrak{p}_2},\ d_{\mathfrak{p}_1}d_{\mathfrak{p}_2}\equiv 1 \pmod{p^{\alpha}}\right\}.
\]
\end{definition}

\begin{definition}\label{LevelSubgroups2}
For any $\alpha\ge1$ and any compact open $K \subseteq G(\bb{A}_f)$ hyperspecial at $p$ we set 

\[
	K_{\mbox{\tiny $\mrm{X}$}}(p^\alpha)
	:=
	\left\{\begin{pmatrix}a&b\\c&d \end{pmatrix}\in K_0(p^\alpha)\Big\lvert\  a_{\mathfrak{p}_1}d_{\mathfrak{p}_1}\equiv  a_{\mathfrak{p}_2}d_{\mathfrak{p}_2},\ d_{\mathfrak{p}_1}a_{\mathfrak{p}_2}\equiv 1 \pmod{p^{\alpha}}\right\},
\]
it is the subgroup of $K_0(p^\alpha)$   generated by $K_{\mbox{\tiny $\mrm{X}$},1}(p^\alpha)$ and matrices $\gamma$ of the form
\[\gamma_v=\mathbbm{1}_2\quad\text{for}\ v\not=p,\qquad\gamma_p=\begin{pmatrix}d^{-1}&\\&d \end{pmatrix}\quad\text{with}\quad d_{\mathfrak{p}_1}\equiv d_{\mathfrak{p}_2} \pmod{p^{\alpha}}.
\] 
\end{definition}

\subsubsection{Geometrically connected components.}
The determinant $\det: G\to D$ induces a bijection between the set of geometric connected components of $S(K)$ and $\text{cl}^+_L(K)=L^\times_+\backslash \bb{A}_{L,f}^{\times}/\det(K)$ the strict class group of $K$.
The natural surjection $\text{cl}^+_L(K)\twoheadrightarrow \text{cl}^+_L$ to the strict ideal class group of $L$ can be used to label the geometrically connected components of  $S(K)$ as follows. 
Fix fractional ideals $\mathfrak{c}_1,\dots,\mathfrak{c}_{h^+_L}$ forming a set of representatives of $\text{cl}^+_L$. For every such ideal $\frak{c}$ choose $[\frak{c}]_K\subseteq G(\bb{A}_f)$ a set of diagonal matrices with lower right entry equal to $1$ and  whose determinants represent the preimage of the class $[\mathfrak{c}]$ in $\text{cl}^+_L(K)$.   By strong approximation there is a decomposition
\[
S(K)(\bb{C})=G(\bb{Q})_+\backslash \mathfrak{H}^{2}\times G(\bb{A}_f)/K=\underset{[\mathfrak{c}]\in \text{cl}^+_L(K)}{\coprod} S^\mathfrak{c}(K)(\bb{C}),
\]
\begin{equation}\label{complex uniformization}
\text{where}\qquad S^\mathfrak{c}(K)(\bb{C})=\underset{g\in[\mathfrak{c}]_K}{\coprod}\Gamma(g,K)\backslash\mathfrak{H}^{2}\qquad\text{for}\qquad \Gamma(g,K)=gKg^{-1}\cap G(\bb{Q})_+.
\end{equation}

\begin{proposition}\label{prop comparison different models}
If $\cal{O}^\times_L$ does not contain a totally positive unit congruent to $-1$ modulo $p$, then the complex uniformizations of $S(K_\diamond(p^\alpha))$ and $S(K_{\mbox{\tiny $\mrm{X}$}}(p^\alpha))$ can be canonically identified:
\[
S(K_\diamond(p^\alpha))(\bb{C})=S(K_{\mbox{\tiny $\mrm{X}$}}(p^\alpha))(\bb{C})\qquad \forall \alpha\ge1.
\]
\end{proposition}
\begin{proof}
	First we note that $\mrm{det}\big(K_\diamond(p^\alpha)\big)$ and $\mrm{det}\big(K_{\mbox{\tiny $\mrm{X}$}}(p^\alpha)\big)$ both equal  $(\bb{Z}_p^\times+p^\alpha\cal{O}_{L,p})\mrm{det}\big(K^p\big)$, thus  we can choose the same set of matrices $[\frak{c}]_{K_\diamond(p^\alpha)}=[\frak{c}]_{K_{\mbox{\tiny $\mrm{X}$}}(p^\alpha)}$
	for any fractional ideal $\frak{c}$. To conclude we claim that for every matrix $g$ in those sets we have
	\[ 
	\Gamma(g,K_\diamond(p^\alpha))=\Gamma(g,K_{\mbox{\tiny $\mrm{X}$}}(p^\alpha)).
	\]
    We only write the argument showing  $\Gamma(g,K_\diamond(p^\alpha))\subseteq\Gamma(g,K_{\mbox{\tiny $\mrm{X}$}}(p^\alpha))$ because the other inclusion is analogous. Let $\gamma$ be a matrix in $\Gamma(g,K_\diamond(p^\alpha))$, then its entries $a_\gamma, b_\gamma, c_\gamma, d_\gamma\in L$ satisfy $(a_\gamma)_p, (d_\gamma)_p\in\bb{Z}_p^\times+p^\alpha\cal{O}_{L,p}$, $(c_\gamma)_p\in p^\alpha\cal{O}_{L,p}$ and its determinant $\det(\gamma)\in\cal{O}^\times_{L,+}$ is a totally positive unit. In order to prove the claimed inclusion we need to show that 
	\[(a_\gamma d_\gamma)_{\mathfrak{p}_1}\equiv_{p^\alpha}(a_\gamma d_\gamma)_{\mathfrak{p}_2}\qquad\text{and}\qquad (d_\gamma)_{\mathfrak{p}_1}\equiv_{p^\alpha}(a_\gamma)_{\mathfrak{p}_2}.
	\]  Since $N_{L/\bb{Q}}\big(\det(\gamma)\big)=1$, one sees that either $a_\gamma d_\gamma- 1\in p^\alpha\cal{O}_{L}$ or $a_\gamma d_\gamma+ 1\in p^\alpha\cal{O}_{L}$. The second option implies that  $\det(\gamma)\equiv_p-1$ that is ruled out by our assumption. Hence we must have $a_\gamma d_\gamma- 1\in p^\alpha\cal{O}_{L}$.
\end{proof}

\subsection{Hecke correspondences}

%
%

We recall the conventions for Hecke correspondences that are used in the rest of this work. For $K\le G(\widehat{\bb{Z}})$ an open compact subgroup and an element $g \in G(\bb{A}_f)$ there is a map
\begin{equation}
\mathfrak{T}_g:
S(K) \longrightarrow S(gKg^{-1}),
\qquad
[x,h]\mapsto[x, hg^{-1}],
\end{equation}
descending to a morphism of $\bb{Q}$-varieties.  The double coset $[KgK]$ defines the following correspondence
\begin{equation}\label{HeckeCorrespondence}
\xymatrix{
	S( g^{-1} K g\cap K)
	\ar[r]^-{\mathfrak{T}_{g}} \ar[d]^{\mrm{pr}} 
	&
	S(  K \cap g K g^{-1} )
	\ar[d]^-{\mrm{pr}'}
	\\
	S(K) 
	&
	S(K)
	}
\end{equation}
acting on the cohomology of $S(K)$ via $(\mrm{pr}')_{*}\circ(\mathfrak{T}_{g})_*\circ (\mrm{pr})^{*}$. Suppose $V(\mathfrak{N})\le K$ for some $\cal{O}_L$-ideal $\mathfrak{N}$, then
for $a,b \in \cal{O}_{L,\mathfrak{N}}^\times$ and $z \in \bb{A}_{L,f}^\times$, we can consider correspondences associated to the double cosets
\begin{equation}
T(a,b) = \left[K\begin{pmatrix} a&0\\0&b \end{pmatrix}K\right],\qquad \langle z \rangle= \left[K\begin{pmatrix} z&0\\0&z \end{pmatrix}K\right].
\end{equation}
Since $z\cdot\mathbbm{1}_2$ belongs to the center, the action of $\langle z\rangle$ is that of
 $(\mathfrak{T}_{z})_* = (\mathfrak{T}_{z^{-1}})^*$. Moreover, if the matrix $D_{a,b}:=\mbox{\tiny $\begin{pmatrix} a&0\\0&b \end{pmatrix}$}$
normalizes the compact open subgroup $K$,
 then $T(a,b)$ acts as
 $(\mathfrak{T}_{D_{a,b}})_* = (\mathfrak{T}_{D_{a,b}^{-1}})^*$. 
 
\begin{definition}\label{diamonds on cohomology}
 Given an element $(z,a)\in\bb{G}_L(K)$ we set
\begin{equation}
\langle z, a \rangle:= \big(\mathfrak{T}_{z}\big)\circ \big(\mathfrak{T}_{D_{a^{-1},1}}\big).
\end{equation}
\end{definition}

\subsubsection{The $U_p$-correspondence.}\label{def Up}
 Recall that $K(p^\alpha)$ denotes the subgroup of $K$ of those matrices whose $p$-component is unipotent modulo $p^\alpha$, that $\varpi_{\mathfrak{p}} \in \bb{A}_{L,f}^\times$ is the element whose $\mathfrak{p}$-component is $p$ and every other component is 1, and $\varpi_p = \varpi_{\mathfrak{p}_1}\varpi_{\mathfrak{p}_2}$.
Consider the matrix $g_p = \mbox{\tiny $\begin{pmatrix}\varpi_p&0\\0&1\end{pmatrix}$}$ satisfying
$g_p^{-1}K(p^\alpha)g_p \cap K(p^\alpha)
=
K(p^\alpha)\cap K_0(p^{\alpha+1})$.
We denote by
\[
\pi_1: S(K(p^\alpha)\cap K_0(p^{\alpha+1}))
\longrightarrow S(K(p^\alpha)),
\qquad
\pi_2: S(K(p^\alpha)\cap K_0(p^{\alpha+1}))
\longrightarrow S(K(p^\alpha)),
\]
the two projections $\pi_1=\mrm{pr}$  and $\pi_2=\mrm{pr}'\circ\mathfrak{T}_{g_p}$.
Then the $U_p$-correspondence and its adjoint $U_p^*$ are given by 
\begin{equation}
U_p = (\pi_{2})_*\circ(\pi_1)^*,\qquad U_p^* = (\pi_{1})_*\circ(\pi_2)^*.
\end{equation}
Analogously, there is a $U_\mathfrak{p}$-correspondence for each prime $\frak{p}$ above $p$. Let 
$ g_\mathfrak{p}
 =
 \mbox{\tiny $\begin{pmatrix} 
 \varpi_{\mathfrak{p}}&0\\
 0&1
 \end{pmatrix}$}$, then
 $g_\mathfrak{p}^{-1}K(p^\alpha)g_\mathfrak{p}
 =
 K(p^\alpha) \cap K_0(\mathfrak{p}^{\alpha+1})$, and there are projections 
 \[
 \pi_{1,\mathfrak{p}}, \pi_{2,\mathfrak{p}}: S(K(p^\alpha) \cap K_0(\mathfrak{p}^{\alpha+1}))\longrightarrow S(K(p^\alpha)).
 \] 
 The $U_{\mathfrak{p}}$ operator is defined as  $U_\mathfrak{p}=(\pi_{2,\mathfrak{p}})_*\circ(\pi_{1,\mathfrak{p}})^*$
 with its adjoint given by $U^*_\mathfrak{p}=(\pi_{1,\mathfrak{p}})_*\circ(\pi_{2,\mathfrak{p}})^*$. The above discussion holds verbatim if we change $K(p^\alpha)$ to any other level subgroup defined in Definitions \ref{LevelSubgroups1} and \ref{LevelSubgroups2}.


\subsubsection{Atkin--Lehner map.}
Recall the matrix 
$\tau_{\mathfrak{p}_2^\alpha} \in G(\bb{A}_{f})$ 
defined by
\[
(\tau_{\mathfrak{p}_2^\alpha})_{\mathfrak{p}_2}=\begin{pmatrix}
	0&-1\\
	\varpi_{\mathfrak{p}_2}^\alpha&0
\end{pmatrix},\qquad
(\tau_{\mathfrak{p}^\alpha})_v=\mathbbm{1}_2\quad \text{for}\; v\not=\mathfrak{p}_2
\] 
which normalizes $K(p^\alpha)$ and induces the morphism
\[
\mathfrak{T}_{\tau_{\mathfrak{p}_2^\alpha}}:
S(K(p^\alpha)) \longrightarrow S(K(p^\alpha))
,\qquad 
[x,h] \mapsto [x,h\tau_{\mathfrak{p}_2^\alpha}^{-1}].
\]
The morphism $\mathfrak{T}_{\tau_{\mathfrak{p}_2^\alpha}}$ interacts with diamonds operators as follows: 
	\begin{equation}\label{AL-diamonds}
	\mathfrak{T}_{\tau_{\mathfrak{p}_2^\alpha}} \circ \big\langle z,a	\big\rangle
	=
	\big\langle z\cdot a_{\mathfrak{p}_2}, a\cdot a_{\mathfrak{p}_2}^{-2}\big\rangle\circ \mathfrak{T}_{\tau_{\mathfrak{p}_2^\alpha}}
	\end{equation}
for any $(z,a)\in \bb{G}_L(K)$. Furthermore,
\begin{equation}
\big(\mathfrak{T}_{\tau_{\mathfrak{p}_2^\alpha}}\big)^2=\big\langle-\varpi_{\mathfrak{p}_2}^\alpha,1\big\rangle.
\end{equation}
It is also useful to know how $\mathfrak{T}_{\tau_{\mathfrak{p}_2^\alpha}}$ interacts with the Hecke operators at $p$. For $\frak{p}$ equal to either $\frak{p}_1$ or $\frak{p}_2$ there are commutative diagrams
 \[\xymatrix{
 S(K(p^\alpha)\cap K_0(p^{\alpha+1}))
 \ar[d]^{\nu_{1,\mathfrak{p}}} \ar[drr]^{\pi_1}
 &&&
 S(K(p^\alpha)\cap K_0(p^{\alpha+1}))
 \ar[d]^{\nu_{2,\mathfrak{p}}} \ar[drr]^{\pi_2}
 & \\
  S(K(p^\alpha)\cap K_0(\mathfrak{p}^{\alpha+1}))
  \ar[rr]_{\qquad \pi_{1,\mathfrak{p}}} 
  && 
  S(K(p^\alpha)),     
  &      
  S(K(p^\alpha)\cap K_0(\mathfrak{p}^{\alpha+1}))
  \ar[rr]_{\qquad \pi_{2,\mathfrak{p}}} 
  &&
  S(K(p^\alpha))
 }\]
 where $\nu_{1,\frak{p}}([x,h])=[x,h]$ is the natural map, while $\nu_{2,\frak{p}}([x,h])=[x,hg_{\frak{p}'}^{-1}]$ for $\frak{p}'\not=\frak{p}$. 
 A direct calculation shows the following relations
\begin{equation}\label{Commute1}
\pi_{1,\mathfrak{p}_2}\circ \mathfrak{T}_{\tau_{\mathfrak{p}_2^{\alpha+1}}}
=
\mathfrak{T}_{\tau_{\mathfrak{p}_2^\alpha}}\circ\pi_{2,\mathfrak{p}_2}
,\qquad
\pi_{2,\mathfrak{p}_2}\circ \mathfrak{T}_{\tau_{\mathfrak{p}_2^{\alpha+1}}}
=
\langle\varpi_{\mathfrak{p}_2}^{-1},1\rangle \circ \mathfrak{T}_{\tau_{\mathfrak{p}_2^\alpha}} \circ \pi_{1,\mathfrak{p}_2}
\end{equation}
and 
\begin{equation}\label{Commute2}
\nu_{1,\mathfrak{p}_1} \circ \mathfrak{T}_{\tau_{\mathfrak{p}_2^{\alpha+1}}}
=
\mathfrak{T}_{\tau_{\mathfrak{p}_2^\alpha}} \circ \nu_{2,\mathfrak{p}_1}
,\qquad
\nu_{2,\mathfrak{p}_1}\circ \mathfrak{T}_{\tau_{\mathfrak{p}_2^{\alpha+1}}}
=
\langle\varpi_{\mathfrak{p}_2}^{-1},1\rangle \circ \mathfrak{T}_{\tau_{\mathfrak{p}_2^\alpha}} \circ \nu_{1,\mathfrak{p}_1}.
\end{equation}
It follows that
\begin{equation}\label{U_pAdjoint}
(\mathfrak{T}_{\tau_{\mathfrak{p}_2^{\alpha}}})_*\circ U_{\mathfrak{p}_2}\circ (\mathfrak{T}_{\tau_{\mathfrak{p}_2^{\alpha}}})^*
=
U^*_{\mathfrak{p}_2}\circ \langle\varpi_{\mathfrak{p}_2},1\rangle.
\end{equation}
Furthermore, it is clear that $U_{\mathfrak{p}_1}$ commutes with $(\mathfrak{T}_{\tau_{\mathfrak{p}_2^{\alpha}}})^*$.

\begin{remark}
	The level subgroups of Definitions $\ref{LevelSubgroups1}, \ref{LevelSubgroups2}$ are related by the Atkin-Lehner map
	\[
	\mathfrak{T}_{\tau_{\mathfrak{p}_2^\alpha}}: S(K_{\diamond,1}(p^\alpha))\longrightarrow S(K_{\mbox{\tiny $\mrm{X}$},1}(p^\alpha)),\qquad
	\mathfrak{T}_{\tau_{\mathfrak{p}_2^\alpha}}: S(K_{\mbox{\tiny $\mrm{X}$}}(p^\alpha)) \longrightarrow
	 S(K_{\diamond,t}(p^\alpha)).
	\]
	\end{remark}

\subsection{On different models of Hilbert modular surfaces}
 From now on we assume that there is no totally positive units in $\cal{O}^\times_{L}$ congruent to $-1$ modulo $p$ and that the open compact subgroup $K\le G(\widehat{\bb{Z}})$ satisfies $\mrm{det}(K)=\widehat{\cal{O}_L}^\times$.
\begin{definition} 
	Given $a\in(\bb{Z}/p^\alpha\bb{Z})^\times$ we denote by $\sigma_a\in\mrm{Gal}(\bb{Q}(\zeta_{p^ \alpha})/\bb{Q})$ the element corresponding to $a$ by class field theory. Specifically, $\sigma_a(\zeta_{p^\alpha})=(\zeta_{p^\alpha})^{a^{-1}}$.
	\end{definition}

\begin{lemma}\label{diffent models}
	There are canonical isomorphisms of $\bb{Q}$-varieties
	\[\xymatrix{
	& S(K(p^\alpha))\ar[dl]_{\sim}\ar[dr]^{\sim}& &\\
	S(K_{\diamond,1}(p^\alpha)) \times_{\bb{Q}}\bb{Q}(\zeta_{p^\alpha})&&
	S(K_{\mbox{\tiny $\mrm{X}$},1}(p^\alpha)) \times_{\bb{Q}}\bb{Q}(\zeta_{p^\alpha}).
	}\]
\end{lemma}
\begin{proof}
	It is well-known that the canonical projection $S(K(p^\alpha))\rightarrow S(K_{?,1}(p^\alpha))$ induce an isomorphism of $\bb{Q}$-varieties
	\[
		S(K(p^\alpha)) \simeq
	S(K_{?,1}(p^\alpha)) \times_{\pi_0(S(K_{?,1}(p^\alpha)))} \pi_0(S(K(p^\alpha)))\qquad \text{for}\ ?=\diamond,\mbox{\tiny $\mrm{X}$}.
	\]
	Therefore in order to prove the lemma we have to show that there is an isomorphism of $\bb{Q}$-group schemes between $\pi_0(S(K(p^\alpha)))$ and $\pi_0(S(K_{?,1}(p^\alpha))) \times_\bb{Q} \bb{Q}(\zeta_{p^\alpha})$. The component groups of $S(K_{\mbox{\tiny $\mrm{X}$},1}(p^\alpha)), S(K_{\diamond,1}(p^\alpha))$ and $S(K_{\mbox{\tiny $\mrm{X}$},1}(p^\alpha))$ are $0$-dimensional Shimura varieties defined over $\bb{Q}$ and the Galois group $\mrm{Gal}(\bb{Q}(\zeta_{p^\alpha})/\bb{Q})$ acts on their points via the Artin map and the diagonal embedding of $\bb{A}_{\bb{Q},f}^\times$ inside $\bb{A}_{L,f}^\times$. Moreover, all points of $\pi_0(S(K_{\mbox{\tiny $\mrm{X}$},1}(p^\alpha)))$ and $\pi_0(S(K_{\diamond,1}(p^\alpha)))$ are defined over $\bb{Q}$ because
	\[
	\det(K_{\diamond,1}(p^\alpha))
	=\det(K_{\mbox{\tiny $\mrm{X}$},1}(p^\alpha))
	= (\bb{Z}^\times_p+p^\alpha\cal{O}_{L,p})\widehat{\cal{O}}_L^{p,\times}.
	\]
The projection $\pi_0(S(K(p^\alpha))) \rightarrow \pi_0(S(K_{?,1}(p^\alpha)))$ corresponds to the natural quotient map
	\[
	L^\times_+\backslash\bb{A}^\times_{L,f}/(1+p^\alpha\cal{O}_{L,p})\widehat{\cal{O}}_L^{p,\times}
	\longrightarrow
	L^\times_+\backslash\bb{A}^\times_{L,f}/(\bb{Z}^\times_p+p^\alpha\cal{O}_{L,p})\widehat{\cal{O}}_L^{p,\times}.
	\]
	Its kernel can be identified with
	\[
	\big(\bb{Z}^\times_p+p^\alpha\cal{O}_{L,p}\big)\bigg/\Big(\big[\cal{O}_{L,+}^\times\cap\big(\bb{Z}^\times_p+p^\alpha\cal{O}_{L,p}\big)\big]\cdot\big(1+p^\alpha\cal{O}_{L,p}\big)\Big)\cong (\bb{Z}/p^\alpha\bb{Z})^\times
	\]
	 because there is no totally positive unit congruent to $-1$ modulo $p$. Hence the Galois action of $\mrm{Gal}(\bb{Q}(\zeta_{p^\alpha})/\bb{Q})$ on the fibers of the projection $\pi_0(S(K(p^\alpha))) \rightarrow \pi_0(S(K_{?,1}(p^\alpha)))$ can be canonically identified with the simply transitive action of $(\bb{Z}/p^\alpha\bb{Z})^\times$ on itself and we conclude that
	 \[
	 \pi_0(S(K(p^\alpha)))\cong\pi_0(S(K_{?,1}(p^\alpha))) \times_\bb{Q} \bb{Q}(\zeta_{p^\alpha})\qquad \text{for}\ ?=\diamond,\mbox{\tiny $\mrm{X}$}.
	 \]
\end{proof}

\begin{proposition}
\label{nu_alpha}
There is an isomorphism of $\bb{Q}$-varieties
\[
\nu_\alpha:
S(K_{\diamond}(p^\alpha))\times_\bb{Q} \bb{Q}(\zeta_{p^\alpha})
\overset{\sim}{\longrightarrow}
S(K_{\mbox{\tiny $\mrm{X}$}}(p^\alpha))\times_\bb{Q} \bb{Q}(\zeta_{p^\alpha})
\]
such that
$
\nu_\alpha\circ(1\times\sigma_a)
=
(\langle 1,(a,a)\rangle\times\sigma_a)\circ\nu_\alpha.
$
\end{proposition}
\begin{proof}
Let $a\in\bb{Z}_p^\times$ and consider the action of $\mbox{\tiny $\begin{pmatrix}(a^{-1},a^{-1})&0\\0&1\end{pmatrix}$}\in \mrm{GL}_2(\cal{O}_{L,p})$ on $S(K(p^\alpha))$. On the one hand, it acts on  
$S(K_{\diamond,1}(p^\alpha)) \times_\bb{Q} \bb{Q}(\zeta_{p^\alpha})$ as 
$1\times \sigma_a$ via the isomorphism of Lemma \ref{diffent models}. On the other hand, it acts as $\langle 1,(a,a)\rangle \times \sigma_a$ on 
$S(K_{\mbox{\tiny $\mrm{X}$},1}(p^\alpha)) \times_\bb{Q} \bb{Q}(\zeta_{p^\alpha})$. In other words, the isomorphism
\[
\tilde{\nu}_\alpha:S(K_{\diamond,1}(p^\alpha))\times_\bb{Q} \bb{Q}(\zeta_{p^\alpha})
\overset{\sim}{\longrightarrow} 
S(K_{\mbox{\tiny $\mrm{X}$},1}(p^\alpha))\times_\bb{Q} \bb{Q}(\zeta_{p^\alpha})
\]
obtained from Lemma \ref{diffent models} satisfies
$
\tilde{\nu}_\alpha\circ(1\times\sigma_a)
=
(\langle 1,(a,a)\rangle\times\sigma_a)\circ\tilde{\nu}_\alpha.
$ 

\noindent The quotient of $S(K_{\diamond,1}(p^\alpha))\times_\bb{Q} \bb{Q}(\zeta_{p^\alpha})$ by the subgroup of matrices $\gamma$ of the form
\[\gamma_v=\mathbbm{1}_2\quad\text{for}\ v\not=p,\qquad\gamma_p=\begin{pmatrix}d^{-1}&\\&d \end{pmatrix}\quad\text{with}\quad d_{\mathfrak{p}_1}\equiv d_{\mathfrak{p}_2} \pmod{p^{\alpha}}
\] 
is $
S(K_{\diamond}(p^\alpha))\times_\bb{Q} \bb{Q}(\zeta_{p^\alpha})
$
because those matrices have determinant $1$. Similarly, the quotient of $S(K_{\mbox{\tiny $\mrm{X}$},1}(p^\alpha))\times_\bb{Q} \bb{Q}(\zeta_{p^\alpha})$ by the same group is
$
S(K_{\mbox{\tiny $\mrm{X}$}}(p^\alpha))\times_\bb{Q} \bb{Q}(\zeta_{p^\alpha})
$
on the target. We denote by $\nu_\alpha$ the resulting isomorphism
\[
\nu_\alpha:
S(K_{\diamond}(p^\alpha))\times_\bb{Q} \bb{Q}(\zeta_{p^\alpha})
\overset{\sim}{\longrightarrow}
S(K_{\mbox{\tiny $\mrm{X}$}}(p^\alpha))\times_\bb{Q} \bb{Q}(\zeta_{p^\alpha})
\]
which also satisfies
$
\nu_\alpha\circ(1\times\sigma_a)
=
(\langle 1,(a,a)\rangle\times\sigma_a)\circ\nu_\alpha.
$  
\end{proof}

\begin{corollary}\label{proj-nu_alpha-commute}
	The isomorphism
\[
\nu_\alpha:
\Big(S(K_\diamond(p^\alpha))\times_\bb{Q}\bb{Q}(\zeta_{p^\alpha})\Big)(\bb{C})
\overset{\sim}{\longrightarrow}
\Big(S(K_{\mbox{\tiny $\mrm{X}$}}(p^\alpha))\times_\bb{Q}\bb{Q}(\zeta_{p^\alpha})\Big)(\bb{C})
\]
is the identity with respect to the complex uniformizations ($\ref{complex uniformization}$). In particular $\nu_\alpha$ commutes with the projections
$\nu_{1,\mathfrak{p}}$ and $\pi_{1,\mathfrak{p}}$ for $\mathfrak{p}\mid p$.
\end{corollary}

\subsubsection{Compatibility with Hecke correspondences.} Even though the complex uniformization of $\nu_\alpha$ is the identity, the map induced in cohomology does not commute in general with the Hecke operators. Indeed, even when $g\in G(\bb{A}_f)$ normalizes the congruence subgroups $K_\diamond(p^\alpha)$ and $K_{\mbox{\tiny $\mrm{X}$}}(p^\alpha)$, the morphisms $\frak{T}_g:S(K_\diamond(p^\alpha))\to S(K_\diamond(p^\alpha))$ and $\frak{T}_g:S(K_{\mbox{\tiny $\mrm{X}$}}(p^\alpha))\to S(K_{\mbox{\tiny $\mrm{X}$}}(p^\alpha))$ may differ. However, there is an important case when the commutativity holds.

\begin{lemma}\label{U_p-nu_alpha-commute}
    The identity
    \[
    U_p\circ(\nu_\alpha)^*
    =
    (\nu_\alpha)^*\circ U_p
    \]
    holds in cohomology.
\end{lemma}
\begin{proof}
  The argument in Proposition $\ref{nu_alpha}$ also yields an isomorphism
    \begin{equation}\label{TwistVariety3}
    \nu_\alpha:
    S\big(K_{\diamond}(p^\alpha)\cap K_0(p^{\alpha+1})\big)\times_\bb{Q} \bb{Q}(\zeta_{p^\alpha})
    \overset{\sim}{\longrightarrow}
    S\big(K_{\mbox{\tiny $\mrm{X}$}}(p^\alpha))\cap K_0(p^{\alpha+1})\big)\times_\bb{Q} \bb{Q}(\zeta_{p^\alpha}),
    \end{equation}
   for any $\alpha\ge1$ factoring through a quotient of $S(K(p^\alpha)\cap K_0(p^{\alpha+1}))$. Since the determinant of the matrix $g_p = \mbox{\tiny $\begin{pmatrix}\varpi_p&0\\0&1\end{pmatrix}$}$  lies in $L^\times_+\cdot(1+p^\alpha\cal{O}_{L,p})\widehat{\cal{O}}_L^{p,\times}$, both projections
    \[
    \pi_{1}, \pi_{2}:
    S\big(K(p^\alpha)\cap K_0(p^{\alpha+1})\big)
    \longrightarrow
    S(K(p^\alpha))
    \] 
	induce the identity map between the component groups. Therefore the projections descend to  $\pi_{1}\times 1$, $\pi_{2}\times 1$ on both domain and target of the isomorphism ($\ref{TwistVariety3}$).
    In other words, both $\pi_{1}\times 1$ and $\pi_{2}\times1$ commute with $\nu_\alpha$ which implies the claim.
\end{proof}

\subsection{Atkin--Lehner correspondence}
\begin{definition}
Let $\frak{p}$ be a $\cal{O}_L$-prime ideal above $p$ and $\alpha\ge 1$ be a positive integer. The Atkin-Lehner correspondence 
$w_{\mathfrak{p}^\alpha_2} :=  (\mathfrak{T}_{\tau_{\mathfrak{p}_2^\alpha}}\times1) \circ \nu_{\alpha}$
 is an isomorphism of $\bb{Q}$-varieties
\[
w_{\mathfrak{p}^\alpha} 
:
S(K_{\diamond}(p^\alpha))\times_\bb{Q} \bb{Q}(\zeta_{p^\alpha})
\overset{\sim}{\longrightarrow}
S(K_{\diamond,t}(p^\alpha))\times_\bb{Q} \bb{Q}(\zeta_{p^\alpha}).
\]
\end{definition}
\noindent By Lemma $\ref{U_p-nu_alpha-commute}$ and equation (\ref{U_pAdjoint}), the cohomological map $(w_{\frak{p}_2^\alpha})^*$ is intertwined with the Hecke correspondences at $p$ through the rule
    \begin{equation}\label{U_pAtkinLehner}
    U_p\circ (w_{\mathfrak{p}_2^\alpha})^*
    =
    (w_{\mathfrak{p}_2^\alpha})^*\circ U_{\mathfrak{p}_1}\circ U_{\mathfrak{p}_2}^*\circ \langle \varpi_{\mathfrak{p}_2},1 \rangle.
    \end{equation}

\subsubsection{Galois action.}
Consider the Galois character
\begin{equation}\label{GalCal}
	\delta _\alpha:\Gamma_\bb{Q}\longrightarrow O[\bb{G}_L^\alpha(K)]^\times
\end{equation} 
obtained by composing the projection $\Gamma_\bb{Q}\rightarrow\mrm{Gal}(\bb{Q}(\zeta_{p^ \alpha})/\bb{Q})$ with the homomorphism  
\begin{equation}
\sigma_a\mapsto\big\langle (1,a^{-1}),(a,a^{-1})\big\rangle.
\end{equation}

\begin{lemma}\label{Cohomology}
There is a $\Gamma_\bb{Q}$-equivariant isomorphism 
\[
(w_{\mathfrak{p}^\alpha_2})_*:
\mrm{H}^{\bfcdot}_\et\big(S(K_{\diamond}(p^\alpha))_{\bar{\bb{Q}}},O\big)
\overset{\sim}{\longrightarrow}
\mrm{H}^{\bfcdot}_\et\big(S(K_{\diamond,t}(p^\alpha))_{\bar{\bb{Q}}},O\big)(\delta _\alpha).
\]
\end{lemma}
\begin{proof}
Since $\mathfrak{T}_{\tau_{\mathfrak{p}_2^\alpha}}$ is defined over $\bb{Q}$ we can use equation ($\ref{AL-diamonds}$) to compute
\begin{equation}\label{ALGalois}
\begin{split}
w_{\mathfrak{p}^\alpha_2}\circ(1\times\sigma_a)
&=
\mathfrak{T}_{\tau_{\mathfrak{p}_2^\alpha}} \circ (\langle 1,(a,a)\rangle\times\sigma_a)\circ\nu_\alpha\\
&=
(\langle(1,a^{-1}),(a,a^{-1})\rangle\times\sigma_a) \circ w_{\mathfrak{p}^\alpha_2}.    
\end{split}\end{equation}
\end{proof}

\section{Hirzebruch--Zagier classes}

We keep the notations and assumptions made at the beginning of the previous chapter. In particular, $M$ is the integer appearing in Theorem \ref{Ichino}. Let $K=V_1(M\cal{O}_L)$, $K' := \mrm{GL}_2(\bb{A}_\bb{Q}) \cap K$ and note that
$K'_0(p^\alpha)=\mrm{GL}_2(\bb{A}_\bb{Q}) \cap K_\diamond(p^\alpha)$.
We consider the following diagram of Shimura varieties over $\bb{Q}$
\begin{equation}\label{cartesiano}
	\xymatrix{ \ar @{} [drr] |{\mbox{\large $\diamond$}}
	Y(K'_0(p^{{\alpha+1}}))
	\ar[d]^{\pi_1}\ar@{^{(}->}[rr]^{\zeta(\mathfrak{p}_1)}
	&&
	S(K_\diamond(p^\alpha)\cap V_0(\mathfrak{p}_1^{\alpha+1}))
	\ar[d]^{\pi_{1,\mathfrak{p}_1}}
	\\
    Y(K'_0(p^{\alpha}))
    \ar@{^{(}->}[rr]^{\zeta}
    &&
    S(K_\diamond(p^\alpha))
}\end{equation}
where the horizontal arrows are induced by the inclusions of groups $K'_0(p^{\alpha})\le K_\diamond(p^\alpha)$ and $K'_0(p^{{\alpha+1}})\le K_\diamond(p^\alpha)\cap V_0(\mathfrak{p}_1^{\alpha+1})$. Since the horizontal arrows are closed embeddings and the vertical ones are finite of degree $p$, diagram \eqref{cartesiano} is cartesian.

\begin{definition}
For $\alpha\ge1$ we set
\[
\Delta^\flat_\alpha
:=
\zeta_*\left[Y(K'_0(p^{\alpha}))\right]\in \mrm{CH}^1\big(S(K_\diamond(p^\alpha)\big)(\bb{Q})
\] 
and
\[
\Delta_\alpha^\flat(\mathfrak{p}_1)
:=
\zeta(\mathfrak{p}_1)_*\left[Y(K'_0(p^{\alpha+1}))\right]
\]
as a codimension one cycle class in $\mrm{CH}^1\big( S(K_\diamond(p^\alpha)\cap V_0(\mathfrak{p}_1^{\alpha+1}))\big)(\bb{Q})$.
\end{definition}

\begin{lemma}\label{firstformula}
We have
\begin{equation}
(\pi_{1,\mathfrak{p}_1})^*\Delta_{\alpha}^{\flat}=\Delta^\flat_{\alpha}(\mathfrak{p}_1)
\end{equation}
 in $\mrm{CH}^1\big(S(K_\diamond(p^\alpha)\cap V_0(\mathfrak{p}_1^{\alpha+1}))\big)(\bb{Q})$.
\end{lemma}
\begin{proof}
Since the diagram ($\ref{cartesiano}$)
is cartesian, the push-pull formula  
\[
(\pi_{1,\mathfrak{p}_1})^*\circ\zeta_*
=
\zeta(\mathfrak{p}_1)_*\circ (\pi_1)^*
\]
implies the claim.
\end{proof}

\subsubsection{Twisting by Atkin--Lehner.}
\begin{definition}
We consider the class of the codimension one cycle
\[
\Delta_{\alpha}^{\mbox{\tiny $\sharp$}}
=
(w_{\mathfrak{p}_{2}^{\alpha}})_*\Delta_{\alpha}^{\flat}\in \mrm{CH}^{1}\big(S(K_{\diamond,t}(p^{\alpha}))\big)(\mathbb{Q}(\zeta_{p^\alpha}))
\]
defined over $\bb{Q}(\zeta_{p^\alpha})$.
\end{definition}
\noindent There are two natural degeneracy maps $\varpi_1, \varpi_2: S(K_{\diamond,t}(p^{\alpha+1}))\to S(K_{\diamond,t}(p^{\alpha}))$ which are described by the following commutative diagrams
    \[\resizebox{\displaywidth}{!}{
	\xymatrix{
    S(K_{\diamond,t}(p^{\alpha+1}))
    \ar[d]^{\mu} \ar[drr]^{\varpi_1}
    &&&
    S(K_{\diamond,t}(p^{\alpha+1}))
    \ar[d]^{\mu} \ar[drr]^{\varpi_2}
    & \\
     S(K_{\diamond,t}(p^\alpha)\cap K_0(p^{\alpha+1}))
     \ar[rr]_{\qquad \pi_{1}} 
     && 
     S(K_{\diamond,t}(p^\alpha)),     
     &      
     S(K_{\diamond,t}(p^\alpha)\cap K_0(p^{\alpha+1}))
     \ar[rr]_{\qquad \pi_{2}} 
     &&
     S(K_{\diamond,t}(p^\alpha)).
    }}\]

\begin{proposition}\label{proposition:sharp-relation}
We have 
\[
 (\varpi_{2})_*\Delta_{\alpha+1}^{\mbox{\tiny $\sharp$}}
 =
 \langle\varpi_{\mathfrak{p}_2}^{-1},1\rangle_*\circ U_{\mathfrak{p}_1}\Delta_\alpha^{\mbox{\tiny $\sharp$}}.
\]
\end{proposition}
\begin{proof}
Combining equation (\ref{Commute2}) with Corollary \ref{proj-nu_alpha-commute}, we have the following diagram commutes
\begin{equation}\label{diagramsharp}
\xymatrix{
    Y(K'_0(p^{\alpha+1}))
    \ar@{=}[d] \ar@{^{(}->}[rr]^{\zeta} 
    &&
    S(K_\diamond(p^{\alpha+1}))  \ar[d]^{\nu_{1,\mathfrak{p}_1}\circ\mu} \ar[rr]^{w_{\mathfrak{p}_2^{\alpha+1}}} 
    &&
    S(K_{\diamond,t}(p^{\alpha+1}))
    \ar[d]^{\nu_{2,\mathfrak{p}_1}\circ\mu} 
    \\
	Y(K'_0(p^{\alpha+1}))
	\ar@{^{(}->}[rr]^{\zeta(\mathfrak{p}_1)} 
	&&
	S(K_\diamond(p^\alpha)\cap K_0(\mathfrak{p}_1^{\alpha+1}))  
	\ar[rr]^{\langle\varpi_{\mathfrak{p}_2}^{-1},1\rangle \circ w_{\mathfrak{p}_2^{\alpha}}} 
	&&
	S(K_{\diamond,t}(p^\alpha)\cap K_0(\mathfrak{p}_1^{\alpha+1}))
	\ar[d]^{\pi_{2,\mathfrak{p}_1}}
	\\
	&& 
	&&
	S(K_{\diamond,t}(p^\alpha)).
	}
\end{equation}
By definition,
$(\varpi_{2})_*\Delta_{\alpha+1}^{\mbox{\tiny $\sharp$}}$
is the pushforward of the cycle class
$\left[Y(K'_0(p^{\alpha+1}))\right]$ along the top arrows and the rightmost vertical arrows. Therefore
\[\begin{split}
 (\varpi_{2})_*\Delta_{\alpha+1}^{\mbox{\tiny $\sharp$}} 
 &=
 (\pi_{2,\mathfrak{p}_1})_*\circ \langle\varpi_{\mathfrak{p}_2}^{-1},1\rangle_*\circ (w_{\mathfrak{p}_2^\alpha})_*\circ(\zeta(\mathfrak{p}_1))_*\left[Y(K'_0(p^{\alpha+1}))\right]\\
 &=
 \langle\varpi_{\mathfrak{p}_2}^{-1},1\rangle_*\circ(\pi_{2,\mathfrak{p}_1})_*\circ(w_{\mathfrak{p}_2^\alpha})_*\Delta_{\alpha}^\flat(\mathfrak{p}_1)\\
 &=
 \langle\varpi_{\mathfrak{p}_2}^{-1},1\rangle_*\circ(\pi_{2,\mathfrak{p}_1})_*\circ(w_{\mathfrak{p}_2^\alpha})_*\circ(\pi_{1,\mathfrak{p}_1})^*\Delta_{\alpha}^{\flat}\\
 &= 
 \langle\varpi_{\mathfrak{p}_2}^{-1},1\rangle_*\circ(\pi_{2,\mathfrak{p}_1})_*(\pi_{1,\mathfrak{p}_1})^*\circ(w_{\mathfrak{p}_2^\alpha})_*\Delta_{\alpha}^{\flat}\\
 &= 
 \langle\varpi_{\mathfrak{p}_2}^{-1},1\rangle_*\circ U_{\mathfrak{p}_1}\Delta_\alpha^{\mbox{\tiny $\sharp$}},
\end{split}\]
where the third equality is due to Lemma \ref{firstformula} and the second to last follows from the fact that $w_{\mathfrak{p}_2^\alpha}$ is an isomorphism commuting with $\pi_{1,\mathfrak{p}_1}$ (Corollary \ref{proj-nu_alpha-commute}).
\end{proof}

\subsection{Hirzebruch--Zagier cycles}
\begin{definition}
 Consider the Shimura threefold 
\begin{equation}
	Z_\alpha(K) =  S(K_{\diamond,t}(p^\alpha))\times X_{1,0}(N,p)
\end{equation} 
where $X_{1,0}(N,p)$ denotes the compactified modular curve $X(V_1(N)\cap V_0(p))$. Then the Hirzebruch--Zagier cycles of level $\alpha\ge1$ is defined by
\begin{equation}
\Delta_{\alpha} 
=
(\langle\varpi_{\mathfrak{p}_2}^{\alpha},1\rangle \circ w_{\mathfrak{p}_2^\alpha}\circ\zeta_\alpha,\ \pi_{1,\alpha})_*[Y(K'_0(p^\alpha))]\in\mrm{CH}^2\big(Z_\alpha(K)\big)(\bb{Q}(\zeta_{p^\alpha})).
\end{equation} 
where $\pi_{1,\alpha}: Y(K'_0(p^\alpha))\longrightarrow X_{1,0}(N,p)$  the natural projection. 
\end{definition}

\begin{remark}
In general Hecke correspondences do not extend to any \emph{fixed} compactification of a Shimura variety. However, Kai-Wen Lan has shown that they do extend to maps between compatible toroidal compactifications (\cite{Lan13}, Proposition 6.4.3.4). In this article we will work with the non-compact Hilbert modular surfaces $S(K_{\diamond,t}(p^\alpha))$'s -- and we will only consider their compactifications whenever strictly necessary -- because the machinery of finite polynomial cohomology has been extended to general smooth varieties in \cite{BLZ}, and because we are ultimately interested in the localization of our Hecke modules at non-Eisenstein maximal ideals.
\end{remark}

\noindent Proposition \ref{proposition:sharp-relation} implies a precise relation between Hirzebruch--Zagier cycles of different levels.

\begin{proposition}\label{proposition:flat-relation}
The following identity holds in $\mrm{CH}^{2}(Z_{\alpha}(K))(\mathbb{Q}(\zeta_{p^\alpha}))$:
\[
(\varpi_{2}, \mrm{id})_{*}\Delta_{\alpha+1}
=
(U_{\mathfrak{p}_1}, \mrm{id})\Delta_{\alpha}.
\]
\end{proposition}

\subsubsection{Galois twisting.}
By equation (\ref{ALGalois}), the Galois group $\Gamma_\bb{Q}$ acts on $\Delta_\alpha$ through the finite quotient $\mrm{Gal}(\bb{Q}(\zeta_{p^\alpha})/\bb{Q})$ as
\[
(\sigma_a)_* \Delta_\alpha 
=
(\langle(1,a),(a^{-1},a)\rangle,\mrm{id})_*\Delta_\alpha.
\]

\begin{definition}
Let $S^\dagger(K_{\diamond,t}(p^\alpha))$ be the twist of the $\bb{Q}$-variety $S(K_{\diamond,t}(p^\alpha))$ by the $1$-cocycle 
\[
\mrm{Gal}(\bb{Q}(\zeta_{p^\alpha})/\bb{Q})\ni\sigma_a
\mapsto
 \langle(1,a^{-1}),(a,a^{-1})\rangle.
\]
We define
$
Z_\alpha^\dagger(K)
=
S^\dagger(K_{\diamond,t}(p^\alpha))\times X(K'_0(p)).
$
By construction 
\begin{equation} \mrm{H}^{\bfcdot}_\et(S^\dagger(K_{\diamond,t}(p^\alpha))_{\bar{\bb{Q}}},O)
\simeq
\mrm{H}^{\bfcdot}_\et(S(K_{\diamond,t}(p^\alpha))_{\bar{\bb{Q}}},O)(\delta _\alpha).
\end{equation}
 and $\Delta_\alpha$ corresponds to a codimension $2$ cycle of $Z_\alpha^\dagger(K)$ defined over $\bb{Q}$.
\end{definition}

\subsubsection{Null-homologous cycles.}
Finally, we apply a suitable correspondence to make Hirzebruch--Zagier cycles null-homologous.
We define a correspondence $\varepsilon_{\msf{f}_\circ}$ on $X_{1,0}(N,p)$ following \cite{DR2} (after equation (47)). Since $p$ is assumed to be non-Eisenstein for $\msf{f}_\circ$, i.e. $\msf{f}_\circ$ is not congruent to an Eisenstein series modulo $p$, there exists an auxiliary prime $\ell\nmid Np$ for which $\ell+1-\msf{a}_p(\ell, \msf{f}_\circ)$ lies in $O^\times$. Then the correspondence
\[
\varepsilon_{\msf{f}_\circ} = (\ell+1-T(\ell))/(\ell+1-\msf{a}_p(\ell, \msf{f}_\circ))
\]
has coefficients in $O$, it annihilates $\mrm{H}^0(X_{1,0}(N,p))$ and $\mrm{H}^2(X_{1,0}(N,p))$ and acts as the identity on the $\msf{f}_\circ$-isotypic subspace $\mrm{H}^1(X_{1,0}(N,p))[\msf{f}_\circ]$.

\begin{definition}
The modified Hirzebruch--Zagier cycle is given by
\[
\Delta_\alpha^\circ := (\mrm{id}, \varepsilon_{\msf{f}_\circ})_*\Delta_\alpha \in \mrm{CH}^2\big(Z^\dagger_\alpha(K)\big)(\bb{Q})\otimes_{\bb{Z}} O.
\]
\end{definition}

\begin{proposition}\label{nullhomo}
The cycle class $\Delta_\alpha^\circ$ is null-homologous.
\end{proposition}
\begin{proof}
Consider the smooth compactification $\iota:Z_\alpha^\dagger(K)\hookrightarrow Z_\alpha^\dagger(K)^\mrm{c}$ obtained by taking the minimal resolution of the Baily-Borel compactification of the Hilbert modular surface, and denote by $\Delta_\alpha^{\circ,\mrm{c}}$ the closure of $\Delta_\alpha^{\circ}$ in $Z_\alpha^\dagger(K)^\mrm{c}$. Thanks to the commuting diagram
\[\xymatrix{
\mrm{CH}^2\big(Z^\dagger_\alpha(K)^\mrm{c}\big)(\bb{Q})\otimes_{\bb{Z}}O\ar[r]^{\mrm{cl}_\et}\ar[d]^{\iota^*} & \mrm{H}_\mrm{et}^4\big(Z_\alpha^\dagger(K)^\mrm{c}_{\bar{\bb{Q}}},O(2)\big)\ar[d]^{\iota^*}\\
\mrm{CH}^2\big(Z^\dagger_\alpha(K)\big)(\bb{Q})\otimes_{\bb{Z}}O\ar[r]^{\mrm{cl}_\et}& \mrm{H}_\mrm{et}^4\big(Z_\alpha^\dagger(K)_{\bar{\bb{Q}}},O(2)\big),
}\]
it suffices to show that $\mrm{cl}_\et\big(\Delta_\alpha^{\circ,\mrm{c}}\big)=0$.
As the integral cohomology of smooth projective curves is torsion free and the minimal resolution of the Baily-Borel compactification of a Hilbert modular surface is simply connected, the group $\mrm{H}_\mrm{et}^4(Z_\alpha^\dagger(K)^\mrm{c}_{\overline{\bb{Q}}},O(2))$ has a K\"unneth decomposition whose every non-zero term is annihilated by the correspondence $(\mrm{id}, \varepsilon_{\msf{f}_\circ})$.
\end{proof}

\subsection{Big cohomology classes}
For any number field $D$, the $p$-adic \'etale Abel--Jacobi map 
\[
\mrm{AJ}_p^\et:
\mrm{CH}^2(Z_\alpha^\dagger(K))_0(D)
\longrightarrow 
\mrm{H}^1\big(D,\mrm{H}^3_\mrm{et}\big(Z_\alpha^\dagger(K)_{\bar{\bb{Q}}},O(2)\big)\big)
=
\mrm{H}^1\big(D,\mrm{H}^3_\mrm{et}\big(Z_\alpha(K)_{\bar{\bb{Q}}},O(2)\big)(\delta_\alpha)\big)
\]
sends null-homologous cycles to Galois cohomology classes. 
\begin{definition}
	We set
	\[
	\bs{\cal{V}}_\alpha(K)
	:=
	e_\mrm{n.o.}\mrm{H}^2_\et\big(S(K_{\diamond,t}(p^\alpha))_{\bar{\bb{Q}}},O(1)\big)(\delta_\alpha)
	\otimes
	\mrm{H}^1_\et\big(X_{1,0}(N,p)_{\bar{\bb{Q}}},O(1)\big).
	\]
\end{definition}
\noindent Then the modified Hirzebruch--Zagier cycles $\Delta^\circ_\alpha$ give rise to cohomology classes 
\begin{equation}
\kappa^\circ_\alpha
:=
\mrm{AJ}_p^\et(\Delta^\circ_\alpha)
\in
\mrm{H}^1\big(\bb{Q},\bs{\cal{V}}_\alpha(K)\big).
\end{equation}

\begin{lemma}\label{compatibility - classes}
	For $\alpha\ge1$ we have 
	\[
	(\varpi_{2}, \mrm{id})_{*}\kappa^\circ_{\alpha+1}
	=
	(U_{\mathfrak{p}_1}, \mrm{id})\kappa^\circ_{\alpha}.
	\]
\end{lemma}
\begin{proof}
	It follows from Proposition $\ref{proposition:flat-relation}$ and the commutativity of cycle class map and correspondences.
\end{proof}

\noindent Since $U_{\mathfrak{p}_1}$ acts invertibly on the nearly ordinary part, we may define
\begin{equation}\label{normalizationHZclasses}
\kappa_{\alpha}^\mrm{n.o.}
:=
\big( U_{\mathfrak{p}_{1}}^{-\alpha}, \mrm{id}\big)\kappa_{\alpha}^{\circ}\in \mrm{H}^{1}\big(\mathbb{Q},\bs{\cal{V}}_\alpha(K)\big),
\end{equation}
then the equality
\begin{equation}\label{compttt}
(\varpi_{2},\mrm{id})_*\kappa_{\alpha+1}^\mrm{n.o.}=\kappa_{\alpha}^\mrm{n.o.},
\end{equation}
 follows directly from the commutativity of $U_{\mathfrak{p}_1}$ with $(\varpi_{{2}})_*$ and Lemma $\ref{compatibility - classes}$.

\begin{definition}\label{BigCohomology}
We consider the inverse limit 
$\bs{\cal{V}}_\infty(K)
:=
\varprojlim_\alpha \bs{\cal{V}}_\alpha(K)$
taken with respect to the trace maps $(\varpi_{2})_*$. Then equation ($\ref{compttt}$) allows us to define 
\[
\boldsymbol{\kappa}_\infty^\mrm{n.o.}=\varprojlim_\alpha\ \kappa_{\alpha}^\mrm{n.o.}\in \mrm{H}^{1}\big(\mathbb{Q},\bs{\cal{V}}_\infty(K)\big).
\]
\end{definition}


\subsection{Specializations of the diagonal restriction of $\Lambda$-adic forms}

Let 
\[
K'_{\det}(p^\alpha)
:=
K_{\mbox{\tiny $\mrm{X}$}}(p^\alpha) \cap \mrm{GL}_2(\bb{A}_{\bb{Q}})
=
\left\{\gamma\in K'_0(p^\alpha)\Big\lvert\  \det(\gamma)\equiv  1 \pmod{p^{\alpha}}\right\},
\]
then the inclusion $K'_{\det}(p^\alpha)\le K_{\mbox{\tiny $\mrm{X}$}}(p^\alpha)$ induces an embedding
$\zeta:Y(K'_{\det}(p^\alpha))\hookrightarrow 
S(K_{\mbox{\tiny $\mrm{X}$}}(p^\alpha))$
of Shimura varieties fitting in the following commutative diagram
\begin{equation}\label{complex points}
\xymatrix{
Y(K'_{\det}(p^{\alpha}))(\bb{C})\ar@{^{(}->}[r]^\zeta  & S(K_{\mbox{\tiny $\mrm{X}$}}(p^\alpha))(\bb{C})\ar[r]^{\mathfrak{T}_{\tau_{\mathfrak{p}_2^\alpha}}} & S(K_{\diamond,t}(p^{\alpha}))(\bb{C})\\
 Y(K'_{0}(p^{\alpha}))(\bb{C})
      \ar@{^{(}->}[r]^{\zeta} &  S(K_\diamond(p^\alpha))(\bb{C}) \ar[u]_{\nu_\alpha}\ar[ru]_{w_{\mathfrak{p}_2^\alpha}}  &.}
\end{equation}

\begin{proposition}\label{analysis comp geometry}
    Let $\mrm{P}\in \cal{A}_{\bs{\chi}}(\mbf{I}_\scr{G})$ be an arithmetic point of weight $(2t_L,t_L)$ and character $(\chi_\circ\theta_L^{-1}\chi^{-1},\mathbbm{1})$ with $\chi$ a of conductor $p^\alpha$. Then
    \[
    e_{\mrm{ord}}\zeta^*\big(d_\frak{p}^{\bfcdot}\breve{\scr{G}}^{\mbox{\tiny $[\cal{P}]$}}\big)^\dagger(\mrm{P})
    =
	\msf{a}_p(\varpi_\mathfrak{p},\breve{\msf{g}}_\mrm{P})^\alpha \cdot G(\theta_{L,\mathfrak{p}}^{-1}\chi_\mathfrak{p}^{-1})^{-1}\cdot
    e_\mrm{ord}\zeta^*\Big[d^{-1}_\mu(w_{\mathfrak{p}_2^\alpha})^*(\breve{\msf{g}}_\mrm{P})^{\mbox{\tiny $[\cal{P}]$}}\Big].
    \]
\end{proposition}
\begin{proof}
Combining Proposition \ref{specializations} and Lemma \ref{Atkin-Lehner} we see that,
\[
e_{\mrm{ord}}\zeta^*\big(d_\frak{p}^{\bfcdot}\breve{\scr{G}}^{\mbox{\tiny $[\cal{P}]$}}\big)^\dagger(\mrm{P})=
\msf{a}_p(\varpi_\mathfrak{p},\breve{\msf{g}}_\mrm{P})^\alpha \cdot G(\theta_{L,\mathfrak{p}}^{-1}\chi_\mathfrak{p}^{-1})^{-1} \cdot
e_\mrm{ord}\zeta^*\Big[d_\mu^{-1}(\breve{\msf{g}}_\mrm{P}\lvert\tau^{-1}_{\mathfrak{p}^\alpha})^{\mbox{\tiny $[\cal{P}]$}}\Big]\otimes\theta_\bb{Q}\chi_{\mbox{\tiny $\spadesuit$}}.
\]
The cuspform
$\breve{\msf{g}}_\mrm{P}$
can be interpreted as a differential form on $S(K_{\diamond,t}(p^\alpha))(\bb{C})$, $\breve{\msf{g}}_\mrm{P}\lvert\tau^{-1}_{\mathfrak{p}^\alpha}=(\mathfrak{T}_{\mathfrak{p}_2^\alpha})^*\breve{\msf{g}}_\mrm{P}$
as a differential on $S(K_{\mbox{\tiny $\mrm{X}$}}(p^\alpha))(\bb{C})$ and  $(w_{\mathfrak{p}_2^\alpha})^*\breve{\msf{g}}_\mrm{P}=\nu_\alpha^*(\breve{\msf{g}}_\mrm{P}\lvert\tau^{-1}_{\mathfrak{p}^\alpha})$ as a differential on $S(K_{\diamond}(p^\alpha))(\bb{C})$. Since the morphism $\nu_\alpha:S(K_{\diamond}(p^\alpha))(\bb{C})\to S(K_{\mbox{\tiny $\mrm{X}$}}(p^\alpha))(\bb{C})$ is the identity with respect to the complex uniformizations (Corollary \ref{proj-nu_alpha-commute}), it preserves classical $q$-expansions. More precisely for all $\xi\in L_+$, every index $i$ we have
\[
a\Big(\xi,\big(\nu_\alpha^*(\breve{\msf{g}}_\mrm{P}\lvert\tau^{-1}_{\mathfrak{p}^\alpha})^{\mbox{\tiny $[\cal{P}]$}}\big)_i\Big)
=a\Big(\xi,(\breve{\msf{g}}_\mrm{P}\lvert\tau^{-1}_{\mathfrak{p}^\alpha})^{\mbox{\tiny $[\cal{P}]$}}_i\Big).
\]
Therefore, interpreting
$\nu_\alpha^*(\breve{\msf{g}}_\mrm{P}\lvert\tau^{-1}_{\mathfrak{p}^\alpha})$ 
and 
$\breve{\msf{g}}_\mrm{P}\lvert\tau^{-1}_{\mathfrak{p}^\alpha}$
as $p$-adic modular forms, we see that
\[
    a\Big(\xi,d^{-1}_\mu\big((w_{\mathfrak{p}_2^\alpha})^*(\breve{\msf{g}}_\mrm{P})^{\mbox{\tiny $[\cal{P}]$}}\big)_i\Big)
    =a\Big(\xi,d^{-1}_\mu(\breve{\msf{g}}_\mrm{P}\lvert\tau^{-1}_{\mathfrak{p}^\alpha})_i^{\mbox{\tiny $[\cal{P}]$}}\Big).
\]
The diagonal restriction $\zeta^*\Big[d^{-1}_\mu(\breve{\msf{g}}_\mrm{P}\lvert\tau^{-1}_{\mathfrak{p}^\alpha})^{\mbox{\tiny $[\cal{P}]$}}\Big]$
is a $p$-adic elliptic cuspform on $Y(K'_{\det}(p^\alpha))$. Its twist $\zeta^*\Big[d^{-1}_\mu(\breve{\msf{g}}_\mrm{P}\lvert\tau^{-1}_{\mathfrak{p}^\alpha})^{\mbox{\tiny $[\cal{P}]$}}\Big]\otimes\theta_\bb{Q}\chi_{\mbox{\tiny $\spadesuit$}}$ has character $(\psi_\circ^{-1},\mathbbm{1})$ and so descends to a $p$-adic cuspform on $Y(K'_0(p^\alpha))$ where it can be compared with $\zeta^*\Big[d^{-1}_\mu(w_{\mathfrak{p}_2^\alpha})^*(\breve{\msf{g}}_\mrm{P})^{\mbox{\tiny $[\cal{P}]$}}\Big]$. By Lemma \ref{twist classical expansion}, twisting by $\theta_\bb{Q}\chi_{\mbox{\tiny $\spadesuit$}}$ does not change the classical $q$-expansion of elliptic cuspforms on the identity component, thus the equality
\[
e_\mrm{ord}\zeta^*\Big[d^{-1}_\mu(\breve{\msf{g}}_\mrm{P}\lvert\tau^{-1}_{\mathfrak{p}^\alpha})^{\mbox{\tiny $[\cal{P}]$}}\Big]\otimes\theta_\bb{Q}\chi_{\mbox{\tiny $\spadesuit$}}=e_\mrm{ord}\zeta^*\Big[d^{-1}_\mu(w_{\mathfrak{p}_2^\alpha})^*(\breve{\msf{g}}_\mrm{P})^{\mbox{\tiny $[\cal{P}]$}}\Big]
\] follows from the $q$-expansion principle since $Y(K'_0(p^\alpha))$ is geometrically connected.
\end{proof}

\section{Review of big Galois representations}

Let $\mbf{Q}_\scr{G}=\mrm{Frac}(\mbf{I}_\scr{G})$, then the ordinary family $\scr{G}$ passing through a choice of ordinary $p$-stabilization  $\msf{g}_\circ^{\mbox{\tiny $(p)$}}$  has an associated big Galois representation $\bs{\varrho}_\scr{G}:\Gamma_L\to\mrm{GL}_2(\mbf{Q}_\scr{G})$ acting on $\mbf{V}_\scr{G}=(\mbf{Q}_\scr{G})^{\oplus2}$ (\cite{HidaGalois}, Theorem 1).
The representation $\boldsymbol{\varrho}_\scr{G}$ is unramified outside $\frak{Q}p$ with determinant  
\[
\det(\boldsymbol{\varrho}_\scr{G})(z)=\phi_{\bs{\chi}}([z,1])\cdot\varepsilon_L(z)\qquad \forall\ z\in\bb{A}_{L,f},
\]
and characteristic polynomial at a prime $\frak{q}\nmid\frak{Q}p$  given by  
\[
\det(1-\boldsymbol{\varrho}_\scr{G}(\mrm{Fr}_\frak{q})X)=1-\scr{G}(\mbf{T}(\frak{q}))X+\phi_{\bs{\chi}}([\varpi_\frak{q},1])\mrm{N}_{L/\bb{Q}}(\frak{q})X^2.
\]
Furthermore, fixing a decomposition group $D_\frak{p}$ in $\Gamma_L$ for an $\cal{O}_L$-prime $\frak{p}$ above $p$, there is an unramified character $\bs{\Psi}_{\scr{G},\frak{p}}:D_\frak{p}\to \mbf{I}_\scr{G}^\times$, $\mrm{Fr}_\frak{p}\mapsto\scr{G}(\mbf{T}(\varpi_\frak{p}))$, such that (\cite{HidaGalois}, Proposition 2.3)
 	\begin{equation}\label{ordinary-G}(\boldsymbol{\varrho}_{\scr{G}})_{\lvert D_\frak{p}}\sim\begin{pmatrix}
 	\bs{\Psi}_{\scr{G},\frak{p}}^{-1}\cdot\det(\bs{\varrho}_\scr{G})_{\lvert D_\frak{p}} &*\\
 	0&\bs{\Psi}_{\scr{G},\frak{p}}
 	\end{pmatrix}.
 	\end{equation}

\begin{definition}
Let \[
\mrm{As}(\mbf{V}_\scr{G}):=\otimes\mbox{-}\mrm{Ind}_L^\bb{Q}\left(\mbf{V}_\scr{G}\right)
\] denote the tensor induction of $\mbf{V}_\scr{G}$ to $\Gamma_\bb{Q}$.
\end{definition}

 \begin{definition}
 		Recall that according to \eqref{splitcyc} and \eqref{GaloisDecomposition}, the non-torsion part of the $p$-adic cyclotomic character $\eta_\bb{Q}:\bb{A}^\times_\bb{Q}\to\Gamma$ is given by $\eta_\bb{Q}(z)=\xi_z^{-1}$. We define  $\boldsymbol{\eta}_\bb{Q}:\Gamma_\bb{Q}\rightarrow \bs{\Lambda}^\times$ to be the Galois character associated to the idele character 
 	\[
	\bb{A}^\times_\bb{Q}\ni z\mapsto \big[\eta_\bb{Q}(z)\big].
 	\] 
Then, the restriction $(\bs{\eta}_\bb{Q})_{\lvert D_p}$  at the decomposition group at $p$ is the Galois character associated by local class field theory to the homomorphism $\bb{Q}_p^\times\to\Gamma\to \bs{\Lambda}^\times$, $x\mapsto \big[\langle x\rangle^{-1}\big]$.
 \end{definition}
\noindent A direct computation shows that
 \begin{equation} \mrm{As}\big(\det(\bs{\varrho}_\scr{G})\big)=\psi_\circ^{-1}\cdot\bs{\eta}_\bb{Q}^2\cdot\eta_\bb{Q}^2.
 \end{equation}
Since $p\cal{O}_L=\mathfrak{p}_1\mathfrak{p}_2$ splits in the real quadratic field $L$, the decomposition group  $D_p$ is contained in $\Gamma_L$. Hence, the characters $\bs{\Psi}_{\scr{G},\frak{p}}$ can be interpreted as characters of $D_p$ and it makes sense to define 
\begin{equation}
\bs{\Psi}_{\scr{G},p}:=\bs{\Psi}_{\scr{G},\frak{p}_1}\bs{\Psi}_{\scr{G},\frak{p}_2}.
\end{equation} 
From equation (\ref{ordinary-G}) we can deduce a concrete description of the action of the decomposition group at $p$ on $\mrm{As}(\mbf{V}_\scr{G})$.

\begin{proposition} \label{AsaiFil}
The restriction of $\mrm{As}(\mbf{V}_\scr{G})$ to $D_p$ admits a three step $D_p$-stable filtration 
\[
\mrm{As}(\mbf{V}_\scr{G})
\supset
\mrm{Fil}^1\mrm{As}(\mbf{V}_\scr{G})
\supset
\mrm{Fil}^2\mrm{As}(\mbf{V}_\scr{G})
\supset
0
\]
with graded pieces \[\mrm{Gr}^0\mrm{As}(\mbf{V}_\scr{G})=\mbf{Q}_\scr{G}\Big(\bs{\Psi}_{\scr{G},p}\Big),\qquad \mrm{Gr}^2\mrm{As}(\mbf{V}_\scr{G})=\mbf{Q}_\scr{G}\Big(\bs{\Psi}_{\scr{G},p}^{-1}\cdot\big(\psi_\circ^{-1}\cdot\bs{\eta}_\bb{Q}^2\cdot\eta_\bb{Q}^2\big)_{\lvert D_p}\Big),\]  \[\resizebox{\displaywidth}{!}{\xymatrix{
\mrm{Gr}^1\mrm{As}(\mbf{V}_\scr{G})=\mbf{Q}_\scr{G}\Big(\big(\bs{\Psi}_{\scr{G},\frak{p}_1}\bs{\Psi}_{\scr{G},\frak{p}_2}^{-1}\big)^{-1}\cdot\big(\psi^{-1}_{\circ,\frak{p}_1}\cdot\bs{\eta}_\bb{Q}\cdot\eta_\bb{Q}\big)_{\lvert D_{p}}\Big)
\oplus
\mbf{Q}_\scr{G}\Big(\bs{\Psi}_{\scr{G},\frak{p}_1}\bs{\Psi}_{\scr{G},\frak{p}_2}^{-1}\cdot\big(\psi^{-1}_{\circ,\frak{p}_2}\cdot\bs{\eta}_\bb{Q}\cdot\eta_\bb{Q}\big)_{\lvert D_{p}}\Big)}}\] 
where, for any prime $\frak{p}\mid p$, we let $\psi_{\circ,\frak{p}}^{-1}:D_p\to O^\times$ be the unramified character determined by $\psi_{\circ,\frak{p}}^{-1}(\mrm{Fr}_p)=\chi_\circ(\mrm{Fr}_\frak{p})$. In particular, $\psi_{\circ,\frak{p}_1}\cdot\psi_{\circ,\frak{p}_2}=\big(\psi_\circ\big)_{\lvert D_p}$.
\end{proposition}
\begin{proof}
Let $\mbf{V}_\scr{G}^+$ denote the subvector space of $\mbf{V}_\scr{G}$ coming from the upper left corner in equation (\ref{ordinary-G}), and $\mbf{V}_\scr{G}^- = \mbf{V}_\scr{G}/\mbf{V}_\scr{G}^+$ the dimension 1 quotient. We fix $\theta\in\Gamma_\bb{Q}\setminus\Gamma_L$ and define
\[
\mrm{Fil}^2\mrm{As}(\mbf{V}_\scr{G})
:=
\mbf{V}_\scr{G}^+ \otimes (\mbf{V}_\scr{G}^{+})^{\theta}\quad\text{and}\quad
\mrm{Fil}^1\mrm{As}(\mbf{V}_\scr{G})
:=
\mbf{V}_\scr{G}^+ \otimes (\mbf{V}_\scr{G})^{\theta}
+ \mbf{V}_\scr{G} \otimes (\mbf{V}_\scr{G}^{+})^{\theta},
\]
then $\mrm{Fil}^2\mrm{As}(\mbf{V}_\scr{G})$ has dimension 1 over $\mbf{Q}_\scr{G}$  while $\mrm{Fil}^1\mrm{As}(\mbf{V}_\scr{G})$ has dimension 3. By the description in (\ref{ordinary-G}), $D_p$ acts on $\mrm{Fil}^2\mrm{As}(\mbf{V}_\scr{G})$ through the character
$\bs{\Psi}_{\scr{G},p}^{-1}\cdot\big(\psi_\circ^{-1}\cdot\bs{\eta}_\bb{Q}^2\cdot\eta_\bb{Q}^2\big)_{\lvert D_p}$,
while it acts on the zero-th graded piece $\mrm{Gr}^0\mrm{As}(\mbf{V}_\scr{G})
=\mbf{V}^-_\scr{G} \otimes (\mbf{V}_\scr{G}^{-})^{\theta}$ through $\bs{\Psi}_{\scr{G},p}$.
Finally, the first graded piece is
\[
\mrm{Gr}^1\mrm{As}(\mbf{V}_\scr{G})=\mbf{Q}_\scr{G}\left(\bs{\Psi}_{\scr{G},\frak{p}_1}^{-1}\bs{\Psi}^\theta_{\scr{G},\frak{p}_1}\cdot\det(\bs{\varrho}_\scr{G})_{\lvert D_{\frak{p}_1}}\right)
\oplus
\mbf{Q}_\scr{G}\left(\bs{\Psi}_{\scr{G},\frak{p}_1}(\bs{\Psi}^{\theta}_{\scr{G},\frak{p}_1})^{-1}\cdot\det(\bs{\varrho}_\scr{G})^\theta_{\lvert D_{\frak{p}_1}}\right).\] 
Using the identification $D_{\frak{p}_1}=D_p$ we see that $\bs{\Psi}^\theta_{\scr{G},\frak{p}_1}=\bs{\Psi}_{\scr{G},\frak{p}_2}$ and that \[
\det(\bs{\varrho}_\scr{G})_{\lvert D_{\frak{p}_1}}=\big(\psi^{-1}_{\circ,\frak{p}_1}\cdot\bs{\eta}_\bb{Q}\cdot\eta_\bb{Q}\big)_{\lvert D_{p}},\qquad \det(\bs{\varrho}_\scr{G})_{\lvert D_{\frak{p}_1}}^\theta= \big(\psi^{-1}_{\circ,\frak{p}_2}\cdot\bs{\eta}_\bb{Q}\cdot\eta_\bb{Q}\big)_{\lvert D_{p}}.
\]
\end{proof}

\begin{definition}\label{selfdual remark}
Let $\mrm{V}_{\msf{f}_\circ}$ denote the representation attached to the elliptic cuspform $\msf{f}_\circ$ and set
 \[
\mrm{As}(\mbf{V}_\scr{G})^\dagger:=\mrm{As}(\mbf{V}_\scr{G})(\theta_\bb{Q}\cdot\bs{\eta}_\bb{Q}^{-1})
 \]
 for $\theta_\bb{Q}$ the torsion part of the $p$-adic cyclotomic character (see \eqref{splitcyc}).
We also set 
 \begin{equation}
\mbf{V}^\dagger_{\scr{G},\msf{f}_\circ}:=\mrm{As}(\mbf{V}_\scr{G})^\dagger(-1)\otimes\mrm{V}_{\msf{f}_\circ}.
 \end{equation}
\end{definition} 

\begin{remark}
 The big Galois representation $\mbf{V}^\dagger_{\scr{G},\msf{f}_\circ}$ interpolates Kummer self-dual representations. The claim follows from the following isomorphisms of Galois representations
  \[
 \mrm{As}(\mbf{V}_\scr{G})_{\mrm{P}}^\dagger(-1)\overset{\sim}{\longrightarrow} \big(\mrm{As}(\mbf{V}_\scr{G})_{\mrm{P}}^\dagger(\psi_\circ)\big)^*(1),\qquad \mrm{V}_{\msf{f}_\circ}\overset{\sim}{\longrightarrow}\big(\mrm{V}_{\msf{f}_\circ}(\psi_\circ^{-1})\big)^*(1),
 \]
where $\mrm{As}(\mbf{V}_\scr{G})_{\mrm{P}}^\dagger$ denotes the specialization at an arithmetic point $\mrm{P}\in\cal{A}_{\bs{\chi}}(\mbf{I}_\scr{G})$ of weight $(2t_L,t_L)$ and character $(\chi_\circ,\theta_L^{-1}\chi^{-1},\mathbbm{1})$. Those isomorphisms are induced by the Hecke equivariant twists of the Poincar\'e pairing for Hilbert modular surfaces and modular curves (see also \eqref{heckeequivariantpoincare}) once the representations are realized in the appropriate \'etale cohomology group. 
\end{remark}

\noindent The explicit realization of $\mbf{V}^\dagger_{\scr{G},\msf{f}_\circ}$ in the cohomology of a tower of Shimura threefolds with increasing level at $p$ plays a crucial role in the understanding of the arithmetic applications of Hirzebruch--Zagier classes. We conclude this section by analyzing the ordinary filtration at $p$. Recall that, analogously to $\mbf{V}_\scr{G}$, the representation $\mrm{V}_{\msf{f}_\circ}$ is unramified outside $Np$ with determinant  $\varepsilon_\bb{Q}\cdot \psi_\circ$, and that there exists an unramified character $\delta_p(\msf{f}_\circ):D_p\to O^\times$ for a choice of decomposition group $D_p$ in $\Gamma_\bb{Q}$ such that 
 	\[\big(\mrm{V}_{\msf{f}_\circ}\big)_{\lvert D_\frak{p}}\sim\begin{pmatrix}
 	\delta_p(\msf{f}_\circ)^{-1}\cdot(\varepsilon_\bb{Q}\cdot \psi_\circ)_{\lvert D_\frak{p}} &*\\
 	0&\delta_p(\msf{f}_\circ)
 	\end{pmatrix}.
 	\]

\begin{lemma}\label{GaloisStructure}
	The restriction to a decomposition group at $p$ of the Galois representation $\mbf{V}^\dagger_{\scr{G},\msf{f}_\circ}$ is endowed with a $4$-steps $D_p$-stable filtration with graded pieces given by
	\[\resizebox{\displaywidth}{!}{\xymatrix{
	\mrm{Gr}^0\mbf{V}^\dagger_{\scr{G},\msf{f}_\circ}=\mbf{Q}_\scr{G}\Big(\bs{\Psi}_{\scr{G},p}\cdot\delta_p(\msf{f}_\circ)\cdot\big(\bs{\eta}_\bb{Q}^{-1}\cdot\eta_\bb{Q}^{-1}\big)_{\lvert D_p}\Big),\qquad \mrm{Gr}^3\mbf{V}^\dagger_{\scr{G},\msf{f}_\circ}=\mbf{Q}_\scr{G}\Big(\bs{\Psi}_{\scr{G},p}^{-1}\cdot \delta_p(\msf{f}_\circ)^{-1} \cdot\big(\bs{\eta}_\bb{Q}\cdot\eta^2_\bb{Q}\cdot\theta_\bb{Q}\big)_{\lvert D_p}\Big),
	}}\]
	\[\resizebox{\displaywidth}{!}{\xymatrix{
	\mrm{Gr}^1\mbf{V}^\dagger_{\scr{G},\msf{f}_\circ}=
\mbf{Q}_\scr{G}\Big(\bs{\Psi}_{\scr{G},p}\cdot\delta_p(\msf{f}_\circ)^{-1}\cdot\big(\bs{\eta}_\bb{Q}^{-1}\cdot\theta_\bb{Q}\cdot\psi_\circ\big)_{\lvert D_p}\Big)	
\oplus \mbf{Q}_\scr{G}\Big(\big(\bs{\Psi}_{\scr{G},\frak{p}_1}\bs{\Psi}_{\scr{G},\frak{p}_2}^{-1}\big)^{-1}\cdot\delta_p(\msf{f}_\circ)\cdot\big(\psi^{-1}_{\circ,\frak{p}_1}\big)_{\lvert D_{p}}\Big)
\oplus
\mbf{Q}_\scr{G}\Big(\bs{\Psi}_{\scr{G},\frak{p}_1}\bs{\Psi}_{\scr{G},\frak{p}_2}^{-1}\cdot\delta_p(\msf{f}_\circ)\cdot\big(\psi^{-1}_{\circ,\frak{p}_2}\big)_{\lvert D_p}\Big),
}}\]
\[\resizebox{\displaywidth}{!}{\xymatrix{
\mrm{Gr}^2\mbf{V}^\dagger_{\scr{G},\msf{f}_\circ}=
\mbf{Q}_\scr{G}\Big(\big(\bs{\Psi}_{\scr{G},\frak{p}_1}\bs{\Psi}_{\scr{G},\frak{p}_2}^{-1}\big)^{-1}\cdot\delta_p(\msf{f}_\circ)^{-1}\cdot\big(\psi_{\circ,\frak{p}_2}\cdot\varepsilon_{\bb{Q}}\big)_{\lvert D_p}\Big)
\oplus
\mbf{Q}_\scr{G}\Big(\bs{\Psi}_{\scr{G},\frak{p}_1}\bs{\Psi}_{\scr{G},\frak{p}_2}^{-1}\cdot\delta_p(\msf{f}_\circ)^{-1}\cdot\big(\psi_{\circ,\frak{p}_1}\cdot\varepsilon_{\bb{Q}}\big)_{\lvert D_p}\Big)	
\oplus \mbf{Q}_\scr{G}\Big(\bs{\Psi}_{\scr{G},p}^{-1}\cdot\delta_p(\msf{f}_\circ)\cdot\big(\bs{\eta}_\bb{Q}\cdot\eta_\bb{Q}\cdot \psi_\circ^{-1}\big)_{\lvert D_p}\Big).
}}\]
\end{lemma}
\begin{proof}
 As $\mbf{V}^\dagger_{\scr{G},\msf{f}_\circ}=\mrm{As}(\mbf{V}_\scr{G})(\bs{\eta}_\bb{Q}^{-1}\cdot\eta^{-1}_\bb{Q})\otimes\mrm{V}_{\msf{f}_\circ}$, its graded pieces are given by
\[
		\mrm{Gr}^0\mbf{V}^\dagger_{\scr{G},\msf{f}_\circ}=\mrm{Gr}^0\mrm{As}(\mbf{V}_\scr{G})(\bs{\eta}_\bb{Q}^{-1}\cdot\eta^{-1}_\bb{Q})\otimes\mrm{Gr}^0\mrm{V}_{\msf{f}_\circ},
		\]
		\[
		\mrm{Gr}^1\mbf{V}^\dagger_{\scr{G},\msf{f}_\circ}=[\mrm{Gr}^0\mrm{As}(\mbf{V}_\scr{G})(\bs{\eta}_\bb{Q}^{-1}\cdot\eta^{-1}_\bb{Q})\otimes\mrm{Gr}^1\mrm{V}_{\msf{f}_\circ}]\oplus [\mrm{Gr}^1\mrm{As}(\mbf{V}_\scr{G})(\bs{\eta}_\bb{Q}^{-1}\cdot\eta^{-1}_\bb{Q})\otimes\mrm{Gr}^0\mrm{V}_{\msf{f}_\circ}],
		\]
		\[
		\mrm{Gr}^2\mbf{V}^\dagger_{\scr{G},\msf{f}_\circ}=[\mrm{Gr}^1\mrm{As}(\mbf{V}_\scr{G})(\bs{\eta}_\bb{Q}^{-1}\cdot\eta^{-1}_\bb{Q})\otimes\mrm{Gr}^1\mrm{V}_{\msf{f}_\circ}]\oplus [\mrm{Gr}^2\mrm{As}(\mbf{V}_\scr{G})(\bs{\eta}_\bb{Q}^{-1}\cdot\eta^{-1}_\bb{Q})\otimes\mrm{Gr}^0\mrm{V}_{\msf{f}_\circ}],
		\]
		\[
		\mrm{Gr}^3\mbf{V}^\dagger_{\scr{G},\msf{f}_\circ}=\mrm{Gr}^2\mrm{As}(\mbf{V}_\scr{G})(\bs{\eta}_\bb{Q}^{-1}\cdot\eta^{-1}_\bb{Q})\otimes\mrm{Gr}^1\mrm{V}_{\msf{f}_\circ}.
		\]
	Hence, the statement follows from Proposition $\ref{AsaiFil}$ and a direct computation.	
\end{proof}

\begin{definition}\label{somedefgal}
	We define the direct summand $\mbf{V}^{\mrm{f}_\circ}_\scr{G}$ of $\mrm{Gr}^2\mbf{V}^\dagger_{\scr{G},\msf{f}_\circ}$ by setting
		\[
		\mbf{V}^{\mrm{f}_\circ}_\scr{G}:=\mbf{Q}_\scr{G}\Big(\bs{\Psi}_{\scr{G},p}^{-1}\cdot\delta_p(\msf{f}_\circ)\cdot\big(\bs{\eta}_\bb{Q}\cdot\eta_\bb{Q}\cdot \psi_\circ^{-1}\big)_{\lvert D_p}\Big).
		\]
	
\end{definition}

\subsection{Geometric realization}\label{geomrealiz}

Recall that $M$ is the integer appearing in Theorem \ref{Ichino}.
\begin{definition}
For the compact open $K=V_{1}(M\cal{O}_L)$ we define the \emph{anemic} Hecke algebra 
\[
\widetilde{\mbf{h}}^\mrm{n.o.}_L(K;O)\subseteq\mbf{h}^\mrm{n.o.}_L(K;O)
\] 
to be the $O$-subalgebra generated by the Hecke operators $\mbf{T}(y)$ with $y_M=1$. 
\end{definition}
\noindent Given the Hida family $\scr{G}:\mbf{h}^\mrm{n.o.}_L(K;O)_{\bs{\chi}}\to \mbf{I}_\scr{G}$ we denote by 
\begin{equation}\label{heartform}
\scr{G}_{\mbox{\tiny $\heartsuit$}}:\widetilde{\mbf{h}}^\mrm{n.o.}_L(K;O)_{\bs{\chi}}\rightarrow\mbf{I}_\scr{G}
\end{equation} 
its restriction  to the anemic Hecke algebra. It is used to single out the part of Hecke modules most relevant for our applications.

\begin{definition}
   We define  $\bs{\cal{V}}_\mathscr{G}(M)$ to be the projective limit of
   	\[
   \bs{\cal{V}}_\mathscr{G}(M)_\alpha:= \cal{V}_\alpha \otimes_{\scr{G}_{\mbox{\tiny $\heartsuit$}}}\mbf{I}_\scr{G}(\theta_\bb{Q}\cdot\bs{\eta}_\bb{Q}^{-1})\qquad 	  \forall\ \alpha\ge1
   	\]
    with respect to the trace maps $(\varpi_2)_*$ where $\cal{V}_\alpha:=e_\mrm{n.o.}\mrm{H}^2_\et\big(S(K_{\diamond,t}(p^\alpha))_{\bar{\bb{Q}}},O(2)\big)$, and set
		 \[
		 	\bs{\cal{V}}_{\mathscr{G},\msf{f}_\circ}(M):= \bs{\cal{V}}_\mathscr{G}(M)(-1)\otimes \mrm{V}_{\msf{f}_\circ}(p)
		 \] 
		 where  $\mrm{V}_{\msf{f}_\circ}(p)$ denotes the $\msf{f}_\circ$-isotypic quotient of  $\mrm{H}^1_\et\big(X_{1,0}(N,p)_{\bar{\bb{Q}}},O(1)\big)$.
	\end{definition} 
	
\noindent  To highlight the dependence on the integer $M$, we  slightly modify the notation introduced in Definition \ref{BigCohomology} by setting $\bs{\cal{V}}_\infty(M):=\bs{\cal{V}}_\infty(K)$ where $K=V_1(M\cal{O}_L)$. Now, let $\bs{\delta}: \Gamma_\bb{Q}\to O\llbracket \bb{G}_L(K)\rrbracket^\times$ be the projective limit of the Galois characters $\delta_\alpha$ defined in ($\ref{GalCal}$). Since the Galois character $\scr{G}_{\mbox{\tiny $\heartsuit$}}\circ\boldsymbol{\delta}:\Gamma_\bb{Q}\to(\mbf{I}_\scr{G})^\times$ satisfies
\[\begin{split}
\scr{G}_{\mbox{\tiny $\heartsuit$}}\circ\boldsymbol{\delta}(\sigma_a) 
&=
\phi_{\bs{\chi}}\big([(1,a^{-1}),(a,a^{-1})]\big)\\
 &=
 \theta_\bb{Q}(a)\cdot\bs{\eta}_\bb{Q}^{-1}(a)
\end{split}
\qquad \forall\ a\in\bb{Z}_p^\times,
\]
we see that there is a natural Galois equivariant surjection
\begin{equation}
	\xymatrix{
\mrm{pr}_{\scr{G},\msf{f}_\circ}:\bs{\cal{V}}_\infty(M)\ar@{->>}[r]& \bs{\cal{V}}_{\mathscr{G},\msf{f}_\circ}(M).
}\end{equation}

\begin{theorem}\label{wishingDimitrov}
	Suppose that $\varrho$ is residually not solvable, then for all odd primes $p$ the Galois module $\bs{\cal{V}}_\mathscr{G}(M)$ is finite free over $\mbf{I}_\scr{G}$ and it satisfies exact control 
	\[
	\bs{\cal{V}}_\mathscr{G}(M)\otimes_{\bs{\Lambda}}\Lambda_\alpha\cong \bs{\cal{V}}_\mathscr{G}(M)_\alpha\qquad \forall\ \alpha\ge1.
	\]
\end{theorem}
\begin{proof}
	This can be proved as (\cite{DimitrovAutSym}, Theorem 3.8 (i),(ii)). The key new input is the recent work of Caraiani--Tamiozzo (\cite{Caraiani-Tamiozzo}, Theorem 7.1.1 $\&$ Corollary 7.1.2).
\end{proof}

\begin{corollary}\label{correctspec}
	Suppose $p>2$ and $\varrho$ residually not solvable, then the Galois representation $\bs{\cal{V}}_\mathscr{G}(M)$ is isomorphic to a direct sum of copies of $\mrm{As}(\mbf{V}_\scr{G})^\dagger$. Moreover, the specialization $\bs{\cal{V}}_\mathscr{G}(M)(-1)\otimes_{\mrm{P}_\circ}E_\wp$ is  a sum of copies of the Artin representation $\mrm{As}(\varrho)$.
\end{corollary}
\begin{proof}
	It follows by Theorem \ref{wishingDimitrov} and a comparison of traces (\cite{Mazur}, Section 5).
\end{proof}

\begin{proposition}\label{somekindoffil}
If the Jordan--Holder factors of the residual representation $\mrm{As}(\mbf{V}_\scr{G})\otimes_{\mbf{I}_\scr{G}}\overline{\bb{F}}_p$ are all distinct, then the $\mbf{I}_{\scr{G}}$-module $\bs{\cal{V}}_\mathscr{G}(M)$ is endowed with a three step $\Gamma_{\bb{Q}_p}$-stable filtration 
\[
\bs{\cal{V}}_\mathscr{G}(M)
\supset
\mrm{Fil}^1\bs{\cal{V}}_\mathscr{G}(M)
\supset
\mrm{Fil}^2\bs{\cal{V}}_\mathscr{G}(M)
\supset
0.
\]
Furthermore, there are $\mbf{I}_\scr{G}$-modules $\mbf{A}, \mbf{B}, \mbf{B}', \mbf{C}$ with trivial $\Gamma_{\bb{Q}_p}$-action such that
\[\resizebox{\displaywidth}{!}{\xymatrix{
\mrm{Gr}^0\bs{\cal{V}}_\mathscr{G}(M)=\mbf{A}\Big(\bs{\Psi}_{\scr{G},p}\cdot (\theta_\bb{Q}\cdot\bs{\eta}_\bb{Q}^{-1})_{\lvert D_p}\Big),
\qquad
\mrm{Gr}^2\bs{\cal{V}}_\mathscr{G}(M)=\mbf{C}\Big(\bs{\Psi}_{\scr{G},p}^{-1}\cdot\big(\psi_\circ^{-1}\cdot\varepsilon_\bb{Q}\cdot\bs{\eta}_\bb{Q}\cdot\eta_\bb{Q}\big)_{\lvert D_p}\Big),
}}\]
and the first graded piece is an extension
\[\resizebox{\displaywidth}{!}{\xymatrix{ \mbf{B}\Big(\big(\bs{\Psi}_{\scr{G},\frak{p}_1}\bs{\Psi}_{\scr{G},\frak{p}_2}^{-1}\big)^{-1}\cdot\big(\psi^{-1}_{\circ,\frak{p}_1}\cdot\varepsilon_\bb{Q}\big)_{\lvert D_{p}}\Big)\ar@{^{(}->}[r]& \mrm{Gr}^1\bs{\cal{V}}_\mathscr{G}(M)\ar@{->>}[r]& \mbf{B}'\Big(\bs{\Psi}_{\scr{G},\frak{p}_1}\bs{\Psi}_{\scr{G},\frak{p}_2}^{-1}\cdot\big(\psi^{-1}_{\circ,\frak{p}_2}\cdot\varepsilon_\bb{Q}\big)_{\lvert D_{p}}\Big).
}}\]
\end{proposition}
\begin{proof}
	By (\cite{BL}, Chapter 3.4 $\&$ \cite{ES-Nekovar}, Theorem 5.20)  there is a Galois equivariant injection
	\begin{equation}\label{nekBL}
		\xymatrix{ \bs{\cal{V}}_\mathscr{G}(M)_\alpha\ar@{^{(}->}[r]& 
	\bigoplus_\mrm{P}\Big(\mrm{As}(\mrm{V}_{\scr{G}_\mrm{P}})(\theta_\bb{Q}\cdot\chi_{\mbox{\tiny $\spadesuit,\mrm{P}$}})\Big)^{\oplus n}
	}\end{equation}
	where the sum taken over arithmetic points $\mrm{P}\in\cal{A}_{\bs{\chi}}(\mbf{I}_\scr{G})$ of weight $(2t_L,t_L)$ and level $p^\alpha$ and where $n$ is the number of divisors of $M/\frak{Q}$.
	 The right-hand side of ($\ref{nekBL}$) is endowed with a nearly ordinary filtration (Proposition $\ref{AsaiFil}$). Therefore, if we set $\mbf{I}_{\scr{G},\alpha}:=\mbf{I}_{\scr{G}}\otimes\Lambda_\alpha$, the Galois module $\bs{\cal{V}}_\mathscr{G}(M)_\alpha$ inherits a three step $\Gamma_{\bb{Q}_p}$-stable filtration consisting of $\mbf{I}_{\scr{G},\alpha}$-modules 
\[
\bs{\cal{V}}_\mathscr{G}(M)_\alpha
\supset
\mrm{Fil}^1\bs{\cal{V}}_\mathscr{G}(M)_\alpha
\supset
\mrm{Fil}^2\bs{\cal{V}}_\mathscr{G}(M)_\alpha
\supset
0.
\]
Moreover, there are $\mbf{I}_{\scr{G},\alpha}$-modules $\mbf{A}_\alpha, \mbf{B}_\alpha, \mbf{B}_\alpha', \mbf{C}_\alpha$ with trivial $\Gamma_{\bb{Q}_p}$-action such that
\[\resizebox{\displaywidth}{!}{\xymatrix{
\mrm{Gr}^0\bs{\cal{V}}_\mathscr{G}(M)_\alpha=\mbf{A}_\alpha\Big(\bs{\Psi}_{\scr{G},p}\cdot (\theta_\bb{Q}\cdot\bs{\eta}_\bb{Q}^{-1})_{\lvert D_p}\Big),
\qquad
\mrm{Gr}^2\bs{\cal{V}}_\mathscr{G}(M)_\alpha=\mbf{C}_\alpha\Big(\bs{\Psi}_{\scr{G},p}^{-1}\cdot\big(\psi_\circ^{-1}\cdot\varepsilon_\bb{Q}\cdot\bs{\eta}_\bb{Q}\cdot\eta_\bb{Q}\big)_{\lvert D_p}\Big),
}}\]
\[\resizebox{\displaywidth}{!}{\xymatrix{
0\ar[r]& \mbf{B}_\alpha\Big(\big(\bs{\Psi}_{\scr{G},\frak{p}_1}\bs{\Psi}_{\scr{G},\frak{p}_2}^{-1}\big)^{-1}\cdot\big(\psi^{-1}_{\circ,\frak{p}_1}\cdot\varepsilon_\bb{Q}\big)_{\lvert D_{p}}\Big)\ar[r]& \mrm{Gr}^1\bs{\cal{V}}_\mathscr{G}(M)_\alpha\ar[r]& \mbf{B}_\alpha'\Big(\bs{\Psi}_{\scr{G},\frak{p}_1}\bs{\Psi}_{\scr{G},\frak{p}_2}^{-1}\cdot\big(\psi^{-1}_{\circ,\frak{p}_2}\cdot\varepsilon_\bb{Q}\big)_{\lvert D_{p}}\Big)\ar[r]&0.
}}\]
Since the Jordan--Holder factors of the residual representation $\mrm{As}(\mbf{V}_\scr{G})\otimes_{\mbf{I}_\scr{G}}\overline{\bb{F}}_p$ are all distinct, the characters appearing in the graded pieces of $\bs{\cal{V}}_\mathscr{G}(M)_\alpha$ are all distinct. It follows that the Galois equivariant transition maps
$\bs{\cal{V}}_\mathscr{G}(M)_{\alpha+1}\rightarrow \bs{\cal{V}}_\mathscr{G}(M)_{\alpha}$
respect the filtration and the claim follows by taking projective limits.
\end{proof}

\begin{remark}
	Recall the primitive eigenform $\msf{g}_\circ\in S_{t_L,t_L}(\mathfrak{Q};\chi_\circ;O)$ and denote by  $\alpha_i,\beta_i$ the eigenvalues of $\varrho_{\msf{g}_\circ}(\mrm{Fr}_{\mathfrak{p}_i})$ for $\mathfrak{p}_1,\mathfrak{p}_2$ the $\cal{O}_L$-prime ideals above $p$.
	Thanks to Proposition $\ref{AsaiFil}$, the Jordan--Holder factors of the residual representation $\mrm{As}(\mbf{V}_\scr{G})\otimes_{\mbf{I}_\scr{G}}\overline{\bb{F}}_p$ are all distinct if and only if  the products $\alpha_1\alpha_2,\ \alpha_1\beta_2,\ \beta_1\alpha_2,\ \beta_1\beta_2$
	are all distinct in $\overline{\bb{F}}_p$.
\end{remark}

\noindent Let $K/\bb{Q}$ be a non-totally real $S_5$-quintic extension whose Galois closure contain a real quadratic field $L$. Recall there is a parallel weight one Hilbert eigenform $\msf{g}_K$ over $L$  such that $\mrm{As}(\varrho_{\msf{g}_K})\cong\mrm{Ind}_K^\bb{Q}\mathbbm{1}-\mathbbm{1}$ (\cite{MicAnalytic}, Corollary 4.3).
\begin{proposition}\label{distinct eigenvalues quintic}
	If $p\not=5$ is a rational prime unramified in $K$ whose Frobenius conjugacy class is that of $5$-cycles in $S_5$, then $p$ splits in $L$ and the residual $\Gamma_{\bb{Q}_p}$-representation $\left(\mrm{As}(\varrho_{\msf{g}_K})\otimes\overline{\bb{F}}_p\right)_{\lvert D_p}$  has distinct Jordan--Holder factors.
\end{proposition}
\begin{proof}
	The representation $\mrm{As}(\varrho_{\msf{g}_K}):\Gamma_\bb{Q}\to\mrm{GL}_4(O)$ factors through the Galois group of the Galois closure of $K$. As a representation of the symmetric group $S_5$ it is isomorphic to the irreducible $4$-dimensional direct summand of the permutation representation of $S_5$ acting on $5$ elements. If $p\not=5$ is a rational prime unramified in $K$ whose Frobenius conjugacy class is that of $5$-cycles in $S_5$, then the decomposition group $D_p$ is cyclic of order $5$ and we can conclude by noting that  $\big(\mrm{As}(\varrho_{\msf{g}_K})\otimes\overline{\bb{F}}_p\big)(\mrm{Fr}_p)$ has four distinct eigenvalues given by the non-trivial $5$-th roots of unity.
\end{proof}

\subsection{Hodge--Tate numerology}
Let $\msf{g}$ be a primitive Hilbert cuspform over $L$ of weight $(\ell t_L,t_L)$
and normalize Hodge--Tate weights by stating that the character $\varepsilon_\bb{Q}$ has weight $-1$. Then for every $\cal{O}_L$-prime ideal $\mathfrak{p}\mid p$, the restriction $\big(\mrm{V}_\msf{g}\big)_{\lvert D_\mathfrak{p}}$ has a $D_\frak{p}$-stable filtration
\[\xymatrix{
0\ar[r]& \mrm{V}^+_\frak{p}\ar[r]& \big(\mrm{V}_\msf{g}\big)_{\lvert D_\mathfrak{p}}\ar[r]& \mrm{V}^-_\frak{p}\ar[r]&0
}\] 
where $\mrm{V}^+_\frak{p}$ is a one-dimensional subrepresentation with Hodge--Tate weights equal to $ 1-\ell$ and $\mrm{V}^-_\frak{p}$ is a one-dimensional quotient with Hodge--Tate weights equal to $0$ (\cite{pHida}, Introduction). Therefore, the twist 
\[
\mrm{V}_\msf{g}^\dagger:= \mrm{V}_\msf{g}\left(\eta_L^{\frac{2-\ell}{2}}\right)
\]
 has Hodge--Tate weights at $\frak{p}$ given by $\{ -\frac{\ell}{2}, \frac{\ell-2}{2}\}_{\tau\in\mrm{I}_{L,\mathfrak{p}}}$. As in Proposition $\ref{AsaiFil}$, when $p\cal{O}_L=\mathfrak{p}_1\mathfrak{p}_2$ splits, the restriction at $p$ of the Asai representation 
 \[
 \mrm{As}\big(\mrm{V}^\dagger_\msf{g}\big)_{\lvert D_p}=\big(\mrm{V}^\dagger_\msf{g}\big)_{\lvert D_{\mathfrak{p}_1}}\otimes\big(\mrm{V}^\dagger_\msf{g}\big)_{\lvert D_{\mathfrak{p}_2}},
 \] 
 is endowed with a $3$-step $D_p$-stable filtration 
\[
\mrm{As}\big(\mrm{V}^\dagger_\msf{g}\big)\supset\mrm{Fil}^1\mrm{As}\big(\mrm{V}^\dagger_\msf{g}\big)\supset\mrm{Fil}^2\mrm{As}\big(\mrm{V}^\dagger_\msf{g}\big)\supset\{0\}
\]
whose graded pieces have dimension $1,2$ and $1$ respectively and whose
Hodge--Tate weights are given in the following table.

\begin{center}
    \begin{tabular}{lc}
    \toprule
    Graded piece & Hodge--Tate weights \\ 
	\midrule
    $\mrm{Gr}^0\mrm{As}(\mrm{V}^\dagger_\msf{g})$
	& $\ell-2$  \\ 
	\midrule
 $\mrm{Gr}^1\mrm{As}(\mrm{V}^\dagger_\msf{g})$
 &  $(-1, -1)$ \\ 
 \midrule
   $\mrm{Gr}^2\mrm{As}(\mrm{V}^\dagger_\msf{g})$
   &  $-\ell$ \\ 
   \bottomrule
    \end{tabular}
\end{center}

\noindent Furthermore, as in Lemma $\ref{GaloisStructure}$, the restriction at $p$ of the $\Gamma_\bb{Q}$-representation 
\[
\mrm{V}^\dagger_{\msf{g},\msf{f}_\circ}=\mrm{As}\big(\mrm{V}^\dagger_\msf{g}\big)(-1)\otimes\mrm{V}_{\msf{f}_\circ}
\] 
inherits a $4$-step $D_p$-stable filtration
\[
\mrm{V}^\dagger_{\msf{g},\msf{f}_\circ}\supset\mrm{Fil}^1\mrm{V}^\dagger_{\msf{g},\msf{f}_\circ}\supset\mrm{Fil}^2\mrm{V}^\dagger_{\msf{g},\msf{f}_\circ}\supset\mrm{Fil}^3\mrm{V}^\dagger_{\msf{g},\msf{f}_\circ}\supset\{0\}
\]
with graded pieces of dimension $1,3,3$ and $1$ respectively and whose Hodge--Tate weights are presented in the following table.

\begin{center}
    \begin{tabular}{lc}
    \toprule
    Graded piece & Hodge--Tate weights \\ 
	\midrule
    $\mrm{Gr}^0\mrm{V}^\dagger_{\msf{g},\msf{f}_\circ}$
	& $\ell-1$  \\ \midrule
 $\mrm{Gr}^1\mrm{V}^\dagger_{\msf{g},\msf{f}_\circ}$
 &  $(\ell-2,\ 0,\ 0)$ \\ \midrule
   $\mrm{Gr}^2\mrm{V}^\dagger_{\msf{g},\msf{f}_\circ}$
   &  $(-1,\ -1,\ 1-\ell)$ \\ \midrule
    $\mrm{Gr}^3\mrm{V}^\dagger_{\msf{g},\msf{f}_\circ}$
	&  $-\ell$ \\ \bottomrule
    \end{tabular}
\end{center}

\begin{corollary}\label{negative HT weights}
	The H--T weights of $\mrm{Fil}^2\mrm{V}^\dagger_{\msf{g},\msf{f}_\circ}$ are all strictly negative if and only if $\ell\ge 2$.
\end{corollary}


\subsection{Local cohomology classes}
From now onward we suppose that the Jordan--Holder factors of the residual representation $\mrm{As}(\varrho_\circ)\otimes_{O}\overline{\bb{F}}_p$ are all distinct. Then, by Proposition \ref{somekindoffil}, the Galois module $\bs{\cal{V}}_{\scr{G},\msf{f}_\circ}(M)$ has a $4$-step $D_p$-stable filtration, and the Galois action on its graded pieces is given by characters appearing in Lemma \ref{GaloisStructure}.

\begin{lemma}\label{fil2 inj}
The natural map
\[
\mrm{H}^1(\mathbb{Q}_p,\mrm{Fil}^2\bs{\cal{V}}_{\scr{G},\msf{f}_\circ}(M))
\longrightarrow
\mrm{H}^1(\mathbb{Q}_p,\bs{\cal{V}}_{\scr{G},\msf{f}_\circ}(M))
\]
induced by the $D_p$-stable filtration is an injection.
\end{lemma}
\begin{proof}
	Lemma $\ref{GaloisStructure}$ implies that
	\[\mrm{H}^0(\bb{Q}_p,\mrm{Gr}^1\bs{\cal{V}}_{\scr{G},\msf{f}_\circ}(M))=0,\qquad \mrm{H}^0(\bb{Q}_p,\mrm{Gr}^0\bs{\cal{V}}_{\scr{G},\msf{f}_\circ}(M))=0.\]
Therefore taking the long exact sequence in Galois cohomology associated with the short exact sequence of $D_p$-modules
	\[\xymatrix{
	0\ar[r]& \mrm{Gr}^1\bs{\cal{V}}_{\scr{G},\msf{f}_\circ}(M)\ar[r]&\bs{\cal{V}}_{\scr{G},\msf{f}_\circ}(M)/\mrm{Fil}^2\ar[r]&\mrm{Gr}^0\bs{\cal{V}}_{\scr{G},\msf{f}_\circ}(M)\ar[r]&0,
	}\]
we deduce that	 $\mrm{H}^0(\bb{Q}_p,\bs{\cal{V}}_{\scr{G},\msf{f}_\circ}(M)/\mrm{Fil}^2)=0$ and the lemma follows.
\end{proof}

\subsubsection{Local properties.}
\begin{definition}
The $\mbf{I}_\scr{G}$-adic cohomology class attached to the pair $(\scr{G}, \msf{f}_\circ)$ is the projection
\[
\boldsymbol{\kappa}_{\scr{G},\msf{f}_\circ}
:=
\mrm{pr}_{\scr{G},\msf{f}_\circ}(\bs{\kappa}^\mrm{n.o.}_\infty)
\in
\mrm{H}^1\big(\mathbb{Q},\bs{\cal{V}}_{\scr{G},\msf{f}_\circ}(M)\big).
\]
We denote its restriction at the decomposition group at $p$ by
\[
\bs{\kappa}_p(\scr{G},\msf{f}_\circ):=\mrm{loc}_p(\boldsymbol{\kappa}_{\scr{G},\msf{f}_\circ})
\in
\mrm{H}^1\big(\mathbb{Q}_p,\bs{\cal{V}}_{\scr{G},\msf{f}_\circ}(M)\big).
\]

\end{definition}

\noindent 	Consider the following element of the ring $\mbf{I}_\scr{G}$ 
	\begin{equation}
	\bs{\xi}_{\scr{G},\msf{f}_\circ} := 
	\Big(1-\alpha_{\msf{f}_\circ}\chi_\circ(\mathfrak{p}_1)\scr{G}\big(\mbf{T}(\varpi_{\mathfrak{p}_1})^{-1}\mbf{T}(\varpi_{\mathfrak{p}_2})\big)\Big)
	\Big(1-\alpha_{\msf{f}_\circ}\chi_\circ(\mathfrak{p}_2)\scr{G}\big(\mbf{T}(\varpi_{\mathfrak{p}_1})\mbf{T}(\varpi_{\mathfrak{p}_2})^{-1}\big)\Big).
	\end{equation}
	
\begin{proposition}\label{classINsel}
	
We have
\[
\bs{\xi}_{\scr{G},\msf{f}_\circ} \cdot \boldsymbol{\kappa}_p(\scr{G},\msf{f}_\circ)
\in
\mrm{H}^1\big(\mathbb{Q}_p,\mrm{Fil}^2\bs{\cal{V}}_{\scr{G},\msf{f}_\circ}(M)\big).
\]
\end{proposition}

\begin{proof}
We follow the argument of (\cite{DR2}, Proposition 2.2).  The module
\[
\bs{\cal{V}}_{\scr{G},\msf{f}_\circ}(M)_\alpha=\bs{\cal{V}}_{\scr{G}}(M)_\alpha\otimes\mrm{V}_{\msf{f}_\circ}(p)
\]
is realized as a quotient of $\mrm{H}_\et^{3}\big(Z^\dagger_\alpha(K)^c_{\bar{\bb{Q}}},O(2)\big)$ for $\iota:Z^\dagger_\alpha(K)\hookrightarrow Z^\dagger_\alpha(K)^c$ the smooth compactification appearing in the proof of Proposition $\ref{nullhomo}$. 
Let
\[
\mrm{H}^1_f\big(\mathbb{Q}_p,\bs{\cal{V}}_{\scr{G},\msf{f}_\circ}(M)_\alpha\big)
\subseteq
\mrm{H}^1_g\big(\mathbb{Q}_p,\bs{\cal{V}}_{\scr{G},\msf{f}_\circ}(M)_\alpha\big)
\subseteq
\mrm{H}^1\big(\mathbb{Q}_p,\bs{\cal{V}}_{\scr{G},\msf{f}_\circ}(M)_\alpha\big)
\]
denote the finite and geometric parts of the local Galois cohomology (\cite{Bloch-Kato}, Section 3).  The purity conjecture for the monodromy filtration holds for the middle cohomology of Hilbert modular varieties by work of Saito and Skinner (\cite{p-adicHodgeHilbert}, \cite{SkinnerHilbert}), and hence for the middle cohomology of $Z^\dagger_\alpha(K)^c$. Therefore, the image of $\boldsymbol{\kappa}_p(\scr{G},\msf{f}_\circ)$ in $\mrm{H}^1\big(\mathbb{Q}_p,\bs{\cal{V}}_{\scr{G},\msf{f}_\circ}(M)_\alpha\big)$ lies in
$\mrm{H}^1_g\big(\mathbb{Q}_p,\bs{\cal{V}}_{\scr{G},\msf{f}_\circ}(M)_\alpha\big)=\mrm{H}^1_f\big(\mathbb{Q}_p,\bs{\cal{V}}_{\scr{G},\msf{f}_\circ}(M)_\alpha\big)$ by (\cite{NekovarAbel-Jacobi}, Theorem 3.1).

\noindent Lemma $\ref{GaloisStructure}$ shows that $\bs{\cal{V}}_{\scr{G},\msf{f}_\circ}(M)_\alpha$ is an ordinary $\Gamma_{\mathbb{Q}_p}$-representation in the sense of (\cite{Weston}, Section 1.1). Therefore, one can deduce that
\[
\mrm{H}^1_f\big(\mathbb{Q}_p,\bs{\cal{V}}_{\scr{G},\msf{f}_\circ}(M)_\alpha\big)
=
\ker \Big(\mrm{H}^1\big(\mathbb{Q}_p,\bs{\cal{V}}_{\scr{G},\msf{f}_\circ}(M)_\alpha\big)
\rightarrow
\mrm{H}^1\big(\mrm{I}_p,\bs{\cal{V}}_{\scr{G},\msf{f}_\circ}(M)_\alpha/\mrm{Fil}^2\big)\Big)
\]
using Corollary $\ref{negative HT weights}$ and the argument in the proof of (\cite{Flach}, Lemma 2). In particular, the class $\bs{\kappa}_p(\scr{G},\msf{f}_\circ)$ has trivial image in the projective limit 
\[
\varprojlim_\alpha\ \mrm{H}^1\Big(\mrm{I}_p,\bs{\cal{V}}_{\scr{G},\msf{f}_\circ}(M)_\alpha/\mrm{Fil}^2\Big)
\cong
\mrm{H}^1\Big(\mrm{I}_p, \bs{\cal{V}}_{\scr{G},\msf{f}_\circ}(M)/\mrm{Fil}^2\Big),
\]
and consequently
the image of $\bs{\kappa}_p(\scr{G},\msf{f}_\circ)$ in $\mrm{H}^1\big(\bb{Q}_p, \bs{\cal{V}}_{\scr{G},\msf{f}_\circ}(M)/\mrm{Fil}^2\big)$ lies in
\[
\mrm{H}^1\Big(\mathbb{Q}_p^\mrm{ur}/\mathbb{Q}_p,\big(\bs{\cal{V}}_{\scr{G},\msf{f}_\circ}(M)/\mrm{Fil}^2\big)^{\mrm{I}_p}\Big)\cong\ker\Big(\mrm{H}^1\big(\bb{Q}_p, \bs{\cal{V}}_{\scr{G},\msf{f}_\circ}(M)/\mrm{Fil}^2\big)\rightarrow \mrm{H}^1\big(\mrm{I}_p, \bs{\cal{V}}_{\scr{G},\msf{f}_\circ}(M)/\mrm{Fil}^2\big)\Big).
\]
Taking into account Lemma $\ref{fil2 inj}$, we are left to show that 
\[
\bs{\xi}_{\scr{G},\msf{f}_\circ}\cdot \mrm{H}^1\Big(\mathbb{Q}_p^\mrm{ur}/\mathbb{Q}_p,\big(\bs{\cal{V}}_{\scr{G},\msf{f}_\circ}(M)/\mrm{Fil}^2\big)^{\mrm{I}_p}\Big)=0.
\]
By choosing the arithmetic Frobenius $\mrm{Fr}_p$ we can make the identification $\mrm{Gal}(\mathbb{Q}_p^\mrm{ur}/\mathbb{Q}_p)\cong\widehat{\mathbb{Z}}$ and compute that
\[
\mrm{H}^1\Big(\mathbb{Q}_p^\mrm{ur}/\mathbb{Q}_p,\big(\bs{\cal{V}}_{\scr{G},\msf{f}_\circ}(M)/\mrm{Fil}^2\big)^{\mrm{I}_p}\Big)
\cong
\big(\bs{\cal{V}}_{\scr{G},\msf{f}_\circ}(M)/\mrm{Fil}^2\big)^{\mrm{I}_p}/(\mrm{Fr}_p-1).
\]
Considering the short exact sequence
\[
0\rightarrow
\mrm{Gr}^1\bs{\cal{V}}_{\scr{G},\msf{f}_\circ}(M)
\rightarrow
\bs{\cal{V}}_{\scr{G},\msf{f}_\circ}(M)/\mrm{Fil}^2
\rightarrow
\mrm{Gr}^0\bs{\cal{V}}_{\scr{G},\msf{f}_\circ}(M)
\rightarrow
0
\]
and the vanishing (Lemma \ref{GaloisStructure} $\&$ Proposition \ref{somekindoffil})
\[
\mrm{H}^0\Big(\mrm{I}_p, \mrm{Gr}^0\bs{\cal{V}}_{\scr{G},\msf{f}_\circ}(M)\Big)=\underset{\leftarrow,\alpha}{\lim}\ \mrm{H}^0\Big(\mrm{I}_p, \mrm{Gr}^0\bs{\cal{V}}_{\scr{G},\msf{f}_\circ}(M)_\alpha\Big)=0,
\] 
we deduce that $\big(\bs{\cal{V}}_{\scr{G},\msf{f}_\circ}(M)/\mrm{Fil}^2\big)^{\mrm{I}_p}
=
\big(\mrm{Gr}^1\bs{\cal{V}}_{\scr{G},\msf{f}_\circ}(M)\big)^{\mrm{I}_p}$ sits in a short exact sequence
\[\resizebox{\displaywidth}{!}{\xymatrix{ \mbf{D}\Big(\big(\bs{\Psi}_{\scr{G},\frak{p}_1}\bs{\Psi}_{\scr{G},\frak{p}_2}^{-1}\big)^{-1}\cdot\delta_p(\msf{f}_\circ)\cdot\big(\psi^{-1}_{\circ,\frak{p}_1}\big)_{\lvert D_{p}}\Big)\ar@{^{(}->}[r]& \Big(\bs{\cal{V}}_{\scr{G},\msf{f}_\circ}(M)/\mrm{Fil}^2\Big)^{\mrm{I}_p}\ar@{->>}[r]& \mbf{D}'\Big(\bs{\Psi}_{\scr{G},\frak{p}_1}\bs{\Psi}_{\scr{G},\frak{p}_2}^{-1}\cdot\delta_p(\msf{f}_\circ)\cdot\big(\psi^{-1}_{\circ,\frak{p}_2}\big)_{\lvert D_p}\Big)
}}\]
where $\mbf{D},\mbf{D}'$ are $\mbf{I}_\scr{G}$-modules with trivial Galois action. Therefore, if we set
\[
	\bs{\xi}_{\scr{G},\msf{f}_\circ} = 
	\Big(1-\alpha_{\msf{f}_\circ}\chi_\circ(\mathfrak{p}_1)\scr{G}\big(\mbf{T}(\varpi_{\mathfrak{p}_1})^{-1}\mbf{T}(\varpi_{\mathfrak{p}_2})\big)\Big)
	\Big(1-\alpha_{\msf{f}_\circ}\chi_\circ(\mathfrak{p}_2)\scr{G}\big(\mbf{T}(\varpi_{\mathfrak{p}_1})\mbf{T}(\varpi_{\mathfrak{p}_2})^{-1}\big)\Big)
\]
 the claim follows.
\end{proof}

\noindent In light of Proposition $\ref{classINsel}$, from now on we replace the ring $\mbf{I}_\scr{G}$ and the various modules over it with their respective localizations at the multiplicative set generated by $\bs{\xi}_{\scr{G},\msf{f}_\circ}$. Observe that the arithmetic specializations of $\bs{\xi}_{\scr{G},\msf{f}_\circ}$ never vanish: for any $\mrm{P}\in\cal{A}(\mbf{I}_\scr{G})$, the specialization 
$\mrm{P}\circ\scr{G}\big(\mbf{T}(\varpi_{\mathfrak{p}_1})^{-1}\mbf{T}(\varpi_{\mathfrak{p}_2})\big)$ is an algebraic integer with complex absolute value 1, whereas $\alpha_{\msf{f}_\circ}$ has complex absolute value $p^{1/2}$. 

\begin{corollary}\label{placing the class}
With the above convention in place,
\[
\boldsymbol{\kappa}_p(\scr{G},\msf{f}_\circ)
\in
\mrm{H}^1\big(\mathbb{Q}_p,\mrm{Fil}^2\bs{\cal{V}}_{\scr{G},\msf{f}_\circ}(M)\big).
\]
\end{corollary}

\begin{definition}\label{TwistedGradedPiece}
	Let $\bs{\cal{V}}^{\msf{f}_\circ}_{\scr{G}}(M)
	:=
	\mrm{Fil}^2 \bs{\cal{V}}_\mathscr{G}(M) (-1) \otimes\mrm{Gr}^0\mrm{V}_{\mrm{f}_\circ}(p)$
and denote by
\[
\boldsymbol{\kappa}^{\msf{f}_\circ}_p(\scr{G}) \in \mrm{H}^1\big(\mathbb{Q}_p,\bs{\cal{V}}^{\msf{f}_\circ}_{\scr{G}}(M)\big)
\]
the image of $\boldsymbol{\kappa}_p(\scr{G},\msf{f}_\circ)$ under the natural surjection 
$\mrm{Fil}^2\bs{\cal{V}}_{\scr{G},\msf{f}_\circ}(M) \twoheadrightarrow \bs{\cal{V}}^{\msf{f}_\circ}_{\scr{G}}(M)$.
\end{definition}
\begin{remark}
The local Galois group $\Gamma_{\bb{Q}_p}$ acts on
$\bs{\cal{V}}^{\msf{f}_\circ}_{\scr{G}}(M)$
through the character
	\[
	\bs{\Psi}_{\scr{G},p}^{-1}\cdot\delta_p(\msf{f}_\circ)\cdot\big(\bs{\eta}_\bb{Q}\cdot\eta_\bb{Q}\cdot \psi_\circ^{-1}\big)_{\lvert D_p}.
	\]
\end{remark}

\section{Big pairing}\label{generalizingOhta}
\subsection{Algebra interlude}
Recall $\Gamma=1+p\bb{Z}_p$, $\Gamma_\alpha=\Gamma/\Gamma^{p^\alpha}$, $\Lambda_\alpha=O[\Gamma_\alpha]$ and $E_\wp=\mrm{Frac}(O)$.
\begin{definition}\label{bigpi}
 Let $\bs{\Pi}_\alpha= (\Lambda_\alpha\otimes_OE_\wp)$ and define
 \begin{equation}
 \bs{\Pi} :=\varprojlim_\alpha\ \bs{\Pi}_\alpha.
 \end{equation}
\end{definition}

\noindent Let $(M_\alpha)_\alpha$ be a projective system where each $M_\alpha$ is a $\bs{\Pi}_\alpha$-module, then the projective limit $\mbf{M} = \varprojlim_\alpha M_\alpha$ inherits a $\bs{\Pi}$-module structure. Any finite order character $\chi:\Gamma \rightarrow \bb{C}^\times_p$ factors through $\Gamma_\alpha$ for some $\alpha\ge1$, and determines a homomorphism 
\begin{equation}
\chi:\mbf{M}\longrightarrow M_\alpha\otimes_{\chi}\bb{C}_p, \qquad x\mapsto \chi(x).
\end{equation}

\begin{lemma}\label{lem: vanishing criterion}
Let $\mbf{M} = \varprojlim_\alpha M_\alpha$ be a projective limit where each $M_\alpha$ is a flat $\bs{\Pi}_\alpha$-module. Then $x\in\mbf{M}$ equals zero if and only if
$\chi(x) = 0$ for every finite order character 
$\chi: \Gamma \to \bb{C}_p^\times$.
\end{lemma}
\begin{proof}
	Let $x=(x_\alpha)_\alpha\in\mbf{M}$ be such that $\chi(x) = 0$ for every finite order character 
$\chi: \Gamma \to \bb{C}_p^\times$. By definition, this means that for every $\alpha\ge1$
\[
\chi(x_\alpha)=0\qquad\forall\ \chi:\Gamma_\alpha\to\bb{C}_p^\times.
\] 
Note that  $\bs{\Pi}_\alpha=E_\wp[\Gamma_\alpha]\subseteq\bb{C}_p[\Gamma_\alpha]\cong \oplus_\chi\bb{C}_p$ where the sum is taken
over all the characters of $\Gamma_\alpha$. 
Therefore, flatness of $M_\alpha$ implies the injectivity of 
\[
\oplus_\chi : M_\alpha \hookrightarrow \oplus_\chi (M_\alpha\otimes_{\chi}\bb{C}_p).
\]
We deduce $x_\alpha=0$ for every $\alpha\ge1$.
\end{proof}

\noindent We are interested in generalizing the vanishing criterion of Lemma $\ref{lem: vanishing criterion}$ to $\bs{\Pi}$-modules of the form
$\mbf{I}\otimes_{\bs{\Lambda}}\bs{\Pi}$ where $\mbf{I}$ is an algebra which is finitely generated and flat over $\bs{\Lambda}$.  We start with a couple of technical lemmas.

\begin{lemma}\label{lemma surjective}
If $A\twoheadrightarrow B$ is a surjective morphism of $\bs{\Lambda}$-modules then
\[\xymatrix{
\varprojlim_\alpha\ \big(A\otimes_{\bs{\Lambda}}\bs{\Pi}_\alpha\big)\ar@{->>}[r]& \varprojlim_\alpha\ \big(B\otimes_{\bs{\Lambda}}\bs{\Pi}_\alpha\big).
}\]
\end{lemma}
\begin{proof}
    Let $Q=\ker(A\to B)$, and $\kappa_\alpha=\ker(A\otimes\bs{\Pi}_\alpha\to B\otimes\bs{\Pi}_\alpha)$ for every $\alpha\ge1$, then there is a commutative diagram
    \[\xymatrix{
    Q\otimes\bs{\Pi}_\alpha\ar@{->>}[dr]\ar[drr]\ar@{->>}[dd] & & &\\
    & \kappa_\alpha\ar@{^{(}->}[r]\ar@{.>>}[dd] & A\otimes\bs{\Pi}_\alpha\ar[r]\ar@{->>}[dd] & B\otimes\bs{\Pi}_\alpha\ar[r]\ar@{->>}[dd]& 0\\
    Q\otimes\bs{\Pi}_{\alpha-1}\ar@{->>}[dr]\ar[drr] & & &\\
    & \kappa_{\alpha-1}\ar@{^{(}->}[r] & A\otimes\bs{\Pi}_{\alpha-1}\ar[r] & B\otimes\bs{\Pi}_{\alpha-1}\ar[r]& 0\\
    }\]
    Since tensoring  by $\otimes_{\bs{\Lambda}}\bs{\Pi}_\alpha$ is right exact, $Q\otimes\bs{\Pi}_\alpha\twoheadrightarrow\kappa_\alpha$ is surjective for every $\alpha\ge1$. Therefore the transition map $\kappa_\alpha\to\kappa_{\alpha-1}$ is surjective for every $\alpha>1$ and  $\varprojlim_\alpha^1\ \kappa_\alpha=0$ as required.
\end{proof}

\begin{lemma}\label{inverse limit}
Suppose $\mbf{I}$ is a finitely generated, flat $\bs{\Lambda}$-module. Then
\[\xymatrix{
\mbf{I}\otimes_{\bs{\Lambda}}\bs{\Pi}\ar[r]^-\sim&\varprojlim_\alpha\ \big(\mbf{I}\otimes_{\bs{\Lambda}}\bs{\Pi}_\alpha)
}\]
is a projective limit of flat $\bs{\Pi}_\alpha$-modules.
\end{lemma}
\begin{proof}
Since $\bs{\Lambda}$ is Noetherian, $\mbf{I}$ is finitely presented, and thus fits in an exact sequence of the form
\begin{equation}\label{presentation}
(\bs{\Lambda})^{\oplus m}\rightarrow (\bs{\Lambda})^{\oplus n}\rightarrow \mbf{I}\rightarrow0.
\end{equation}
Tensoring is right exact, hence we deduce presentations for $\mbf{I}\otimes_{\bs{\Lambda}}\bs{\Pi}$ and $\mbf{I}\otimes_{\bs{\Lambda}}\bs{\Pi}_\alpha$:
\[
(\bs{\Pi})^{\oplus m}\rightarrow (\bs{\Pi})^{\oplus n}\rightarrow \mbf{I}\otimes_{\bs{\Lambda}}\bs{\Pi}\rightarrow0, \qquad
(\bs{\Pi}_\alpha)^{\oplus m}\rightarrow (\bs{\Pi}_\alpha)^{\oplus n}\rightarrow \mbf{I}\otimes_{\bs{\Lambda}}\bs{\Pi}_\alpha\rightarrow0.
\]
Now it suffices to show that 
\[
(\bs{\Pi})^{\oplus m}\rightarrow (\bs{\Pi})^{\oplus n}\rightarrow \underset{\leftarrow,\alpha}{\lim}\ (\mbf{I}\otimes_{\bs{\Lambda}}\bs{\Pi}_\alpha)\rightarrow0
\]
is exact. The surjectivity of $(\bs{\Pi})^{\oplus n}\twoheadrightarrow \varprojlim_\alpha\ (\mbf{I}\otimes_{\bs{\Lambda}}\bs{\Pi}_\alpha)$ follows from Lemma $\ref{lemma surjective}$, so we are left to prove exactness at the middle of the sequence. Let $Q=\ker(\bs{\Lambda}^{\oplus n}\to\mbf{I})$, then $Q$ comes equipped with a surjection $\bs{\Lambda}^{\oplus m}\twoheadrightarrow Q$ because ($\ref{presentation}$) is a presentation. Moreover, 
\[
Q\otimes_{\bs{\Lambda}}\bs{\Pi}_\alpha=\ker\big((\bs{\Pi}_\alpha)^{\oplus n}\rightarrow \mbf{I}\otimes_{\bs{\Lambda}}\bs{\Pi}_\alpha\big)
\]
because $\mrm{Tor}_1^{\bs{\Lambda}}(\mbf{I},\bs{\Pi}_\alpha)=0$ as $\mbf{I}$ is a flat $\bs{\Lambda}$-module. Therefore, 
\[
\underset{\leftarrow,\alpha}{\lim}\ (Q\otimes_{\bs{\Lambda}}\bs{\Pi}_\alpha)=\ker\big( (\bs{\Pi})^{\oplus n}\twoheadrightarrow \underset{\leftarrow,\alpha}{\lim}\ (\mbf{I}\otimes_{\bs{\Lambda}}\bs{\Pi}_\alpha)\big)
\]
and we are left to show that the induced map $\bs{\Pi}^{\oplus m}\to \varprojlim_\alpha\ (Q\otimes_{\bs{\Lambda}}\bs{\Pi}_\alpha)$ is surjective, which it is because of Lemma $\ref{lemma surjective}$.
\end{proof}

\noindent Lemma $\ref{inverse limit}$ shows that the vanishing criterion of Lemma $\ref{lem: vanishing criterion}$ applies to any $\bs{\Pi}$-module of the form $\mbf{I}\otimes_{\bs{\Lambda}}\bs{\Pi}$ for $\mbf{I}$ a finitely generated and flat $\bs{\Lambda}$-module.
When comparing the automorphic and motivic $p$-adic $L$-functions in Section $\ref{Sect: Comparison}$ we will be interested in the case  $\mbf{I}=\mbf{I}_\scr{G}$ and we will have information about the specializations at arithmetic points of weight $2$. 
Observe that if $\mrm{P}\in\cal{A}_{\bs{\chi}}(\mbf{I}_\scr{G})$ has weight $(2t_L,t_L)$ and level $p^\alpha$ then it induces a map
\[
\mrm{P}:\mbf{I}_\scr{G}\otimes_{\bs{\Lambda}}\bs{\Pi}\longrightarrow\bb{C}_p,\qquad x\mapsto \mrm{P}(x).
\]
The following vanishing criterion will be of crucial importance.

\begin{theorem}\label{cor: vanishing criterion}
    An element $x \in \mbf{I}_\scr{G} \otimes_{\bs{\Lambda}} \bs{\Pi}$ equals zero if and only if
$\mrm{P}(x) = 0$ for all arithmetic points $\mrm{P}\in\cal{A}_{\bs{\chi}}(\mbf{I}_\scr{G})$ of weight $(2t_L,t_L)$.
\end{theorem}
\begin{proof}
  By Lemma $\ref{inverse limit}$, $\mbf{I}_\scr{G} \otimes_{\bs{\Lambda}} \bs{\Pi} 
\overset{\sim}{\to}
\varprojlim_\alpha (\mbf{I}_\scr{G} \otimes_{\bs{\Lambda}} \bs{\Pi}_\alpha)$, thus an element  $x\in \mbf{I}_\scr{G} \otimes_{\bs{\Lambda}} \bs{\Pi}$ is equal to zero if and only if its image $x_\alpha\in \mbf{I}_\scr{G} \otimes_{\bs{\Lambda}} \bs{\Pi}_\alpha$ is equal to zero for every $\alpha\ge1$. As in the proof of Lemma $\ref{lem: vanishing criterion}$ there is an injection
\[
\mbf{I}_\scr{G} \otimes_{\bs{\Lambda}} \bs{\Pi}_\alpha\hookrightarrow\bigoplus_{\chi}\ \big(\mbf{I}_\scr{G}\otimes_{\bs{\Lambda},\chi}\bb{C}_p\big)
\]
because $\mbf{I}_\scr{G}$ is $\bs{\Lambda}$-flat. Let $\frak{p}_\chi$ be the kernel of $\chi:\bs{\Lambda}\to\bb{C}_p$ and denote by $\bs{\Lambda}_{\mathfrak{p}_\chi}$ the localization, then $\mbf{I}_\scr{G}\otimes_{\bs{\Lambda}}\bs{\Lambda}_{\mathfrak{p}_\chi}$ is finite \'etale over $\bs{\Lambda}_{\mathfrak{p}_\chi}$ because Hida families are finite \'etale over arithmetic points of weight $\ge 2t_L$ (\cite{nearlyHida}). Therefore
\[
(\mbf{I}_\scr{G}\otimes_{\bs{\Lambda}}\bs{\Lambda}_{\mathfrak{p}_\chi})\otimes_{\bs{\Lambda}_{\mathfrak{p}_\chi}}\bb{C}_p\cong \mbf{I}_\scr{G}\otimes_{\bs{\Lambda},\chi}\bb{C}_p
\]
is a finite product of copies of $\bb{C}_p$, indexed by the arithmetic points $\mrm{P}\in\cal{A}_{\bs{\chi}}(\mbf{I}_\scr{G})$ above $\mathfrak{p}_\chi$. In summary
\[
\mbf{I}_\scr{G}\otimes_{\bs{\Lambda},\chi}\bb{C}_p\cong\bigoplus_{\mrm{P}}\bb{C}_p
\]
the sum over the arithmetic points $\mrm{P}\in\cal{A}_{\bs{\chi}}(\mbf{I}_\scr{G})$ of weight $(2t_L,t_L)$ and the claim follows.
\end{proof}

\subsection{Generalizing Ohta}
\begin{definition}\label{def 2 I-adic forms}
	For a $\Lambda$-algebra $\mbf{I}$ we define the space of ordinary $\mbf{I}$-adic cuspforms of tame level $K$ and character $\bs{\chi}$ to be
	\[
	\overline{\mathbf{S}}_L^{\text{ord}}(K;\bs{\chi};\mbf{I}):=\mrm{Hom}_{\bs{\Lambda}\mbox{-}\mrm{mod}}\big(\mbf{h}^\mrm{n.o.}_L(K;O)\otimes_{\phi_{\bs{\chi}}}\bs{\Lambda}, \mbf{I}\big).
	\]
\end{definition}
\noindent Let $\scr{G}\in \overline{\mathbf{S}}_L^{\text{ord}}(K;\bs{\chi};\bs{\Lambda})$ then exact control for the nearly ordinary Hecke algebra implies that extending scalars for $\scr{G}$ to $\Lambda_\alpha$  produces a $\Lambda_\alpha$-module homomorphism
\[
\underline{\scr{G}}_\alpha:\mrm{h}^\mrm{ord}_{2t_L,t_L}(K_{\diamond,t}(p^\alpha);O)\longrightarrow \Lambda_\alpha.
\]
Since $\Lambda_\alpha=\bigoplus_{\sigma\in\Gamma_\alpha}O\cdot[\sigma]$, we may write $\underline{\scr{G}}_\alpha=\bigoplus_{\sigma \in \Gamma_\alpha}\scr{G}_{\alpha,\sigma^{-1}}\cdot[\sigma]$, and the $\Lambda_\alpha$-linearity implies $\scr{G}_{\alpha,\sigma^{-1}}(-)=\scr{G}_{\alpha,1}([\sigma]-)$. In order to lighten the notation we write $\scr{G}_\alpha$ for the Hilbert cuspform 
\[
\scr{G}_\alpha:=\scr{G}_{\alpha,1}\in S^\mrm{ord}_{2t_L,t_L}(K_{\diamond,t}(p^\alpha);O).
\] 
The compatibility
\[\xymatrix{
\mrm{h}^\mrm{ord}_{2t_L,t_L}(K_{\diamond,t}(p^{\alpha+1});O)\ar[d]\ar[rr]^-{\underline{\scr{G}}_{\alpha+1}}& &\Lambda_{\alpha+1}\ar[d]\\
\mrm{h}^\mrm{ord}_{2t_L,t_L}(K_{\diamond,t}(p^\alpha);O)\ar[rr]^-{\underline{\scr{G}}_\alpha}& & \Lambda_\alpha
}\]
translates into 
\[
\sum_{\sigma\in\ker(\Gamma_{\alpha+1}\to\Gamma_\alpha)}\scr{G}_{\alpha+1}([\sigma]-)=\scr{G}_\alpha(-),
\]
or equivalently,
\begin{equation}\label{tracecompatible}
(\mu)_*\scr{G}_{\alpha+1} = (\pi_1)^*\scr{G}_{\alpha}.
\end{equation}
Recall that in Section \ref{ontheconjectures} of the introduction we set
\[
\cal{V}_\alpha:=e_\mrm{n.o.}\mrm{H}^2_{\et,c}\big(S(K_{\diamond,t}(p^\alpha))_{\bar{\bb{Q}}},O(2)\big).
\]
Its $0$-th graded piece $\mrm{Gr}^0\cal{V}_\alpha$ of the \'etale cohomology, with respect to the ordinary filtration, 
is an unramified $\Gamma_{\bb{Q}_p}$-representation, therefore 
\[
\bb{D}\big(\mrm{Gr}^0\cal{V}_\alpha\big):=\big(\mrm{Gr}^0\cal{V}_\alpha\otimes\widehat{\bb{Z}}_p^\mrm{ur}\big)^{\Gamma_{\bb{Q}_p}}
\] is a lattice in the de-Rham cohomology $\mrm{D}_\mrm{dR}\big(\mrm{Gr}^0\cal{V}_\alpha\otimes_OE_\wp\big)$. The following conjectures compares the integral structures for the de Rham cohomology of Hilbert modular surfaces coming from integral \'etale cohomology, and ordinary Hilbert modular forms of parallel weight two.

\begin{conjecture}\label{wishingOhta}
	For every large enough prime $p$ and every $\alpha\ge1$ the image of the natural map
	\[
	S^\mrm{ord}_{2t_L,t_L}\big(K_{\diamond,t}(p^\alpha);O\big)
	\longrightarrow
	\mrm{D}_\mrm{dR}\big(\mrm{Gr}^0\cal{V}_\alpha\otimes_OE_\wp\big)
	\]
	is contained in the lattice $\bb{D}\big(\mrm{Gr}^0\cal{V}_\alpha\big)$.
\end{conjecture}

\noindent We consider the projective limits 
	\begin{equation}
	\bb{D}\big(\mrm{Gr}^0\cal{V}_\infty\big):=
	\underset{\leftarrow, \varpi_{2}}{\lim}\ \bb{D}\big(\mrm{Gr}^0\cal{V}_\alpha\big),\qquad \mbf{D}_\mrm{dR}\big(\mrm{Gr}^0\mrm{V}_\infty\big)
	:=
	\underset{\leftarrow, \varpi_{2}}{\lim}\ \mrm{D}_\mrm{dR}\big(\mrm{Gr}^0\cal{V}_\alpha\otimes_OE_\wp\big).
	\end{equation}

\begin{lemma}\label{BigPeriodMap}
	There is a Hecke-equivariant morphism 
	\[
	\raisebox{\depth}{\scalebox{1}[-1]{$ \Omega $}}_\infty:\overline{\mathbf{S}}_L^{\text{ord}}(K;\bs{\chi};\mbf{I}_\scr{G})
	\longrightarrow 
	\mbf{D}_\mrm{dR}\big(\mrm{Gr}^0\mrm{V}_\infty\big)\otimes_{\bs{\Lambda}}\mbf{I}_\scr{G}.
	\]
	Further, assuming Conjecture \ref{wishingOhta}, the homomorphism $\raisebox{\depth}{\scalebox{1}[-1]{$ \Omega $}}_\infty$ takes values in $\bb{D}\big(\mrm{Gr}^0\cal{V}_\infty\big)\otimes_{\bs{\Lambda}}\mbf{I}_\scr{G}$.

\end{lemma}
\begin{proof}
	Pushing forward equation ($\ref{tracecompatible}$) along $\pi_2$ gives $(\varpi_{2})_*\scr{G}_{\alpha+1} = U_p\scr{G}_{\alpha}$. Hence, the collection $(U_p^{-\alpha}\scr{G}_{\alpha})_\alpha$ is compatible under projection along $\varpi_{2}$. The homomorphism
	\[
	\overline{\mathbf{S}}_L^{\text{ord}}(K;\bs{\chi};\bs{\Lambda}) \longrightarrow \underset{\leftarrow,\varpi_2}{\lim}\ S^\mrm{ord}_{2t_L,t_L}(K_{\diamond,t}(p^\alpha);O),
	\qquad 
	\scr{G}\mapsto\big(U_p^{-\alpha}\scr{G}_{\alpha}\big)_\alpha
	\] 
combined with the natural map $S^\mrm{ord}_{2t_L,t_L}(K_{\diamond,t}(p^\alpha);E_\wp)
	\longrightarrow \mrm{D}_\mrm{dR}\big(\mrm{Gr}^0\cal{V}_\alpha\otimes_OE_\wp\big)$
	gives
	\begin{equation}\label{periodlambda}
	\raisebox{\depth}{\scalebox{1}[-1]{$ \Omega $}}_\infty:\overline{\mathbf{S}}_L^{\text{ord}}(K;\bs{\chi};\bs{\Lambda})
	\longrightarrow 
	\mbf{D}_\mrm{dR}\big(\mrm{Gr}^0\mrm{V}_\infty\big).
	\end{equation}
Since $\bs{\Lambda}$ is Noetherian and $\mbf{h}^\mrm{n.o.}_L(K;O)\otimes_{\phi_{\bs{\chi}}}\bs{\Lambda}$ is finite over $\bs{\Lambda}$, the Hecke algebra is also finitely presented as a $\bs{\Lambda}$-module. As $\mbf{I}_\scr{G}$ is flat over $\bs{\Lambda}$ it follows that
\[
\overline{\mathbf{S}}_L^{\text{ord}}(K;\bs{\chi};\mbf{I}_\scr{G})
\simeq \overline{\mathbf{S}}_L^{\text{ord}}(K;\bs{\chi};\bs{\Lambda}) \otimes_{\bs{\Lambda}} \mbf{I}_\scr{G}.
\]  
The claimed Hecke equivariant morphism is obtained from ($\ref{periodlambda}$) by extension of scalars.
\end{proof}

\begin{definition}\label{BigGPeriodMap}
	For $M$ a $\widetilde{\mbf{h}}^\mrm{n.o.}_L(K;O)_{\bs{\chi}}$-module we denote by 
	\[
	M[\scr{G}_{\mbox{\tiny $\heartsuit$}}]:=\big\{m\in M\lvert\ Tm=\scr{G}_{\mbox{\tiny $\heartsuit$}}(T)m\ \  \forall T\in \widetilde{\mbf{h}}^\mrm{n.o.}_L(K;O)\big\}
	\]
	its $\scr{G}_{\mbox{\tiny $\heartsuit$}}$-isotypic submodule, where $\scr{G}_{\mbox{\tiny $\heartsuit$}}$ was defined in \eqref{heartform}.  With a small abuse of notation, we write $\mbf{D}_\mrm{dR}\big(\mrm{Gr}^0\mrm{V}_\infty\big)[\scr{G}_{\mbox{\tiny $\heartsuit$}}]$ for $\big(\mbf{D}_\mrm{dR}\big(\mrm{Gr}^0\mrm{V}_\infty\big)\otimes_{\bs{\Lambda}}\mbf{I}_\scr{G}\big)[\scr{G}_{\mbox{\tiny $\heartsuit$}}]$ and consider the induced morphism
  \[
    \raisebox{\depth}{\scalebox{1}[-1]{$ \Omega $}}_\scr{G}:
    \overline{\mathbf{S}}_L^{\text{ord}}(K;\bs{\chi};\mbf{I}_\scr{G})[\scr{G}_{\mbox{\tiny $\heartsuit$}}]
	\longrightarrow 
	\mbf{D}_\mrm{dR}\big(\mrm{Gr}^0\mrm{V}_\infty\big)[\scr{G}_{\mbox{\tiny $\heartsuit$}}].
    \]
	\end{definition}

\subsubsection{Big pairing.}
Let $m_\alpha\in\bb{A}_{L,f}^\times$ be the image of the integer $Mp^\alpha$ and consider the matrix
\begin{equation}
\tau_\alpha
=
\begin{pmatrix}
0& -1\\
m_\alpha& 0
\end{pmatrix}\in G_L(\bb{A}_{f}).
\end{equation}
If we denote by 
$(-)^*:G_L(\bb{A}_{f})\to G_L(\bb{A}_{f})$ the involution $g^*=\det(g)^{-1}g$, then for every $\alpha\ge1$ there is a morphism 
$\lambda_\alpha:S(K_{\diamond,t}(p^\alpha))\to S(K_{\diamond,t}(p^\alpha))$ defined as the composition
\begin{equation}
\xymatrix{
S(K_{\diamond,t}(p^\alpha))\ar@{.>}[rr]^{\lambda_\alpha}\ar[rd]^{\mathfrak{T}_{\tau_\alpha}}& & S(K_{\diamond,t}(p^\alpha))\\
& S(^{\tau_\alpha} K_{\diamond,t}(p^\alpha))\ar[ru]^{(-)^*}& &.
}\end{equation}

\begin{lemma}\label{AtkinLehnerGalois}
 If $\sigma\in \Gamma_{\bb{Q}}$ corresponds to $a\in\bb{A}_{f}^\times$ under the global Artin map, then
  \[
\langle \Delta(a^{-1}),1\rangle\circ\lambda_\alpha \circ \sigma
=
\sigma \circ \lambda_\alpha
\]
where $\Delta:\bb{A}_{f}\hookrightarrow \bb{A}_{L,f}$ is the natural inclusion.  In particular, $\lambda_\alpha$ is defined over $\bb{Q}(\zeta_{Mp^\alpha})$. 
\end{lemma}
\begin{proof}
    This is a standard computation using the reciprocity laws of Shimura varieties at CM points. We will sketch a proof below following the notations of (\cite{Milne-SVI}, Chapters 12 and 13).
    
\noindent    Let $x\in \frak{H}^2$ be a CM point defined over $E$. Suppose $\sigma \in \mrm{Aut}(\bb{C}/E)$ and choose $s\in\bb{A}_E^\times$ so that $s$ corresponds to $\sigma_{\lvert\bar{\bb{Q}}}$ under the reciprocal of the global Artin map. Then we have the following commutative diagram
    \[\xymatrix{
    [x,h]\in S(^{\tau_\alpha} K_{\diamond,t}(p^\alpha))
    \ar[d]^{\sigma}\ar[rrr]^{(-)^*} &&&
    [x,\det(h)^{-1}h]\in S(K_{\diamond,t}(p^\alpha))\ar[d]^{\sigma}\\
    [x,r_x(s)h]\in S(^{\tau_\alpha} K_{\diamond,t}(p^\alpha))
    \ar[rrr]^{(-)^*\circ\langle \det(r_x(s)),1\rangle} &&& 
    [x,\det(h)^{-1}r_x(s)h]\in S(K_{\diamond,t}(p^\alpha))
}\]
where $r_x(s) = N_{E/\bb{Q}}(\mu_x(s_f))$ as defined as in (\cite{Milne-SVI}, Chapter 12, equation (52)), and $\mu_x:\bb{G}_m\rightarrow G$ is the $E$-rational cocharacter of $G$ characterized by $\mu_x(z)=h_{x,\bb{C}}(z,1)$ for $z\in\bb{C}$. In the case of Hilbert modular varieties, $\det(\mu_x)$ is the natural embedding $\bb{G}_m\hookrightarrow \mrm{Res}_{L/\bb{Q}}\bb{G}_m$ and thus $\det(r_x(s))=\mrm{N}_{E/\bb{Q}}(s_f)\in \bb{A}_{f}^\times$. 
By functoriality of class field theory, $N_{E/\bb{Q}}(s_f)\in \bb{A}_{f}^\times$ corresponds  to $\sigma_{\lvert\bar{\bb{Q}}}$, seen inside $\Gamma_{\bb{Q}}$.
Therefore for any $\sigma\in \Gamma_\bb{Q}$, corresponding to $a\in\bb{A}_{f}^\times$ under the global Artin map and fixing the reflex field of a CM point, we have
\[
\langle \Delta(a^{-1}),1\rangle\circ\lambda_\alpha \circ \sigma 
=
\sigma \circ \lambda_\alpha
\]
because points of the form $[x,h]$ for a fixed $x$ are Zariski dense  (\cite{Milne-SVI}, Lemma 13.5), and the map $\mathfrak{T}_{\tau_\alpha}$ is defined over $\bb{Q}$.
Clearly such $\sigma$'s generate $\mrm{Aut}(\bb{C}/\bb{Q})$, and thus the above relation holds for all $\sigma \in \mrm{Aut}(\bb{C}/\bb{Q})$. 
\end{proof}

\noindent A direct calculation shows that
\begin{equation}\label{formula-lambdapi}
\pi_{1,p}\circ\lambda_{\alpha+1}=\lambda_\alpha\circ\pi_{2,p},\qquad \pi_{2,p}\circ\lambda_{\alpha+1}=\lambda_\alpha\circ\pi_{1,p},
\end{equation}
so that
\begin{equation}\label{UandU}
U_p=(\lambda_{\alpha})_*\circ U^*_p\circ(\lambda_\alpha)^*.
\end{equation}
Consider the group 
\[
\bb{G}^{\alpha}_{\diamond,t}(K)
:=
K_0(p^\alpha)\cal{O}_L^\times/K_{\diamond,t}(p^\alpha)\cal{O}_L^\times
\]
of diamond operators acting on 
$S(K_{\diamond,t}(p^\alpha))$.
There is an inclusion 
\[
\Gamma_\alpha 
\hookrightarrow
\bb{G}^{\alpha}_{\diamond,t}(K),
\qquad
z \mapsto 
\begin{pmatrix}
\Delta(z)&0\\0& \Delta(z)\\
\end{pmatrix},
\]
where $\Delta:1+p\bb{Z}_p\hookrightarrow 1+p\cal{O}_p$ denotes the diagonal embedding. Under this inclusion, an element $z\in \Gamma_\alpha$ acts on cohomology as the diamond operator $\langle \Delta(z),1 \rangle$. To shorten the notation, we set
\[
\mrm{H}^2_{\et,?}\big(K_{\diamond,t}(p^\alpha);O\big):=\mrm{H}^2_{\et,?}\big(S(K_{\diamond,t}(p^\alpha))_{\bar{\bb{Q}}},O\big),\qquad ?\in\{\emptyset,c\},
\]
and define a twisted group-ring-valued pairing 
\[
\{\ ,\}_\alpha:
\mrm{H}^2_\et\big(K_{\diamond,t}(p^\alpha);O(2)\big)
\times
\mrm{H}^2_{\et,c}\big(K_{\diamond,t}(p^\alpha); O\big)
\longrightarrow \Lambda_\alpha
\]
by 
\begin{equation}\label{FiniteLevelPairing}
\{x_\alpha,y_\alpha\}_\alpha
=
\sum_{z\in\Gamma_\alpha}\Big\langle\langle \Delta(z),1\rangle^* x_\alpha, (\lambda_\alpha)^*\circ U_p^\alpha y_\alpha\Big\rangle_\alpha [z^{-1}],
\end{equation}
where the Poincar\'e pairing $\langle\ , \rangle_\alpha$ is defined modulo torsion. 

\begin{proposition}\label{basicsPairing}
The pairing $\{\ , \}_\alpha$ is $\Lambda_\alpha$-bilinear and all the Hecke operators are self-adjoint with respect to it. In particular, $\{\ , \}_\alpha$ induces a pairing on nearly ordinary parts. 

\end{proposition}
\begin{proof}
The Hecke operator $U_p$ is self-adjoint with respect to $\{\ , \}_\alpha$ because $U^*_p$ and $U_p$ are adjoint with respect the Poincar\'e pairing and equation ($\ref{UandU}$). A similar argument works for all the other Hecke operators $T(\varpi_\mathfrak{q})$ (or $U(\varpi_\mathfrak{q})$ if $\mathfrak{q}\mid Mp$). 
To check $\Lambda_\alpha$-bilinearity, let $b$ be any element in $\Gamma_\alpha$, then we have 
$\langle \Delta(b),1\rangle\circ\lambda_\alpha=\lambda_\alpha\circ\langle \Delta(b)^{-1},1\rangle$,
which implies 
\[
\lambda_\alpha^*\circ\langle \Delta(b),1\rangle^*=\langle \Delta(b),1\rangle_*\circ\lambda_\alpha^*.
\] 
Therefore
\[\begin{split}
\{x_\alpha, \langle \Delta(b),1\rangle^*y_\alpha\}_\alpha 
&= 
\sum_{z\in \Gamma_\alpha}\Big\langle\langle \Delta(z),1\rangle^* x_\alpha, (\lambda_\alpha)^*\circ U_p^\alpha\circ\langle \Delta(b),1\rangle^* y_\alpha\Big\rangle_\alpha [z^{-1}]\\
&= 
\sum_{z\in \Gamma_{\alpha}}\Big\langle\langle \Delta(zb),1\rangle^* x_\alpha, (\lambda_\alpha)^*\circ U_p^\alpha y_\alpha\Big\rangle_\alpha [z^{-1}]\\
&=[b]\{x_\alpha,y_\alpha\}_\alpha.
\end{split}
\]
and similarly
\[
\{\langle \Delta(b),1\rangle^* x_\alpha, y_\alpha\}_\alpha 
=
[b]\{x_\alpha,y_\alpha\}_\alpha.
\]
\end{proof}

\begin{lemma}\label{lemma: invariant}
The finite Galois covering
$\mu:S(K_{\diamond,t}(p^{\alpha+1}))
\rightarrow S(K_{\diamond,t}(p^\alpha)\cap K_0(p^{\alpha+1}))$
induces an isomorphism
\[
\mu^*:\mrm{H}^2_\et\big(K_{\diamond,t}(p^\alpha)\cap K_0(p^{\alpha+1});E_\wp(2)\big)\longrightarrow
\mrm{H}^2_\et\big(K_{\diamond,t}(p^{\alpha+1});E_\wp(2)\big)^{\mathfrak{I}_{\diamond,t}^{\alpha+1}(K)},
\]
where
$\mathfrak{I}_{\diamond,t}^{\alpha+1}(K) =  \ker (\bb{G}^{\alpha+1}_{\diamond,t}(K) \to \bb{G}^{\alpha}_{\diamond,t}(K))$
is the Galois group.
\end{lemma}
\begin{proof}
The claim follows by analyzing the Hochschild-Serre spectral sequence
\[
\mrm{E}_2^{p,q}
=
\mrm{H}^p\big(\mathfrak{I}_{\diamond,t}^{\alpha+1}(K),\mrm{H}_\et^q(K_{\diamond,t}(p^{\alpha+1});E_\wp(2))\big)\implies \mrm{H}_\et^{p+q}\big(K_{\diamond,t}(p^\alpha)\cap K_0(p^{\alpha+1});E_\wp(2)\big).
\]
It degenerates at the second page because  $\mrm{E}_2^{p,q}=0$  for all $p>0$ as $\mathfrak{I}_{\diamond,t}^{\alpha+1}(K)$ is a finite group and $\mrm{H}_\et^q(K_{\diamond,t}(p^{\alpha+1});E_\wp(2))$ is an $E_\wp$-vector space.
\end{proof}

\begin{proposition}\label{compatibility}
Let
$p_{\alpha+1}: \Lambda_{\alpha+1}\to \Lambda_\alpha$ be the homomorphism induced by the natural projection
$\Gamma_{\alpha+1}\to \Gamma_{\alpha}.$
Then the diagram
\[
\xymatrix{
\mrm{H}^2_\et\big(K_{\diamond,t}(p^{\alpha+1});O(2)\big)
\times
\mrm{H}^2_{\et,c}\big(K_{\diamond,t}(p^{\alpha+1});O\big)
\ar[rr]^-{\{ , \}_{\alpha+1}} \ar[d]^{(\varpi_{2})_*\times(\varpi_{2})_*}
&&
\Lambda_{\alpha+1} \ar[d]^{p_{\alpha+1}}\\
	\mrm{H}^2_\et\big(K_{\diamond,t}(p^\alpha);O(2)\big)
	\times
	\mrm{H}^2_{\et,c}\big(K_{\diamond,t}(p^\alpha);O\big) \ar[rr]^-{\{ ,\}_\alpha} && \Lambda_{\alpha}
	}
\]
commutes.
\end{proposition}
\begin{proof}
We prove the proposition through a direct computation. Since the pairing is defined modulo torsion, it suffices to prove the lemma after inverting $p$. We have
\begin{equation*}
\begin{split}
p_{\alpha+1}\big(\{x_{\alpha+1},y_{\alpha+1}\}_{\alpha+1}\big)
&=
p_{\alpha+1}\Big(\sum_{z\in\Gamma_{\alpha+1}}\Big\langle\langle \Delta(z),1\rangle^* x_{\alpha+1}, (\lambda_{\alpha+1})^*\circ U_p^{\alpha+1} y_{\alpha+1}\rangle\Big\rangle_{\alpha+1} [z^{-1}]\Big) \\
& =\sum_{b\in \Gamma_{\alpha}}\Big\langle \sum_{z\in\Gamma_{\alpha+1},\ z\mapsto b}\langle \Delta(z),1\rangle^*x_{\alpha+1},\ (\lambda_{\alpha+1})^*\circ U_p^{\alpha+1} y_{\alpha+1}\Big\rangle_{\alpha+1 } [b^{-1}].
\end{split}
\end{equation*}
Note that
\[
\sum_{z\in\Gamma_{\alpha+1},\ z\mapsto b} \langle \Delta(z),1\rangle^*x_{\alpha+1}
=\sum_{z\in\Gamma_{\alpha+1},\ z\mapsto b} \langle \Delta(z)^{-1},1\rangle_*x_{\alpha+1}
\]
is invariant under the action of $\mathfrak{I}_{\diamond,t}^{\alpha+1}(K)$, and thus equals to $(\mu)^*\eta_b$ for some
$\eta_b \in \mrm{H}^2_\et\big(K_{\diamond,t}(p^\alpha)\cap K_0(p^{\alpha+1});E_\wp(2)\big)$
by Lemma $\ref{lemma: invariant}$. We compute that
\begin{equation}\label{eq: number 1}
	\begin{split}
p_{\alpha+1}\circ\{x_{\alpha+1},y_{\alpha+1}\}_{\alpha+1}
&=
\sum_{b\in \Gamma_{\alpha}}\Big\langle (\mu)^*\eta_b,\ (\lambda_{\alpha+1})^*\circ U_p^{\alpha+1} y_{\alpha+1}\Big\rangle_{\alpha+1} [b^{-1}]\\
&= 
\sum_{b\in \Gamma_{\alpha}}\Big\langle \eta_b,\ (\mu)_*\circ(\lambda_{\alpha+1})^*\circ U_p^{\alpha+1} y_{\alpha+1}\Big\rangle_{\alpha} [b^{-1}]\\
&= 
\sum_{b\in \Gamma_{\alpha}}\Big\langle \eta_b,(\pi_2)^*\circ\ (\lambda_{\alpha})^*\circ U_p^{\alpha}\circ(\varpi_{2})_* y_{\alpha+1}\Big\rangle_{\alpha} [b^{-1}]\\
&= \sum_{b\in \Gamma_{\alpha}}\Big\langle (\pi_{2})_*\eta_b,\ (\lambda_{\alpha})^*\circ U_p^{\alpha}\circ(\varpi_{2})_* y_{\alpha+1}\Big\rangle_{\alpha} [b^{-1}]
\end{split}\end{equation}
using $(\mu)_*\circ U_p=(\pi_1)^*\circ(\varpi_2)_*$ and $\eqref{formula-lambdapi}$ to obtain the third equality. Observing that
\[\begin{split}
	(\pi_{2})_*\eta_b
	&= \frac{1}{\deg(\mu)}\cdot(\pi_{2})_*\circ(\mu)_*\Big(\sum_{z\in\Gamma_{\alpha+1},\ z\mapsto b}  \langle \Delta(z)^{-1},1\rangle_*x_{\alpha+1}\Big)\\
	&=
	\langle\Delta(b)^{-1},1\rangle_*\circ(\varpi_{2})_*x_{\alpha+1}\\
	&=
	\langle \Delta(b),1\rangle^*\circ(\varpi_{2})_*x_{\alpha+1},
\end{split}
\]
the claim follows
\[
p_{\alpha+1}\circ\{x_{\alpha+1},y_{\alpha+1}\}_{\alpha+1}
= \{(\varpi_{2})_*x_{\alpha+1},(\varpi_{2})_*y_{\alpha+1}\}_{\alpha}.
\]
\end{proof}

%
%

Consider $\mrm{V}_\infty^\mrm{dR}(-2)=\underset{\leftarrow, \varpi_{2}}{\lim}\ \cal{V}_\alpha(-2)\otimes_OB_\mrm{dR}$, then extending scalars in the pairing $\{,\}_\alpha$ in (\ref{FiniteLevelPairing}), restricting to the $\scr{G}_{\mbox{\tiny $\heartsuit$}}$-isotypic subspace in the second argument and taking projective limits gives
\[
\{\ , \}_\scr{G}\colon \Big(\bs{\cal{V}}_\mathscr{G}(M)(\theta_\bb{Q}^{-1}\cdot\bs{\eta}_\bb{Q})\Big)\widehat{\otimes}_{\bb{Z}_p} \widehat{\bb{Z}}_p^\mrm{ur}\
\times\
\mrm{V}_\infty^\mrm{dR}(-2)[\scr{G}_{\mbox{\tiny $\heartsuit$}}]
\longrightarrow 
 B_\mrm{dR}\llbracket\Gamma\rrbracket\otimes_{\bs{\Lambda}}\mbf{I}_\scr{G}.
\]

\begin{proposition}\label{prop: Big pairing}
Let $\bs{\Theta}=(\bs{\eta}_\bb{Q}\cdot\eta_\bb{Q})_{\lvert\Gamma_{\bb{Q}_p}}$, then the pairing 
\[
	\{\ , \}_\scr{G}\colon \bs{\cal{V}}_\mathscr{G}(M)(-1)(\bs{\Theta}^{-1})\widehat{\otimes}_{\bb{Z}_p} \widehat{\bb{Z}}_p^\mrm{ur}\
	\times \
	\mrm{V}_\infty^\mrm{dR}[\scr{G}_{\mbox{\tiny $\heartsuit$}}]
	\longrightarrow 
	B_\mrm{dR}\llbracket\Gamma\rrbracket\otimes_{\bs{\Lambda}}\mbf{I}_\scr{G}(-\psi_\circ)
\]
is $\Gamma_{\bb{Q}_p}$-equivariant.
\end{proposition}
\begin{proof}
  Let $\sigma\in \Gamma_{\bb{Q}_p}$ such that $a\in\bb{Q}_p^\times$ corresponds to $\sigma$ via the local Artin map. Lemma \ref{AtkinLehnerGalois} implies
    \[
    (\lambda_\alpha)^*\circ\sigma^* 
    =
    \langle \Delta(a^{-1}),1\rangle^*\circ\sigma^*\circ (\lambda_\alpha)^*,
    \]
    which, combined with (\ref{FiniteLevelPairing}) yields
    \[\begin{split}
    \big\{\sigma^*(x_\alpha), \sigma^*(y_\alpha)\big\}_\alpha 
    &= 
    \sum_{z\in \Gamma_\alpha}\Big\langle\sigma^*\circ\langle \Delta(z),1\rangle^* x_\alpha, 
    \sigma^*\circ\langle \Delta(a^{-1}),1\rangle^*\circ(\lambda_\alpha)^*\circ U_p^\alpha y_\alpha\Big\rangle_\alpha [z^{-1}]\\
    &= 
    \sum_{z\in \Gamma_\alpha}\Big\langle\langle \Delta(za),1\rangle^* x_\alpha, 
    (\lambda_\alpha)^*\circ U_p^\alpha y_\alpha\Big\rangle_\alpha [z^{-1}]\\
    &= 
    [a]\big\{x_\alpha, y_\alpha\big\}_\alpha 
    \end{split}\]
   where in the second equality we used the Galois equivariance of the Poincar\'e pairing
    \[
    \langle\ ,\rangle_\alpha\colon
    \mrm{H}^2_\et(S(K_{\diamond,t}(p^\alpha))_{\bar{\bb{Q}}},O(2))
    \times
    \mrm{H}^2_{\et,c}(S(K_{\diamond,t}(p^\alpha))_{\bar{\bb{Q}}},O)
    \longrightarrow
    O.
    \]
   As
    \[ \scr{G}_{\mbox{\tiny $\heartsuit$}}([\Delta(a),1])=\boldsymbol{\chi}(\Delta(a))\big[\xi_{\Delta(a)}^{-t_L}\big]
    =
    \psi^{-1}_\circ(a)\cdot\theta_\bb{Q}^{-2}(a)\cdot\bs{\eta}^{2}_\bb{Q}(a),
    \]
    we see that
    \[
   \big\{\sigma^*(x_\alpha), \sigma^*(y_\alpha)\big\}_\alpha
    = \psi^{-1}_\circ(a)\cdot\theta_\bb{Q}^{-2}(a)\cdot\bs{\eta}^{2}_\bb{Q}(a)\cdot\{x_\alpha, y_\alpha\}_\alpha.
    \]
	Therefore, the pairing $\{\ , \}_\scr{G}$
	\[ \bs{\cal{V}}_\mathscr{G}(M)(\theta_\bb{Q}^{-1}\cdot\bs{\eta}_\bb{Q})\widehat{\otimes}_{\bb{Z}_p} \widehat{\bb{Z}}_p^\mrm{ur}
	\times 
	\mrm{V}_\infty^\mrm{dR}(-2)[\scr{G}_{\mbox{\tiny $\heartsuit$}}]
	\longrightarrow 
	B_\mrm{dR}\llbracket\Gamma\rrbracket\otimes_{\bs{\Lambda}}\mbf{I}_\scr{G}(\psi^{-1}_\circ\cdot\theta_\bb{Q}^{-2}\cdot\bs{\eta}^{2}_\bb{Q}).
	\]
	is  $\Gamma_{\bb{Q}_p}$-equivariant and the claim follows by twisting.
\end{proof}

\subsection{On Dieudonn\'e modules}
Given $\msf{f}^*_\circ\in S^\mrm{ord}_{2,1}(N;\psi_\circ^{-1};\overline{\bb{Q}})$  an ordinary elliptic cuspform, one can define a linear map
\[
\varphi\colon S^\mrm{ord}_{2,1}(V_{1,0}(N,p);\psi_\circ^{-1};\overline{\bb{Q}})\longrightarrow\overline{\bb{Q}},\qquad \msf{h}\mapsto\frac{\big\langle \msf{h},\msf{f}_\circ^{*\mbox{\tiny $(p)$}}\big\rangle}{\big\langle \msf{f}_\circ^{*\mbox{\tiny $(p)$}}, \msf{f}_\circ^{*\mbox{\tiny $(p)$}}\big\rangle}
\]
which satisfies $\varphi(T^*(\ell)\msf{h})=\msf{a}_p(\ell,\msf{f}_\circ)\cdot\varphi(\msf{h})$. As in  (\cite{DR2}, Section 2.3 $\&$ Equation (118)) we can give the following definition.

\begin{definition}\label{def eta}
	Let $\eta_\circ\in \mrm{H}^1_\mrm{dR}(X_{1,0}(N,p))^{\mrm{ord},\mrm{ur}}[\msf{f}^*_\circ]$ denote the  unique class that satisfies
	\[
	\Phi(\eta_\circ)=\alpha_{\msf{f}^*_\circ}\cdot\eta_\circ
	\]
	and for any $\msf{h}\in S^\mrm{ord}_{2,1}(V_{1,0}(N,p);\psi_\circ^{-1};\overline{\bb{Q}})$
	\[
	\big\langle\omega_\msf{h},(\lambda_1)^*\eta_\circ\big\rangle_\mrm{dR}=\varphi(\msf{h}).
	\]
\end{definition}

\begin{remark}
	For any $\phi\in S^\mrm{ord}_{2,1}(V_{1,0}(N,p);\psi_\circ;\overline{\bb{Q}})$ we have
	\[
	\big\langle\omega_\phi,\eta_\circ\big\rangle_\mrm{dR}=\varphi\big((\lambda_1)_*\phi\big)=\frac{\big\langle (\lambda_1)_*\phi,\msf{f}_\circ^{*\mbox{\tiny $(p)$}}\big\rangle}{\big\langle \msf{f}_\circ^{*\mbox{\tiny $(p)$}}, \msf{f}_\circ^{*\mbox{\tiny $(p)$}}\big\rangle}.
	\]
\end{remark}

\noindent The Hecke equivariant twist of the Poincar\'e pairing
\begin{equation}\label{heckeequivariantpoincare}
\{\ ,\}_\bb{Q}\colon
\mrm{H}^1_\et(X_{1,0}(N,p)_{\bar{\bb{Q}}},O(1))
\times \mrm{H}^1_\et(X_{1,0}(N,p)_{\bar{\bb{Q}}},O)
\longrightarrow O,
\qquad
\big\{x,y\big\}_\bb{Q} = \big\langle x,(\lambda_1)^* y\big\rangle_\mrm{dR}
\end{equation}
 induces a $\Gamma_{\bb{Q}_p}$-equivariant perfect pairing on $\msf{f}_\circ$-isotypic components
\[
    \{\ ,\}_{\msf{f}_\circ}\colon
\mrm{V}_{\msf{f}_\circ}(p) \times \mrm{V}_{\msf{f}_\circ}(p)(-1)
\longrightarrow O(\psi_\circ).
\]
Furthermore, by looking at the Galois action, one sees that $\mrm{Fil}^1\mrm{V}_{\msf{f}_\circ}(p)$ and  $\mrm{Fil}^1\mrm{V}_{\msf{f}_\circ}(p)(-1)$ are orthogonal with respect to $ \{\ ,\}_{\msf{f}_\circ}$. Therefore there is an induced perfect pairing
\begin{equation}\label{QPairing}
    \{,\}_{\msf{f}_\circ}\colon
\mrm{Gr}^0\mrm{V}_{\msf{f}_\circ}(p) \times \mrm{Fil}^1\mrm{V}_{\msf{f}_\circ}(p)(-1)
\longrightarrow O(\psi_\circ).
\end{equation}
which we can use to make the identification
\[
\mrm{D}_\mrm{dR}\big(\mrm{Fil}^1\mrm{V}_{\msf{f}_\circ}(p)(-1)\big)\overset{\sim}{\longrightarrow}\mrm{Hom}_{E_\wp}\Big(\mrm{D}_\mrm{dR}\big(\mrm{Gr}^0\mrm{V}_{\msf{f}_\circ}(p)\big),E_\wp\Big).
\]

\begin{definition}\label{etaprime}
We denote by 
\[
\eta_\circ'\in \mrm{D}_\mrm{dR}\big(\mrm{Fil}^1\mrm{V}_{\msf{f}_\circ}(p)(-1)\big)
\]
the element corresponding to the homomorphism
\[ \mrm{D}_\mrm{dR}\big(\mrm{Gr}^0(\mrm{V}_{\msf{f}_\circ}(p))\big)\longrightarrow E_\wp,\qquad
\omega_\phi\mapsto\big\langle\omega_\phi,\eta_\circ\big\rangle_\mrm{dR}.
\]
It satisfies
\[
\big\{\omega,\eta_\circ'\big\}_{\msf{f}_\circ}= \big\langle\omega,\eta_\circ\big\rangle_\mrm{dR}\qquad \forall\ \omega\in \mrm{D}_\mrm{dR}\big(\mrm{Gr}^0\mrm{V}_{\msf{f}_\circ}(p)\big).
\]
\end{definition}

\begin{proposition}\label{prop: huge period map}
Let	$\bs{\cal{U}}_\scr{G}^{\msf{f}_\circ}(M)
	=
	\mrm{Fil}^2 \bs{\cal{V}}_\mathscr{G}(M)(-1) (\boldsymbol{\Theta}^{-1}) \otimes
	\mrm{Gr}^0\mrm{V}_{\mrm{f}_\circ}(p)$, then 
 there exists a homomorphism of $\mbf{I}_\scr{G}$-modules
	\[
	\Big\langle\ ,\omega_{\breve{\scr{G}}}\otimes\eta_\circ'\Big\rangle\colon \bb{D}\big(\bs{\cal{U}}_\scr{G}^{\msf{f}_\circ}(M)\big)\longrightarrow \bs{\Pi} \otimes_{\bs{\Lambda}} \mbf{I}_\scr{G}
	\]
	whose specialization at any arithmetic point $\mrm{P}\in\cal{A}_{\bs{\chi}}(\mbf{I}_\scr{G})$ of weight 2 is 
	\[
	\mrm{P}\circ\Big\langle\ ,\omega_{\breve{\scr{G}}}\otimes\eta_\circ'\Big\rangle
	=
	\Big\langle\ ,(\lambda_\alpha)^*\omega_{\breve{\scr{G}}_{\msf{P}}}\otimes \eta_\circ\Big\rangle_\mrm{dR}\colon \mrm{D}_\mrm{dR}\big(\cal{U}_{\scr{G}_\mrm{P}}^{\msf{f}_\circ}(M)\big) \longrightarrow \bb{C}_p.
	\]
	where $\cal{U}_{\scr{G}_\mrm{P}}^{\msf{f}_\circ}(M)=\bs{\cal{U}}_{\scr{G}}^{\msf{f}_\circ}(M)\otimes_{\mbf{I}_\scr{G},\mrm{P}}E_\wp$.
\end{proposition}

\begin{proof}
By tensoring the $\Gamma_{\bb{Q}_p}$-equivariant pairing of Proposition \ref{prop: Big pairing}
 with the $\Gamma_{\bb{Q}_p}$ equivariant pairing in (\ref{QPairing}),
we obtain
\[
\bs{\cal{U}}_\scr{G}^{\msf{f}_\circ}(M)\widehat{\otimes}_{\bb{Z}_p} \widehat{\bb{Z}}_p^\mrm{ur}\ 
\times\ 
\mrm{V}_\infty^\mrm{dR}[\scr{G}_{\mbox{\tiny $\heartsuit$}}]
\otimes_O \mrm{Fil}^1\mrm{V}_{\msf{f}_\circ}(p)(-1)
\longrightarrow
B_\mrm{dR}\llbracket\Gamma\rrbracket \otimes_{\bs{\Lambda}}\mbf{I}_\scr{G}.
\]
 Restricting the pairing to $\Gamma_{\bb{Q}_p}$-invariants we obtain
\begin{equation}\label{hereisthepairing}
\big\langle\ , \big\rangle\colon\bb{D}\big(\bs{\cal{U}}_\scr{G}^{\msf{f}_\circ}(M)\big)\ 
\times\ 
\mbf{D}_\mrm{dR}\big(\mrm{Gr}^0\mrm{V}_\infty\big)[\scr{G}_{\mbox{\tiny $\heartsuit$}}]
\otimes_{\bb{Q}_p}
\mrm{D}_\mrm{dR}\big(\mrm{Fil}^1\mrm{V}_{\msf{f}_\circ}(p)(-1)\big)
\longrightarrow 
\bs{\Pi}\otimes_{\bs{\Lambda}} \mbf{I}_\scr{G}
\end{equation}
because
\[
\Big(B_\mrm{dR}\llbracket\Gamma\rrbracket\Big)^{\Gamma_{\bb{Q}_p}} \otimes_{\bs{\Lambda}}\mbf{I}_\scr{G}\cong\Big(\varprojlim_{\alpha} \bb{Q}_p[\Gamma_\alpha]\Big)\otimes_{\bs{\Lambda}}\mbf{I}_\scr{G}
=
\bs{\Pi}\otimes_{\bs{\Lambda}} \mbf{I}_\scr{G}.
\]
Let 
$\omega_{\breve{\scr{G}}} := \raisebox{\depth}{\scalebox{1}[-1]{$ \Omega $}}_\scr{G}(\breve{\scr{G}})\in \mbf{D}_\mrm{dR}\big(\mrm{Gr}^0\mrm{V}_\infty\big)[\scr{G}_{\mbox{\tiny $\heartsuit$}}]$ be the class represented by the compatible collection $(U_p^{-\alpha}\breve{\scr{G}}_\alpha)_\alpha$ of cuspforms, and let $\eta_\circ'\in \mrm{D}_\mrm{dR}\big(\mrm{Fil}^1\mrm{V}_{\msf{f}_\circ}(p)(-1)\big)$ be the class of Definition $\ref{etaprime}$. Then
evaluating the pairing ($\ref{hereisthepairing}$) at $\omega_{\breve{\scr{G}}}\otimes\eta_\circ'$ gives the homomorphism
\begin{equation}\label{dR pairing}
\big\langle\ ,\omega_{\breve{\scr{G}}}\otimes\eta_\circ'\big\rangle\colon
\bb{D}\big(\bs{\cal{U}}_\scr{G}^{\msf{f}_\circ}(M)\big)
\longrightarrow \bs{\Pi}\otimes_{\bs{\Lambda}} \mbf{I}_\scr{G}.
\end{equation}
Now we study the specialization of the pairing at  arithmetic points $\mrm{P}\in\cal{A}_{\bs{\chi}}(\mbf{I}_\scr{G})$ of weight $(2t_L,t_L)$ and character $(\chi_\circ\theta_L^{-1}\chi^{-1})$ of level $p^\alpha$. Let 
\[
\bs{z}=\sum_i x_i\otimes y_i
\in \bb{D}\big(\bs{\cal{U}}_\scr{G}^{\msf{f}_\circ}(M)\big)
\]
be any element, then by construction
\[
\Big\langle\bs{z},\ \omega_{\breve{\scr{G}}}\otimes\eta_\circ'\Big\rangle 
=
\sum_i\big\{x_i,\omega_{\breve{\scr{G}}}\big\}_\scr{G}\cdot\big\{y_i,\eta_\circ'\big\}_{\msf{f}_\circ}.
\]
Firstly we note that
\[
\big\{y_i,\eta_\circ'\big\}_{\msf{f}_\circ}=\big\langle y_i, \eta_\circ\big\rangle_\mrm{dR},
\] 
then we observe that the projection of 
$\big\{x_i,\omega_{\breve{\scr{G}}}\big\}_\scr{G}$
to level $\alpha$ is  
\[\begin{split}
\big\{x_{i,\alpha},U_p^{-\alpha}\breve{\scr{G}}_\alpha\big\}_\alpha
&=
\sum_{z\in \Gamma_\alpha}
\Big\langle \langle \Delta(z),1\rangle^*x_{i,\alpha}, (\lambda_\alpha)^*\circ U_{p}^\alpha U_p^{-\alpha}\breve{\scr{G}}_\alpha\Big\rangle[z^{-1}]\\
&=
\sum_{z\in \Gamma_\alpha}
\Big\langle x_{i,\alpha}, (\lambda_\alpha)^*\circ\langle \Delta(z),1\rangle^*\breve{\scr{G}}_\alpha\Big\rangle[z^{-1}].
\end{split}\]
Therefore
\[\begin{split}
\mrm{P}\circ\big\{x_i,\omega_{\breve{\scr{G}}}\big\}_\scr{G}
&=
\Big\langle x_{i,\alpha},
(\lambda_\alpha)^*\sum_{z}\chi_{\mbox{\tiny $\spadesuit$}}(z^{-1})\langle \Delta(z)^{-1},1\rangle_* \breve{\scr{G}}_\alpha\Big\rangle\\
&=
\Big\langle x_{i,\alpha},(\lambda_\alpha)^*\breve{\scr{G}}_\mrm{P}\Big\rangle
=
\Big\langle x_{i,\mrm{P}},(\lambda_\alpha)^*\breve{\scr{G}}_\mrm{P}\Big\rangle
\end{split}\]
where the last equality results from the $\Lambda_\alpha$-equivariance of the twisted Poincar\'e pairing. It follows that
\[
\mrm{P}\circ \Big\langle\bs{z},\ \omega_{\breve{\scr{G}}}\otimes\eta_\circ'\Big\rangle = \Big\langle\bs{z}_\mrm{P},\ (\lambda_\alpha)^*\omega_{\breve{\scr{G}}_{\msf{P}}}\otimes \eta_\circ \Big\rangle_\mrm{dR}.
\]
\end{proof}

\begin{remark}\label{remintegral}
Under Conjecture \ref{wishingOhta}, the homomorphism $\big\langle\ ,\omega_{\breve{\scr{G}}}\otimes\eta_\circ'\big\rangle\colon \bb{D}\big(\bs{\cal{U}}_\scr{G}^{\msf{f}_\circ}(M)\big)\rightarrow \bs{\Pi} \otimes_{\bs{\Lambda}} \mbf{I}_\scr{G}$ of Proposition \ref{prop: huge period map} is actually $\mbf{I}_\scr{G}$-valued, and we are allowed to compose it with the arithmetic point $\mrm{P}_\circ\in\cal{A}_{\bs{\chi}}(\mbf{I}_\scr{G})$ of parallel weight one.
Then, for any $\bs{z}\in\bb{D}\big(\bs{\cal{U}}_\scr{G}^{\msf{f}_\circ}(M)\big)$ we have
\[
	\mrm{P}_\circ\circ\Big\langle\bs{z} ,\ \omega_{\breve{\scr{G}}}\otimes\eta_\circ'\Big\rangle
	=\Big\langle\bs{z}_{\mrm{P}_\circ} ,\ \omega_{\breve{\scr{G}}_{\mrm{P}_\circ}}\otimes\eta_\circ'\Big\rangle
\]
\end{remark}

\section{Motivic $p$-adic $L$-functions}
\subsection{Perrin-Riou's regulator}

Let $\mrm{P}\in\cal{A}_{\bs{\chi}}(\mbf{I}_\scr{G})$ be an arithmetic point of weight $(\ell t_L,t_L)$ and character $(\chi_\circ\theta_L^{1-\ell}\chi^{-1},\mathbbm{1})$, then \[\bs{\Theta}(\mrm{P})=\chi_{\mbox{\tiny $\spadesuit$}}^{-1}\cdot\big(\eta_\bb{Q}^{\ell-1}\big)_{\lvert D_p}\] 
 has negative Hodge--Tate weight if $\ell\ge 2$.  In the case of the arithmetic point $\mrm{P}_\circ$ of weight $(t_L,t_L)$ and character $(\chi_\circ,\mathbbm{1})$, the specialization has Hodge--Tate weight equal to zero:
 \[ 
 \boldsymbol{\Theta}(\mrm{P}_\circ)\equiv 1.
 \]
 The Galois modules $\bs{\cal{U}}^{\msf{f}_\circ}_\scr{G}(M)$ is unramified and recalling Definition \ref{TwistedGradedPiece} we see that
\[
\bs{\cal{V}}^{\msf{f}_\circ}_\scr{G}(M)=\bs{\cal{U}}^{\msf{f}_\circ}_\scr{G}(M)(\bs{\Theta}).
\]
Therefore, if we let $\cal{V}^{\msf{f}_\circ}_{\msf{g}_\mrm{P}}(M)=\bs{\cal{V}}^{\msf{f}_\circ}_\scr{G}(M)\otimes_{\mbf{I}_\scr{G},\mrm{P}}E_\wp$, there are isomorphisms
\begin{equation}\label{step four}
\begin{split}
&\log_\mrm{BK}:\mrm{H}^1\big(\bb{Q}_p, \cal{V}^{\msf{f}_\circ}_{\msf{g}_\mrm{P}}(M)\big)\overset{\sim}{\longrightarrow}\mrm{D}_\mrm{dR}\big(\cal{V}^{\msf{f}_\circ}_{\msf{g}_\mrm{P}}(M)\big),\qquad\quad \text{if}\ \mrm{P}\ \text{has weight}\ \ell\ge 2,\\
&\exp^*_\mrm{BK}:\mrm{H}^1\big(\bb{Q}_p, \cal{V}^{\msf{f}_\circ}_{\msf{g}_{\mrm{P}_\circ}}(M)\big)\overset{\sim}{\longrightarrow}\mrm{D}_\mrm{dR}\big(\cal{V}^{\msf{f}_\circ}_{\msf{g}_{\mrm{P}_\circ}}(M)\big), \qquad \text{if}\ \mrm{P}=\mrm{P}_\circ,
\end{split}
\end{equation}
since $\cal{V}^{\msf{f}_\circ}_{\msf{g}_\mrm{P}}(M)$ never contains $\bb{Q}_p(1)$ nor the trivial $1$-dimensional representation.

\begin{lemma}
Let $\beta:\bb{Z}_p^\times\to E_\beta^\times$ be a finite order character of conductor $p^\alpha$, corresponding to a Galois character $\beta:\Gamma_{\bb{Q}_p}\to E_\beta^\times$ factoring through $\text{Gal}(\bb{Q}_p(\zeta_{p^\alpha})/\bb{Q}_p)$. Consider the $\Gamma_{\bb{Q}_p}$-representation $E_\beta\big(\beta+j\big)$, then the $E_\beta$-vector space 
	\[
	\mrm{D}_\mrm{dR}\big(E_\beta(\beta+j)\big)= E_\beta\cdot b_{\beta,j}.
	\]
has a canonical basis $b_{\beta,j}$.	
\end{lemma}
\begin{proof}
	For any $j\in\bb{Z}$, the choice of a compatible sequence of $p$-power roots of unity $\bs{\zeta}:=\{\zeta_{p^\alpha}\}_{\alpha\ge0}$ determines a basis $\bs{\zeta}^j$ of the $\Gamma_{\bb{Q}_p}$-representation $\bb{Q}_p(j)$ and an element $t^{-j}\in\mrm{B}_\mrm{dR}$ such that the element $\bs{\zeta}^j\otimes t^{-j}$ gives a canonical basis of $\mrm{D}_\mrm{dR}\big(\bb{Q}_p(j))$ independent of $\bs{\zeta}$. We consider models of the $\Gamma_{\bb{Q}_p}$-representations $E_\beta(\beta)$, $E_\beta(-\beta)$ appearing in the Galois modules $E_\beta\otimes_{\bb{Q}_p}\bb{Q}_p(\zeta_{p^\alpha})$ where $\Gamma_{\bb{Q}_p}$ acts only on the second factor by Galois automorphisms. For a character $\alpha:\text{Gal}(\bb{Q}_p(\zeta_{p^\alpha})/\bb{Q}_p)\to E_\beta^\times$ the element
	\[
	\theta_{\alpha}=\sum_{\tau\in \text{Gal}(\bb{Q}_p(\zeta_{p^\alpha})/\bb{Q}_p)}\alpha^{-1}(\tau)\otimes\zeta_{p^\alpha}^\tau\in E_\beta\otimes_{\bb{Q}_p}\bb{Q}_p(\zeta_{p^\alpha})
\]
satisfies $(\theta_{\alpha})^\sigma=\alpha(\sigma)\theta_{\alpha}$ for all $\sigma\in \Gamma_{\bb{Q}_p}$. Then $E_\beta(\beta)\cong E_\beta\cdot\theta_{\beta}$ and $E_\beta(-\beta)\cong E_\beta\cdot\theta_{\beta^{-1}}$. We choose the model $E_\beta\cdot\theta_\beta\otimes\bs{\zeta}^j$ of the $\Gamma_{\bb{Q}_p}$-representation $E_\beta(\beta+j)$ and we note that 
\[
b_{\beta,j}:=(\theta_\beta\otimes\bs{\zeta}^j)\otimes_{E_\beta}(\theta_{\beta^{-1}} \otimes t^{-j})\in E_\beta(\beta+j)\otimes_{\bb{Q}_p}\mrm{B}_\mrm{dR}
\] 
is invariant under the $\Gamma_{\bb{Q}_p}$-action and it is independent of the choice of $\bs{\zeta}$. Therefore, we deduce that $\mrm{D}_\mrm{dR}\big(E_\beta(\beta+j)\big)$ has $b_{\beta,j}$ as canonical $E_\beta$-basis.
\end{proof}

\noindent	Write $\bs{\Lambda}_\Gamma$ for $O\llbracket\bb{Z}_p^\times\rrbracket$, then by (\cite{KLZ}, Theorem 8.2.3) and (\cite{LZ14}, Theorems 4.15, B.5) there is a $\big(\mbf{I}_\scr{G}\widehat{\otimes} \bs{\Lambda}_\Gamma\big)$-linear map 
	\[
	\bs{\cal{L}}: \mrm{H}^1\big(\bb{Q}_p,\bs{\cal{U}}_\scr{G}^{\msf{f}_\circ}(M)\widehat{\otimes}\bs{\Lambda}_\Gamma(-\mbf{j})\big)\longrightarrow\bb{D}\big(\bs{\cal{U}}_\scr{G}^{\msf{f}_\circ}(M)\big)\widehat{\otimes}\bs{\Lambda}_\Gamma
	\]
	 such that for all points $\mrm{P}\in\mrm{Hom}(\mbf{I}_\scr{G},\overline{\bb{Q}}_p)$ and all characters of $\bb{Z}^\times_p$ of the form $\eta\cdot\varepsilon_\bb{Q}^j$ where $j\in\bb{Z}$ and $\eta$ has finite order, we have a commutative diagram 
\[\xymatrix{
\mrm{H}^1\big(\bb{Q}_p,\bs{\cal{U}}_\scr{G}^{\msf{f}_\circ}(M)\widehat{\otimes}\bs{\Lambda}_\Gamma(-\mbf{j})\big)\ar[d]\ar[r]^{\bs{\cal{L}}} & \bb{D}\big(\bs{\cal{U}}_\scr{G}^{\msf{f}_\circ}(M)\big)\widehat{\otimes}\bs{\Lambda}_\Gamma\ar[d]\\
\mrm{H}^1\big(\bb{Q}_p,\cal{U}_{\scr{G}_\mrm{P}}^{\msf{f}_\circ}(M)(-j-\eta)\big)\ar[r] & \mrm{D}_\mrm{dR}\big(\cal{U}_{\scr{G}_\mrm{P}}^{\msf{f}_\circ}(M)(-j-\eta)\big)
}\]	
where the rightmost vertical map is \[\bb{D}(\bs{\cal{U}}_\scr{G}^{\msf{f}_\circ}(M))\widehat{\otimes}\bs{\Lambda}_\Gamma\longrightarrow\mrm{D}_\mrm{dR}\big(\cal{U}_{\scr{G}_\mrm{P}}^{\msf{f}_\circ}(M)(-j-\eta)\big),
\qquad
\bs{x}\otimes[u]\mapsto\eta\varepsilon_\bb{Q}^j(u)\cdot\bs{x}_\mrm{P}\otimes b_{\eta^{-1},-j},
\]
and the bottom horizontal map is given by 
\begin{equation}\label{specialization biglog}
\begin{cases}
\left(1-\alpha_{1,\msf{g}_\mrm{P}}\alpha_{2,\msf{g}_\mrm{P}}\alpha_{\msf{f}^*_\circ}^{-1}\cdot \eta(\mrm{Fr}_p)p^j\right)\left(1-\alpha^{-1}_{1,\msf{g}_\mrm{P}}\alpha^{-1}_{2,\msf{g}_\mrm{P}}\alpha_{\msf{f}^*_\circ}\cdot \eta^{-1}(\mrm{Fr}_p)p^{-j-1}\right)^{-1} & \text{cond}(\eta)=0\\
\\
\left(\alpha_{1,\msf{g}_\mrm{P}}\alpha_{2,\msf{g}_\mrm{P}}\alpha_{\msf{f}^*_\circ}^{-1}\cdot\eta(p)p^{1+j}\right)^{\text{cond}(\eta)}G(\eta)^{-1} & \text{cond}(\eta)>0
\end{cases}
\end{equation}
\[
\times\qquad\begin{cases}
\frac{(-1)^{-j-1}}{(-j-1)!}\cdot\log_\mrm{BK} & j<0\\
\\
j!\cdot\exp_\mrm{BK}^* & j\ge 0.
\end{cases}
\]
As in the proof of (\cite{KLZ}, Theorem 8.2.8), we pull back the map $\bs{\cal{L}}$ by the automorphism 
\[
1\otimes[z]\mapsto\bs{\Theta}(z)^{-1}\cdot 1\otimes[z]
\]
of $\mbf{I}_\scr{G}\widehat{\otimes} \bs{\Lambda}_\Gamma$. By functoriality of the construction of $\bs{\cal{L}}$ we obtain 
	\begin{equation}\label{eq: biglog diagram}
		\xymatrix{ \mrm{H}^1\big(\bb{Q}_p,\bs{\cal{V}}^{\msf{f}_\circ}_\scr{G}(M)\widehat{\otimes}\bs{\Lambda}_\Gamma(-\mbf{j})\big)\ar[d]\ar@{.>}[r]^{\bs{\cal{L}}'}&\Big(\bb{D}\big(\bs{\cal{U}}_\scr{G}^{\msf{f}_\circ}(M)\big)\widehat{\otimes}\bs{\Lambda}_\Gamma\Big)\widehat{\otimes}_{\bs{\Theta}}\bs{\Lambda}_\Gamma\ar[d]^{\mrm{id}\otimes\bs{\Theta}^{-1}}\\ \mrm{H}^1\big(\bb{Q}_p,\bs{\cal{U}}_\scr{G}^{\msf{f}_\circ}(M)\widehat{\otimes}\bs{\Lambda}_\Gamma(-\mbf{j})\big)\ar[r]^{\bs{\cal{L}}}&\bb{D}\big(\bs{\cal{U}}_\scr{G}^{\msf{f}_\circ}(M)\big)\widehat{\otimes}\bs{\Lambda}_\Gamma.}\end{equation}

	\begin{proposition}\label{prop: big log}
		There is an homomorphism 
		\[
		\bs{\cal{L}}_\scr{G}^{\msf{f}_\circ}: \mrm{H}^1\big(\bb{Q}_p,\bs{\cal{V}}^{\msf{f}_\circ}_\scr{G}(M)\big)\longrightarrow \bb{D}\big(\bs{\cal{U}}_\scr{G}^{\msf{f}_\circ}(M)\big)
		\]
		 satisfying the following properties: 
		\begin{itemize}
			\item[(i)] For all arithmetic points $\mrm{P}\in\cal{A}_{\bs{\chi}}(\mbf{I}_\scr{G})$ of weight $(\ell t_L, t_L)$, $\ell\ge2$, and character $(\chi_\circ\theta_L^{1-\ell}\chi^{-1},\mathbbm{1})$,
			\[
			\nu_{\mrm{P}}\circ\bs{\cal{L}}^{\msf{f}_\circ}_\scr{G}=\frac{(-1)^{\ell-2}}{(\ell-2)!}\cdot\Upsilon(\mrm{P})\cdot\big(\log_\mrm{BK}\circ\ \mrm{P}\big)
			\]
			where $\Upsilon(\mrm{P})=\left(\alpha_{1,\msf{g}_{\mrm{P}}}\alpha_{2,\msf{g}_{\mrm{P}}}\alpha_{\msf{f}^*_\circ}^{-1}p^{2-\ell}\right)^\alpha\cdot G\big(\chi_{\mbox{\tiny $\spadesuit$}}\cdot\theta^{\ell-1}_{\bb{Q}\lvert D_p}\big)^{-1}$.
			
			\item[(ii)] For the arithmetic point $\mrm{P}_\circ\in\cal{A}_{\bs{\chi}}(\mbf{I}_\scr{G})$ of weight one,
			\[
			\nu_{\mrm{P}_\circ}\circ\bs{\cal{L}}^{\msf{f}_\circ}_\scr{G}= \Upsilon(\mrm{P}_\circ)\cdot\big(\exp^*_\mrm{BK}\circ\ \mrm{P}_\circ\big)
			\]
			where $\Upsilon(\mrm{P}_\circ)=\big(1-\alpha_{1,\msf{g}_{\mrm{P}_\circ}}\alpha_{2,\msf{g}_{\mrm{P}_\circ}}\alpha_{\msf{f}^*_\circ}^{-1}\big)\big(1-\alpha^{-1}_{1,\msf{g}_{\mrm{P}_\circ}}\alpha^{-1}_{2,\msf{g}_{\mrm{P}_\circ}}\alpha_{\msf{f}^*_\circ}\cdot p^{-1}\big)^{-1}$.
		\end{itemize}
	\end{proposition}
	\begin{proof}
		First, we note that the map $\bs{\Theta}\otimes\mrm{id}:\bs{\Lambda}_\Gamma\widehat{\otimes}_{\bs{\Theta}}\bs{\Lambda}_\Gamma\overset{\sim}{\to} \bs{\Lambda}_\Gamma$ is an isomorphism, hence $\bs{\cal{L}}'$ can be seen as a homomorphism
		$\bs{\cal{L}}': \mrm{H}^1\big(\bb{Q}_p,\bs{\cal{V}}^{\msf{f}_\circ}_\scr{G}(M)\widehat{\otimes}\bs{\Lambda}_\Gamma(-\mbf{j})\big)\longrightarrow \bb{D}\big(\bs{\cal{U}}_\scr{G}^{\msf{f}_\circ}(M)\big)\widehat{\otimes}\bs{\Lambda}_\Gamma$.
For any arithmetic point $\mrm{P}\in\cal{A}_{\bs{\chi}}(\mbf{I}_\scr{G})$, there is a commutative diagram
\[\xymatrix{
\mrm{H}^1\big(\bb{Q}_p,\bs{\cal{V}}^{\msf{f}_\circ}_\scr{G}(M)\widehat{\otimes}\bs{\Lambda}_\Gamma(-\mbf{j})\big)\ar[d]\ar[rr]^{\bs{\cal{L}}'} && \bb{D}\big(\bs{\cal{U}}_\scr{G}^{\msf{f}_\circ}(M)\big)\widehat{\otimes}\bs{\Lambda}_\Gamma\ar[d]\\
\mrm{H}^1\big(\bb{Q}_p,\cal{V}^{\msf{f}_{\circ}}_{\msf{g}_{\mrm{P}}}(M)\big)\ar[rr] && \mrm{D}_\mrm{dR}\big(\cal{V}^{\msf{f}_{\circ}}_{\msf{g}_{\mrm{P}}}(M)\big)
}\]	
obtained by composing ($\ref{eq: biglog diagram}$) with specialization at point $\mrm{P}$ and character \[\bs{\Theta}(\mrm{P})^{-1}=\chi_{\mbox{\tiny $\spadesuit$}}\cdot\big(\varepsilon_\bb{Q}^{1-\ell}\cdot\theta_{\bb{Q}}^{\ell-1}\big)_{\lvert D_p}.
\] Then, using ($\ref{specialization biglog}$), the bottom horizontal map can be computed to be
\[
\frac{(-1)^{\ell-2}}{(\ell-2)!}\cdot\left(\alpha_{1,\msf{g}_{\mrm{P}}}\alpha_{2,\msf{g}_{\mrm{P}}}\alpha_{\msf{f}^*_\circ}^{-1}p^{2-\ell}\right)^\alpha G\big(\chi_{\mbox{\tiny $\spadesuit$}}\cdot\theta^{\ell-1}_{\bb{Q}\lvert D_p}\big)^{-1}\cdot\log_\mrm{BK}.
\]
Similarly, when considering the arithmetic point $\mrm{P}_\circ$, the relevant character is the trivial character $\bs{\Theta}(\mrm{P}_\circ)^{-1}\equiv 1$, and the bottom horizontal map can be seen to be
\[
\big(1-\alpha_{1,\msf{g}_{\mrm{P}_\circ}}\alpha_{2,\msf{g}_{\mrm{P}_\circ}}\alpha_{\msf{f}^*_\circ}^{-1}\big)\big(1-\alpha^{-1}_{1,\msf{g}_{\mrm{P}_\circ}}\alpha^{-1}_{2,\msf{g}_{\mrm{P}_\circ}}\alpha_{\msf{f}^*_\circ}\cdot p^{-1}\big)^{-1}\cdot\exp^*_\mrm{BK}.
\]
In order to define the claimed homomorphism $\bs{\cal{L}}_\scr{G}^{\msf{f}_\circ}$, we note that if we consider
\[
\beta:\mbf{I}_\scr{G}\widehat{\otimes}\bs{\Lambda}_\Gamma\longrightarrow \mbf{I}_\scr{G},\qquad 1\otimes[u]\mapsto \langle u\rangle[u],
\]
then for all $\mrm{P}\in\cal{A}_{\bs{\chi}}(\mbf{I}_\scr{G})$ the following diagram commutes
\[\xymatrix{
\mbf{I}_\scr{G}\widehat{\otimes}\bs{\Lambda}_\Gamma\ar[drr]_{\mrm{P}\otimes\bs{\Theta}(\mrm{P})^{-1}}\ar[rr]^\beta && \mbf{I}_\scr{G}\ar[d]^{\mrm{P}} \\
&& O.
}
\]
Therefore, the composition   \[\xymatrix{
\mrm{H}^1\big(\bb{Q}_p,\bs{\cal{V}}^{\msf{f}_\circ}_\scr{G}(M)\big)\ar[d]^{\mrm{H}^1\big(\bb{Q}_p,\mrm{id}\otimes 1\big)}\ar@{.>}[rr]^{\bs{\cal{L}}_\scr{G}^{\msf{f}_\circ}}&& \bb{D}\big(\bs{\cal{U}}_\scr{G}^{\msf{f}_\circ}(M)\big)\\
\mrm{H}^1\big(\bb{Q}_p,\bs{\cal{V}}^{\msf{f}_\circ}_\scr{G}(M)\widehat{\otimes}\bs{\Lambda}_\Gamma(-\mbf{j})\big)\ar[rr]^{\bs{\cal{L}}'}&& \bb{D}\big(\bs{\cal{U}}_\scr{G}^{\msf{f}_\circ}(M)\big)\widehat{\otimes}\bs{\Lambda}_\Gamma\ar[u]_\beta
}\]
satisfies the claimed properties.
\end{proof}

\subsubsection{The motivic $p$-adic $L$-function.}\label{motivic p-adic L-function}
The class 
$\boldsymbol{\kappa}_p^{\msf{f}_\circ}(\scr{G}) \in \mrm{H}^1\big(\mathbb{Q}_p,\bs{\cal{V}}^{\msf{f}_\circ}_{\scr{G}}(M)\big)$, presented in Definition \ref{TwistedGradedPiece} and arising from cycles on Shimura threefolds, is the key input to define the motivic $p$-adic $L$-function.
\begin{definition}\label{motpadicLfun}
The motivic $p$-adic $L$-function is given by
\[
\scr{L}^\mrm{mot}_p(\breve{\scr{G}},\msf{f}_\circ)
:=
\Big\langle \bs{\cal{L}}_\scr{G}^{\msf{f}_\circ}\big(\boldsymbol{\kappa}_p^{\msf{f}_\circ}(\scr{G}) \big),\
\omega_{{\breve{\scr{G}}}}\otimes\eta_\circ'\Big\rangle
\in \bs{\Pi}\otimes_{\bs{\Lambda}} \mbf{I}_\scr{G}.
\]	
\end{definition}

\begin{lemma}\label{cruximplicat}
Assume Conjecture \ref{wishingOhta}, then the motivic $p$-adic $L$-function belongs to $\mbf{I}_\scr{G}$, and 
\[
\scr{L}_p^\mrm{mot}(\breve{\scr{G}},\msf{f}_\circ)(\mrm{P}_\circ)\not=0\quad\implies\quad\mrm{exp}^*_\mrm{BK}\big(\bs{\kappa}^{\msf{f}_\circ}_p(\scr{G})(\mrm{P}_\circ)\big)\not=0.
\]
\end{lemma}
\begin{proof}
	This is a direct consequence of Remark \ref{remintegral} and Proposition \ref{prop: big log}.
\end{proof}

\section{$p$-adic Gross-Zagier formulas}
The main goal of this section is to give a formula for certain values of the syntomic Abel--Jacobi map in terms of $p$-adic modular forms. Given the definition of the motivic $p$-adic $L$-function in terms of the pairing introducted in Section $\ref{generalizingOhta}$, we will be interested in the values
\[
\mrm{AJ}_{\mrm{syn}}\big(\Delta_\alpha^\circ\big)\big((\lambda_\alpha)^*\omega_{\breve{\scr{G}}_\mrm{P}}\otimes\eta_\circ\big)\in\bb{C}_p
\]

\subsection{$P$-syntomic cohomology}
Let $K/\bb{Q}_p$ be finite extension, we denote by $K_0$ the maximal unramified subfield of $K$ and by $q$ be the cardinality of the residue field.
\begin{definition}
	A filtered $(\varphi,N,\Gamma_K)$-module over $K$ is a finite dimensional $\bb{Q}_p^\mrm{ur}$-vector space $D$ endowed with a $\bb{Q}_p^\mrm{ur}$-semilinear bijective Frobenius endomorphism $\varphi$ and a $\bb{Q}_p^\mrm{ur}$-linear monodromy operator $N$ satisfying $N\varphi=p\varphi N$. The absolute Galois group $\Gamma_K$ acts $\bb{Q}_p^\mrm{ur}$-semilinearly on $D$
and there is a decreasing, separated, exhaustive filtration of the $K$-vector space \[D_K:=\big(D\otimes_{\bb{Q}_p^\mrm{ur}}\bar{\bb{Q}}_p\big)^{\Gamma_K}\] by $K$-vector subspaces $\mrm{Fil}^iD_K$.
\end{definition}

\noindent One writes $D_\mrm{st}$ for the $K_0$-vector space $D^{\Gamma_K}$ of $\Gamma_K$-invariant elements. A filtered $(\varphi,N,\Gamma_K)$-module $D$ such that the $\Gamma_K$-action is unramified and $N=0$ is said to be \emph{crystalline}. In this case one can show that
\[
D=D_\mrm{st}\otimes_{K_0}\bb{Q}_p^\mrm{ur}\qquad\text{and}\qquad D_K=D_\mrm{st}\otimes_{K_0}K.
\]
\begin{definition}
	A crystalline filtered $(\varphi,N,\Gamma_K)$-module over $K$ is said to be \emph{convenient} for a choice of polynomial $P(T)\in 1+TK[T]$ if $P(\Phi)$ and $P(q\Phi)$ are bijective endomorphisms of $D_K$, where $\Phi$ denotes the extension of scalars of the $K_0$-linear operator $\varphi^{[K_0:\bb{Q}_p]}$ on $D_\mrm{st}$.
\end{definition}

\noindent For a variety $X_{/K}$, we let $\mrm{H}^{\bfcdot}_\mrm{HK}(X_h)$ and $\mrm{H}^{\bfcdot}_\mrm{dR}(X_h)$ be the extensions of Hyodo--Kato and de-Rham cohomologies defined by Beilinson \cite{Beilinson} and by 
\[
\iota_\mrm{dR}^\mrm{B}:\mrm{H}^{\bfcdot}_\mrm{HK}(X_h)\otimes_{K_0}K\longrightarrow\mrm{H}^{\bfcdot}_\mrm{dR}(X_h)
\]
the comparison morphism relating them (which is an isomorphism if $X$ has a semistable model over $\cal{O}_K$). For the filtered $(\varphi,N,\Gamma_K)$-modules $D^{\bfcdot}(X_h)=\mbf{D}_\mrm{pst}\big(\mrm{H}^{\bfcdot}_\mrm{et}(X_{\overline{K}},\bb{Q}_p)\big)$ it was shown by Beilinson \cite{Beilinson} that 
\[
D^{\bfcdot}(X_h)_\mrm{st}=\mrm{H}^{\bfcdot}_\mrm{HK}(X_h)\qquad\text{and}\qquad D^{\bfcdot}(X_h)_K=\mrm{H}^{\bfcdot}_\mrm{dR}(X_h).
\]
\subsubsection{Cohomology of filtered $(\varphi, N, \Gamma_K)$-modules.} For a polynomial $P(T)\in 1+TK[T]$ one can define the complex 
\[
C^{\bfcdot}_{\mrm{st},P}(D):\qquad D_{\mrm{st},K}\oplus\mrm{Fil}^0D_K\longrightarrow D_{\mrm{st},K}\oplus D_{\mrm{st},K}\oplus D_K\longrightarrow D_{\mrm{st},K}
\]
where the first map is $(u,v)\mapsto(P(\Phi)u, Nu, u-v)$, and the second is $(w,x,y)\mapsto Nw- P(q\Phi)x$. The cohomology of this complex is denoted by \[
\mrm{H}^{\bfcdot}_{\mrm{st},P}(D):=\mrm{H}^{\bfcdot}\big(C^{\bfcdot}_{\mrm{st},P}(D)\big).
\]

\begin{theorem}(\cite{BLZ} Theorem 2.1.2)
	There is a $P$-syntomic descent spectral sequence
	\[
	E_2^{i,j}=\mrm{H}^i_{\mrm{st},P}\big(D^j(X_h)(r)\big)\implies \mrm{H}^{i+j}_{\mrm{syn},P}(X_h,r)
	\]
	compatible with cup products.
\end{theorem}
\subsection{Syntomic Abel--Jacobi map}
Let $X_{/K}$ be a smooth $d$-dimensional variety. The commutativity of the following diagram 
\begin{equation}\label{syn to et}\xymatrix{
& \mrm{CH}^i(X)\ar[ld]_{\mrm{cl}_\mrm{syn}}\ar[rd]^{\mrm{cl}_\mrm{et}} &\\
\mrm{H}^{2i}_\mrm{syn}(X_{K,h},i)\ar@{.>}[rr]^{\rho_\mrm{syn}}\ar@{->>}[d] & & \mrm{H}^{2i}_\mrm{et}(X_{K},\bb{Q}_p(i))\ar@{->>}[d]\\
\mrm{Gr}^0_{\mrm{syn}}\ar@{^{(}->}[d]\ar@{.>}[rr]& & \mrm{Gr}^0_{\mrm{et}}\ar@{=}[d]\\
\mrm{H}_{\mrm{st},1-T}^0(D^{2i}(X_h)(i))\ar@{.>}[rr]^\sim & & \mrm{H}^{2i}_\mrm{et}(X_{\overline{K}},\bb{Q}_p(i))^{\Gamma_K}
}\end{equation}
where $\rho_\mrm{syn}$ is Nekov\'a$\check{r}$--Niziol period morphism, follows from the compatibility of the syntomic descent spectral sequence for syntomic cohomology and the Hochschild--Serre spectral sequence for \'etale cohomology (\cite{BLZ} Theorem 2.1.2). 
\begin{remark}
	The bottom horizontal map of diagram ($\ref{syn to et}$) is an isomorphism by (\cite{BLZ} Theorem 1.1.4), hence the middle horizontal map is injective.
\end{remark}

\noindent Therefore, if we let 
\[
\mrm{CH}^i(X)_0:=\ker\left(\mrm{cl}_\mrm{et}:\mrm{CH}^i(X)\longrightarrow \mrm{H}^{2i}_\mrm{et}(X_{\overline{K}},\bb{Q}_p(i))^{\Gamma_K}\right)
\]
denote the subgroup of null-homologous cycles, then the syntomic and the $p$-adic \'etale Abel--Jacobi maps can be compared 
\[\xymatrix{
& \mrm{CH}^i(X)_0\ar[ld]_{\mrm{AJ}_\mrm{syn}}\ar[dr]^{\mrm{AJ}^\mrm{et}_p} &\\
\mrm{H}_{\mrm{st},1-T}^1(D^{2i-1}(X_h)(i))\ar[rr]^{\exp_\mrm{st}} & & \mrm{H}^1(K, \mrm{H}^{2i-1}_\mrm{et}(X_{\overline{K}},\bb{Q}_p(i)))
}\]
through the generalized Bloch--Kato exponential map. 
If $V$ is a quotient of $\mrm{H}^{2i-1}_\mrm{et}(X_{\overline{K}},\bb{Q}_p(i))$ such that $D=\mbf{D}_\mrm{pst}(V)$ is a convenient quotient of $D^{2i-1}(X_h)(i)$ with respect to the polynomial $1-T$, then the natural inclusion $D_K\hookrightarrow C^1_{\mrm{st},1-T}(D)$ induces an isomorphism
\[
\frac{D_K}{\mrm{Fil}^0D_K}\cong\mrm{H}^1_{\mrm{st},1-T}(D),
\]
and one can refine the comparison to
\[\xymatrix{
& \mrm{CH}^i(X)_0\ar[ld]_{\mrm{AJ}_{\mrm{syn},D}}\ar[dr]^{\mrm{AJ}^\mrm{et}_{p,V}} &\\
D_K/\mrm{Fil}^0\ar[rr]^{\exp_\mrm{BK}} & & \mrm{H}_e^1(K, V). 
}\]
\begin{remark}
Recall (\cite{Bloch-Kato} Definition 3.10) that the Bloch--Kato exponential surjects onto  $\mrm{H}_e^1(K, V)$ and its kernel is given by $\mrm{D}_\mrm{cris}(V)^{\varphi=1}/\mrm{H}^0(K,V)$. In particular, when $\mrm{D}_\mrm{cris}(V)^{\varphi=1}=0$ we can write 
\[
\mrm{AJ}_{\mrm{syn},D}=\log_\mrm{BK}\circ\mrm{AJ}_{p,V}^\mrm{et}\qquad\text{where}\qquad \log_\mrm{BK}=\exp_\mrm{BK}^{-1}.
\]
\end{remark}
\noindent If we let $D^*(1)$ denote the Tate dual of $D$, which is a submodule of $D^{2(d-i)+1}(X_h)(d+1-i)$, then $D_K/\mrm{Fil}^0=\big(\mrm{Fil}^0D^*(1)_K\big)^\vee$ and we can write 
\begin{equation}
	\mrm{AJ}_{\mrm{syn},D}: \mrm{CH}^i(X)_0\longrightarrow \big(\mrm{Fil}^0D^*(1)_K\big)^\vee.
\end{equation}

\subsubsection{Evaluation using $P$-syntomic cohomology.} Let $\Delta\in\mrm{CH}^i(X)_0$ be a null-homologous cycle. For any class 
\[
\eta\in\mrm{Fil}^0D^*(1)_K\subset\mrm{Fil}^{d+1-i}\mrm{H}^{2(d-i)+1}_\mrm{dR}(X_h),
\] choose a polynomial $P(T)\in 1+TK[T]$ such that $P(1)\not=0$, $P(q^{-1})\not=0$ and $\eta\in\mrm{H}^0_{\mrm{st},P}(D^*(1))$ (\cite{BLZ} Proposition 1.4.3). Suppose that $\eta$ is in the kernel of the "knight's move" map \[\mrm{H}^0_{\mrm{st},P}\big(D^{2(d-i)+1}(X_h)(d+1-i)\big)\longrightarrow \mrm{H}^2_{\mrm{st},P}\big(D^{2(d-i)}(X_h)(d+1-i)\big),\]
so that $\eta$ can be lifted to syntomic cohomology. Then, for any lift $\tilde{\eta}\in \mrm{H}^{2(d-i)+1}_{\mrm{syn},P}\big(X_h,d+1-i\big)$ we can write
\begin{equation}\label{evaluation}
	\mrm{AJ}_{\mrm{syn},D}(\Delta)(\eta)=\mrm{tr}_{X,\mrm{syn},P}\big(\mrm{cl}_\mrm{syn}(\Delta)\cup\tilde{\eta}\big)
\end{equation}
 thanks to the compatibility of the $P$-syntomic descent spectral sequence with cup products 
\[\resizebox{\displaywidth}{!}{\xymatrix{
\mrm{Fil}^1\mrm{H}^{2i}_{\mrm{syn},1-T}\big(X_h,i\big)\ar[d]& \times & \mrm{H}^{2(d-i)+1}_{\mrm{syn},P}\big(X_h,d+1-i\big)\ar[d]\ar[r]& \mrm{H}^{2d+1}_{\mrm{syn},P}\big(X_h,d+1\big)/\mrm{Fil}^2\cong K\ar[d]^\sim \\
\mrm{H}^1_{\mrm{st},1-T}\big(D^{2i-1}(X_h)(i)\big)\ar@{->>}[d]& \times & \mrm{H}^0_{\mrm{st},P}\big(D^{2(d-i)+1}(X_h)(d+1-i)\big)\ar[r]& \mrm{H}^1_{\mrm{st},P}\big(\mbf{D}_\mrm{pst}(\bb{Q}_p(1))\big)\cong K\ar@{=}[d]\\
\mrm{H}^1_{\mrm{st},1-T}(D) &\times& \mrm{H}^0_{\mrm{st},P}(D^*(1))\ar@{^{(}->}[u]\ar[r]& \mrm{H}^1_{\mrm{st},P}\big(\mbf{D}_\mrm{pst}(\bb{Q}_p(1))\big)\cong K.
}}\]

\subsection{Abel--Jacobi map of Hirzebruch--Zagier cycles}
For every arithmetic point $\mrm{P}\in \cal{A}_{\bs{\chi}}(\mbf{I}_\scr{G})$ of weight 2 and level $p^\alpha$, the Galois representation $\mrm{V}_{\msf{g}_\mrm{P}}$ is crystalline as a $\Gamma_{\bb{Q}_p(\zeta_{p^\alpha})}$-representation. Throughout this subsection, we will consider all our geometric structures, including the moduli schemes and the cycles, to be defined over $F_\alpha:=\bb{Q}_p(\zeta_{p^\alpha})$. Similarly, we will regard all Galois representations, as well as Dieudonn\'e functors, to be defined with respect to the absolute Galois group $\Gamma_{F_\alpha}$.

\noindent The specialization at $\mrm{P}$ of 
 $\bs{\cal{V}}_{\scr{G},\msf{f}_\circ}(M)$ 
is a quotient of $\mrm{H}_\et^3(Z_\alpha(K)_{\overline{F}},\bb{Q}_p(2))$ such that
\begin{equation}
	D_{\msf{g}_\mrm{P},\msf{f}_\circ}
:=
\mbf{D}_\mrm{pst}\big(\cal{V}_{\scr{G}_\mrm{P},\mrm{f}_\circ}(M)\big)
\end{equation}
is a convenient quotient of $D^3(Z_\alpha(K))(2)$. Furthermore, 
$D_{\msf{g}_\mrm{P},\msf{f}_\circ}\cong (D_{\msf{g}_\mrm{P},\msf{f}_\circ})^*(1)$ since the Galois representation $\cal{V}_{\scr{G}_\mrm{P},\mrm{f}_\circ}(M)$ is Kummer self-dual.

\begin{definition}
Let 
$\omega_\mrm{P} \in e_\mrm{n.o.}
\mrm{Fil}^2\mrm{H}^2_\mrm{dR}(S(K_{\diamond,t}(p^\alpha))/F_\alpha)$
be the de Rham cohomology class associated with the specialization 
\[
\breve{\msf{g}}_\mrm{P}\in S^\mrm{n.o.}_{2t_L,t_L}\big(K_{\diamond,t}(p^\alpha);\chi_\circ\theta_L^{-1}\chi^{-1},\mathbbm{1};O\big)
\] of the $\mbf{K}_\scr{G}$-adic cuspform $\breve{\scr{G}}$.
\end{definition}

\noindent Recall the class $\eta_\circ$ of Definition $\ref{def eta}$, then the tensor product $(\lambda_\alpha)^*\omega_\mrm{P}\otimes\eta_\circ$ belongs to the convenient $(\varphi, N,\Gamma_{F_\alpha})$-module $(D_{\msf{g}_\mrm{P},\msf{f}_\circ})^*(1)_{F_\alpha}\cong (D_{\msf{g}_\mrm{P},\msf{f}_\circ})_{F_\alpha}$ and it makes sense to try to evaluate
\[
\mrm{AJ}_{\mrm{syn}}\big(\Delta_\alpha^\circ\big)\big((\lambda_\alpha)^*\omega_\mrm{P}\otimes\eta_\circ\big)
=\mrm{AJ}_{\mrm{syn}}\Big((\lambda_\alpha,\mrm{id})_*\Delta_\alpha^\circ\Big)\big(\omega_\mrm{P}\otimes\eta_\circ\big).
\]
We use the functoriality of the formation of the syntomic Abel--Jacobi map to move the computation on a Hilbert--Blumenthal variety where the theory of overconvergent $p$-adic Hilbert cuspforms with level at $p$ was developed by Kisin and Lai in (\cite{Kisin-Lai}).
\begin{lemma}\label{ablemma}
    The following equality holds
    \[
    \lambda_\alpha \circ w_{\mathfrak{p}_2^\alpha} 
    =
    \langle\varpi_{\mathfrak{p}_2}^\alpha,1\rangle \circ w_{\mathfrak{p}_2^\alpha} \circ \lambda_\alpha.
    \]
\end{lemma}
\begin{proof}
  By definition 
    $w_{\mathfrak{p}_2^\alpha} 
    =
    \mathfrak{T}_{\tau_{\mathfrak{p}_2}} \circ \nu_\alpha$
    and
    $
    \lambda_\alpha = (-)^* \circ \mathfrak{T}_{\tau_\alpha}
    $, thus by a direct calculation using complex uniformizations one sees that
    \[
    \lambda_\alpha\circ \mathfrak{T}_{\tau_{\mathfrak{p}_2}}
    =
    \langle\varpi_{\mathfrak{p}_2}^\alpha,1\rangle\circ \mathfrak{T}_{\tau_{\mathfrak{p}_2}}\circ\lambda_\alpha.
    \]
  The claim follows because $\nu_\alpha\circ\lambda_\alpha=\lambda_\alpha\circ\nu_\alpha$ as the determinant of $\tau_\alpha$ defining $\lambda_\alpha$ is $Mp^\alpha\in\bb{Q}^\times_+$.
\end{proof}
\noindent The diagonal embedding $\zeta:Y(K'_0(p^\alpha))\to S(K_\diamond(p^\alpha))$ naturally factors 
\[\xymatrix{
Y(K'_0(p^\alpha))\ar[r]^\zeta\ar[dr]_\zeta& S^*(K^*_\diamond(p^\alpha))\ar@{.>}[d]^\xi\\
&S(K_\diamond(p^\alpha))
}\]
through a map to a Hilbert--Blumenthal variety $\zeta:Y(K'_0(p^\alpha))\to S^*(K^*_\diamond(p^\alpha))$, denoted by the same symbol. 
\begin{lemma}\label{TwistedCycle2}
Let $Z^*_\diamond(p^\alpha):=S^*(K^*_\diamond(p^\alpha))\times X_{1,0}(N,p)$ and consider the null-homologous cycle 
\[
\Xi_\alpha^\circ:=(\mrm{id}, \varepsilon_{\msf{f}^*_\circ})_*( \lambda_\alpha\circ \zeta,  \pi_{1,\alpha})_* [Y(K'_0(p^\alpha))]\in\mrm{CH}^2(Z^*_\diamond(p^\alpha))_0(F_\alpha)\otimes_\bb{Z}\bb{Z}_p
\] then
   \[
    (\lambda_\alpha, \mrm{id})_*\Delta^\circ_\alpha= (   w_{\mathfrak{p}_2^\alpha}\circ\xi,  \mrm{id})_* \Xi_\alpha^\circ.
   \]
\end{lemma}
\begin{proof}
Using Lemma $\ref{ablemma}$ we compute
\[
\begin{split}
(\lambda_\alpha, \mrm{id})_*\Delta^\circ_\alpha
 &=
 (\lambda_{\alpha}, \mrm{id})_*(\mrm{id}, \varepsilon_{\msf{f}_\circ})_* (\langle\varpi_{\mathfrak{p}_2}^\alpha,1\rangle \circ w_{\mathfrak{p}_2^\alpha} \circ \zeta,\ \pi_{1,\alpha})_*[Y(K'_0(p^\alpha))] \\
 &=
 (\mrm{id}, \varepsilon_{\msf{f}^*_\circ})_*(\langle\varpi_{\mathfrak{p}_2}^{-\alpha},1\rangle \circ \lambda_\alpha \circ w_{\mathfrak{p}_2^\alpha}\circ \zeta,\ \pi_{1,\alpha})_*[Y(K'_0(p^\alpha))] \\
 &=
 (\mrm{id}, \varepsilon_{\msf{f}^*_\circ})_*(\langle\varpi_{\mathfrak{p}_2}^{-\alpha},1\rangle \circ \langle\varpi_{\mathfrak{p}_2}^{\alpha},1\rangle \circ   w_{\mathfrak{p}_2^\alpha} \circ \lambda_\alpha \circ \zeta,\  \pi_{1,\alpha})_*[Y(K'_0(p^\alpha))] \\
 &= 
  (\mrm{id}, \varepsilon_{\msf{f}^*_\circ})_*(   w_{\mathfrak{p}_2^\alpha}  \circ\lambda_\alpha\circ \zeta,\  \pi_{1,\alpha})_* [Y(K'_0(p^\alpha))] \\
 &= 
 ( w_{\mathfrak{p}_2^\alpha}\circ\xi,\  \mrm{id})_* \Xi_\alpha^\circ.
\end{split}
\]
\end{proof}
 
\noindent Therefore we are left to compute the right-hand side of the following equation
\begin{equation}\label{firstreduction}
\mrm{AJ}_{\mrm{syn}}\Big((\lambda_\alpha,\mrm{id})_*\Delta_\alpha^\circ\Big)\big(\omega_\mrm{P}\otimes\eta_\circ\big)= \mrm{AJ}_{\mrm{syn}}\big(\Xi^\circ_\alpha\big)\big(\omega^\diamond_\mrm{P}\otimes\eta_\circ\big)
\end{equation}
where the de Rham cohomology class 
\[
\omega^\diamond_\mrm{P}:=(w_{\frak{p}_2^\alpha}\circ\xi)^*\omega_\mrm{P}\in\mrm{Fil}^2\mrm{H}^2_\mrm{dR}(S^*(K^*_\diamond(p^\alpha))/F_\alpha)
\]
is associated with the cuspform
\begin{equation}\label{def prop cuspform}
\breve{\msf{g}}_\mrm{P}^\diamond:=(\nu_\alpha\circ\xi)^*(\breve{\msf{g}}_\mrm{P}\lvert\tau^{-1}_{\frak{p}_2^\alpha})\in S_{2t_L,t_L}\big(K^*_\diamond(p^\alpha);O\big).
\end{equation}

\subsubsection{Back to syntomic cohomology.} 
 If $R(T), Q(T)\in 1+T\cdot F_\alpha[T]$ are polynomials such that the cohomology classes
 $R(p^{-2}\Phi)\omega^\diamond_\mrm{P}$ and $Q(\Phi)\eta_\circ$ are zero, then
\[
\omega^\diamond_\mrm{P}
\in
\mrm{H}^0_{\mrm{st},R}\big(D^2(S^*(K^*_{\diamond}(p^\alpha)))(2)\big),\qquad 
\eta_\circ
\in
\mrm{H}^0_{\mrm{st},Q}\big(D^1(X_{1,0}(N,p))\big)
\]
and it is often possible to lift them to syntomic cohomology.
\begin{lemma}
	Suppose that $Q(p)\not=0$, then there exists lifts
	\[
	\tilde{\omega}^\diamond_\mrm{P} \in \mrm{H}^2_{\mrm{syn},R}(S^*(K^*_{\diamond}(p^\alpha)),2),\qquad\tilde{\eta}_\circ\in\mrm{H}^1_{\mrm{syn},Q}(X_{1,0}(N,p),0)
	\]
	of $\omega^\diamond_\mrm{P}$ and $\eta_\circ$ to syntomic cohomology.
\end{lemma}
\begin{proof}
Let $S^*(K^*_{\diamond}(p^\alpha))^c$ denote the minimal resolution of the Baily-Borel compactification  of $S^*(K^*_{\diamond}(p^\alpha))$. As the class $\omega^\diamond_\mrm{P}$ extends to the smooth compactification (\cite{Geer}, Chapter III, Proposition 3.7), it suffices to show $\omega^\diamond_\mrm{P}$ can be lifted to the syntomic cohomology of $S^*(K^*_{\diamond}(p^\alpha))^c$. The algebraic surface $S^*(K^*_{\diamond}(p^\alpha))^c$ is simply connected, hence $\mrm{H}^2_{\mrm{st},R}\big(D^1(S^*(K^*_{\diamond}(p^\alpha))^c)(2)\big)=0$ and the descent spectral sequence (\cite{BLZ} Theorem 2.1.2) produces a surjection
\[\xymatrix{
\mrm{H}^2_{\mrm{syn},R}(S^*(K^*_{\diamond}(p^\alpha))^c,2)
\ar@{->>}[r]&
\mrm{H}^0_{\mrm{st},R}\big(D^2(S^*(K^*_{\diamond}(p^\alpha))^c)(2)\big)
}\]
which proves the first claim. In the modular curve case, we compute that
\[
\mrm{H}^2_{\mrm{st},Q}\big(D^0(X_{1,0}(N,p))\big)=\mrm{H}^2_{\mrm{st},Q}\big(\mbf{D}_{\mrm{pst}}(\bb{Q}_p)\big)=F_\alpha/Q(p)F_\alpha
\]
is zero if $Q(p)\not=0$. Hence,  there is a surjection
$\mrm{H}^1_{\mrm{syn},Q}(X_{1,0}(N,p),0)\twoheadrightarrow \mrm{H}^0_{\mrm{st},Q}\big(D^1(X_{1,0}(N,p))\big)$ whenever $Q(p)$ is non-zero.
\end{proof}

\noindent Out of $R(T)$ and $Q(T)$ we can form the polynomial $(R\star Q)(T)\in 1+T\cdot F_\alpha[T]$ whose roots are all the products of a root of $R(T)$ with a root of $Q(T)$.
It follows from equations (\ref{evaluation}) and (\ref{firstreduction}) that if $(R\star Q)(1)\not=0$ and $(R\star Q)(p^{-1})\not=0$ we can evaluate the syntomic Abel--Jacobi map as
\[
\mrm{AJ}_{\mrm{syn}}\Big((\lambda_\alpha,\mrm{id})_*\Delta_\alpha^\circ\Big)\big(\omega_\mrm{P}\otimes\eta_\circ\big)= \mrm{tr}_{Z^*_\diamond(p^\alpha),R\star Q}\Big(\mrm{cl}_\mrm{syn}\big(\Xi_\alpha^\circ\big)\cup\big(\tilde{\omega}^\diamond_\mrm{P}\otimes\tilde{\eta}_\circ\big)\Big).
\]
Moreover, the projection formula (\cite{BLZ}, Theorem 2.5.3) computes the syntomic trace for $Z_\diamond(p^\alpha)$ as a syntomic trace for the curve $Y_\alpha=Y(K'_\diamond(p^\alpha))$
\begin{equation}\label{AJ101}
\mrm{AJ}_{\mrm{syn}}\Big((\lambda_\alpha,\mrm{id})_*\Delta_\alpha^\circ\Big)\big(\omega_\mrm{P}\otimes\eta_\circ\big)=\Big\langle (\zeta\circ \lambda_\alpha)^*\tilde{\omega}^\diamond_\mrm{P},\ (\pi_{1,\alpha})^*\tilde{\eta}_\circ\Big\rangle_{Y_\alpha,R\star Q}
	\end{equation}
where we used the equality $(\varepsilon_{\msf{f}_\circ})^*\eta_\circ=\eta_\circ$.

\subsubsection{Explicit cup product formulas.}\label{cupproduct}
To continue the computation of syntomic regulators it is convenient to make the choice of lifts more explicit. According to (\cite{BLZ} Section 2.4) any lift $\tilde{\omega}^\diamond_{\mrm{P}}\in \mrm{H}^2_{\mrm{syn},R}(S^*(K^*_\diamond(p^\alpha)),2)$ of $\omega_\mrm{P}^\diamond$ can be described by a tuple $[u,v;w,x,y;z]$ where   
\begin{center}\begin{tabular}{lllll}
$u\in\bb{R}\Gamma^{B,2}_\mrm{HK}(S^*(K^*_\diamond(p^\alpha)))$, &&$v\in \mrm{Fil}^2\hspace{1mm}\bb{R}\Gamma^{2}_\mrm{dR}(S^*(K^*_\diamond(p^\alpha)))$, &&$z\in\bb{R}\Gamma^{B,0}_\mrm{HK}(S^*(K^*_\diamond(p^\alpha)))$,\\
\\
$w,x\in \bb{R}\Gamma^{B,1}_\mrm{HK}(S^*(K^*_\diamond(p^\alpha)))$,&&$y\in \bb{R}\Gamma^{1}_\mrm{dR}(S^*(K^*_\diamond(p^\alpha)))$,&&\\
\end{tabular}\end{center}
satisfy the relations
\begin{center}
    \begin{tabular}{lllll}
$du=0$,& &$dv=0$,& &$dw=R(p^{-2}\Phi)u$,\\
$dx=Nu$, &&$dy=\iota_\mrm{dR}^B(u)-v$, &&$dz=Nw-R(p^{-1}\Phi)x$.
\end{tabular}
	\end{center}
We choose a lift $\tilde{\omega}^\diamond_{\mrm{P}}$ such that $\iota^B_\mrm{dR}(u)=v=\omega^\diamond_\mrm{P}$ so that we can assume $y=0$. Then, we obtain
\[
(\zeta\circ \lambda_\alpha)^*\tilde{\omega}^\diamond_\mrm{P}=\big[0,0;(\zeta\circ \lambda_\alpha)^*w,(\zeta\circ \lambda_\alpha)^*x,0;(\zeta\circ \lambda_\alpha)^*z\big]
\]
for dimension reasons. Similarly, any lift $\tilde{\eta}_\circ\in \mrm{H}^1_{\mrm{syn},Q}(X_{1,0}(N,p),0)$ of $\eta_\circ$ can be described by a tuple $[u',v';w',x',y';0]$ where 
\begin{center}\begin{tabular}{lll}
$u'\in\bb{R}\Gamma^{B,1}_\mrm{HK}(X_{1,0}(N,p))$, &&$v'\in \bb{R}\Gamma^{1}_\mrm{dR}(X_{1,0}(N,p))$,\\
\\
$w',x'\in \bb{R}\Gamma^{B,0}_\mrm{HK}(X_{1,0}(N,p))$,&&$y'\in \bb{R}\Gamma^{0}_\mrm{dR}(X_{1,0}(N,p))$
\end{tabular}\end{center}
satisfy the relations
\begin{center}
    \begin{tabular}{lllll}
$du'=0$,& &$dv'=0$,& &$dw'=Q(\Phi)u'$,\\
$dx'=Nu'$, &&$dy'=\iota_\mrm{dR}^B(u')-v'$. &&
\end{tabular}
	\end{center}
As one might expect, another tuple $[u'',v'';w'',x'',y'';z'']$ represents the cup product
\[ 
(\zeta\circ \lambda_\alpha)^*\tilde{\omega}^\diamond_\mrm{P}\cup (\pi_{1,\alpha})^*\tilde{\eta}_\circ\in \mrm{H}^3_{\mrm{syn},R\star Q}(Y_\alpha,2)\]
which can be described explicitly (\cite{BLZ}, Proposition 2.4.1). For our computation, the relevant entries are $w''$ and $y''$ (see \cite{BLZ}, Equation (2)).
Given polynomials  $a(T_1,T_2), b(T_1,T_2)$ with
	\[
	(R\star Q)(T_1T_2)=a(T_1,T_2)R(T_1)+b(T_1,T_2)Q(T_2),
	\]
we then compute that
		 \[
w''=a(p^{-2}\Phi,\Phi)((\zeta\circ \lambda_\alpha)^*w\cup (\pi_{1,\alpha})^*u')\qquad\&\qquad y''=0.
\]
 Therefore, combining ($\ref{AJ101}$) with the definition of the syntomic trace map (\cite{BLZ}, Definition 3.1.2) we obtain
\begin{equation}\label{AJformula1}
    \begin{split}
	\mrm{AJ}_{\mrm{syn}}\Big((\lambda_\alpha,\mrm{id})_*\Delta_\alpha^\circ\Big)\big(\omega_\mrm{P}\otimes\eta_\circ\big)&= \mrm{tr}_{Y_\alpha,R\star Q}\Big( (\zeta\circ \lambda_\alpha)^*\tilde{\omega}^\diamond_\mrm{P}\cup (\pi_{1,\alpha})^*\tilde{\eta}_\circ\Big)\\ 
		&=-\iota_\mrm{dR}^B\big((R\star Q)(\Phi)^{-1}w''\big)
	\\
		&=-\frac{a(p^{-2}\Phi,\alpha_{\msf{f}_\circ^*})}{(R\star Q)(p^{-1})} \cdot \big[( \lambda_\alpha)^*\zeta^*[\iota_\mrm{dR}^B(w)]\cup_\mrm{dR} (\pi_{1,\alpha})^*\eta_\circ\big]
    \end{split}
\end{equation}
where the last equality follows the facts that $\eta_\circ$ is a eigenvector for $\Phi$ of eigenvalue $\alpha_{\msf{f}_\circ^*}$ and the Frobenius endomorphism of $\mbf{D}_\mrm{pst}(\bb{Q}_p(1))_{\mrm{st},F_\alpha}$ is multiplication by $p^{-1}$.

\subsection{Relation to $p$-adic modular forms}\label{section AJ p-adic}
 The modular curve $X_{1,0}(N,p)$ admits a proper regular model $\scr{X}_{1,0}(N,p)_{/\bb{Z}_p}$ whose special fiber is the union of two curves, each isomorphic to the special fiber of $X(V_{1}(N))$. One writes $\{\scr{W}_\infty, \scr{W}_0\}$ for the standard admissible covering of $X_{1,0}(N,p)=\scr{X}_{1,0}(N,p)^\mrm{an}$ by wide open neighborhoods obtained as the inverse image under the specialization map of the two distinguished curves in the special fiber. The two opens are interchanged by the involution $\lambda_1:X_{1,0}(N,p)\to X_{1,0}(N,p)$ defined over $\bb{Q}_p(\zeta_p)$. 
 The de
Rham cohomology group $\mrm{H}^1_\mrm{dR}(X_{1,0}(N,p)/\bb{Q}_p)$ is endowed with an action of a Frobenius map $\Phi$ commuting with the $U_p$ operator and the ordinary unit root subspace
\[
\mrm{H}^1_\mrm{dR}(X_{1,0}(N,p)/\bb{Q}_p)^{\mrm{ord},\mrm{ur}}\subseteq e_\mrm{ord}\mrm{H}^1_\mrm{dR}(X_{1,0}(N,p)/\bb{Q}_p)
\]
is spanned by the eigenvectors of $\Phi$ whose eigenvalue is a $p$-adic unit. As in (\cite{DR2}, Lemma 4.2) the natural map induced by restriction
\[
\mrm{res}_{\scr{W}_\infty}:\mrm{H}^1_\mrm{dR}(X_{1,0}(N,p)/\bb{Q}_p)^{\mrm{ord},\mrm{ur}}\overset{0}{\longrightarrow} \mrm{H}^1_\mrm{rig}(\scr{W}_\infty)
\]  is trivial, while 
\[
\mrm{res}_{\scr{W}_0}:\mrm{H}^1_\mrm{dR}(X_{1,0}(N,p)/\bb{Q}_p)^{\mrm{ord},\mrm{ur}}[\phi]
\overset{\sim}{\longrightarrow} \mrm{H}^1_\mrm{rig}(\scr{W}_0)[\phi]
\]
is an isomorphism for any eigenform $\phi\in S_{2,1}(V_{1,0}(N,p);\overline{\bb{Q}})$.

\smallskip
 \noindent  For any $\alpha\ge2$  let $\scr{W}_\infty(p^\alpha)$ be the open subset $(\pi_{1,\alpha})^{-1}(\scr{W}_\infty)$ of $X(K'_0(p^\alpha))$. 
\begin{lemma}
If $\omega\in \mrm{H}^1_\mrm{dR}(X(K'_0(p^\alpha))/\bb{Q}_p)$, then the de Rham pairing 
\[
\big\langle (\lambda_\alpha)^*\omega,(\pi_{1,\alpha})^*\eta_\circ\big\rangle_{\mrm{dR}}=\big\langle e_\mrm{ord}\omega,(\pi_{2,\alpha})^*(\lambda_1)^*\eta_\circ\big\rangle_\mrm{dR}
\]
 depends only on the overconvergent cuspform associated to $\mrm{res}_{\scr{W}_\infty(p^\alpha)}\big(e_{\mrm{ord}}\omega\big)$.
\end{lemma}
\begin{proof}
Since the class $\eta_\circ$ ordinary and $e_\mrm{ord}\circ(\pi_{1,\alpha})^*=(\pi_{1,\alpha})^*\circ e_\mrm{ord}$ we can compute
	\[\begin{split}
	\big\langle (\lambda_\alpha)^*\omega,(\pi_{1,\alpha})^*\eta_\circ\big\rangle_\mrm{dR}
	&=
	\big\langle e^*_\mrm{ord}(\lambda_\alpha)^*\omega,(\pi_{1,\alpha})^*\eta_\circ\big\rangle_\mrm{dR}\\
	&=
	\big\langle (\lambda_\alpha)^*e_\mrm{ord}\omega,(\pi_{1,\alpha})^*\eta_\circ\big\rangle_\mrm{dR}\\
\text{as}\quad\lambda_\alpha^{-1}=\lambda_\alpha\circ\langle-1,1\rangle\qquad	&=
	\big\langle e_\mrm{ord}\omega,(\pi_{1,\alpha}\circ\lambda_\alpha)^*\eta_\circ\big\rangle_\mrm{dR}\\
	&=
	\big\langle e_\mrm{ord}\omega,(\lambda_1\circ\pi_{2,\alpha})^*\eta_\circ\big\rangle_\mrm{dR}.\\
	\end{split}\]
	Then, the pairing depends only on  the overconvergent cuspform associated to $\mrm{res}_{\scr{W}_\infty(p^\alpha)}\big(e_{\mrm{ord}}\omega\big)$  because the class $\eta_\circ$ is supported on the wide open $\scr{W}_0$,  the involution $\lambda_1:X_{1,0}(N,p)\to X_{1,0}(N,p)$ interchanges $\scr{W}_0$ with $\scr{W}_\infty$, and because of the explicit description of the Poincar\'e pairing given in (\cite{DR2}, Equation (109)).
\end{proof}

\noindent We are left to describe the restriction of $e_\mrm{ord}\zeta^*[\iota_\mrm{dR}^B(w)]$ to  $\scr{W}_\infty(p^\alpha)$ in terms of $p$-adic cuspforms.

\begin{definition}
Let $S^*_\diamond(p^\alpha)$ denote  the $\bb{Q}_p$-scheme $S^{*}(K_\diamond^*(p^\alpha))$ with its minimal compactification   $\overline{S}_{\diamond}^*(p^\alpha)^{\mbox{\tiny $\mrm{min}$}}$. For a choice  of toroidal compactification $\overline{S}_{\diamond}^{*}(p^\alpha)^{\mbox{\tiny $\mrm{tor}$}}$ we write $\mbox{\small $D$}$ for the boundary divisor $\overline{S}_{\diamond}^{*}(p^\alpha)^{\mbox{\tiny $\mrm{tor}$}}\setminus S^*_\diamond(p^\alpha)$ and we will use the same symbols to denote the associated rigid analytic spaces.
\end{definition} 

\noindent Following (\cite{Kisin-Lai}, Section 3.2.2) we let
\begin{equation}
j_{\mbox{\tiny $\mrm{tor}$}}: \overline{S}_{\diamond}^{*}(p^\alpha)^{\mbox{\tiny $\mrm{tor}$}}_{\mbox{\tiny $\mrm{ord}$}}\hookrightarrow \overline{S}_{\diamond}^{*}(p^\alpha)^{\mbox{\tiny $\mrm{tor}$}},\qquad 
j_{\mbox{\tiny $\mrm{min}$}}: \overline{S}_{\diamond}^{*}(p^\alpha)^{\mbox{\tiny $\mrm{min}$}}_{\mbox{\tiny $\mrm{ord}$}}\hookrightarrow \overline{S}_{\diamond}^{*}(p^\alpha)^{\mbox{\tiny $\mrm{min}$}}
\end{equation}
denote the open immersions of the rigid analytic subspaces parametrizing ordinary abelian surfaces and the unramified cusps. By (\cite{Kisin-Lai}, Equation 3.2.4) the pullback of a fundamental system of strict neighborhoods of the ordinary loci at level $p^0$ provides a fundamental system of strict neighborhoods at level $p^\alpha$. In particular,  $\overline{S}_{\diamond}^{*}(p^\alpha)^{\mbox{\tiny $\mrm{min}$}}_{\mbox{\tiny $\mrm{ord}$}}$ has a fundamental system of strict neighborhoods consisting of affinoid subdomains.

\noindent 	The overconvergent cuspform associated to the restriction of $e_\mrm{ord}\zeta^*[\iota_\mrm{dR}^B(w)]$ to  $\scr{W}_\infty(p^\alpha)$ depends only on the image of $\iota_\mrm{dR}^B(w)$ in the rigid complex of $\overline{S}_{\diamond}^{*}(p^\alpha)^{\mbox{\tiny $\mrm{tor}$}}_{\mbox{\tiny $\mrm{ord}$}}$ because the modular interpretation of $\zeta: X(K_0'(p^\alpha))\to S_\diamond^*(p^\alpha)^{\mbox{\tiny $\mrm{tor}$}}$ ensures that the inverse image of neighborhoods of the ordinary locus of $\overline{S}_{\diamond}^{*}(p^\alpha)^{\mbox{\tiny $\mrm{tor}$}}_{\mbox{\tiny $\mrm{ord}$}}$ are neighborhoods of the ordinary locus of $\scr{W}_\infty(p^\alpha)$. Furthermore, as in characteristic zero there exists a Hecke equivariant projection from the de Rham complex to the dual BGG complex which is a quasi-isomorphism of filtered complexes (\cite{BGG}, Sections 5.3-5.4), we are reduced to describing the projection of $\iota_\mrm{dR}^B(w)$  using overconvergent cuspforms for the group $G^*$ (defined as in \cite{Hilbert}, Section 3.3).
Concretely, the relevant complex is given by
	\[\xymatrix{
	S^\dagger_{0,0}(K^*_\diamond(p^\alpha)) \ar[rr]^-{(d_1,d_2)}&&
	 S^\dagger_{(2,0),(1,0)}(K^*_\diamond(p^\alpha))\oplus S^\dagger_{(0,2),(0,1)}(K^*_\diamond(p^\alpha)) \ar[rr]^-{-d_2\oplus d_1}&&
	S^\dagger_{2t_L,t_L}(K^*_\diamond(p^\alpha)),
	}\]
it computes the rigid cohomology groups (see also \cite{Hilbert}, Theorem 3.5)
	\[
	\mrm{H}_{\mrm{rig},c}^{\bfcdot}\Big(\overline{S}_{\diamond}^{*}(p^\alpha)^{\mbox{\tiny $\mrm{tor}$}}_{\mbox{\tiny $\mrm{ord}$}}\Big)
	:=\bb{H}^{\bfcdot}\Big(\overline{S}_{\diamond}^{*}(p^\alpha)^{\mbox{\tiny $\mrm{tor}$}}, (j_\mrm{tor})^\dagger\ \Omega^{\star}_{\overline{S}_{\diamond}^{*}(p^\alpha)^{\mbox{\tiny $\mrm{tor}$}}}(-\mbox{\small $D$})\Big),
	\]
and the projection of $\omega_\mrm{P}^\diamond$ in $S^\dagger_{2t_L,t_L}(K^*_\diamond(p^\alpha))$ is the cuspform $\breve{\msf{g}}_\mrm{P}^\diamond$ defined in \eqref{def prop cuspform}.

\begin{lemma}\label{lemma: representative} 
The image of $\iota_\mrm{dR}^B(w)$ in $S^\dagger_{(2,0),(1,0)}(K^*_\diamond(p^\alpha))\oplus S^\dagger_{(0,2),(0,1)}(K^*_\diamond(p^\alpha))$ is given by a pair $(H_1,H_2)$ of overconvergent cuspforms satisfying 
\begin{equation}\label{star relation}
d_1H_2-d_2H_1=R(V(p))\breve{\msf{g}}_\mrm{P}^\diamond.
\end{equation}
\end{lemma} 
\begin{proof}
Recall that the operator $R(p^{-2}\Phi)$ was required to annihilate the cohomology class $\omega_\mrm{P}^\diamond$. As the action of $p^{-2}\Phi$ on de Rham cohomology corresponds to the action of the $V(p)$--operator on overconvergent cuspforms, we deduce the triviality of the class $R(V(p))\breve{\msf{g}}_\mrm{P}^\diamond$ in $\mrm{H}_{\mrm{rig},c}^{2}\big(\overline{S}_{\diamond}^{*}(p^\alpha)^{\mbox{\tiny $\mrm{tor}$}}_{\mbox{\tiny $\mrm{ord}$}}\big)$.
\end{proof}


\subsubsection{Choice of the polynomial $R$.}
Recall that the cuspform $\breve{\msf{g}}_\mrm{P}\in S_{2t_L,t_L}(K_{\diamond,t}(p^\alpha);O)$ is an eigenform for the Hecke operators $U_{\mathfrak{p}_1},U_{\mathfrak{p}_2}^*$ for the two primes above $p$ with eigenvalues  
\[
U_{\mathfrak{p}_1}\breve{\msf{g}}_\mrm{P}=\alpha_{1,\msf{g}_\mrm{P}}\cdot\breve{\msf{g}}_\mrm{P}\qquad \text{and}\qquad U_{\mathfrak{p}_2}^*\breve{\msf{g}}_\mrm{P}= \overline{\alpha}_{2,\msf{g}_\mrm{P}}\cdot\breve{\msf{g}}_\mrm{P}=\alpha_{2,\msf{g}_\mrm{P}}^{-1}p\cdot \breve{\msf{g}}_\mrm{P}.
\]
Thus, the cuspform $(w_{\frak{p}_2^\alpha})^*\breve{\msf{g}}_\mrm{P}\in S_{2t_L,t_L}(K_\diamond(p^\alpha);O)$ is an eigenform for the $U_p$-operator and equation (\ref{U_pAtkinLehner}) allows us to compute the Hecke eigenvalue by
\[\begin{split}
U_p \cdot (w_{\frak{p}_2^\alpha})^*\breve{\msf{g}}_\mrm{P}
&=
(w_{\mathfrak{p}_2^\alpha})^* U_{\mathfrak{p}_1}U_{\mathfrak{p}_2}^* \langle \varpi_{\mathfrak{p}_2},1 \rangle \breve{\msf{g}}_\mrm{P}\\
&=
\Big(\chi_\circ\theta_L^{-1}\chi^{-1}(\varpi_{\mathfrak{p}_2})\cdot\alpha_{1,\msf{g}_\mrm{P}}\alpha_{2,\msf{g}_\mrm{P}}^{-1}p\Big)
\cdot(w_{\mathfrak{p}_2^\alpha})^*   \breve{\msf{g}}_\mrm{P}\\
&= \chi_\circ\theta_L^{-1}\chi^{-1}(\varpi_{\mathfrak{p}_2})\cdot\alpha_{1,\msf{g}_\mrm{P}}\alpha_{2,\msf{g}_\mrm{P}}^{-1}p\cdot (w_{\frak{p}_2^\alpha})^*\breve{\msf{g}}_\mrm{P}.
\end{split}\]
We suppose from now on that $\mathfrak{p}_1$ is narrowly principal in $\cal{O}_L$, $\frak{p}_1=(p_1)\cal{O}_L$, for $p_1\in\cal{O}_L$ a totally positive generator and we
set $p_2 := p/p_1 \in \cal{O}_L$. Then, there are equalities of Hecke operators 
    \[
    U(p) =  \langle 1,p^{-1}\varpi_{p}\rangle U_{p},\qquad
    U(p_i) = \langle 1,p_i^{-1}\varpi_{\mathfrak{p}_i}\rangle U_{\mathfrak{p}_i}\qquad \text{for}\ i=1,2
    \]
  Moreover, $U(p_1), U(p_2)$ commute with $(\nu_\alpha)^*$ as one can see  arguing as in Lemma \ref{U_p-nu_alpha-commute}.
    \begin{lemma}
    The cuspform $\breve{\msf{g}}^\diamond_\mrm{P}$ is an eigenform for the Hecke operators $U(p_1)$ and $U(p_2)$:
    \[
    U(p_1)\cdot \breve{\msf{g}}^\diamond_\mrm{P}
    =
    \alpha^\diamond_1 \cdot \breve{\msf{g}}^\diamond_\mrm{P},\qquad 
     U(p_2)\cdot \breve{\msf{g}}^\diamond_\mrm{P}
    =
    \alpha^\diamond_2 \cdot \breve{\msf{g}}^\diamond_\mrm{P}
    \]
     where 
     \[ \alpha^\diamond_1:=\chi_\circ\theta_L^{-1}\chi^{-1}((p_1)_{\mathfrak{p}_2}^{-1})\cdot\alpha_{1,\msf{g}_\mrm{P}}
     \qquad \text{and}\qquad  \alpha_2^\diamond:=\chi_\circ\theta_L^{-1}\chi^{-1}((p_2)_{\mathfrak{p}_2}^{-1}\varpi_{\frak{p}_2})\cdot\alpha_{2,\msf{g}_\mrm{P}}^{-1}p.
     \]
    \end{lemma}
    \begin{proof}
    By definition the matrices $\mbox{\tiny $\begin{pmatrix}p_i&0\\0&1\end{pmatrix}$}$ for $i=1,2$ belong to
    $G(\bb{Q})^+G^*(\bb{A}_{f})$, hence the operators $U(p_1)$ and $U(p_2)$ commute with the natural map $\xi:S^*(K_\diamond^*(p^\alpha))\rightarrow S(K_\diamond(p^\alpha))$ (\cite{LLZ}, Definition 2.2.4). First, we observe that 
     \[\begin{split}
    U(p_1)\cdot \breve{\msf{g}}^\diamond_\mrm{P}
    &=
    \xi^*\circ  U(p_1)\circ(w_{\frak{p}_2^\alpha})^*\breve{\msf{g}}_\mrm{P}\\
    &=
    \xi^*\circ(\nu_\alpha)^*\circ U(p_1)\circ(\mathfrak{T}_{\tau_{\mathfrak{p}_2}})^*\breve{\msf{g}}_\mrm{P}.
    \end{split}\]
    As
    \[\begin{split}
    U(p_1)\circ(\mathfrak{T}_{\tau_{\mathfrak{p}_2}})^*
    &=
    U_{\mathfrak{p}_1}\circ  \langle 1,p_1^{-1}\varpi_{\mathfrak{p}_1}\rangle\circ (\mathfrak{T}_{\tau_{\mathfrak{p}_2}})^*\\
    (\text{Equation}\ (\ref{AL-diamonds}))\qquad&=
    U_{\mathfrak{p}_1}\circ (\mathfrak{T}_{\tau_{\mathfrak{p}_2}})^*\circ
    \langle (p_1)_{\mathfrak{p}_2}^{-1},p_1^{-1}\varpi_{\mathfrak{p}_1}\cdot(p_1)_{\mathfrak{p}_2}^{2}\rangle\\
    &=
    (\mathfrak{T}_{\tau_{\mathfrak{p}_2}})^*\circ U_{\mathfrak{p}_1}\circ \langle (p_1)_{\mathfrak{p}_2}^{-1},p_1^{-1}\varpi_{\mathfrak{p}_1}\cdot(p_1)_{\mathfrak{p}_2}^2\rangle,
    \end{split}\]
     we obtain
    
    \[\begin{split}
     U(p_1)\cdot \breve{\msf{g}}^\diamond_\mrm{P}
     &=
    \xi^*\circ(\nu_\alpha)^*\circ(\mathfrak{T}_{\tau_{\mathfrak{p}_2}})^*\circ U_{\mathfrak{p}_1}\circ \langle (p_1)_{\mathfrak{p}_2}^{-1},\varpi_{\mathfrak{p}_1}\cdot(p_1)_{\mathfrak{p}_2}^{2}\rangle \breve{\msf{g}}_\mrm{P}\\
    &=
    \chi_\circ\theta_L^{-1}\chi^{-1}((p_1)_{\mathfrak{p}_2}^{-1})\cdot\alpha_{1,\msf{g}_\mrm{P}}\cdot \breve{\msf{g}}^\diamond_\mrm{P}.
    \end{split}\]
 Similarly one computes
    \[
    U(p_2)\cdot \breve{\msf{g}}^\diamond_\mrm{P}
    = \chi_\circ\theta_L^{-1}\chi^{-1}((p_2)_{\mathfrak{p}_2}^{-1}\varpi_{\frak{p}_2})\cdot\alpha_{2,\msf{g}_\mrm{P}}^{-1}p\cdot \breve{\msf{g}}^\diamond_\mrm{P}. 
    \]
    \end{proof}
  \begin{remark}
  	Note that 
	\begin{equation}
	\alpha_1^\diamond\alpha_2^\diamond=\alpha_{1,\msf{g}_\mrm{P}}\alpha_{2,\msf{g}_\mrm{P}}^{-1}p.
	\end{equation}
  \end{remark}

\begin{definition}
	We set 
	\[
	R(T):=(1-\alpha^\diamond_1\alpha^\diamond_2T).
	\]
\end{definition}

\begin{lemma}\label{lemma: R is ok}
	If $R(T)=(1-\alpha^\diamond_1\alpha^\diamond_2T)$, then in $\mrm{H}_{\mrm{rig},c}^{2}\big(\overline{S}_{\diamond}^{*}(p^\alpha)^{\mbox{\tiny $\mrm{tor}$}}_{\mbox{\tiny $\mrm{ord}$}}\big)$ we have
	\[
	[R(V(p))\breve{\msf{g}}_\mrm{P}^\diamond]=0.
	\]
\end{lemma}
\begin{proof}
Consider the operators
    \[
    V(p_1)
    :=
    \langle 1,p_1\varpi_{\mathfrak{p}_1}^{-1}\rangle V_{\mathfrak{p}_1}
    \qquad \mrm{and} \qquad
    V(p_2)
    :=
    \langle 1,p_2^{-1}\varpi_{\mathfrak{p}_2}\rangle V_{\mathfrak{p}_2},
    \]
    acting on automorphic forms for $G^*$. They satisfy $V(p)=V(p_1)V(p_2)$ and  $U(p_i)V(p_i) = 1$ for $i=1,2$. The polynomials
\[
R_i(T_i)=(1-\alpha^\diamond_iT_i)\quad \text{for}\quad i=1,2
\]
can be used to write 
\[
R(T_1T_2)=R_1(T_1)R_2(T_2)+\alpha^\diamond_2T_2R_1(T_1)+ \alpha^\diamond_1T_1R_2(T_2).
\]
Therefore, the cuspform $R(V(p))\breve{\msf{g}}_\mrm{P}^\diamond$ can be expressed as a sum of depleted cuspforms
\[
R(V(p))\breve{\msf{g}}_\mrm{P}^\diamond=(\breve{\msf{g}}_\mrm{P}^\diamond)^{\mbox{\tiny$[\cal{P}]$}}+\alpha^\diamond_2V(p_2)(\breve{\msf{g}}_\mrm{P}^\diamond)^{\mbox{\tiny$[\mathfrak{p}_1]$}}+\alpha^\diamond_1V(p_1)(\breve{\msf{g}}_\mrm{P}^\diamond)^{\mbox{\tiny$[\mathfrak{p}_2]$}}.
\]
We deduce that $U(p)\cdot R(V(p))\breve{\msf{g}}_\mrm{P}^\diamond=0$. This implies the claim because $U(p)$ acts invertibly on the finite dimensional cohomology group $\mrm{H}_{\mrm{rig},c}^{2}\big(\overline{S}_{\diamond}^{*}(p^\alpha)^{\mbox{\tiny $\mrm{tor}$}}_{\mbox{\tiny $\mrm{ord}$}}\big)$ having a right inverse (the $V(p)$--operator).
\end{proof}

\subsection{The formula} 

	We choose the polynomial $Q(T)=(1-\alpha_{\msf{f}_\circ^*}^{-1}T)$ to annihilate the class $\eta_\circ$ because by definition $\Phi(\eta_\circ)=\alpha_{\msf{f}_\circ^*}\cdot\eta_\circ$. Clearly $Q(p)\not=0$, and $R\star Q(1)\not=0$, $R\star Q(p^{-1})\not=0$ because the Weil conjectures imply that the roots of $R(T)$ have complex absolute values $p^{-1}$. Observe that
	\begin{equation}
	(R\star Q)(T)= R(\alpha_{\msf{f}_\circ^*}^{-1}T) =(1-\alpha_{1,\msf{g}_\mrm{P}}\alpha_{2,\msf{g}_\mrm{P}}^{-1} \alpha_{\msf{f}_\circ^*}^{-1}p T).
	\end{equation}
	 and if we write
	\[
	(R\star Q)(T_1T_2) = 
	a(T_1,T_2)R(T_1)
	+b(T_1,T_2)Q(T_2),
	\]
	as in Section $\ref{cupproduct}$, then $a(T_1,\alpha_{\msf{f}_\circ^*})=1$ because
	\[
	R(T_1)=(R\star Q)(T_1\cdot\alpha_{\msf{f}_\circ^*})=a(T_1,\alpha_{\msf{f}_\circ^*})R(T_1).
	\]
Therefore, equation ($\ref{AJformula1}$) simplifies to the following expression
\begin{equation}\label{secondreduction}
\mrm{AJ}_{\mrm{syn}}\Big((\lambda_\alpha,\mrm{id})_*\Delta_\alpha^\circ\Big)\big(\omega_\mrm{P}\otimes\eta_\circ\big)=\frac{-1}{(R\star Q)(p^{-1})} \Big\langle e_\mrm{ord}\zeta^*[\iota_\mrm{dR}^B(w)], (\pi_{2,\alpha})^*(\lambda_1)^*\eta_\circ\Big\rangle_{\mrm{dR}, Y_\alpha}.
   \end{equation}
By Lemma \ref{lemma: representative} and the straightforward computation $e_\mrm{ord}\zeta^*(d_1f, d_2f)=e_\mrm{ord}d\zeta^*f=0$ for any overconvergent function $f\in S^\dagger_{0,0}(K^*_\diamond(p^\alpha))$, we obtain
\begin{equation}\label{express in classical forms}
e_\mrm{ord}\zeta^*[\iota_\mrm{dR}^B(w)]
=\omega_{e_\mrm{ord}\zeta^*\big(H_1, H_2\big)}.
\end{equation}

\subsubsection{Replacing overconvergent cuspforms with $p$-adic ones.}
Recall the expression 
\[
R(V(p))\breve{\msf{g}}_\mrm{P}^\diamond=(\breve{\msf{g}}_\mrm{P}^\diamond)^{\mbox{\tiny$[\cal{P}]$}}+\alpha^\diamond_2V(p_2)(\breve{\msf{g}}_\mrm{P}^\diamond)^{\mbox{\tiny$[\mathfrak{p}_1]$}}+\alpha^\diamond_1V(p_1)(\breve{\msf{g}}_\mrm{P}^\diamond)^{\mbox{\tiny$[\mathfrak{p}_2]$}}
\]
obtained in the proof of Lemma \ref{lemma: R is ok}. Each term of the right hand side is trivial in the cohomology group $\mrm{H}_{\mrm{rig},c}^{2}\big(\overline{S}_{\diamond}^{*}(p^\alpha)^{\mbox{\tiny $\mrm{tor}$}}_{\mbox{\tiny $\mrm{ord}$}}\big)$, thus there exist three pairs of overconvergent cuspforms
\[
(A_1,A_2),\ (B_1,B_2),\ (C_1,C_2)\ \in S^\dagger_{(2,0),(1,0)}(K^*_\diamond(p^\alpha))\oplus S^\dagger_{(0,2),(0,1)}(K^*_\diamond(p^\alpha))
\] 
satisfying the relations
\[
d_1A_2-d_2A_1=(\breve{\msf{g}}_\mrm{P}^\diamond)^{\mbox{\tiny$[\cal{P}]$}},\quad
d_1B_2-d_2B_1=\alpha^\diamond_2V(p_2)(\breve{\msf{g}}_\mrm{P}^\diamond)^{\mbox{\tiny$[\mathfrak{p}_1]$}},\quad
d_1C_2-d_2C_1=\alpha^\diamond_1V(p_1)(\breve{\msf{g}}_\mrm{P}^\diamond)^{\mbox{\tiny$[\mathfrak{p}_2]$}}.
\]
We can further assume that $(A_1,A_2)$ consists of $\cal{P}$-depleted forms, $(B_1,B_2)$ consists of $\mathfrak{p}_1$-depleted forms, and $(C_1,C_2)$ consists of $\mathfrak{p}_2$-depleted forms because depletions operators are idempotents commuting with differential operators. Then, without loss of generality, we can suppose that $(H_1,H_2)=(A_1+B_1+C_1, A_2+B_2+C_2)$. Finally, the next lemma shows that in evaluating $e_\mrm{ord}\zeta^*[\iota_\mrm{dR}^B(w)]$, we can replace $\big(H_1, H_2\big)$  with the pair  of $p$-adic cuspforms
\[
\Big(-d_2^{-1}(\breve{\msf{g}}^\diamond_\mrm{P})^{\mbox{\tiny $[\cal{P}]$}}- \alpha_1^\diamond V(p_1)d_2^{-1}(\breve{\msf{g}}^\diamond_\mrm{P})^{\mbox{\tiny $[\frak{p}_2]$}},\quad \alpha_2^\diamond V(p_2) d_1^{-1}(\breve{\msf{g}}^\diamond_\mrm{P})^{\mbox{\tiny $[\frak{p}_1]$}}\Big).
\]

\begin{lemma}\label{PrimitiveVanishing}
Suppose $(\msf{g}_1,\msf{g}_2)\in 
S_{(2,0),(1,0)}^\mrm{p\mbox{-}adic}(K^*_\diamond(p^\alpha))\oplus S_{(0,2),(0,1)}^\mrm{p\mbox{-}adic}(K^*_\diamond(p^\alpha))$
such that 
\[
d_1 \msf{g}_2 -d_2 \msf{g}_1 =0,
\]
and that either $\msf{g}_1$ is $\mathfrak{p}_1$-depleted or $\msf{g}_2$ is $\mathfrak{p}_2$-depleted.
Then \[e_\mrm{ord}\zeta^*(\msf{g}_1, \msf{g}_2)=0.\]
\end{lemma}
\begin{proof}
Since this is a statement involving only sections on the identity component of the Hilbert--Blumenthal surface, we may carry out the computation using the classical $q$-expansion.

\noindent First, suppose $\msf{g}_1$ is $\mathfrak{p}_1$-depleted. For $\lambda\in L$, write $q^\lambda =\exp(2\pi i(\lambda z_1+\bar{\lambda}z_2)),$ $\bar{\lambda}\in L$ is the algebraic conjugate.
Suppose $\msf{g}_1 = \sum_\lambda a_\lambda q^\lambda$ and $\msf{g}_2 = \sum_\lambda b_\lambda q^\lambda$, where $\lambda$ runs through the totally positive elements in some lattice in $L$. Then $d_1\msf{g}_2-d_2\msf{g}_1=0$ implies
\[
\lambda b_\lambda -\bar{\lambda}a_\lambda=0,\qquad\text{or equivalently}\qquad b_\lambda = \frac{\bar{\lambda}}{\lambda}\cdot a_\lambda.
\]
Here $a_\lambda$ and $b_\lambda$ are understood to be in $\overline{\bb{Q}}_p,$ as well as $\lambda$ and $\bar{\lambda}$ via our fixed embedding $\overline{\bb{Q}}\hookrightarrow \overline{\bb{Q}}_p,$ corresponding to a place of $\bar{\bb{Q}}_p$ above $\mathfrak{p}_1$. As $\msf{g}_1$ is $\mathfrak{p}_1$-depleted, 
$d_1^{-1}\msf{g}_1\in S_{0t_L,0t_L}^\mrm{p\mbox{-}adic}(K^*_\diamond(p^\alpha))$
is well-defined and we can compute
\[
\zeta^*(\msf{g}_1, \msf{g}_2)
=\sum_{n>0}\Big(\sum_{\lambda+\bar{\lambda}=n}(a_\lambda+b_\lambda)\Big)q^n=-\sum_{n>0}n\Big(\sum_{\lambda+\bar{\lambda}=n}\frac{a_\lambda}{\lambda}\Big)q^n=-d\zeta^*\big[d_1^{-1}\msf{g}_1\big].
\]
 Therefore, $e_\mrm{ord}\zeta^*(\msf{g}_1, \msf{g}_2)=0$. The case when $\msf{g}_2$ is $\mathfrak{p}_2$-depleted is completely analogous.
\end{proof}

\begin{corollary}\label{corollcomputation}
Let $\Omega_\mrm{P}$ denote the differential form associated to  $e_\mrm{ord}\zeta^*\big[ d_2^{-1}\big(\nu_\alpha^*(\breve{\msf{g}}_\mrm{P}\lvert\tau^{-1}_{\frak{p}_2^\alpha})\big)^{\mbox{\tiny $[\cal{P}]$}}\big]$, then the following equality holds
	\[
	e_\mrm{ord}\zeta^*[\iota_\mrm{dR}^B(w)]=-\Omega_\mrm{P}.
	\]
\end{corollary}
\begin{proof}
	Combining equation ($\ref{express in classical forms}$) and Lemma $\ref{PrimitiveVanishing}$ we deduce that
	\[
	e_\mrm{ord}\zeta^*[\iota_\mrm{dR}^B(w)]
	=
	\omega_{e_\mrm{ord}\zeta^*\Big( -d_2^{-1}(\breve{\msf{g}}^\diamond_\mrm{P})^{\mbox{\tiny $[\cal{P}]$}}- \alpha_1^\diamond V(p_1)d_2^{-1}(\breve{\msf{g}}^\diamond_\mrm{P})^{\mbox{\tiny $[\frak{p}_2]$}},\ \alpha_2^\diamond V(p_2) d_1^{-1}(\breve{\msf{g}}^\diamond_\mrm{P})^{\mbox{\tiny $[\frak{p}_1]$}}\Big)}.
	\] 
	Then, the claim follows by applying (\cite{BlancoFornea}, Proposition 2.11 $\&$ Lemma 3.10) as in the proof of (\cite{BlancoFornea}, Theorem 5.14) to obtain  the vanishing 
	\[
	e_\mrm{ord}\zeta^*\Big( V(p_2) d_1^{-1}(\breve{\msf{g}}^\diamond_\mrm{P})^{\mbox{\tiny $[\frak{p}_1]$}}\Big)=0,
	\qquad 
	e_\mrm{ord}\zeta^*\Big( V(p_1)d_2^{-1}(\breve{\msf{g}}^\diamond_\mrm{P})^{\mbox{\tiny $[\frak{p}_2]$}}\Big)=0.
	\]
\end{proof}

\noindent Now we are ready to compute the syntomic Abel--Jacobi map of Hirzebruch--Zagier cycles.

\begin{theorem}\label{AJ formula}
Suppose that $p$ splits in $L$ with narrowly principal factors and that there is no totally positive unit in $L$ congruent to $-1$ modulo $p$, then
\[
\mrm{AJ}_{\mrm{syn}}\Big((\lambda_\alpha,\mrm{id})_*\Delta_\alpha^\circ\Big)\big(\omega_\mrm{P}\otimes\eta_\circ\big)
=
\frac{\alpha_{\msf{f}^*_\circ}^{\alpha-1}\cdot \alpha_{2,\msf{g}_\mrm{P}}^{-\alpha}\cdot G(\theta_{L,\mathfrak{p}}^{-1}\chi_\mathfrak{p}^{-1})}{\big(1-\alpha_{1,\msf{g}_\mrm{P}}\alpha_{2,\msf{g}_\mrm{P}}^{-1}\alpha_{\msf{f}_\circ^*}^{-1}\big)}\cdot\frac{\Big\langle e_{\mrm{ord}}\zeta^*\big(d_\mu^{\bfcdot}\breve{\scr{G}}^{\mbox{\tiny $[\cal{P}]$}}\big)^\dagger(\mrm{P}),\ \msf{f}_\circ^{*\mbox{\tiny $(p)$}}\Big\rangle}{\Big\langle\msf{f}_\circ^{*\mbox{\tiny $(p)$}}, \msf{f}_\circ^{*\mbox{\tiny $(p)$}}\Big\rangle}.
\]
\end{theorem}
\begin{proof}
	By equation ($\ref{secondreduction}$) and Corollary $\ref{corollcomputation}$ we have 
	\[
	\mrm{AJ}_{\mrm{syn}}\Big((\lambda_\alpha,\mrm{id})_*\Delta_\alpha^\circ\Big)\big(\omega_\mrm{P}\otimes\eta_\circ\big)=
	\frac{1}{\big(1-\alpha_{1,\msf{g}_\mrm{P}}\alpha_{2,\msf{g}_\mrm{P}}^{-1}\alpha_{\msf{f}_\circ^*}^{-1}\big)}\cdot \Big\langle \Omega_\mrm{P},\hspace{1mm} (\pi_{2,\alpha})^*(\lambda_1)^*\eta_\circ\Big\rangle_{\mrm{dR},Y_\alpha}
	\]
	Corollary $\ref{diagonal restriction family}$ and Proposition $\ref{analysis comp geometry}$ show that the cuspform $e_\mrm{ord}\zeta^*\Big[ d_2^{-1}\big(\nu_\alpha^*(\breve{\msf{g}}_\mrm{P}\lvert\tau^{-1}_{\frak{p}_2^\alpha})\big)^{\mbox{\tiny $[\cal{P}]$}}\Big]$ is of level $K'_0(p)$, and since $U_p^{\alpha-1}\eta_\circ=\alpha_{\msf{f}^*_\circ}^{\alpha-1}\cdot \eta_\circ$
	\[
	\Big\langle \Omega_\mrm{P},\hspace{1mm}  (\pi_{2,\alpha})^*(\lambda_1)^*\eta_\circ\Big\rangle_{\mrm{dR},Y_\alpha}=
	\Big\langle \Omega_\mrm{P},\hspace{1mm} (\lambda_1)^*U_p^{\alpha-1}\eta_\circ\Big\rangle_{\mrm{dR},Y_1}=
	\alpha_{\msf{f}^*_\circ}^{\alpha-1}\cdot \Big\langle \Omega_\mrm{P},\hspace{1mm} (\lambda_1)^*\eta_\circ\Big\rangle_{\mrm{dR},Y_1}.
	\]
	 Finally, by Definition $\ref{def eta}$ 
\[
\Big\langle \Omega_\mrm{P},\hspace{1mm} (\lambda_1)^*\eta_\circ\Big\rangle_{\mrm{dR},Y_1}
=
\frac{\Big\langle e_\mrm{ord}\zeta^*\Big[ d_2^{-1}\big(\nu_\alpha^*(\breve{\msf{g}}_\mrm{P}\lvert\tau^{-1}_{\frak{p}_2^\alpha})\big)^{\mbox{\tiny $[\cal{P}]$}}\Big],\msf{f}_\circ^{*\mbox{\tiny $(p)$}}\Big\rangle}{\Big\langle\msf{f}_\circ^{*\mbox{\tiny $(p)$}}, \msf{f}_\circ^{*\mbox{\tiny $(p)$}}\Big\rangle}
\]
which together with Proposition $\ref{analysis comp geometry}$ implies the result.	
\end{proof}

\section{Comparison}\label{Sect: Comparison}
Recall our running assumptions: there is an ordinary prime $p$ for $\msf{f}_\circ$ such that
	\begin{itemize}
		\item[\bfcdot] $p$ splits in $L$ with narrowly principal factors;
		\item[\bfcdot] there is no totally positive unit in $L$ congruent to $-1$ modulo $p$;
		\item[\bfcdot] the eigenvalues of $\mrm{Fr}_p$ on $\mrm{As}(\varrho)$ are all distinct modulo $p$.
	\end{itemize}
Consider the element
\[
\bs{\zeta}_{\scr{G},\msf{f}_\circ}:=\alpha_{\msf{f}_\circ}\cdot\Big(1-\scr{G}(\mbf{T}(\varpi_{\frak{p}_1})\mbf{T}(\varpi_{\frak{p}_2}^{-1}))\alpha_{\msf{f}^*_\circ}^{-1}\Big)\in\mbf{I}_\scr{G}.
\]
\begin{theorem}\label{comparison aut-mot}
There is an equality of $p$-adic $L$-functions
\[
\bs{\zeta}_{\scr{G},\msf{f}_\circ}\cdot \scr{L}^\mrm{mot}_p(\breve{\scr{G}},\msf{f}_\circ)
=
\scr{L}^\mrm{an}_p(\breve{\scr{G}},\msf{f}_\circ)\qquad \text{in}\qquad
\bs{\Pi}\otimes_{\bs{\Lambda}} \mbf{I}_\scr{G}. 
\]
In particular, $\scr{L}^\mrm{mot}_p(\breve{\scr{G}},\msf{f}_\circ)$ belongs to  $\mbf{I}_\scr{G}[\bs{\zeta}_{\scr{G},\msf{f}_\circ}^{-1}]$.
\end{theorem}
\begin{proof}
By Theorem \ref{cor: vanishing criterion} it suffices to show that the elements we want to compare have the same specialization at every arithmetic point of $\mbf{I}_\scr{G}$  of weight $(2t_L,t_L)$.
Let $\mrm{P}\in\cal{A}_{\bs{\chi}}(\mbf{I}_\scr{G})$ be an arithmetic point of weight $(2t_L,t_L)$ and character $(\chi_\circ\theta_L^{-1}\chi^{-1},\mathbbm{1})$ of conductor $p^\alpha$ we have (Proposition \ref{prop: big log})
			\[
			\mrm{P}\circ\bs{\cal{L}}^{\msf{f}_\circ}_\scr{G}=\Upsilon(\mrm{P})\cdot\big(\log_\mrm{BK}\circ\ \mrm{P}\big)
			\qquad\text{for}\qquad
\Upsilon(\mrm{P})=\left(\alpha_{1,\msf{g}_{\mrm{P}}}\alpha_{2,\msf{g}_{\mrm{P}}}\alpha_{\msf{f}^*_\circ}^{-1}\right)^\alpha\cdot G\big(\chi_{\mbox{\tiny $\spadesuit$}}\cdot\theta_{\bb{Q}\lvert D_p}\big)^{-1}.
\]
By Propositions \ref{prop: huge period map} 
\[
\begin{split}
\scr{L}^\mrm{mot}_p(\breve{\scr{G}},\msf{f}_\circ)(\mrm{P})
&=
\Big\langle \mrm{P}\circ\bs{\cal{L}}_\scr{G}^{\msf{f}_\circ}\big(\boldsymbol{\kappa}_p^{\msf{f}_\circ}(\scr{G})\big),\
(\lambda_\alpha)^*\omega_{\breve{\scr{G}}_\mrm{P}}\otimes (\lambda_1)^*\eta_\circ\Big\rangle_\mrm{dR}\\
&=
\Upsilon(\mrm{P})\cdot\Big\langle\log_\mrm{BK}\big(\boldsymbol{\kappa}_p^{\msf{f}_\circ}(\scr{G})(\mrm{P})\big),\
(\lambda_\alpha)^*\omega_{\breve{\scr{G}}_\mrm{P}}\otimes (\lambda_1)^*\eta_\circ\Big\rangle_\mrm{dR}.\\
\end{split}\]
From the definition of $\kappa_\alpha^\mrm{n.o.}$ (Equation ($\ref{normalizationHZclasses}$)) we continue with
\begin{equation}\label{L-1}\begin{split}
\scr{L}^\mrm{mot}_p(\breve{\scr{G}},\msf{f}_\circ)(\mrm{P})
&=
\alpha_{1,\msf{g}_{\mrm{P}}}^{-\alpha}\cdot\Upsilon(\mrm{P})\cdot\Big\langle\mrm{AJ}_\mrm{syn}(\Delta^\circ_\alpha),\
(\lambda_\alpha)^*\omega_{\breve{\scr{G}}_\mrm{P}}\otimes (\lambda_1)^*\eta_\circ\Big\rangle_\mrm{dR}\\
&=
\alpha_{1,\msf{g}_{\mrm{P}}}^{-\alpha}\cdot\Upsilon(\mrm{P})\cdot\mrm{AJ}_\mrm{syn}\Big((\lambda_\alpha,\lambda_1)_*\Delta^\circ_\alpha\Big)\big(
\omega_{\breve{\scr{G}}_\mrm{P}}\otimes \eta_\circ\big)\\
(\text{Theorem}\ \ref{AJ formula})\qquad
&= 
\frac{\alpha_{\msf{f}^*_\circ}^{-1}}{\big(1-\alpha_{1,\msf{g}_\mrm{P}}\alpha_{2,\msf{g}_\mrm{P}}^{-1}\alpha_{\msf{f}_\circ^*}^{-1} \big)}\cdot\frac{\Big\langle e_{\mrm{ord}}\zeta^*\big(d_\mu^{\bfcdot}\breve{\scr{G}}^{\mbox{\tiny $[\cal{P}]$}}\big)^\dagger(\mrm{P}),\ \msf{f}_\circ^{*\mbox{\tiny $(p)$}}\Big\rangle}{\Big\langle\msf{f}_\circ^{*\mbox{\tiny $(p)$}}, \msf{f}_\circ^{*\mbox{\tiny $(p)$}}\Big\rangle}.
\end{split}\end{equation}
We found that
\begin{equation}\label{eq111}
\begin{split}
\scr{L}^\mrm{mot}_p(\breve{\scr{G}},\msf{f}_\circ)(\mrm{P})
&= \frac{\alpha_{\msf{f}^*_\circ}^{-1}}{\big(1-\alpha_{1,\msf{g}_\mrm{P}}\alpha_{2,\msf{g}_\mrm{P}}^{-1}\alpha_{\msf{f}_\circ^*}^{-1} \big)}\cdot \scr{L}^\mrm{an}_p(\breve{\scr{G}},\msf{f}_\circ)(\mrm{P}).
\end{split}\end{equation}
  If we let
\[
\bs{\zeta}_{\scr{G},\msf{f}_\circ}:=\alpha_{\msf{f}_\circ}\cdot\Big(1-\scr{G}(\mbf{T}(\varpi_{\frak{p}_1})\mbf{T}(\varpi_{\frak{p}_2}^{-1}))\alpha_{\msf{f}^*_\circ}^{-1}\Big)\in\mbf{I}_\scr{G}
\]
then we can rewrite ($\ref{eq111}$) as
\[
 \bs{\zeta}_{\scr{G},\msf{f}_\circ}(\mrm{P})\cdot\scr{L}^\mrm{mot}_p(\breve{\scr{G}},\msf{f}_\circ)(\mrm{P})
 = \scr{L}^\mrm{an}_p(\breve{\scr{G}},\msf{f}_\circ)(\mrm{P}).
\]
The equality $ \bs{\zeta}_{\scr{G},\msf{f}_\circ}\cdot\scr{L}^\mrm{mot}_p(\breve{\scr{G}},\msf{f}_\circ)=  \scr{L}^\mrm{an}_p(\breve{\scr{G}},\msf{f}_\circ)$ follows from Theorem \ref{cor: vanishing criterion}.
\end{proof}

\begin{corollary}\label{second step}
	For any arithmetic point $\mrm{P}\in\cal{A}_{\bs{\chi}}(\mbf{I}_\scr{G})$
	\[
	\scr{L}_p^\mrm{aut}(\breve{\scr{G}},\msf{f}_\circ)(\mrm{P})
	\not=0\qquad\iff\qquad
	\scr{L}_p^\mrm{mot}(\breve{\scr{G}},\msf{f}_\circ)(\mrm{P})\not=0.
	\]
\end{corollary}
\begin{proof}
	The claim follows from Theorem $\ref{comparison aut-mot}$ and the fact that $\bs{\zeta}_{\scr{G},\msf{f}_\circ}(\mrm{P})\not=0$ for any arithmetic point $\mrm{P}\in\cal{A}_{\bs{\chi}}(\mbf{I}_\scr{G})$.
\end{proof}

\begin{remark}
	We are not assuming Conjecture \ref{wishingOhta} for  Theorem \ref{comparison aut-mot} and Corollary \ref{second step}.
\end{remark}

\subsection{Specialization in parallel weight one}

\begin{proposition}\label{nontriv-specialization}
Suppose that the running assumptions on the prime $p$, the representation $\varrho$, and Conjecture \ref{wishingOhta} hold. If the special $L$-value $L(\msf{f}_\circ,\mrm{As}(\varrho),1)$ does not vanish, then there is a surjection
		\[
		\xymatrix{	 \bs{\cal{V}}_\mathscr{G}(M)(-1)\otimes_{\mrm{P}_\circ}E_\wp\ar@{->>}[r]
	& \mrm{As}(\varrho).
		}\]
such that $\bs{\kappa}^{\msf{f}_\circ}_p(\scr{G})(\mrm{P}_\circ)$ has non-trivial image under the induced map 
\[
\mrm{H}^1\big(\bb{Q}_p,\bs{\cal{V}}^{\msf{f}_\circ}_{\scr{G}_{\mrm{P}_\circ}}(M)\big)\longrightarrow\mrm{H}^1(\bb{Q}_p,\mrm{As}(\varrho)\otimes\mrm{Gr}^0\mrm{V}_{\msf{f}_\circ}).
\] 
\end{proposition}
\begin{proof}
	Corollary \ref{firststep} shows that the non-vanishing of the special $L$-value  $L(\msf{f}_\circ,\mrm{As}(\varrho),1)$ is equivalent to the non-vanishing of the automorphic $p$-adic $L$-function $\scr{L}_p^\mrm{aut}(\breve{\scr{G}},\msf{f}_\circ)$ at the arithmetic point $\mrm{P}_\circ\in\cal{A}_{\bs{\chi}}(\mbf{I}_\scr{G})$ of parallel weight one corresponding to  $\msf{g}_\circ^{\mbox{\tiny $(p)$}}$. Then
		\[\begin{split}
		L\Big(\msf{f}_\circ,\mrm{As}(\varrho),1\Big)\not=0\qquad
		\iff&\qquad \scr{L}_p^\mrm{aut}(\breve{\scr{G}},\msf{f}_\circ)(\mrm{P}_\circ)\not=0\\
		\mbox{\tiny $(\text{Corollary}\ \ref{second step})$}\qquad \iff&\qquad
		\scr{L}_p^\mrm{mot}(\breve{\scr{G}},\msf{f}_\circ)(\mrm{P}_\circ)\not=0\\
		\mbox{\tiny $(\text{Lemma}\ \ref{cruximplicat})$}\qquad\implies&\qquad
		\mrm{exp}^*_\mrm{BK}\big(\bs{\kappa}^{\msf{f}_\circ}_p(\scr{G})(\mrm{P}_\circ)\big)\not=0\\
		\mbox{\tiny $(\text{Equation}\ (\ref{step four}))$}\qquad \iff&\qquad
		\bs{\kappa}^{\msf{f}_\circ}_p(\scr{G})(\mrm{P}_\circ)\not=0.\\
		\end{split}\]
 Therefore, by Corollary \ref{correctspec}, the   non-trivial class $\bs{\kappa}^{\msf{f}_\circ}_p(\scr{G})(\mrm{P}_\circ)\in \mrm{H}^1\big(\bb{Q}_p,\bs{\cal{V}}^{\msf{f}_\circ}_{\scr{G}_{\mrm{P}_\circ}}(M)\big)$ maps non-trivially to some copy of $\mrm{H}^1(\bb{Q}_p,\mrm{As}(\varrho)\otimes\mrm{Gr}^0\mrm{V}_{\msf{f}_\circ})$.
		
\end{proof}

\noindent The choice of surjection $\bs{\cal{V}}_{\scr{G}_{\mrm{P}_\circ}}(M)(-1)\otimes_OE_\wp\twoheadrightarrow\mrm{As}(\varrho)$ in Proposition \ref{nontriv-specialization} induces a map
\[
\mrm{H}^1\big(\bb{Q}_p,\bs{\cal{V}}_{\scr{G}_{\mrm{P}},\msf{f}_\circ}(M)\big)\longrightarrow\mrm{H}^1(\bb{Q}_p,\mrm{V}_{\varrho,\msf{f}_\circ}).
\]
We denote the image of $\boldsymbol{\kappa}_{\scr{G},\msf{f}_\circ}$ under such map by 
\begin{equation}
	\kappa(\msf{g}_\circ^{\mbox{\tiny $(p)$}},\msf{f}_\circ)\in \mrm{H}^1\big(\bb{Q}, \mrm{V}_{\varrho,\msf{f}_\circ}\big).
\end{equation}  
The quotient map $\mrm{V}_{\varrho,\msf{f}_\circ}\twoheadrightarrow \mrm{As}(\varrho)\otimes\mrm{Gr}^0\mrm{V}_{\msf{f}_\circ}$ induces a homomorphism 
\[
\partial_p: 
\mrm{H}^1\big(\bb{Q}_p, \mrm{V}_{\varrho,\msf{f}_\circ}\big)
\longrightarrow 
\mrm{H}^1\big(\bb{Q}_p, \mrm{As}(\varrho)\otimes\mrm{Gr}^0\mrm{V}_{\msf{f}_\circ}\big)
\] 
 whose kernel is the local Selmer group at $p$,
as one can see by analyzing the Hodge--Tate weights, 
\[
\mrm{H}^1_f\big(\bb{Q}_p, \mrm{V}_{\varrho,\msf{f}_\circ}\big)=\ker(\partial_p).
\]

\begin{theorem}\label{criterion crystalline}
Suppose that the running assumptions on the prime $p$, the representation $\varrho$, and Conjecture \ref{wishingOhta} hold.  Let $\msf{g}_\circ^{\mbox{\tiny $(p)$}}$ be any ordinary $p$-stabilization of $\msf{g}_\circ$. If the special $L$-value $L(\msf{f}_\circ,\mrm{As}(\varrho),1)$ does not vanish, then the global cohomology class $\kappa(\msf{g}_\circ^{\mbox{\tiny $(p)$}},\msf{f}_\circ)$ is not crystalline at $p$. Furthermore,
\[
\partial_p\big(\kappa(\msf{g}_{\circ}^{\mbox{\tiny $(p)$}},\msf{f}_\circ)\big)\in \mrm{H}^1\Big(\bb{Q}_p,\mrm{As}(\varrho)^{\beta_p}\otimes\mrm{Gr}^0\mrm{V}_{\msf{f}_\circ}\Big)
\]
where $\mrm{As}(\varrho)^{\beta_p}$ is the subpace where $\mrm{Fr}_p$ acts as multiplication by $\beta_p=\beta_1\beta_2$.
\end{theorem}
\begin{proof}
	 It follows from the definitions that
	\[
	\mrm{Im}\Big(\bs{\cal{V}}_{\scr{G}}^{\msf{f}_\circ}(M)\longrightarrow \mrm{As}(\varrho)\otimes\mrm{Gr}^0\mrm{V}_{\msf{f}_\circ}\Big)=\mrm{Im}\Big(\mrm{Fil}^2 \bs{\cal{V}}_{\scr{G},\msf{f}_\circ}(M)\longrightarrow \mrm{As}(\varrho)\otimes\mrm{Gr}^0\mrm{V}_{\msf{f}_\circ}\Big),
	\]
	and that the image of  $\bs{\kappa}^{\msf{f}_\circ}_p(\scr{G})(\mrm{P}_\circ)$ coincides with $\partial_p\big(\kappa_p(\msf{g}_{\circ}^{\mbox{\tiny $(p)$}},\msf{f}_\circ)\big)$. Therefore, invoking Proposition \ref{nontriv-specialization} we deduce the first claim. We obtain the second claim by observing that Proposition $\ref{AsaiFil}$ implies that
	\[
	\mrm{As}(\varrho)^{\beta_p}=\mrm{Fil}^2\mrm{As}(\varrho).
	\]
\end{proof}

\section{On the equivariant BSD-conjecture}

Let $(\mrm{W},\varrho)$ be a $d$-dimensional self-dual Artin representation with coefficients in a number field $D$. Suppose $\varrho$ factors through the Galois group $G(H/\bb{Q})$ of a number field $H$
\[\xymatrix{
\Gamma_\bb{Q}\ar[rr]^{\varrho}\ar@{->>}[dr]&& \mrm{GL}_d(D)\\
& G(H/\bb{Q})\ar@{^{(}->}[ru]&.
}\]
Let $E_{/\bb{Q}}$ be a rational elliptic curve, then its algebraic rank with respect to the Artin representation $\varrho$ is defined as
\[
r_\mrm{alg}(E,\varrho)=\mrm{dim}_D\ E(H)^\varrho_D,
\]
the dimension of $E(H)^\varrho_D=\mrm{Hom}_{G(H/\bb{Q})}(\varrho, E(H)\otimes D)$ the $\varrho$-isotypic component of the Mordell-Weil group. For $p$ a rational prime and $\wp\mid p$ an $O_D$-prime ideal, we can consider the $p$-adic Galois representations
\[
\mrm{V}_\wp(E)=\mrm{H}^1_\et(E_{\bar{\bb{Q}}},D_\wp(1)), \qquad \mrm{W}_\wp=\mrm{W}\otimes_DD_\wp.
\]
As the Artin representation $\varrho$ is self-dual, the Kummer map allows us to identify the group $E(H)^\varrho_D$ with a subgroup of the Bloch--Kato Selmer group $\mrm{H}^1_f(\bb{Q},\mrm{W}_\wp\otimes \mrm{V}_\wp(E))$. Then, local Tate duality together with the global Poitou-Tate exact sequence can be used to show that global cohomology classes not crystalline at $p$ bound the size of the $\varrho$-isotypic component of the Mordell-Weil group of $E_{/\bb{Q}}$ (\cite{DR2}, Section 6.1).

\begin{lemma}\label{zerolocalization}
	Let $\kappa_1,\dots,\kappa_d\in\mrm{H}^1(\bb{Q},  \mrm{W}_\wp\otimes \mrm{V}_\wp(E))$ be global cohomology classes with linearly independent images in the singular quotient $\mrm{H}^1_\mrm{sing}(\bb{Q}_p, \mrm{W}_\wp\otimes \mrm{V}_\wp(E))$ at $p$. Then the $\varrho$-isotypic part of $E(H)$ is trivial:
		\[
		r_\mrm{alg}(E,\varrho)=0.
		\]

\end{lemma}
\begin{proof}
	The local cohomology group $\mrm{H}^1(\bb{Q}_p, \mrm{W}_\wp\otimes \mrm{V}_\wp(E))$ is a $2d$-dimensional $D_\wp$-vector space and the local Tate pairing induces a perfect duality of $d$-dimensional spaces (\cite{DR2}, Lemma 6.1)
	\[
	\langle\ , \rangle: \mrm{H}^1_f(\bb{Q}_p, \mrm{W}_\wp\otimes \mrm{V}_\wp(E))\times\mrm{H}^1_\mrm{sing}(\bb{Q}_p,  \mrm{W}_\wp\otimes \mrm{V}_\wp(E))\longrightarrow D_\wp.
	\]
	The global Poitou-Tate exact sequence implies that the image of the localization at $p$
	\[
	\mrm{loc}_p:\mrm{H}^1(\bb{Q}, \mrm{W}_\wp\otimes \mrm{V}_\wp(E))\longrightarrow\mrm{H}^1(\bb{Q}_p, \mrm{W}_\wp\otimes \mrm{V}_\wp(E))
	\]
	is $d$-dimensional (\cite{DR2}, Lemma 6.2). Therefore, the existence of global cohomology classes $\kappa_1,\dots,\kappa_d$ whose localizations generate the singular quotient at $p$ implies that the restriction of $\mrm{loc}_p$ to the Bloch--Kato Selmer group $\mrm{H}^1_f(\bb{Q},\mrm{V}_\wp(E)\otimes \mrm{W}_\wp)$ is the zero map. The commutativity of the diagram
\[\xymatrix{
E(H)^{\varrho}_D\ar[r]\ar@{^{(}->}[d]& \oplus_{\mathfrak{p}\mid p}\mrm{Hom}_{G(H_\mathfrak{p}/\bb{Q}_p)}(\mrm{W}_\wp, E(H_\mathfrak{p})\otimes D_\wp)\ar@{^{(}->}[d]\\
\mrm{H}^1_f(\bb{Q},  \mrm{W}_\wp\otimes\mrm{V}_\wp(E))\ar[r]^0& \mrm{H}^1(\bb{Q}_p,\mrm{W}_\wp\otimes \mrm{V}_\wp(E))
}\]
and the injectivity of the vertical Kummer maps imply the triviality of the top horizontal morphism.  Since $E(H)\otimes D\hookrightarrow E(H_\mathfrak{p})\otimes D_\wp$ is injective for all $O_H$-prime ideal $\mathfrak{p}\mid p$, we deduce that $\dim_DE(H)_D^{\varrho}=0$.
\end{proof}

\subsubsection{On twisted triple products.}
The goal of this section is apply the idea of $p$-adic deformation to the setting where the self-dual Artin representation is $4$-dimensional and arises as the tensor induction
\[
\mrm{As}(\varrho)=\otimes\mbox{-}\mrm{Ind}_L^\bb{Q}(\varrho)
\] 
of a totally odd, irreducible two-dimensional Artin representation $\varrho:\Gamma_L\to\mrm{GL}_2(D)$  of the absolute Galois group of a real quadratic field $L$. We suppose that $\varrho$ has conductor $\mathfrak{Q}$ and that the tensor induction of the determinant $\det(\varrho)$ is the trivial character. Let $E_{/\bb{Q}}$ be a rational elliptic curve of conductor $N$, and for any rational prime $p$, we consider  the Kummer self-dual $p$-adic Galois representation of $\Gamma_\bb{Q}$ 
\[
\mrm{V}_{\varrho,E}=\mrm{As}(\varrho)\otimes \mrm{V}_\wp(E).
\] 
By modularity (\cite{W}, \cite{TW}, \cite{PS}) there are a primitive Hilbert cuspform $\msf{g}_\varrho\in S_{t_L,t_L}(\frak{Q};D)$ of parallel weight one, and a primitive elliptic cuspform $\msf{f}_E\in S_{2,1}(N;\bb{Q})$ of weight $2$ associated to $\varrho$ and $E$ respectively.
The twisted $L$-function $L\big(E,\mrm{As}(\varrho),s\big)$ has meromorphic continuation to $\bb{C}$ and a functional equation centered at $s=1$, at which the $L$-function is holomorphic. 


\begin{theorem}\label{Main Theorem}
Suppose that $N$ is coprime to $\mathfrak{Q}$, split in $L$, and there exists an ordinary prime $p\nmid 2N\cdot\frak{Q}$ for $E_{/\bb{Q}}$ such that  
	\begin{itemize}
		\item[($1$)] $p$ splits in $L$ with narrowly principal factors;
		\item[($2$)] there is no totally positive unit in $L$ congruent to $-1$ modulo $p$;
		\item[($3$)] the eigenvalues of $\mrm{Fr}_p$ on $\mrm{As}(\varrho)$ are all distinct modulo $p$.
	\end{itemize}
If, additionally, $\varrho$ is residually not solvable and Conjecture \ref{wishingOhta} holds, then 
\[
r_\mrm{an}\big(E,\mrm{As}(\varrho)\big)=0\quad\implies\quad r_\mrm{alg}\big(E,\mrm{As}(\varrho)\big)=0.
\]
\end{theorem}
\begin{proof}
	As the eigenvalues of $\mrm{Fr}_p$ on $\mrm{As}(\varrho)$ are all distinct modulo $p$, the cuspform $\msf{g}_\varrho$ has $4$ distinct $p$-stabilizations 
	\[
	\msf{g}_\varrho^{\mbox{\tiny $(\alpha_1\alpha_2)$}},
	\quad 
	\msf{g}_\varrho^{\mbox{\tiny $(\alpha_1\beta_2)$}},
	\quad 
	\msf{g}_\varrho^{\mbox{\tiny $(\beta_1\alpha_2)$}}
	\quad \text{and}\quad 
	\msf{g}_\varrho^{\mbox{\tiny $(\beta_1\beta_2)$}}.
	\]
	Therefore, assuming $r_\mrm{an}\big(E,\mrm{As}(\varrho)\big)=0$, Theorem \ref{criterion crystalline} produces $4$ global cohomology classes
	\[
	\kappa\big(\msf{g}_\varrho^{\mbox{\tiny $(\alpha_1\alpha_2)$}},\msf{f}_E\big),
	\quad 
	\kappa\big(\msf{g}_\varrho^{\mbox{\tiny $(\alpha_1\beta_2)$}},\msf{f}_E\big),
	\quad 
	\kappa\big(\msf{g}_\varrho^{\mbox{\tiny $(\beta_1\alpha_2)$}},\msf{f}_E\big),
	\quad 
	\kappa\big(\msf{g}_\varrho^{\mbox{\tiny $(\beta_1\beta_2)$}},\msf{f}_E\big)\quad\in\ \mrm{H}^1\big(\bb{Q},\mrm{V}_{\varrho,E}\big)
	\]
	whose images in the singular quotient $\mrm{H}^1_\mrm{sing}\big(\bb{Q}_p,\mrm{V}_{\varrho,E}\big)$ are linearly independent because they belong to different eigenspaces. The result follows by invoking Lemma \ref{zerolocalization}.
\end{proof}

\subsection{Rational elliptic curves over quintic fields}
In this section we show that in many cases of interest there are infinitely many ordinary primes $p\nmid 2N\cdot\frak{Q}$ for a rational elliptic curve $E_{/\bb{Q}}$ satisfying assumptions ($1$),($2$),($3$) of Theorem \ref{Main Theorem}.

\subsubsection{Narrowly principal prime factors.} By class field theory, a prime is  split in $L$ with narrowly principal prime factors if and only if it totally splits in the narrow class field $H_L^+$ of $L$. We recall the lattice of subfields of $\bb{Q}(\zeta_{16})$ 
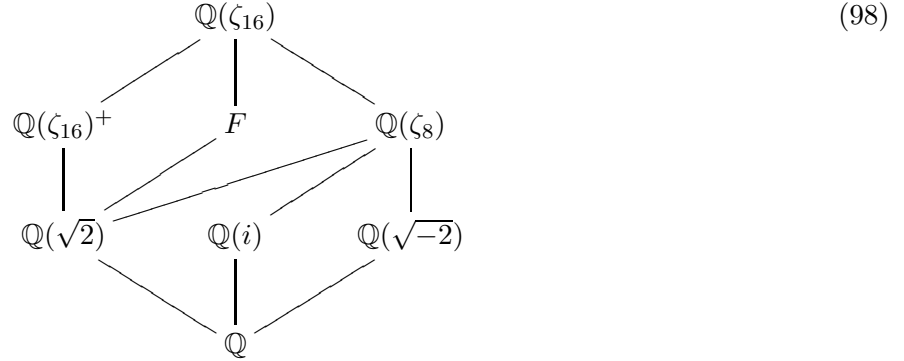
\begin{equation}\label{lattice of subfields}
	\xymatrix{
	& \bb{Q}(\zeta_{16})\ar@{-}[d]\ar@{-}[dl]\ar@{-}[dr]& &\\
	\bb{Q}(\zeta_{16})^+\ar@{-}[d]& F\ar@{-}[dl]& \bb{Q}(\zeta_8)\ar@{-}[dll]\ar@{-}[dl]\ar@{-}[d]\\
	\bb{Q}(\sqrt{2})\ar@{-}[dr]& \bb{Q}(i)\ar@{-}[d]& \bb{Q}(\sqrt{-2})\ar@{-}[dl]\\
	& \bb{Q}& &
	}
\end{equation}
where $F/\bb{Q}$ is the splitting field of the polynomial $X^4+4X^2+2$.
\begin{lemma}\label{narrow class}
	Let $L=\bb{Q}(\sqrt{d})$ be a real quadratic field. 
	\begin{itemize}
	\item[\bfcdot] If $d\equiv_41$ then  $H_L^+\cap\bb{Q}(\zeta_8)=\bb{Q}$.
	\item[\bfcdot] If $d\equiv_43$ then $\bb{Q}(i)\subseteq H_L^+$ and $H_L^+\cap\bb{Q}(\sqrt{\pm2})=\bb{Q}$.
	\item[\bfcdot] If $d\equiv_8\pm2$ then $H_L^+\cap \bb{Q}(i)=\bb{Q}$, $H_L^+\cap\bb{Q}(\sqrt{\mp2})=\bb{Q}$ and $\bb{Q}(\sqrt{\pm2})\subseteq H_L^+$.
	\end{itemize}
\end{lemma}
\begin{proof}
	 If $d\equiv_41$, the field $L$ is unramified at $2$, while $\bb{Q}(\sqrt{\bullet})$ is ramified at $2$ for any $\bullet\in\{-1,\pm2\}$. We deduce that $L(\sqrt{\bullet})/L$ is ramified at $2$, hence $H_L^+\cap\bb{Q}(\zeta_8)=\bb{Q}$. 

 \noindent 	If $d\equiv_43$, both $\bb{Q}(i)$ and $L$ are ramified at $2$, while $\bb{Q}(\sqrt{-d})$ is not. Therefore, the prime $2$ cannot ramify $L(i)/L$. As $\bb{Q}(i)$ is only ramified at $2$, we deduce that $\bb{Q}(i)\subseteq H_L^+$. Next, we note that $L(\sqrt{\pm2})/\bb{Q}$ is totally ramified at $2$ because  all its proper subfields are ramified at $2$. In particular, $L(\sqrt{\pm2})/L$ is ramified at $2$, thus $H_L^+\cap\bb{Q}(\sqrt{\pm2})=\bb{Q}$.
 	
\noindent 	If $d\equiv_42$, all the proper subfields of the biquadratic field $L(i)$ are ramified $2$, hence $2$ is totally ramified in $L(i)$. In particular, $L(i)/L$ is ramified at $2$ and $H_L^+\cap \bb{Q}(i)=\bb{Q}$. Furthermore, if $d\equiv_8\pm2$ the extension $L(\sqrt{\mp2})/\bb{Q}$ is totally ramified at $2$, while $L(\sqrt{\pm2})/L$ is unramified because $\pm d/2\equiv_41$. Hence, $H_L^+\cap\bb{Q}(\sqrt{\mp2})=\bb{Q}$ and $\bb{Q}(\sqrt{\pm2})\subseteq H_L^+$.
\end{proof}

\begin{proposition}\label{inf narrow princ}
	Let $L=\bb{Q}(\sqrt{d})$ be a real quadratic field, then the primes $p\equiv_{16}9$ which are split in $L$ with narrowly principal factors have positive density.
\end{proposition}
\begin{proof}
It follows directly from Lemma $\ref{narrow class}$ that:
	\begin{itemize}
		\item [\bfcdot] if $d\equiv_41$ then $\bb{Q}(\zeta_{16})\cap H_L^+=\bb{Q}$;
		\item [\bfcdot] if $d\equiv_43$ then $\bb{Q}(\zeta_{16})\cap H_L^+=\bb{Q}(i)$;
		\item [\bfcdot] if $d\equiv_86$ then $\bb{Q}(\zeta_{16})\cap H_L^+=\bb{Q}(\sqrt{-2})$.
	\end{itemize}
When $d\equiv_82$ we claim that $\bb{Q}(\zeta_{16})\cap H_L^+=\bb{Q}(\sqrt{2})$. Indeed, in this case the intersection could be either $\bb{Q}(\sqrt{2})$, $\bb{Q}(\zeta_{16})^+$ or $F$ and we show that the latter two options cannot occur. Let $A$ denote either $\bb{Q}(\zeta_{16})^+$ or $F$. When $d=2$ then $\bb{Q}(\zeta_{16})\cap H_L^+=\bb{Q}(\sqrt{2})$ because $A/\bb{Q}(\sqrt{2})$ is ramified at $2$. When $d\not=2$, then $L(\sqrt{2})/\bb{Q}(\sqrt{2})$ is a proper extension unramified at $2$. It follows that $L(\sqrt{2})\cdot A/L(\sqrt{2})$ is ramified at $2$ and cannot be contained in $H_L^+$.

\noindent Since the rational primes $p\equiv_{16}9$ are those totally split in $\bb{Q}(\zeta_8)$ and inert in the extension $\bb{Q}(\zeta_8)\subseteq\bb{Q}(\zeta_{16})$, the analysis above of the intersection $\bb{Q}(\zeta_{16})\cap H_L^+$ together with Chebotarev's density theorem finishes the proof.
\end{proof}

\subsubsection{Congruences for totally positive units.} Let $\epsilon\in\cal{O}_{L,+}^\times$ be a generator of the totally positive units and $p$ a rational prime split in $L$. Then requiring that there is no totally positive unit congruent to $-1$ modulo $p$ is equivalent to ask that, for $\frak{p}\mid p$, the subgroup $\langle \bar\epsilon\rangle$ of $(\cal{O}_L/\frak{p})^\times$, generated by the reduction of $\epsilon$, has odd order.

\begin{lemma}\label{intermediatefields}
	Let $\epsilon\in\cal{O}_{L,+}^\times$ be a generator of the totally positive units of $L=\bb{Q}(\sqrt{d})$, then the totally real number field $L(\sqrt{\epsilon})$ is either equal to $L$ or it is biquadratic over $\bb{Q}$. Suppose that $L(\sqrt{\epsilon})$ is biquadratic and write $\epsilon=a+b\sqrt{d}$ for $a,b\in\bb{N}$, then the subfields of $L(\sqrt{\epsilon})$  are
	
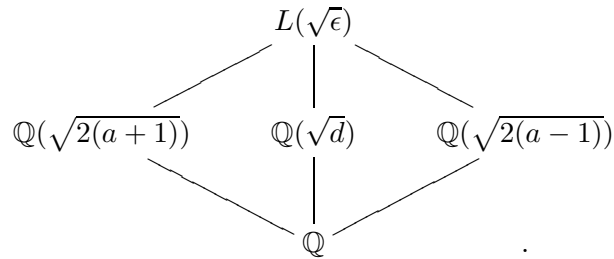
\[\xymatrix{
	& L(\sqrt{\epsilon})\ar@{-}[d]\ar@{-}[dl]\ar@{-}[dr]&\\
	\bb{Q}(\sqrt{2(a+1)})\ar@{-}[dr]& \bb{Q}(\sqrt{d})\ar@{-}[d]& \bb{Q}(\sqrt{2(a-1)})\ar@{-}[dl]\\
	&\bb{Q}&.
	}\]
\end{lemma}
\begin{proof}
	If the fundamental unit of $L$ is not totally positive, then $\epsilon$ is a square in $L$ and $L(\sqrt{\epsilon})=L$. If $\epsilon$ is the fundamental unit, then the number field $L(\sqrt{\epsilon})$  is the splitting field of the polynomial 
	\[
	X^4-\mrm{Tr}_{L/\bb{Q}}(\epsilon)X^2+1  = (X^2-\epsilon)(X^2-1/\epsilon),
	\]
	hence it is biquadratic over $\bb{Q}$ and totally real. Using the relation $\mrm{N}_{L/\bb{Q}}(\epsilon)=1$ one sees that 
	\[
	\left(\sqrt{\frac{a+1}{2}}+\sqrt{\frac{a-1}{2}}\right)^2=\epsilon
	\]
	and the claim follows.
\end{proof}

\begin{remark}
The number field $L(\sqrt[8]{\epsilon})$ is not Galois over $\bb{Q}$. Its Galois closure is obtained by adding an $8$-th root of unity. Indeed  $J=L(\sqrt[8]{\epsilon},\zeta_8)$ is the splitting field of the polynomial
\[
X^{16}-\mrm{Tr}_{L/\bb{Q}}(\epsilon)X^8+1  = (X^8-\epsilon)(X^8-1/\epsilon).
\]
It is clear from this description that $J/\bb{Q}$ is a solvable extension.
\end{remark}

\begin{lemma}\label{nounits}
Let  $\epsilon\in\cal{O}_{L,+}^\times$ be  a generator of the totally positive units and $J/\bb{Q}$ the Galois closure of $L(\sqrt[8]{\epsilon})$. Then for all but finitely many primes $p\equiv_{16}9$ which are totally split in $J$, there is no totally positive unit in $L$ congruent to $-1$ modulo $p$. 
\end{lemma}
\begin{proof}
	Suppose that $p\equiv_{16}9$ and totally split in $J$. If $\frak{p}$ is an $\cal{O}_L$-prime ideal above $p$ then $(\cal{O}_L/\frak{p})^\times\cong(\bb{Z}/p\bb{Z})^\times$ and for all but finitely many such primes the reduction $\bar{\epsilon}$ of $\epsilon$ modulo $\frak{p}$ is an $8$-th power. It follows that $\bar{\epsilon}$ generates a subgroup of order dividing $(p-1)/8$. Since $p\equiv_{16}9$ that order is odd and the subgroup cannot contain $-1$. 
	\end{proof}

\begin{corollary}\label{narrow cong}
	Let $L$ be a real quadratic field, then the primes $p$ split in $L$ with narrowly principal factors and such that there is no totally positive unit congruent to $-1$ mod $p$ have positive density.
\end{corollary}	
\begin{proof}
	By Proposition $\ref{inf narrow princ}$ and Lemma $\ref{nounits}$, all the primes which are totally split in $J$, $H_L^+$, $\bb{Q}(\zeta_8)$ and inert in the extension $\bb{Q}(\zeta_8)\subseteq\bb{Q}(\zeta_{16})$ satisfy the requirements. Clearly the splitting conditions for $J$ and $H_L^+$ are compatible, and from Proposition $\ref{inf narrow princ}$ we know that also the splitting conditions for $H_L^+$ and $\bb{Q}(\zeta_{16})$ are compatible too. We are left to understand $J\cap\bb{Q}(\zeta_{16})$.
Clearly $\bb{Q}(\zeta_8)$ is contained in the intersection because $J=L(\sqrt[8]{\epsilon},\zeta_8)$. One can check that 
\[
[J:\bb{Q}]=\begin{cases}
	16\cdot 2& \text{if}\ \bb{Q}(\sqrt{2})\subseteq L(\sqrt{\epsilon})\\
	16\cdot 4& \text{if}\ \bb{Q}(\sqrt{2})\not\subseteq L(\sqrt{\epsilon}),\\
\end{cases}
\]
$[L(\sqrt[4]{\epsilon},i):\bb{Q}]=16$, and
\[
L(\sqrt[4]{\epsilon},i)\cap\bb{Q}(\zeta_{16})=\begin{cases}
	\bb{Q}(\zeta_8) & \text{if}\ \bb{Q}(\sqrt{2})\subseteq L(\sqrt{\epsilon})\\
	\bb{Q}(i)& \text{if}\ \bb{Q}(\sqrt{2})\not\subseteq L(\sqrt{\epsilon}).\\
\end{cases}
\]
Suppose by contradiction that $\bb{Q}(\zeta_{16})\subseteq J$. Then $J=\bb{Q}(\zeta_{16})\cdot L(\sqrt[4]{\epsilon},i)$ because $\bb{Q}(\zeta_{16})\cdot L(\sqrt[4]{\epsilon},i)$ is a subfield of the same degree as $J$. Therefore the natural injection 
\[
G(J/\bb{Q})\hookrightarrow G(\bb{Q}(\zeta_{16})/\bb{Q})\times G(L(\sqrt[4]{\epsilon},i)/\bb{Q})
\]
produces a contradiction because $G(J/\bb{Q})$ contains an element of order $8$ while the other two Galois groups have exponent $4$. 
In summary, we showed that $J\cap\bb{Q}(\zeta_{16})=\bb{Q}(\zeta_8)$ so that all the required splitting conditions are compatible. Chebotarev's density theorem finishes the proof.
\end{proof}

\noindent When $\varrho$ is one of the Artin representations constructed in \cite{MicAnalytic}, the next proposition shows that there are infinitely many primes satisfying assumptions ($1$),($2$),($3$) of Theorem \ref{Main Theorem}.
\begin{proposition}\label{choiceofp}
Let $K/\bb{Q}$ be an $S_5$-quintic extension whose Galois closure $\widetilde{K}/\bb{Q}$ contains a real quadratic field $L$. Suppose $E_{/\bb{Q}}$ is a rational elliptic curve, then there are infinitely many ordinary primes $p$ for $E_{/\bb{Q}}$ such that
\begin{itemize}
	\item[\bfcdot] $p$ splits in $L$ with narrowly principal factors;
	\item[\bfcdot] there is no totally positive unit in $L$ congruent to $-1$ modulo $p$;
	\item[\bfcdot] the conjugacy class of $\mrm{Fr}_p$ in $G(\widetilde{K}/\bb{Q})\cong S_5$ is that of $5$-cycles.	
\end{itemize}   
\end{proposition}
\begin{proof}
Since $\widetilde{K}/L$ is a non-solvable extension, we deduce that $\widetilde{K}\cap J=L$ and $\widetilde{K}\cap H_L^+=L$. Moreover,  $\widetilde{K}\cap\bb{Q}(\zeta_{16})$ is either $\bb{Q}(\sqrt{2})$ or $\bb{Q}$ according to whether $L=\bb{Q}(\sqrt{2})$ or not.
 Given that $5$-cycles are in the kernel of the surjection $G(\widetilde{K}/\bb{Q})\twoheadrightarrow G(L/\bb{Q})$, one can prove the existence of a set of positive density consisting of rational primes satisfying the listed conditions as in Corollary $\ref{narrow cong}$. It then remains to show that infinitely many of such primes are of good ordinary reduction for the given elliptic curve.
 When $E_{/\bb{Q}}$ does not have complex multiplication, the ordinary primes have density one so there are infinitely many ordinary primes that satisfy the listed conditions. When the elliptic curves $E_{/\bb{Q}}$  has complex multiplication by a quadratic imaginary field $B$, a prime is ordinary for $E_{/\bb{Q}}$ if it splits in $B$. As this new splitting requirement is compatible with those coming from the conditions above, Chebotarev's density theorem gives the claim. 
\end{proof}

\begin{corollary}\label{finalquintic}
	Let $K/\bb{Q}$ be a non-totally real $S_5$-quintic extension whose Galois closure contains a real quadratic field $L$. Suppose that $N$ is odd, unramified in $K/\bb{Q}$ and split in $L$, and that Conjecture \ref{wishingOhta} holds, then 
		\[
		r_\mrm{an}(E/K)=r_\mrm{an}(E/\bb{Q})\quad\implies\quad r_\mrm{alg}(E/K)=r_\mrm{alg}(E/\bb{Q}).
		\]
\end{corollary}
\begin{proof}
By (\cite{MicAnalytic}, Corollary 4.2) there exists a parallel weight one Hilbert eigenform $\msf{g}_K$ over $L$ of level $\frak{Q}$ prime to $N$  such that $\varrho_{\msf{g}_K}$ is residually not solvable and $\mrm{As}(\varrho_{\msf{g}_K})\cong\mrm{Ind}_K^\bb{Q}\mathbbm{1}-\mathbbm{1}$. From the Artin formalism of $L$-functions we deduce that
	\[
	r_\mrm{an}\big(E,\mrm{As}(\varrho_{\msf{g}_K})\big)=r_\mrm{an}(E/K)-r_\mrm{an}(E/\bb{Q})\quad\text{and}\quad r_\mrm{alg}\big(E,\mrm{As}(\varrho_{\msf{g}_K})\big)=r_\mrm{alg}(E/K)-r_\mrm{alg}(E/\bb{Q}).
	\]
The result follows from Theorem $\ref{Main Theorem}$ after invoking Proposition \ref{choiceofp}.
\end{proof}

\bibliography{p3L_GZ}

\begin{thebibliography}{{Nek}18}

\bibitem[BCF19]{BlancoFornea}
I.~Blanco-Chac\'on and M.~Fornea.
\newblock Twisted triple product $\text{p}$-adic $\text{L}$-functions and
  $\text{H}$irzebruch-$\text{Z}$agier cycles.
\newblock {\em Journal of the Inst. of Math. of Jussieu}, 2019.

\bibitem[BDR15]{BDR}
M.~Bertolini, H.~Darmon, and V.~Rotger.
\newblock Beilinson-{F}lach elements and {E}uler systems {II}: the
  {B}irch-{S}winnerton-{D}yer conjecture for {H}asse-{W}eil-{A}rtin
  {$L$}-series.
\newblock {\em J. Algebraic Geom.}, 24(3):569--604, 2015.

\bibitem[Bei13]{Beilinson}
A.~Beilinson.
\newblock On the crystalline period map.
\newblock {\em Camb. J. Math.}, 1(1):1--51, 2013.

\bibitem[Bha10]{Bar5}
M.~Bhargava.
\newblock The density of discriminants of quintic rings and fields.
\newblock {\em Ann. of Math. (2)}, 172(3):1559--1591, 2010.

\bibitem[BK90]{Bloch-Kato}
S.~Bloch and K.~Kato.
\newblock {$L$}-functions and {T}amagawa numbers of motives.
\newblock In {\em The {G}rothendieck {F}estschrift, {V}ol. {I}}, volume~86 of
  {\em Progr. Math.}, pages 333--400. Birkh\"{a}user Boston, Boston, MA, 1990.

\bibitem[BL84]{BL}
J.-L. Brylinski and J.-P. Labesse.
\newblock Cohomologie d'intersection et fonctions $l$ de certaines vari\'et\'es
  de shimura.
\newblock {\em Annales scientifiques de l'\'Ecole Normale Sup\'erieure}, 4e
  s{\'e}rie, 17(3):361--412, 1984.

\bibitem[BL22]{BhattLurie}
B.~Bhatt and J.~Lurie.
\newblock Absolute prismatic cohomology, 2022.

\bibitem[BLZ16]{BLZ}
A.~Besser, D.~Loeffler, and S.~Zerbes.
\newblock Finite polynomial cohomology for general varieties.
\newblock {\em Ann. Math. Qu\'{e}.}, 40(1):203--220, 2016.

\bibitem[BMS18]{BMS}
B.~Bhatt, M.~Morrow, and P.~Scholze.
\newblock Integral $p$-adic {H}odge theory.
\newblock {\em Publ. math. IHES}, 128:219--397, 2018.

\bibitem[BS22]{Prismatic}
B.~Bhatt and P.~Scholze.
\newblock Prisms and prismatic cohomology.
\newblock {\em Ann. of Math.}, 196(3):1135--1275, 2022.

\bibitem[BSW15]{BSW}
M.~{Bhargava}, A.~{Shankar}, and X.~{Wang}.
\newblock {Geometry-of-numbers methods over global fields {I}: Prehomogeneous
  vector spaces}.
\newblock {\em ArXiv e-prints}, dec 2015.

\bibitem[CT23]{Caraiani-Tamiozzo}
A.~Caraiani and M.~Tamiozzo.
\newblock On the \'etale cohomology of {H}ilbert modular varieties with torsion
  coefficients.
\newblock {\em Comp. Math.}, 159(11):2279--2325, 2023.

\bibitem[CW77]{CoatesWiles}
J.~Coates and A.~Wiles.
\newblock On the conjecture of {B}irch and {S}winnerton-{D}yer.
\newblock {\em Invent. Math.}, 39(3):223--251, 1977.

\bibitem[Dim13]{DimitrovAutSym}
M.~Dimitrov.
\newblock Automorphic symbols, {$p$}-adic {$L$}-functions and ordinary
  cohomology of {H}ilbert modular varieties.
\newblock {\em Amer. J. Math.}, 135(4):1117--1155, 2013.

\bibitem[DR14]{DR}
H.~Darmon and V.~Rotger.
\newblock Diagonal cycles and {E}uler systems {I}: {A} {$p$}-adic
  {G}ross-{Z}agier formula.
\newblock {\em Ann. Sci. \'Ec. Norm. Sup\'er. (4)}, 47(4):779--832, 2014.

\bibitem[DR17]{DR2}
H.~Darmon and V.~Rotger.
\newblock Diagonal cycles and {E}uler systems {II}: {T}he {B}irch and
  {S}winnerton-{D}yer conjecture for {H}asse-{W}eil-{A}rtin {$L$}-functions.
\newblock {\em J. Amer. Math. Soc.}, 30(3):601--672, 2017.

\bibitem[Fal05]{FaltingsSiegel}
G~Faltings.
\newblock {Arithmetic {E}isenstein {C}lasses on the {S}iegel {S}pace: {S}ome
  {C}omputations}.
\newblock {\em van der Geer, G., Moonen, B., Schoof, R. (eds) Number Fields and
  Function Fields—Two Parallel Worlds. Progress in Mathematics}, 239, 2005.

\bibitem[Fla90]{Flach}
M.~Flach.
\newblock A generalisation of the {C}assels-{T}ate pairing.
\newblock {\em J. Reine Angew. Math.}, 412:113--127, 1990.

\bibitem[For19]{MicAnalytic}
M.~Fornea.
\newblock {Growth of the analytic rank of modular elliptic curves over quintic
  extensions}.
\newblock {\em Math. Research Letters}, 26:1571--1586, 2019.

\bibitem[GZ86]{GZformula}
B.~H. Gross and D.~B. Zagier.
\newblock Heegner points and derivatives of {$L$}-series.
\newblock {\em Invent. Math.}, 84(2):225--320, 1986.

\bibitem[Hid89a]{HidaGalois}
H.~Hida.
\newblock Nearly ordinary {H}ecke algebras and {G}alois representations of
  several variables.
\newblock In {\em Algebraic analysis, geometry, and number theory ({B}altimore,
  {MD}, 1988)}, pages 115--134. Johns Hopkins Univ. Press, Baltimore, MD, 1989.

\bibitem[Hid89b]{nearlyHida}
H.~Hida.
\newblock On nearly ordinary {H}ecke algebras for {${GL}(2)$} over totally real
  fields.
\newblock In {\em Algebraic number theory}, volume~17 of {\em Adv. Stud. Pure
  Math.}, pages 139--169. Academic Press, Boston, MA, 1989.

\bibitem[Hid91]{pHida}
H.~Hida.
\newblock On {$p$}-adic {$L$}-functions of {${GL}(2)\times {GL}(2)$} over
  totally real fields.
\newblock {\em Ann. Inst. Fourier (Grenoble)}, 41(2):311--391, 1991.

\bibitem[Ich08]{I}
A.~Ichino.
\newblock Trilinear forms and the central values of triple product
  {$L$}-functions.
\newblock {\em Duke Math. J.}, 145(2):281--307, 2008.

\bibitem[KL05]{Kisin-Lai}
M.~Kisin and F.~Lai.
\newblock Overconvergent hilbert modular forms.
\newblock {\em American Journal of Mathematics}, 127:735--783, 2005.

\bibitem[KLZ17]{KLZ}
G.~Kings, D.~Loeffler, and S.~Zerbes.
\newblock Rankin-{E}isenstein classes and explicit reciprocity laws.
\newblock {\em Camb. J. Math.}, 5(1):1--122, 2017.

\bibitem[Kol88]{Koly}
V.~A. Kolyvagin.
\newblock Finiteness of {E(Q)} and {S}h({E},{Q}) for a subclass of {W}eil
  curves.
\newblock {\em Izv. Akad. Nauk SSSR Ser. Mat.}, 52(3):522--540, 670--671, 1988.

\bibitem[Lan13]{Lan13}
K.-W. Lan.
\newblock {\em Arithmetic compactifications of {PEL}-type {S}himura varieties},
  volume~36 of {\em London Mathematical Society Monographs Series}.
\newblock Princeton University Press, Princeton, NJ, 2013.

\bibitem[Liu16]{YLiu}
Y.~Liu.
\newblock Hirzebruch-{Z}agier cycles and twisted triple product {S}elmer
  groups.
\newblock {\em Invent. Math.}, 205(3):693--780, 2016.

\bibitem[LLZ18]{LLZ}
A.~Lei, D.~Loeffler, and S.~Zerbes.
\newblock Euler systems for {H}ilbert modular surfaces.
\newblock {\em Forum Math. Sigma}, 6:e23, 67, 2018.

\bibitem[Lon06]{Longo}
M.~Longo.
\newblock On the {B}irch and {S}winnerton-{D}yer conjecture for modular
  elliptic curves over totally real fields.
\newblock {\em Ann. Inst. Fourier (Grenoble)}, 56(3):689--733, 2006.

\bibitem[LP18]{BGG}
K.-W. Lan and P.~Polo.
\newblock Dual {BGG} complexes for automorphic bundles.
\newblock {\em Math. Res. Lett.}, 25(1):85--141, 2018.

\bibitem[LZ14]{LZ14}
D.~Loeffler and S.~Zerbes.
\newblock Iwasawa theory and {$p$}-adic {$L$}-functions over
  {$\Bbb{Z}_p^2$}-extensions.
\newblock {\em Int. J. Number Theory}, 10(8):2045--2095, 2014.

\bibitem[Maz97]{Mazur}
B.~Mazur.
\newblock An introduction to the deformation theory of {G}alois
  representations.
\newblock In {\em Modular forms and {F}ermat's last theorem ({B}oston, {MA},
  1995)}, pages 243--311. Springer, New York, 1997.

\bibitem[Mil05]{Milne-SVI}
J.~S. Milne.
\newblock Introduction to {S}himura varieties.
\newblock In {\em Harmonic analysis, the trace formula, and {S}himura
  varieties}, volume~4 of {\em Clay Math. Proc.}, pages 265--378. Amer. Math.
  Soc., Providence, RI, 2005.

\bibitem[Nek00]{NekovarAbel-Jacobi}
J.~Nekov\'{a}\v{r}.
\newblock {$p$}-adic {A}bel-{J}acobi maps and {$p$}-adic heights.
\newblock In {\em The arithmetic and geometry of algebraic cycles ({B}anff,
  {AB}, 1998)}, volume~24 of {\em CRM Proc. Lecture Notes}, pages 367--379.
  Amer. Math. Soc., Providence, RI, 2000.

\bibitem[Nek12]{LevelRaisingNekovar}
J.~Nekov\'{a}\v{r}.
\newblock Level raising and anticyclotomic {S}elmer groups for {H}ilbert
  modular forms of weight two.
\newblock {\em Canad. J. Math.}, 64(3):588--668, 2012.

\bibitem[{Nek}18]{ES-Nekovar}
J.~{Nekov\'ar}.
\newblock {Eichler-Shimura relations and semisimplicity of \'etale cohomology
  of quaternionic Shimura varieties}.
\newblock {\em Ann. Sci. E.N.S.}, 51, 2018.

\bibitem[Oht95]{Ohta95}
M.~Ohta.
\newblock On the {$p$}-adic {E}ichler-{S}himura isomorphism for
  {$\Lambda$}-adic cusp forms.
\newblock {\em J. Reine Angew. Math.}, 463:49--98, 1995.

\bibitem[Pra90]{prasadtrilinear}
D.~Prasad.
\newblock Trilinear forms for representations of {${GL}(2)$} and local
  {$\epsilon$}-factors.
\newblock {\em Compositio Math.}, 75(1):1--46, 1990.

\bibitem[Pra92]{epsilonprasad}
D.~Prasad.
\newblock Invariant forms for representations of {${GL}_2$} over a local field.
\newblock {\em Amer. J. Math.}, 114(6):1317--1363, 1992.

\bibitem[PS16]{PS}
V.~Pilloni and B.~Stroh.
\newblock Surconvergence, ramification et modularit\'e.
\newblock {\em Ast\'er.}, 2016.

\bibitem[Sai09]{p-adicHodgeHilbert}
T.~Saito.
\newblock Hilbert modular forms and {$p$}-adic {H}odge theory.
\newblock {\em Compos. Math.}, 145(5):1081--1113, 2009.

\bibitem[Ski09]{SkinnerHilbert}
C.~Skinner.
\newblock A note on the {$p$}-adic {G}alois representations attached to
  {H}ilbert modular forms.
\newblock {\em Doc. Math.}, 14:241--258, 2009.

\bibitem[SS23]{SangiovanniSkinner}
M.~Sangiovanni and C.~Skinner.
\newblock Generalizations of {O}hta’s theorem for holomorphic modular forms
  on certain {PEL}-type shimura varieties.
\newblock {\em Preprint}, 2023.

\bibitem[TW95]{TW}
R.~Taylor and A.~Wiles.
\newblock Ring-theoretic properties of certain {H}ecke algebras.
\newblock {\em Ann. of Math. (2)}, 141(3):553--572, 1995.

\bibitem[TX16]{Hilbert}
Y.~Tian and L.~Xiao.
\newblock {$p$}-adic cohomology and classicality of overconvergent {H}ilbert
  modular forms.
\newblock {\em Ast\'erisque}, 2016.

\bibitem[vdG88]{Geer}
G.~van~der Geer.
\newblock {\em Hilbert modular surfaces}, volume~16 of {\em Ergebnisse der
  Mathematik und ihrer Grenzgebiete (3) [Results in Mathematics and Related
  Areas (3)]}.
\newblock Springer-Verlag, Berlin, 1988.

\bibitem[Wes05]{Weston}
T.~Weston.
\newblock Iwasawa invariants of {G}alois deformations.
\newblock {\em Manuscripta Math.}, 118(2):161--180, 2005.

\bibitem[Wil88]{Wiles-rep}
A.~Wiles.
\newblock On ordinary {$\lambda$}-adic representations associated to modular
  forms.
\newblock {\em Invent. Math.}, 94(3):529--573, 1988.

\bibitem[Wil95]{W}
A.~Wiles.
\newblock Modular elliptic curves and {F}ermat's last theorem.
\newblock {\em Ann. of Math. (2)}, 141(3):443--551, 1995.

\bibitem[Zha01]{Heights}
S.~Zhang.
\newblock Heights of {H}eegner points on {S}himura curves.
\newblock {\em Ann. of Math. (2)}, 153(1):27--147, 2001.

\end{thebibliography}
\bibliographystyle{alpha}

\end{document}